%% file: PECNF_E.tex
\documentclass[12pt,twoside]{article}
\usepackage{etex}
\usepackage[top = 30truemm,bottom = 30truemm,left = 25truemm,right = 25truemm]{geometry}
\usepackage{url,amsmath,amssymb,amsthm,enumerate,mdwlist,mathrsfs,mathtools}
\usepackage{hyperref}
\usepackage{graphicx, xcolor}
\usepackage{bxbase}
\usepackage{multicol}
\usepackage{fancybox}
\usepackage{verbatim}
\usepackage{euscript}
\usepackage[all]{xy}
\usepackage{tikz} 
\usetikzlibrary{calc}
\usetikzlibrary{fadings}
\usetikzlibrary{patterns}
\usetikzlibrary{backgrounds}
\usepackage{fancyhdr}
\makeatletter
\def\blfootnote{\gdef\@thefnmark{}\@footnotetext}
\makeatother

\renewcommand{\theequation}{\arabic{section}.\arabic{equation}} 
\numberwithin{equation}{section} 
\theoremstyle{definition}
\newtheorem{definition}{Definition}[section]
\newtheorem{setting}[definition]{Setting}

\newtheorem{remark}[definition]{Remark}

\newtheorem{example}[definition]{Example}
\theoremstyle{plain}
\newtheorem{theorem}[definition]{Theorem}
\newtheorem{mtheorem}{Theorem}

\newtheorem{proposition}[definition]{Proposition}

\newtheorem{lemma}[definition]{Lemma}

\newtheorem{corollary}[definition]{Corollary}
\newcommand{\sign}{\sigma} 
\newcommand{\super}[1]{^{(#1 )}}

\newcommand{\calB}{\mathcal{B}}
\newcommand{\calC}{\mathcal{C}}
\newcommand{\CC}{\mathbb{C}}
\newcommand{\rd}{\mathrm{d}}
\newcommand{\DD}{\mathcal{D}}
\newcommand{\ee}{\boldsymbol{\epsilon}}
\newcommand{\E}{\mathbf{e}}
\newcommand{\Ea}{\mathbf{e}^{\mathfrak a}}
\newcommand{\EE}{\mathbb{E}}
\newcommand{\FF}{\mathbb{F}}
\newcommand{\HH}{\mathscr{H}}

\newcommand{\calL}{\mathcal{L}}
\newcommand{\LL}{\mathbb{L}}
\newcommand{\calM}{\mathcal{M}}
\newcommand{\calQ}{\mathcal{Q}}
\newcommand{\calT}{\mathcal{T}}
\newcommand{\NN}{\mathbb{N}}
\newcommand{\Nrm}{\mathbf{N}}
\newcommand{\Nelm}{N_{K/\mathbb Q}}
\newcommand{\scrN}{\mathscr{N}}
\newcommand{\scrQ}{\mathscr{Q}}
\newcommand{\OO}{\mathcal{O}}
\newcommand{\Or}{\mathcal O} 
\newcommand{\rmI}{\mathrm{I}}
\newcommand{\rmII}{\mathrm{II}}
\newcommand{\conductor}{f}

\newcommand{\cond}{\conductor}
\newcommand{\If}{I_K^\cond }
\newcommand{\OK}{\mathcal{O}_K}
\newcommand{\OKnz}{\mathcal{O}_K\setminus \{0\}}
\newcommand{\OKt}{\mathcal{O}_K^{\times}}
\newcommand{\OKbar}{\overline{\mathcal{O}_K^{\times}}}

\newcommand{\PP}{\mathcal{P}}
\newcommand{\PI}{\mathrm{PI}}
\newcommand{\Pavs}{\PP_{\ideala;\tau,s}}
\newcommand{\QQ}{\mathbb{Q}}
\newcommand{\calR}{\mathcal{R}}
\newcommand{\RR}{\mathbb{R}}
\newcommand{\eS}{\EuScript{S}}
\newcommand{\eT}{\EuScript{T}}
\newcommand{\ZZ}{\mathbb{Z}}
\newcommand{\calZ}{\mathcal{Z}}

\newcommand{\vph}{\varphi}
\newcommand{\bv}{\boldsymbol{v}}
\newcommand{\bw}{\boldsymbol{w}}
\newcommand{\bz}{\boldsymbol{z}}
\newcommand{\bflog}{\mathbf{log}}
\newcommand{\omom}{\boldsymbol{\omega}}
\newcommand{\Ideals}{\mathrm{Ideals}}
\newcommand{\Cl}{\mathrm{Cl}} 
\newcommand{\pr}{\mathrm{pr}}
\newcommand{\lmugen}{\ell_{\infty}}
\newcommand{\Mobius}{M{\"o}bius\ }
\newcommand{\Sz}{Szemer{\' e}di}

\newcommand{\RMST}{the relative multidimensional Szemer{\'e}di theorem}

\newcommand{\RCL}{relative counting lemma}
\newcommand{\Cheb}{Chebotarev}
\newcommand{\chebden}{the \Cheb\ density theorem}

\newcommand{\counting}{count}
\newcommand{\havethat}{have}
\newcommand{\obtainthat}{obtain}
\newcommand{\naturaldensityversionofthe}{}

\newcommand{\Kfp}{K_{\cond +}^{\times}}
\newcommand{\Clf}{\mathrm{Cl}_K^f}
\newcommand{\hcond}{h_\cond }
\newcommand{\SpecOK}{|\Spec \OK |}

\newcommand{\OKf}{\mathcal O_{K,\cond +}^{\times}}
\newcommand{\ideala}{\mathfrak a}
\newcommand{\idealb}{\mathfrak b}
\newcommand{\idealc}{\mathfrak c}
\newcommand{\ideald}{\mathfrak d}
\newcommand{\idealp}{\mathfrak p}
\newcommand{\pideala}{\alpha}
\newcommand{\pidealb}{\beta}
\newcommand{\pidealc}{\gamma}
\newcommand{\pidealp}{\mathfrak{p}}
\newcommand{\ppart}{^{(p)} }
\newcommand{\lpart}{^{(\ell)}}
\newcommand{\kyo}{\sqrt{-1}}
\newcommand{\chihat}{\widehat{\chi}}
\newcommand{\uxi}{\underline{\xi}}
\newcommand{\ueta}{\underline{\eta}}
\newcommand{\Fourier}{\mathcal{F}^{\ast}}
\renewcommand{\Re}{\mathrm{Re}}
\renewcommand{\Im}{\mathrm{Im}}
\renewcommand{\ker}{\mathrm{ker}}
\newcommand{\ichi}{\mathbf 1}
\newcommand{\tide}{\tilde{\delta}}

\newcommand{\idealae}{\idealc \OK }

\newcommand{\signF}{\epsilon _{\idealc }} 

\newcommand{\thering}{\mathcal O}

\newcommand{\idealq}{\mathfrak q}
\newcommand{\ClOr}{\Cl ^+(\Or)}
\newcommand{\signs}{\{ \pm 1 \}}
\newcommand{\aE}{\tilde{E}}
\newcommand{\uvarsigma}{u} 
\newcommand{\etaUpsilon}{\eta} 
\newcommand{\RRCC}{\RR ^{r_1}\times \CC ^{r_2}}
\DeclareMathOperator{\Aff}{Aff}
\DeclareMathOperator{\coker}{coker}
\DeclareMathOperator{\Hom}{Hom}
\DeclareMathOperator{\Spec}{Spec}
\DeclareMathOperator{\Li}{Li}
\DeclareMathOperator{\sgn}{sgn}
\DeclareMathOperator{\rank}{rank}
\DeclareMathOperator{\supp}{supp}

\DeclareMathOperator{\lebesgue}{Leb}

\newcommand{\relmiddle}[1]{\mathrel{}\middle#1\mathrel{}}
\newcommand{\compati}{$(\rho,M,\bv,S)$-condition}
\newcommand{\compatiW}{$(\rho,W,M,\bv,S)$-condition}
\newcommand{\compatia}{$(\rho,\uvarsigma,M,\bv,S,a)$-condition}
\newcommand{\compatiaW}{$(\rho,W,\uvarsigma,M,\bv,S,a)$-condition}

\newcommand{\logpseu}{\mathrm{S}\Psi_{\log}(\ideala)}
\newcommand{\logpseua}{\mathrm{S}\Psi_{\log}^{\mathrm{SI}}(\ideala)}
\newcommand{\logpseuOK}{\mathrm{S}\Psi_{\log}(\OK)}
\newcommand{\logpseuaOK}{\mathrm{S}\Psi_{\log}^{\mathrm{SI}}(\OK)}

\title{Constellations in prime elements of number fields} 
\author{Wataru Kai, Masato Mimura, Akihiro Munemasa, \\ Shin-ichiro Seki, Kiyoto Yoshino}
\date{}
\begin{document}
\maketitle
\begin{center}
\includegraphics{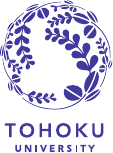} 
\end{center}
\begin{abstract}
Given \emph{any} number field, we prove that there exist arbitrarily shaped constellations consisting of pairwise non-associate prime elements of the ring of integers.
This result extends the celebrated Green--Tao theorem on arithmetic progressions of rational primes and Tao's theorem on constellations of Gaussian primes.
Furthermore, we prove a constellation theorem on prime representations of binary quadratic forms with integer coefficients.
More precisely, for a non-degenerate primitive binary quadratic form $F$ which is not negative definite, there exist arbitrarily shaped constellations consisting of pairs of integers $(x,y)$ for which $F(x,y)$ is a rational prime.
The latter theorem is obtained by extending the framework from the ring of integers to the pair of an order and its invertible fractional ideal.
\end{abstract}
\setcounter{tocdepth}{3}
\input{chapter1}
\input{acknowledgement}
\tableofcontents
\input{chapter2}
\input{chapter3}
\input{chapter4}
\input{chapter5}
\input{chapter6}
\input{chapter7}

\input{chapter8}
\input{chapter9}
\input{chapter10}
\appendix
\input{appendix_quadratic}

\bibliographystyle{amsalpha}
\bibliography{PECNF_E.bib}
\input{address}
\end{document}

%% file: chapter1.tex
\section{Introduction}\label{section=introduction}
The following theorem %
of Green and Tao is a monumental work %
in additive number theory.
\begin{theorem}[{The Green--Tao theorem \cite{Green-Tao08}}]\label{theorem=Green-Tao}
There exist arithmetic progressions of primes of arbitrary length.
\end{theorem}
In order to consider multidimensional generalizations of this result, we introduce a terminology of \emph{constellations}.
For a finite subset $S$ of a $\ZZ$-module $\calZ$ (we will consider only a free module of finite rank), we call a set of the form $\alpha+kS\coloneqq\{\alpha+ks : s\in S\}$ a \emph{constellation with the shape $S$}.
Here, $\alpha$ is an element of $\calZ$ and $k$ is a positive integer.
In this paper, we abbreviate it as an \emph{$S$-constellation}; it is also known as a \emph{homothetic copy} of $S$.
In the literature, some %
authors allow $k$ to be a negative integer.
When a subset $A$ of $\calZ$ contains an $S$-constellation for any finite subset $S$ of $\calZ$, we say that ``there exist constellations of arbitrary shape in $\calZ$ consisting of elements of $A$'' or ``the constellation theorem holds for $A$.''
Note that for a subset $A\subseteq\ZZ$,  the existence of arithmetic progressions of arbitrary length is equivalent to that of constellations of arbitrary shape.

As the Gaussian counterpart of the Green--Tao theorem, Tao established the following.
\begin{theorem}[Constellation theorem in the Gaussian primes \cite{Tao06Gaussian}]\label{theorem=Gaussian-primes}
There exist constellations of arbitrary shape in the ring of Gaussian integers $\ZZ[\sqrt{-1}]$ consisting of Gaussian primes.
\end{theorem}
In the same paper, Tao \cite[12 Discussion]{Tao06Gaussian} conjectured that Theorem~\ref{theorem=Green-Tao} could be extended in the following two ways. 
\begin{enumerate}[(1)]
\item Extension to other number fields (= the constellation theorem in the prime elements of an arbitrary number field).\label{conj=MAIN}
\item Extension to a relatively dense subset $A$ of the direct product $\PP^n$ of the set of primes (= the multidimensional Szemer\'edi theorem holds in the primes).
More precisely, there exist constellations of arbitrary shape in $\ZZ^n$ consisting of elements of $A$.\label{conj=PMST}
\end{enumerate}
Recently, the second conjecture has been settled independently by three research groups, Tao--Ziegler \cite{Tao-Ziegler15}, Fox--Zhao \cite{Fox-Zhao15} and Cook--Magyar--Titichetrakun \cite{Cook-Magyar-Titichetrakun18}.

In this paper, we resolve the first conjecture in the affirmative.
The main theorem in this paper is stated in its simplest form as follows; various refined statements will be introduced in Section~\ref{section=organization}.
We denote by $\OK$ the ring of integers of a number field $K$.

\begin{theorem}[Constellation theorem in the prime elements of a number field]\label{theorem=primeconstellationssimple}
Let $K$ be a number field.
Then there exist constellations of arbitrary shape in $\OK$ consisting of prime elements of $\OK$.
\end{theorem}
In the above theorem, the statement for the case of $K=\QQ$ is equivalent to the Green--Tao theorem (Theorem~\ref{theorem=Green-Tao}); that for the case of $K=\QQ(\kyo)$ is exactly the constellation theorem in the Gaussian primes (Theorem~\ref{theorem=Gaussian-primes}).
The definitions and facts on number fields that appear in this section are summarized in Section~\ref{section=preliminarynumbertheory}.

Tao remarks in \cite[12 Discussion]{Tao06Gaussian} that his method of proving Theorem~\ref{theorem=Gaussian-primes} is likely to extend to $K$ at least if the class number of $K$ is $1$ and the unit group $\OKt$ is finite.
There exist only nine such number fields other than $\QQ$ and all of them are imaginary quadratic by Dirichlet's unit theorem and the Baker--Heegner--Stark theorem; see for instance, \cite{Stark67}.
For a general number field, the class number may be strictly greater than $1$ or the unit group may be infinite.
Both of these two cause problems for formulating an appropriate statement of generalizations of Theorem~\ref{theorem=Gaussian-primes}.

If the class number is greater than $1$, then not all irreducible elements are prime elements.
As already mentioned in Theorem~\ref{theorem=primeconstellationssimple}, the prime elements suffice to guarantee the existence of constellations of arbitrary shape. 
The unit group acts on the set of prime elements by multiplication.
In Corollary~\ref{corollary=primeconstellationssimple}, we strengthen Theorem~\ref{theorem=primeconstellationssimple} by showing the existence of constellations consisting of primes from distinct orbits.
The original method of Tao \cite{Tao06Gaussian} does not extend to the proof of Theorem~\ref{theorem=primeconstellationssimple}%
in a straightforward manner due to the above two obstacles.

We will describe the difficulties in the latter part of this introduction.

In this paper, for a number field $K$, we denote by $\PP_K$ the set of all prime elements of the ring of integers $\OK$.
We employ some concepts in order to extend conjecture~\eqref{conj=MAIN} to a `Szemer\'edi-type' theorem and to refine the statement in the case that the unit group is infinite.
For an integral basis $\omom$ of $K$, we denote by $\|\cdot\|_{\infty,\omom}\colon \OK\to \ZZ_{\geq 0}$ the \emph{$\lmugen$-length} with respect to the basis $\omom$; see Definition~\ref{definition=lmugenlength}.
For a non-empty set $X\subseteq \OK$ and its subset $A\subseteq X$, we define the \emph{relative upper asymptotic density measured by the $\lmugen$-length $\|\cdot\|_{\infty,\omom}$} of $A$ in $X$ by
\[
 \overline{d}_{X,\omom}(A)\coloneqq\limsup_{M\to\infty}\frac{\#(A\cap \OO_K(\omom,M))}{\#(X \cap \OO_K(\omom,M))},
\]
where $\OO_K(\omom,M)\coloneqq\{\alpha\in \OK : \|\alpha\|_{\infty,\omom}\leq M\}$.
We say that two elements of $\OKnz$ are \emph{associate} if they lie in the same orbit for the action $\OKt \curvearrowright \OKnz$ by multiplication.
We call a two-point subset $\{\alpha,\beta\}$ of $\OKnz$ an \emph{associate pair} if $\alpha$ and $\beta$ are associate.
The following theorem is a strengthening of Theorem~\ref{theorem=primeconstellationssimple}.
\begin{theorem}[Szemer\'edi-type theorem in the prime elements of a number field]\label{theorem=primeconstellationsdensesemiprecise}
Let $K$ be a number field and $\omom$ an integral basis of $K$.
Assume that a subset $A$ of $\PP_K$ has a positive relative upper asymptotic density measured by $\|\cdot\|_{\infty,\omom}$ in $\PP_K$, namely, $ \overline{d}_{\PP_K,\omom}(A)>0$.
Then there exist constellations of arbitrary shape in $\OK$ consisting only of elements of $A$ without associate pairs.
\end{theorem}
In Subsection~\ref{subsection=mainpreciseversion}, we state Theorem~\ref{mtheorem=primeconstellationsfinite}, which may be seen as a version of Theorem~\ref{theorem=primeconstellationsdensesemiprecise} in a finitary setting.
Theorem~\ref{mtheorem=primeconstellationsfinite} is the first main theorem of the present paper.
As a corollary to Theorem~\ref{theorem=primeconstellationsdensesemiprecise}, we obtain the following.
\begin{corollary}\label{corollary=primeconstellationssimple}
In the statement of Theorem~$\ref{theorem=primeconstellationssimple}$, we can take constellations that do not admit associate pairs.
\end{corollary}
If the unit group is finite, then we see that Theorem~\ref{theorem=primeconstellationssimple} and Corollary~\ref{corollary=primeconstellationssimple} are equivalent in a simple argument using the pigeonhole principle.
On the other hand, if the unit group is infinite, then Corollary~\ref{corollary=primeconstellationssimple} seems stronger than Theorem~\ref{theorem=primeconstellationssimple}.
Although Corollary~\ref{corollary=primeconstellationssimple} is derived immediately from Theorem~\ref{theorem=primeconstellationsdensesemiprecise}, we prove it prior to Theorem~\ref{theorem=primeconstellationsdensesemiprecise}.
More precisely, we prove Corollary~\ref{corollary=primeconstellationssimple} by using Theorem~\ref{theorem=primeconstellationsfinite} and the existence of a `good' fundamental domain (Section~\ref{section=NLC}); see Subsection~\ref{subsection=proofofmaintheorem}.

Next, we briefly discuss the technical problems of the proofs in the case of general number fields.
Recall that the class number of $K$ can be greater than $1$, in which case prime element factorization in $\OK $ fails.
From this viewpoint, it may be said that prime elements are `few.'
One of the key steps to the proofs of Theorem~\ref{theorem=Green-Tao} and Theorem~\ref{theorem=Gaussian-primes} is to prove \emph{Goldston--Y\i ld\i r\i m type asymptotic formulas}; these are used to confirm the hypotheses of a relative version of the multidimensional Szemer\'edi theorem.
Since the proofs of Goldston--Y\i ld\i r\i m type asymptotic formulas involve the existence and uniqueness of factorizations, it is a non-trivial problem to extend the proof to the case where the class number of $K$ is not $1$.
In the work of Green--Tao and Tao, they consider some variants of the von Mangoldt function to obtain Goldston--Y\i ld\i r\i m type asymptotic formulas. 
However, if the unit group $\OKt$ is infinite, naive generalizations of their variants do \emph{not} make any sense; in their summations, an element would be summed for infinitely many times.

To address the two difficulties above, we switch the framework from that of elements in $\OK$ to that of \emph{ideals}.
This is the standard approach in algebraic number theory since Dedekind.
It also enables us to treat our problems of all number fields \emph{in a unified manner}.

The role of the prime number theorem in the case of $\ZZ $ is played by the \emph{Chebotarev density theorem}; it asserts that principal prime ideals account for a certain proportion in prime ideals.
From this viewpoint, the prime elements are `not too few.'
If we take a fundamental domain for $\OKt\curvearrowright\OK\setminus\{0\}$,
then each prime element in this domain exactly corresponds to each (non-zero) principal prime ideal.
We need to count prime elements with respect to $\lmugen$-length, while prime ideals are counted with respect to (ideal) norms.
To connect these two scales, we introduce the notion of %
\emph{norm-length compatibility} (\emph{NL-compatibility} for short) of fundamental domains;
We will have
a desired estimate of numbers of prime elements measured by $\lmugen$-length in an NL-compatible fundamental domain. Then the \emph{relative multidimensional \Sz \ theorem} applies, and we establish our constellation theorem for this domain.
Next, we prove Theorem~\ref{theorem=primeconstellationsdensesemiprecise}, whose statement does not involve fundamental domains.
For the proof, we will establish
a certain \emph{reduction theorem} of this case to the case with a fundamental domain; see Theorem~\ref{theorem=no_DD_to_with_DD}.
The reduction theorem is proved with the aid of the geometry of numbers.

On the full resolution of the conjecture \eqref{conj=MAIN} in Tao's paper \cite{Tao06Gaussian}, the main novel points are summarized as follows.
\begin{itemize}
\item{[Pseudorandom part]} We formulate the Goldston--Y\i ld\i r\i m type asymptotic formula (Theorem~\ref{Th:Goldston_Yildirim}) by focusing on ideals of $\OK$ instead of elements of $\OK$.
\item{[Counting part]} We employ an NL-compatible fundamental domain for counting of prime elements.
Then we reduce a general case to this setting.
\end{itemize}
In this manner, we can treat the case where the class number is greater than $1$ or the unit group is infinite.

In the last part of this section, we describe an application to binary quadratic forms with integer coefficients, 
which is obtained as a corollary to refinements of our theorems for quadratic fields.
We say that $F\colon\ZZ^2\to\ZZ$ is a \emph{primitive $($binary$)$ quadratic form} if it is of the form $F(x,y)=ax^2+bxy+cy^2$, where $a,b,c$ are integers with $\mathrm{gcd}(a,b,c)=1$.
A fundamental problem in number theory asks which primes, or $-1$ multiples of them, are represented by $F$.
In this paper, motivated by this problem, we obtain a combinatorial theorem for pairs $(x,y)\in \ZZ^2$ satisfying $F(x,y)\in\PP_{\QQ}$.
The detailed statement is presented as Theorem~\ref{mtheorem=quadraticform} in Subsection~\ref{subsection=mainpreciseversion}; the following theorem is a simplified version of it.
The discriminant $D_F$ of $F$ is defined by $D_F\coloneqq b^2-4ac$.
\begin{theorem}[Constellation theorem on prime representations of binary quadratic forms]\label{theorem=quadraticform}
Let $F\colon \ZZ^2\to \ZZ$ be a primitive quadratic form.
Assume that its discriminant $D_F$ is not a perfect square and that $F$ is not negative definite.
Then, for a given finite set $S\subseteq \ZZ^2$, there exists an $S$-constellation $\eS$ in $\ZZ^2$ such that the function $F(x,y)$ takes distinct prime values on $\eS$.
\end{theorem}
The above condition on $D_F$ is necessary.
Indeed, if $D_F$ is a perfect square, then $F(x,y)$ is not irreducible over $\ZZ$.
If $F$ is indefinite, the above theorem also implies the existence of an $S$-constellation on which $F(x,y)$ takes distinct negative prime values.
In order to prove Theorem~\ref{theorem=quadraticform} for general coefficients $(a,b,c)$, we extend the framework of our constellation theorem.
More precisely, we consider a pair $(\OO, \idealc)$, where $\OO$ is an order in $K$ and $\idealc$ is an invertible fractional  ideal of $\OO$.
The original case is where $\OO$ and $\idealc$ both equal $\OK$.

%% file: acknowledgement.tex
\subsection*{Acknowledgments}
The authors are grateful to Seiichi Azuma, Toshiki Matsusaka, Kota Saito, Keiju Sono and Yuta Suzuki for discussions. 
Wataru Kai is supported in part by JSPS KAKENHI grant number JP18K13382.
Masato Mimura is supported in part by JSPS KAKENHI grant number JP17H04822 and 21K03241.
Akihiro Munemasa is supported in part by JSPS KAKENHI grant number JP20K03527.
Shin-ichiro Seki is supported in part by JSPS KAKENHI grant number JP18J00151 and JP21K13762.
Kiyoto Yoshino is supported by JSPS KAKENHI grant number JP21J14427.

%% file: chapter2.tex
\section{Precise statements of main theorems and the outline of the proofs}\label{section=organization}
In this section, we state three main theorems in the present paper.
Theorem~\ref{mtheorem=primeconstellationsfinite} is a finitary version of Theorem~\ref{theorem=primeconstellationsdensesemiprecise}.
Theorem~\ref{mtheorem=TaoZieglergeneral} is a short interval version of Theorem~\ref{mtheorem=primeconstellationsfinite}.
Theorem~\ref{mtheorem=quadraticform} is a precise version of Theorem~\ref{theorem=quadraticform}.

In Subsection~\ref{subsection=ideasofproof}, we give 
an overview of the proofs of our constellation theorems.
This mainly follows the ingenious method of Green--Tao and Tao; we simplify 
some detail, which includes the use of a recent result of Conlon--Fox--Zhao.
In our main argument, we construct a fundamental diagram among `five worlds'; see Subsection~\ref{subsection=ideasofproof} for details.
\subsection{Main theorems}\label{subsection=mainpreciseversion}
In this subsection, we present the statements of our main theorems, Theorems~\ref{mtheorem=primeconstellationsfinite}, \ref{mtheorem=TaoZieglergeneral} and \ref{mtheorem=quadraticform}.
First, we define the (\emph{ideal}) \emph{norm} of a non-zero element $\alpha$ of $\OK$ for a number field $K$; see also Remark~\ref{remark=idealnorm}.
As mentioned in Introduction, the \emph{$\lmugen$-length} of $\alpha\in\OK$ is defined for a fixed integral basis $\omom$ of $K$.
Here we state the exact definitions of them.
These are two distinct scales on $\OK$ used in this paper. 
\begin{definition}[Norm]\label{definition=norm}
Let $K$ be a number field of degree $n$.
For $\alpha\in\OKnz$, we define the \emph{norm} $\Nrm(\alpha)$ by
\[
\Nrm(\alpha)\coloneqq\#\left(\OO_K/\alpha \OK\right)\quad (<+\infty).
\]
For a non-negative real number $L$, we set
\[
\OK(L)\coloneqq\{\alpha\in \OKnz : \Nrm (\alpha)\leq L\}.
\]
\end{definition}
\begin{definition}[$\lmugen$-length]\label{definition=lmugenlength}
Let $\calZ$ be a free $\ZZ$-module of rank $n\in\NN$.
Let $\bv=(v_1,v_2,\ldots,v_n)$ be a $\ZZ$-basis of $\calZ$.
Then, we define the {$\lmugen$-length} of $\alpha\in \calZ$ for $\bv$ by
\[
\|\alpha\|_{\infty,\bv}\coloneqq\max_{1\leq i\leq n}|a_i|,
\]
where $\alpha=\sum\limits_{1\leq i\leq n}a_iv_i$.
For a non-negative real number $M$, we set
\[
\calZ(\bv,M)\coloneqq\{\alpha\in \calZ : \|\alpha\|_{\infty,\bv}\leq M\}.
\]
\end{definition}
As is well known, the ring of integers $\OK$ of a number field $K$ of degree $n$ is a free $\ZZ$-module of rank $n$.
We use the symbol $\omom$ for its integral basis in this paper.
In particular, $\lmugen$-length $\|\cdot\|_{\infty,\omom}$ on $\OK$ and the set $\OK(\omom,M)\subseteq\OK$ are defined by Definition~\ref{definition=lmugenlength}.
Furthermore, Definition~\ref{definition=lmugenlength} applies to the case where $\calZ$ is a non-zero ideal $\ideala$ of $\OK$.
For $\alpha\in\OKnz$, its norm $\Nrm(\alpha)$ and its $\lmugen$-length $\|\alpha\|_{\infty,\omom}$ are both positive integers.
Nevertheless, we allow the parameters $L,M$ to be non-negative real numbers in the definitions of $\OO_K(L)$ and $\OO_K(\omom,M)$ above.
This is for avoiding inessential issues of integrality.

We introduce the notion of standard shapes; this is useful for estimates of the number of constellations in our main theorems.

\begin{definition}[Standard shape, the number of $S$-constellations]\label{definition=standardshape}
Let $\calZ$ be a $\ZZ$-module and $S\subseteq \calZ$ a finite set. 
\begin{enumerate}[$(1)$]
\item\label{en:standard_shape} The set $S$ is called a \emph{standard shape} if the following hold: $0\in S$, $S=-S$, and $S$ generates $\calZ$ as a $\ZZ$-module. 
\item\label{en:kosuu} Assume that $S\ne\varnothing$.
For a finite subset $X\subseteq\calZ$, $\scrN_S(X)$ denotes the number of distinct $S$-constellations in $X$.
\end{enumerate}
\end{definition}
Let $\ideala$ be a non-zero ideal of $\OK$.
Then, for $\alpha,\beta \in \ideala \setminus \{0\}$ we say that they are \emph{associate} if these elements lie in the same orbit of the action $\OKt\curvearrowright \ideala\setminus \{0\}$ by multiplication.
For $A\subseteq \ideala \setminus \{0\}$, we define an \emph{associate pair} in $A$ to be a two-point subset $\{\alpha,\beta\}$ of $A$ consisting of associate elements.
We have already defined this concept in Introduction in the case where $\ideala=\OK$.
In this paper, we study the existence of constellations \emph{without associate pairs}.
Furthermore, if the shape $S$ is standard, then we evaluate the number of $S$-constellations without associate pairs.
\begin{definition}\label{definition=hidohan_seiza}
Let $K$ be a number field and $\ideala$ a non-zero ideal of $\OK$.
Let $S$ be a non-empty finite subset of $\OK$.
For a finite set $X\subseteq\ideala$, $\scrN_S^{\sharp}(X)$ denotes the number of $S$-constellations in $X$  without associate pairs.
\end{definition}
Now we exhibit our first main theorem.
This is a finitary version of Theorem~\ref{theorem=primeconstellationsdensesemiprecise} in Introduction.
\begin{mtheorem}[Szemer\'edi-type theorem in the prime elements of a number field: finitary version]\label{mtheorem=primeconstellationsfinite}
Let $K$ be a number field and $\omom$ an integral basis of $K$.
Let $\delta$ be a positive real number and $S$ a finite subset of $\OK$.
Then there exists a positive integer $M_0$ depending on $\omom,\delta$ and $S$ such that the following holds true.
\begin{enumerate}[$(1)$]
\item\label{en:TheoremA_seiza}
If $M\geq M_0$ and a subset $A$ of $\PP_K\cap\OK(\omom,M)$ satisfies
\begin{equation}\label{eq:condition_for_A_thmA}
\#A\geq\delta\cdot\#(\PP_K \cap\OK(\omom,M)),
\end{equation}
then there exists an $S$-constellation in $A$ without associate pairs. 
\item \label{en:TheoremA_seiza_kosuu}
If $S$ is a standard shape, then there exists a constant $\gamma>0$ depending on $\omom,\delta$ and $S$ such that in \eqref{en:TheoremA_seiza},
\[
	\scrN_{S}^{\sharp}(A)\geq\gamma\cdot \frac{M^{n+1}}{(\log M)^{\#S}}
\]
holds true. Here, $n$ is the degree of $K$.
\end{enumerate}
\end{mtheorem}
On the estimate of the number of $S$-constellations in the main theorem, Theorem~\ref{mtheorem=primeconstellationsfinite}~\eqref{en:TheoremA_seiza_kosuu}, it seems natural that the negative power of $\log$ appears 
in view of 
the Chebotarev density theorem (Theorem~\ref{theorem=Chebotarev}~\eqref{Chebotarev}).
For a given finite set $S\subseteq\OK$, we can construct a standard shape by inflating $S$ as follows:
add some elements of a basis of $\OK$ to $S$ if $S$ does not generate $\OK$.
Let $\eS$ be the resulting set, and consider $\eS\cup (-\eS) \cup \{0\}$. 
Note that for some $S$, the inflating process above may be done in a better manner.
For this reason, the assumption on $S$ in Theorem~\ref{mtheorem=primeconstellationsfinite}~\eqref{en:TheoremA_seiza_kosuu} does not lose its generality.
In general, for a finite set $X\subseteq \OKnz$, the inequality $\scrN_{S}^{\sharp}(X)\leq \scrN_{S}(X)$ holds.
Hence, we have also a lower bound of $\scrN_{S}(A)$ by Theorem~\ref{mtheorem=primeconstellationsfinite}~\eqref{en:TheoremA_seiza_kosuu}.

We use the terms `finitary versions' and `infinitary versions' in the following standard manner:
a statement of the existence of constellations in a certain subset of the set of the form $\calZ(\bv,M)$ for sufficiently large $M$ is called a `finitary' statement.
A statement on the existence of constellations in a subset $A$ of a certain subset $X$ of $\calZ$ where $A$ has a positive relative upper asymptotic density in $X$ is called
an `infinitary' one. 
Theorem~\ref{mtheorem=primeconstellationsfinite} is an example of the former; Theorem~\ref{theorem=primeconstellationsdensesemiprecise} is one of the latter.

Our second main theorem may be seen as a `short interval version' of Theorem~\ref{mtheorem=primeconstellationsfinite}.
\begin{definition}[$\lmugen$-interval]\label{definition=shortinterval}
Let $\calZ$ be a free $\ZZ$-module of finite rank and $\bv$ its $\ZZ$-basis.
For $x\in\calZ$ and a positive real number $M$, the \emph{$\lmugen$-interval $\calZ(\bv,x,M)$} is defined to be
\[
\calZ(\bv,x,M)\coloneqq\{\alpha\in\calZ : \|\alpha-x\|_{\infty,\bv}\leq M\}.
\]
\end{definition}
\begin{mtheorem}[A short interval version of the Szemer\'edi-type theorem in the prime elements of a number field: finitary version = Theorem~\ref{theorem=shortinterval}]\label{mtheorem=TaoZieglergeneral}
Let $K$ be a number field and $\omom$ an integral basis of $K$.
Let $\delta$ be a positive real number and $S$ a finite subset of $\OK$.
Take a real number $a$ with $0<a<1$. 
Then there exist a positive integer $M_0$ depending on $\omom,\delta,S$ and $a$, and a positive real number $\etaUpsilon >0$ depending only on $\omom$ and $\delta$ such that the following hold true.
\begin{enumerate}[$(1)$]
\item\label{en:tan_seiza_thmB}  If $M\geq M_0$ and a subset $A$ of $\PP_K\cap\OO_K(\omom,M)$ satisfies
\begin{equation}\label{eq:thmB_cond_for_A}
\#A\geq\delta\cdot\#(\PP_K \cap\OO_K(\omom,M)),
\end{equation}
then there exists $x\in A$ with
\begin{equation}\label{eq=condition_for_x_in_short_interval_thmB}
	 \etaUpsilon M \leq \|x\|_{\infty,\omom} \leq M
\end{equation}
such that $A\cap\OK(\omom,x,\|x\|_{\infty,\omom}^a)$
contains an $S$-constellation without associate pairs.
\item\label{en:tan_seiza_kosuu_thmB} If $S$ is a standard shape, then there exists a constant $\gamma>0$ depending on $\omom,\delta,S$ and $a$ such that the following holds: in \eqref{en:tan_seiza_thmB}, we can take $x$ in such a way that
\[
\mathscr{N}_S^{\sharp}(A \cap \OO_K(\omom,x,\|x\|_{\infty,\omom}^a))\geq \gamma \cdot  \frac{M^{a(n+1)}}{(\log M)^{\#S}}
\]
holds. Here, $n$ is the degree of $K$.
\end{enumerate}
\end{mtheorem}
An infinitary version of Theorem~\ref{mtheorem=TaoZieglergeneral} will be stated as Corollary~\ref{corollary=shortinterval} in Section~\ref{section=slidetrick}.
The case $K=\QQ$ of Corollary~\ref{corollary=shortinterval} is written in \cite[Remark~2.4]{Tao-Ziegler08} in more general `polynomial progression' setting; see Remark~\ref{remark=TaoZiegler}.

In the last part of this subsection, we state a precise version of Theorem~\ref{theorem=quadraticform} as an application to binary quadratic forms with integer coefficients.
Let $F(x,y)=ax^2+bxy+cy^2\in \ZZ[x,y]$ be  a quadratic form and $D_F=b^2-4ac$ its discriminant.
Assume that $D_F$ is not a perfect square.
By multiplying $-1$ if necessary, we may assume that $a>0$.
If $D_F<0$, then $F$ is positive definite and if $D_F>0$, then $F$ is indefinite.
In this paper, $\PP=\{2,3,5,7,11,\ldots\}$ denotes the set of \emph{positive} rational prime numbers. We consider the set $F^{-1}(\PP)$ $($respectively, $F^{-1}(-\PP)$$)$ of elements $(x,y)$ at which the value of $F$ $($respectively, $-F$$)$ is a prime number:
\[
F^{-1}(\PP)= \{(x,y)\in \ZZ^2 : F(x,y)\in \PP\},\quad\textrm{and} \quad  F^{-1}(-\PP)=\{(x,y)\in \ZZ^2 : -F(x,y)\in \PP\}.
\]

The following is our third main theorem.
\begin{mtheorem}[Szemer\'edi-type  theorem on prime representations of binary quadratic forms
]\label{mtheorem=quadraticform}
Let $F(x,y)\coloneqq ax^2+bxy+cy^2 \in \ZZ[x,y]$ be a  primitive quadratic form whose discriminant $D_F$ is not a perfect square.
Assume that $a>0$.
Let $\boldsymbol{u}$ be the standard basis of $\ZZ^2$.
\begin{enumerate}[$(1)$]
\item\label{en:positiveprime} Let $A\subseteq F^{-1}(\PP)$ be a set which has a positive relative upper asymptotic density measured by $\|\cdot\|_{\infty,\boldsymbol{u}}$ in $F^{-1}(\PP)$, that means
\[
\overline{d}_{F^{-1}(\PP),\boldsymbol{u}}(A)\coloneqq \limsup_{M\to \infty}\frac{\#(A\cap \ZZ^2(\boldsymbol{u},M))}{\# (F^{-1}(\PP)\cap \ZZ^2(\boldsymbol{u},M))}>0.
\]
Then, for every finite set $S\subseteq \ZZ^2$, there exists an $S$-constellation $\eS$ in $A$.
\item\label{en:negativeprime}
Assume that $D_F>0$.
Let $A\subseteq F^{-1}(-\PP)$ be a set which has a positive relative upper asymptotic density measured by $\|\cdot\|_{\infty,\boldsymbol{u}}$ in $F^{-1}(-\PP)$, that means
\[
\overline{d}_{F^{-1}(-\PP),\boldsymbol{u}}(A)\coloneqq \limsup_{M\to \infty}\frac{\#(A\cap \ZZ^2(\boldsymbol{u},M))}{\# (F^{-1}(-\PP)\cap \ZZ^2(\boldsymbol{u},M))}>0.
\]
Then, for every finite set $S\subseteq \ZZ^2$, there exists an $S$-constellation $\eS$ in $A$.
\end{enumerate}
In both \eqref{en:positiveprime} and  \eqref{en:negativeprime},
we can furthermore take $\eS$ in such a way that
the function $F(x,y)$ takes distinct prime values on $\eS$.
\end{mtheorem}

Theorem~\ref{mtheorem=quadraticform}~\eqref{en:positiveprime} says the constellation theorem holds for $F^{-1}(\PP)$; %
Theorem~\ref{mtheorem=quadraticform}~\eqref{en:negativeprime} says if moreover $D_F>0$, then the constellation theorem also holds for $F^{-1}(-\PP)$. %
See Theorem~\ref{theorem=normform} for a general statement on norm forms. 
\subsection{Constellations in prime elements inside a fundamental domain}\label{subsection=domain}
In this subsection, we state Theorem~\ref{theorem=primeconstellationsfinite} (a finitary version) and Corollary~\ref{corollary=primeconstellationsupperdense} (an infinitary version), 
restricted forms of Theorem~\ref{mtheorem=primeconstellationsfinite} and Theorem~\ref{theorem=primeconstellationsdensesemiprecise}.
Proving these theorems is the \emph{first goal of this paper}, and their proofs contain most of the new ideas
of this paper.

The difference between their settings is
whether we consider sets inside a fundamental domain for the action $\OKt\curvearrowright \OKnz$ by multiplication.
If we take a fundamental domain $\DD$, then the correspondence $\alpha\mapsto \alpha\OK$ gives a bijection from $\DD$ to the set of non-zero principal ideals and hence counting of elements is reduced to that of ideals.

We call a fundamental domain for the action $\OKt\curvearrowright \OKnz$ an \emph{$\OKt$-fundamental domain} in this paper.
We will use this terminology without referring to the action any further.
\begin{definition}[$\OKt$-fundamental domain]\label{definition=fundamentaldomain}
A set $\DD\subseteq\OKnz$ is called an \emph{$\OKt$-fundamental domain}, if $\OKnz$ is decomposed as the following disjoint union:
\[
\OKnz=\bigsqcup_{\eta\in \OKt} \eta \DD.
\]
\end{definition}
We formulate a prime element constellation theorem inside an $\OKt$-fundamental domain $\DD$ as follows: we consider $\PP_K \cap \DD$ instead of $\PP_K$, and take a subset $A$ of it.
Here we warn that if $\#(\OKt)=\infty$, then there exists an $\OKt$-fundamental domain $\DD$ such that the constellation theorem for $\PP_K\cap \DD$ does \emph{not} hold.
This fact will be proved as Proposition~\ref{proposition=badchoicedomain}.
Thus, the following question arises: ``for which $\OKt$-fundamental domain can we ensure that a constellation theorem holds?''
We answer this question by introducing the notion of \emph{NL-compatible} fundamental domains.
The NL-compatibility is defined for a subset of $\OKnz$ as follows.
\begin{definition}[NL-compatibility]\label{definition=normrespecting}
Let $K$ be a number field of degree $n$.
A set $X\subseteq\OKnz$ is \emph{NL-compatible} (\emph{norm-length compatible}) if the following condition is satisfied: there exist an integral basis $\omom$ of $K$ and a constant $C=C(\omom,X)>0$ such that, for every $\alpha\in X$,
\begin{equation}\label{eq:NL-comp}
C \|\alpha\|_{\infty,\omom}^n\leq\Nrm(\alpha)
\end{equation}
holds.
\end{definition}

The existence of a constant $C$ above is independent of the choice of an integral basis $\omom$; the exact value of $C$ depends.
Note that the opposite inequality always holds:
there exists $C'=C'(\omom)>0$ such that for all $\alpha\in \OKnz$, we have 
$ %
\Nrm (\alpha)\leq C' \|\alpha\|_{\infty,\omom}^n .
$ %
The NL-compatibility has a basis-free characterization %
in terms of the geometry of numbers. %
If the unit group $\OKt$ is infinite, then not all $\OKt$-fundamental domains are NL-compatible.
Nevertheless,
NL-compatible $\OKt$-fundamental domains always exist.
These results are shown in Section \ref{section=NLC}.
\begin{theorem}[Theorem~\ref{mtheorem=primeconstellationsfinite} restricted to an NL-compatible fundamental domain = Theorem~\ref{theorem=primeconstellationsfiniteagain}]\label{theorem=primeconstellationsfinite}
Let $K$ be a number field, $\omom$ an integral basis of $K$ and $\DD$ an NL-compatible $\OKt$-fundamental domain $($which exists by Proposition~$\ref{proposition=normrespectingfundamentaldomain})$.
Let $\delta$ be a positive number and $S$ a finite subset of $\OK$.
Then there exists a positive integer $M_0$ depending on $\omom,\DD,\delta$ and $S$ such that the following hold:
if $M\geq M_0$ and a subset $A$ of  $\PP_K\cap\DD\cap\OK(\omom,M)$ satisfies
\begin{equation}\label{eq:condition_for_A_fundamentaldomain_thm2.8}
\#A\geq\delta\cdot\#(\PP_K\cap \DD \cap\OO_K(\omom,M)),
\end{equation}
then there exists an $S$-constellation in $A$. 
Furthermore, if $S$ is a standard shape, then there exists a constant $\gamma>0$ depending only on $\omom,\DD,\delta$ and $S$ such that in the setting above,
\[
\scrN_{S}(A)\geq\gamma\cdot \frac{M^{n+1}}{(\log M)^{\#S}}
\]
holds true.
Here $n$ is the degree of $K$.

\end{theorem}
Note that, for a finite subset $A$ of an $\OKt$-fundamental domain, we have $\scrN_{S}(A)=\scrN_{S}^{\sharp}(A)$ because constellations in $A$ never admit associate pairs.

The proof of Theorem~\ref{theorem=primeconstellationsfinite} is completed in Section~\ref{section=positiveweighteddensity}.
It is used to prove Corollary~\ref{corollary=primeconstellationssimple} in Subsections~\ref{subsection=proofofmaintheorem}.
Let us deduce an infinitary consequence of Theorem~\ref{theorem=primeconstellationsfinite}.
In order to state it, we define the relative upper asymptotic density for both the norm scale and the $\lmugen$-length scale.
The latter is defined also in Introduction.
\begin{definition}\label{definition=relativedensity}
Let $X$ be a non-empty subset of $\OKnz$ and $A$ a subset of $X$.
\begin{enumerate}[$(1)$]
\item\label{en:normdensity} Assume that, for every non-negative real number $L$,  $X$ satisfies $\#(X\cap\OK(L))<\infty$.
Then the \emph{relative upper asymptotic density of $A$ measured by norm} in $X$ is defined by
\[
\overline{d}_{X}(A)\coloneqq\limsup_{L\to\infty}\frac{\#(A\cap \OO_K(L))}{\#(X \cap \OO_K(L))}.
\]
\item\label{en:lenghtdensity} Let $\omom$ be an integral basis of $K$.
Then the \emph{relative upper asymptotic density of $A$ measured by $\lmugen$-length} in $X$ is defined by
\[
\overline{d}_{X,\omom}(A)\coloneqq\limsup_{M\to\infty}\frac{\#(A\cap \OO_K(\omom,M))}{\#(X \cap \OO_K(\omom,M))}.
\]
\end{enumerate}
\end{definition}
For an $\OKt$-fundamental domain $\DD$, the set $X=\PP_K \cap \DD$ satisfies the assumption in \eqref{en:normdensity}, while $X=\PP_K$ does not if $\#(\OKt)=\infty$.
The norm scale naturally appears in algebraic number theory.
However, to prove constellation theorems, we need to convert this setting to that of the $\lmugen$-length scale.
The NL-compatibility enables this conversion. %
\begin{corollary}[Corollary~\ref{corollary=primeconstellationssimple} restricted to an NL-compatible fundamental domain]\label{corollary=primeconstellationsupperdense}
Let $K$ be a number field and $\DD$ an NL-compatible $\OKt$-fundamental domain.
Assume that a set $A\subseteq\PP_K\cap\DD$ satisfies either $\overline{d}_{\PP_K\cap \DD}(A)>0$ or $\overline{d}_{\PP_K\cap \DD,\omom}(A)>0$ for an integral basis $\omom$ of $K$.
Then there exist constellations of an arbitrary shape in $\OK$ consisting of elements of $A$.
\end{corollary}
Corollary~\ref{corollary=primeconstellationsupperdense} immediately follows from Theorem~\ref{theorem=primeconstellationsfinite}
at least if we assume the second condition $\overline{d}_{\PP_K\cap \DD,\omom}(A)>0$; the deduction is written down in full in Subsection \ref{subsection=proofofmaintheorem}.
In fact, the conditions $\overline{d}_{\PP_K\cap \DD}(A)>0$ and $\overline{d}_{\PP_K\cap \DD,\omom}(A)>0$ are equivalent; see Subsection~\ref{subsection=normlengthcomparison}.
\subsection{The idea of proofs and the organization of this paper}\label{subsection=ideasofproof}
We give an overview of the proof of our first major goal, Theorem~\ref{theorem=primeconstellationsfinite}.
We use the relative hypergraph removal lemma of Conlon--Fox--Zhao (\cite[Theorem~2.12]{Conlon-Fox-Zhao15}, Theorem~\ref{thm:RHRL}) as a black box.
Contrastingly, we do not appeal to the existing constellation theorems, Theorem~\ref{theorem=Green-Tao} or Theorem~\ref{theorem=Gaussian-primes}, in the proof.
Our proofs yield these theorems as special cases.

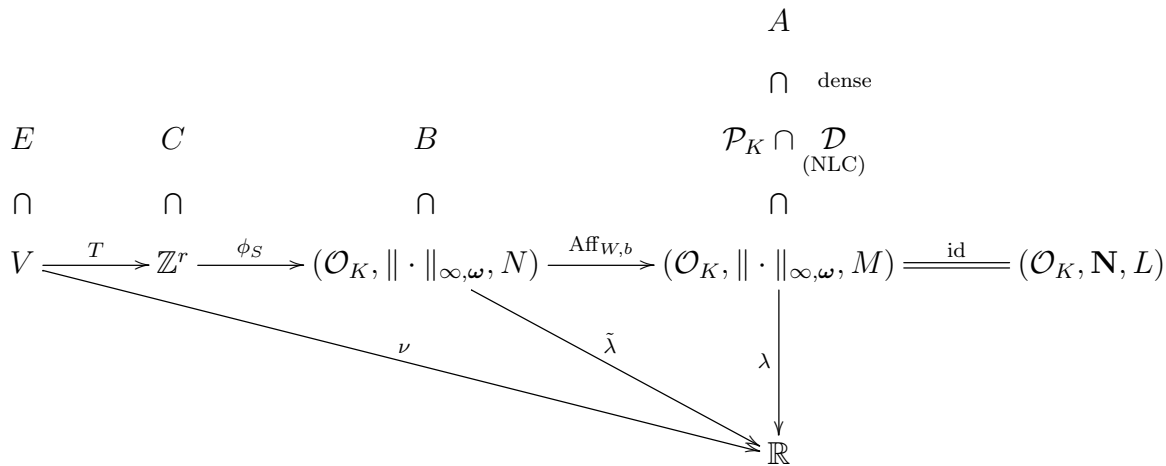
\begin{figure} 
	\xymatrix@C=40pt@R=25pt{
		& & & A \ar@{}[d]^-{\quad \text{dense}}|{\bigcap} &
		\\
		E \ar@{}[d]|{\bigcap} & C \ar@{}[d]|{\bigcap} & B \ar@{}[d]|{\bigcap} & \quad\PP_K\cap\underset{(\text{NLC})}{\DD} \ar@{}[d]|{\bigcap}  & 
		\\
		V \ar[r]^-T\ar[ddrrr]^{\nu} & \ZZ^r \ar[r]^-{\phi_S} & (\OK,\|\cdot\|_{\infty,\omom},N)\ar[r]^-{\Aff_{W,b}}\ar[ddr]^{\tilde{\lambda}} & (\OK,\|\cdot\|_{\infty,\omom},M) \ar@{=}[r]^-{\mathrm{id}}\ar[dd]_{\lambda} & (\OK, \Nrm,L) 
		\\
		{}
		\\
		& & & \RR &
	}
\caption{The fundamental diagram}\label{figure:the_fundamental_diagram}
\end{figure}

In what follows, we fix a number field $K$, its integral basis $\omom$ and a standard shape $S\subseteq\OK$.
Our goal is to find an $S$-constellation in the given set $A\subseteq \PP _K\cap \DD $. 
We will do so by going back and forth among the `five worlds' 
in the `fundamental diagram'; see Figure~\ref{figure:the_fundamental_diagram}.
Let us describe these worlds in the left-to-right order.
\begin{itemize}
\item\emph{The world of hypergraphs} $V$: $V$ is an ($r$-uniform) hypergraph system which is constructed in Subsection~\ref{subsec:psuedo-random-measure} 
following Solymosi's idea.
Each vertex in $V$ represents a hyperplane in $\ZZ^r$.
The mapping $T$, which connects two worlds $V$ and $\ZZ^r$, maps a hyperedge in $V$ to the intersection of the corresponding $r$ hyperplanes.\footnote{Actually we do not construct a single mapping $T$, but a family of mappings $T_j$ ($1\leq j\leq r+1$). We use a similar convention for $E$ and $\nu$. See Subsection~\ref{subsec:psuedo-random-measure} for details.}
\item\emph{The world of higher dimension} $\ZZ^r$: this is a free $\ZZ$-module of rank higher than (or equal to) that of $\OK\simeq\ZZ^n$.
The rank $r$ is determined by the relation $r+1=\#S$. %
Via a homomorphism $\phi _S\colon \ZZ ^r \to \OK $ associated to $S$,
the standard basis and the origin of $\ZZ ^r$ provide a canonical lift of the shape $S$.
A constellation in $\ZZ ^r$ with this shape is called a \emph{corner}.
\item\emph{The $N$-world} $(\OO_K,\|\cdot\|_{\infty,\omom},N)$: the following three worlds are all $\OK$ as sets.
In this world, we construct a pseudorandom measure $\tilde{\lambda}$ by `$W$-trick.'
We use the $\lmugen$-length scale $\|\cdot\|_{\infty,\omom}$ and the parameter $N$ to limit the scale.
\item\emph{The $M$-world} $(\OO_K,\|\cdot\|_{\infty,\omom},M)$: this is the world where the given set $A$ lives, and it is connected to the $N$-world by an affine transformation.
We use the $\lmugen$-length scale $\|\cdot\|_{\infty,\omom}$ and the parameter $M$, which will be larger than $N$ by a relatively small factor.
\item\emph{The $L$-world} $(\OK, \Nrm,L)$: this world is connected to the $M$-world by the identity map.
The difference between these two worlds is that we take the (ideal) norm scale $\Nrm (\cdot )$ in the $L$-world.
We use the parameter $L$, which will differ from $M^n$ by a constant factor.
\end{itemize}
The sketch of the proofs of Theorem~\ref{theorem=primeconstellationsfinite} and Corollary~\ref{corollary=primeconstellationsupperdense} goes as follows.

\

\noindent\textbf{Step~1}: We start the proof from the `$M$-world.'
Take an NL-compatible $\OKt$-fundamental domain $\DD$ and a relatively dense subset $A$ of $\PP_K\cap\DD$.
More precisely, the relative density of $A$ measured by $\lmugen$-length $\|\cdot\|_{\infty,\omom}$ in $\PP_K\cap\DD$ is greater than a certain positive constant.
Our goal is to prove that if the parameter $M$ is large enough, then there exists an $S$-constellation in $A\cap \OK(\omom, M)$.
Since $\PP_K$ is sparse in $\OK$, we cannot apply the classical multidimensional Szemer\'edi theorem (Theorem~\ref{theorem=multiSzemeredi}) directly to $A \cap \OK(\omom, M)$.
Instead, we aim to show that the `weighted density' of $A$ with respect to a certain weight function $\lambda$ is not small.
We define the weight function $\lambda$ by using a variant of the von Mangoldt function.
Since the norm scale $\Nrm(\cdot)$ is easier to measure the desired density than the $\lmugen$-length scale $\|\cdot\|_{\infty,\omom}$, we switch to the `$L$-world.'

\

\noindent\textbf{Step~2}:
Apply a version of the Chebotarev density theorem (Theorem~\ref{theorem=Chebotarev}~\eqref{Chebotarev}) in the `$L$-world,' and deduce the following: if $A$ is relatively dense measured by the norm $\Nrm$ in $\PP_K\cap\DD$, then the weighted density of $A$ with respect to $\lambda$ 
is greater than a certain constant.

\

\noindent\textbf{Step~3}: Then we return to the `$M$-world.'
The NL-compatibility of $\DD$ relates the density by norm to that by $\lmugen$-length.
Hence, we conclude that the weighted density of $A$ with respect to $\lambda$ measured by $\lmugen$-length is greater than a certain positive constant if the parameter $M$ is large enough.

To the best of our knowledge, in order to apply known Szemer\'edi-type theorems to this $A$, we need to confirm an extra condition on the measure $\lambda$.
More precisely, we aim to employ the relative multidimensional Szemer\'edi theorem, originating from Green--Tao~\cite{Green-Tao08}; $\lambda$ is required to be a \emph{pesudorandom measure}.
Since it is difficult to show the pseudorandomness of $\lambda$, we precompose a certain affine transformation to $\lambda$ and investigate this modified function instead of $\lambda$.
Via this affine transformation, we switch from the `$M$-world' to the `$N$-world.'

\

\noindent\textbf{Step~4}: Choose an appropriate parameter $w$, and define $W$ to be the product $W=\prod _{p\le w} p $ of prime numbers not exceeding $w$.
We perform the following `$W$-trick':
pick $b\in \OK$ according to $A$, $w$ and $M$. 
We define $B\subseteq \OK $ as the inverse image of $A$ under the affine transformation $\Aff_{W,b}; \beta\mapsto W\beta +b$.
We consider $\lambda \circ \Aff_{W,b}$ and define $\tilde{\lambda}$ 
by multiplying a normalizing factor.
Then we can prove the pseudorandomness of $\tilde{\lambda}$ by using the \emph{Goldston--Y\i ld\i r\i m type asymptotic formula} (Theorem~\ref{Th:Goldston_Yildirim}).
In its proof, the $W$-trick helps to eliminate the bias caused by small prime numbers.
In the `$N$-world,' we can show that the weighted density of $B$ with respect to the weight $\tilde{\lambda}$ is still greater than a constant.
Thus, it is possible to apply the \emph{relative multidimensional Szemer\'edi theorem} (Theorem~\ref{thm:RMST}) and obtain an $S$-constellation in $B$.
In what follows, we also describe how this latter theorem is proved. %
For this purpose, the `world of higher dimension $\ZZ^r$' shows up.

\

\noindent\textbf{Step~5}: Let $C\subseteq \ZZ ^r$ be the inverse image of $B$ under $\phi_S$ in the `world of higher dimension $\ZZ^r$.'
Then the weighted density of $C$ with respect to the weight $\tilde{\lambda}\circ\phi_S$ is still greater than a constant.

\

\noindent\textbf{Step~6}: A weighted hypergraph $\nu$ on $V$ is constructed from $\tilde{\lambda}\circ\phi_S$, and the pseudorandomness of $\nu$ follows from that of $\tilde{\lambda}$.
Let $E$ be the inverse image of $C$ under $T$.
The elements of $E$ are hyperedges. 
Since the weighted density of $C$ is greater than a certain constant, removing hyperedges from $E$ with small weighted density with respect to the weight $\nu$ does not completely eliminate isomorphic copies of $K_{r+1}^{(r)}$.
Here $K_{r+1}^{(r)}$ denotes the $(r+1)$-vertex complete $r$-graph.
By the \emph{relative hypergraph removal lemma} (Theorem~\ref{thm:RHRL}), this implies that
there exists an isomorphic copy of $K_{r+1}^{(r)}$ whose $r+1$ hyperedges are sent to distinct $r+1$ points by $T$, provided that $M$ is large enough.

\

\noindent\textbf{Step~7}: By sending such a copy of $K_{r+1}^{(r)}$ in $E$ under $T$, we obtain a corner consisting of elements of $C$ in $\ZZ^r$.
The image of the resulting corner in $C$ under $\phi_S$ is an $S$-constellation in $B$.
This completes the rough description of the proof of our relative multidimensional Szemeredi theorem.

Finally, the image of this $S$-constellation under $\Aff_{W,b}$ is a desired $S$-constellation in $A$!

It should be noted %
that this paper is not actually written in the order described in Steps~1--7.
At the beginning of Sections~\ref{section=relativeSzemeredi}--\ref{section=positiveweighteddensity}, we indicate the corresponding steps in this overview. 

After achieving the first goal, we prove that Theorem~\ref{mtheorem=primeconstellationsfinite} 
and Theorem~\ref{theorem=primeconstellationsfinite} 
are equivalent to each other (and so are Corollary~\ref{corollary=primeconstellationsupperdense} and Theorem~\ref{theorem=primeconstellationsdensesemiprecise}); see Remark~\ref{remark=toruno_toranaino}.
For this proof, we need a further counting argument, which is based on the geometry of numbers:
\begin{theorem}[= Corollary~\ref{corollary=no_DD_to_with_DD_re}]\label{theorem=no_DD_to_with_DD}
Let $K$ be a number field, $\omom$ its integral basis.
Assume that $A\subseteq \PP_K$ satisfies $\overline{d}_{\PP_K,\omom}(A)>0$.
Then, there exists an NL-compatible $\OKt$-fundamental domain $\DD=\DD(A,\omom)$ such that
\[
\overline{d}_{\PP_K\cap \DD,\omom}(A\cap \DD)>0
\]
holds.
\end{theorem}
See Theorem~\ref{theorem=fundamental_Omega} in the case of the finitary version and Theorem~\ref{theorem=fundamental_Omega_infinite} in the case of the infinitary version in full generality.
In Section~\ref{section=maintheoremfull} where the above theorem is presented, we  focus on the deduction `counting condition + pseudorandomness $\Longrightarrow$ constellation theorem' in our setting and axiomatize it.
In Section~\ref{section=slidetrick}, we refine the axiomatization formulated in Section~\ref{section=maintheoremfull}, and prove Theorem~\ref{mtheorem=TaoZieglergeneral}, a short interval version of Theorem~\ref{mtheorem=primeconstellationsfinite}.
There, a `slide trick,' a form of the pigeonhole principle, is in addition employed in order to take an appropriate $b$ in Step 4; see Lemma~\ref{lemma=slidetrick}.

These axiomatizations can be further extended to the case where the total space is a non-zero ideal $\ideala$ of $\OK$.
This leads to a constellation theorem for a pair of an order in $K$ and its invertible fractional ideal %
(Theorem~\ref{th:constellations-in-ideals}).
This theorem derives Theorem~\ref{mtheorem=quadraticform}, via the correspondence between binary quadratic forms and pairs of quadratic orders and their invertible fractional ideals (Theorem~\ref{th:quad-classical}).
We present a proof of the correspondence in the appendix for the reader's convenience.
We remark that to define the counterpart of prime elements is non-trivial in our constellation theorem; see Definitions~\ref{def:of-set-P}. 
For instance, if $\ideala$ is not principal, then non-principal prime ideals show up in the definitions, in contrast to the case of Theorem~\ref{theorem=primeconstellationsdensesemiprecise}.

\

We briefly summarize here the differences between the previous work \cite{Green-Tao08,Tao06Gaussian} and this paper.
We have already discussed novelty and ideas for Goldston--Y\i ld\i r\i m type asymptotic formulas, the NL-compatibility and the setting without a fundamental domain in Introduction.
The simplification of the proof of Goldston--Y\i ld\i r\i m type asymptotic formulas by using a smooth function $\chi$ was introduced in an unpublished note by Tao and subsequently used in 
\cite{Tao06Gaussian}
and this method is also used in this paper.

There exist several different formulations of the `relative multidimensional Szemer\'{e}di theorem' (RMST) in the literature.
In this paper, we establish Theorem~\ref{thm:RMST}; seemingly, the present paper may be the first place where the RMST of this form is explicitly stated.
In the work of Green--Tao~\cite{Green-Tao08} and Tao~\cite{Tao06Gaussian}, a condition called the \emph{correlation condition} was imposed on their definitions of pseudorandomness in addition to the \emph{linear forms condition}.
Conlon--Fox--Zhao~\cite{Conlon-Fox-Zhao15} succeeded in removing this correlation condition from their formulation of the RMST.
To do this, they proved the relative hypergraph removal lemma (RHRL) which only requires the linear forms condition.
Note that the RMST in \cite[Theorem~3.1]{Conlon-Fox-Zhao15} is stated in terms of finite additive groups, as is \cite[Theorem~3.5]{Green-Tao08} and  \cite[Theorem~2.18]{Tao06Gaussian}.
For this reason, some additional arguments were needed in their work to switch from $\ZZ^n$ to finite additive groups, and to go back. %
In this paper, although we appeal to the RHRL of Conlon--Fox--Zhao, we do \emph{not} transfer our setting to that of finite additive groups.
Instead, we follow the argument of Solymosi; see Steps 5--7 in the above overview.
Moreover, our argument of estimating a weighted density seems more straightforward than that in the previous work.

Due to the simplifications mentioned above, the complete proofs of Theorem~\ref{mtheorem=primeconstellationsfinite} (the finitary version) and Theorem~\ref{mtheorem=TaoZieglergeneral} (the short interval version) require \emph{no} technical complication beyond that of Theorem~\ref{theorem=primeconstellationssimple}.
In some earlier work on constellation theorems for $K=\QQ$ or $\QQ (\sqrt{-1})$, detailed proofs of the corresponding theorems were omitted.

\ 

\noindent\textbf{Plan.}
This paper is organized as follows:
in Section~\ref{section=preliminarynumbertheory}, we briefly summarize some facts in algebraic number theory needed in this paper.

In Section~\ref{section=NLC}, we study the NL-compatibility by using the geometry of numbers, and give a geometric characterization of it (Theorem~\ref{theorem=NLcompatible}).
We also construct an NL-compatible $\OKt$-fundamental domain $\DD_K(\ee,\sigma)$ from a fixed system of fundamental units $\ee$ of $K$ and a field embedding $\sigma\colon K\hookrightarrow \CC$.
We also prove some facts on the relation between the counts of prime principal ideals and of prime elements.

In Section~\ref{section=relativeSzemeredi}, we formulate and prove our relative multidimensional Szemer\'{e}di theorem.

In Section~\ref{section=GoldstonYildirim}, we prove the Goldston--Y\i ld\i r\i m type asymptotic formula %
in the number field setting.
To formulate it, we need algebraic backgrounds from Section \ref{section=preliminarynumbertheory}.

In Section~\ref{section=positiveweighteddensity}, we present the full proof of Theorem~\ref{theorem=primeconstellationsfinite}.
For this proof, we construct a pseudorandom measure with the aid of the Goldston--Y\i ld\i r\i m type asymptotic formula.
Then we make an estimate of the weighted density of a well-chosen set and apply the relative multidimensional Szemer\'edi theorem.

In Section~\ref{section=maintheoremfull}, we prove Theorem~\ref{mtheorem=primeconstellationsfinite}.
In the proof, we axiomatize the argument in the proof of Theorem~\ref{theorem=primeconstellationsfinite}.
By employing Lemma~\ref{lemma=OKt_orbit}, we reduce the general setting of Theorem~\ref{mtheorem=primeconstellationsfinite} to that with a fundamental domain; see Theorem~\ref{theorem=fundamental_Omega} for details. 
Theorem~\ref{theorem=primeconstellationsdensesemiprecise} is also verified.

In Section~\ref{section=slidetrick}, we demonstrate Theorem~\ref{mtheorem=TaoZieglergeneral}. 
The proof uses an additional argument, the `slide trick,' to that of Theorem~\ref{mtheorem=primeconstellationsfinite}.

In Section~\ref{section=quadraticform}, we formulate and prove our prime element constellation theorem with respect to the pair of an order and its invertible fractional ideal (Theorem~\ref{th:constellations-in-ideals}). It derives a constellation theorem for a norm form (Theorem~\ref{theorem=normform}). 
By combining this with the classical correspondence between binary quadratic forms and the pairs where the orders are quadratic  (Theorem~\ref{th:quad-classical}), we establish Theorem~\ref{mtheorem=quadraticform}.

In Appendix, we present a proof of the correspondence above.

\ 

\noindent\textbf{Notation.}
Let $\NN=\{1,2,3,\ldots\}$ denote the set of positive integers and $\PP=\{2,3,5,\dots\}$ the set of rational prime numbers.
A subset of $\PP$ truncated by %
a real number $x$ 
is expressed by %
a subscript. 
For example,
\[
\PP_{\leq x}=\{p\in \PP : p\leq x\},\quad \textrm{and} \quad \PP_{>x}=\{p\in \PP : p>x\}.
\]
For $m\in\NN$, set $[m]\coloneqq\{1,2,3,\ldots,m\}$.
For $m_1,m_2\in\ZZ$ satisfying $m_1\leq m_2$, set $[m_1,m_2]\coloneqq\{l\in \ZZ : m_1\leq l\leq m_2\}$.
When we use $[a_1,a_2]$ in the sense of a real closed interval, we write $[a_1,a_2]_{\RR}$ to distinguish it.
Similarly in the case of (half-)open intervals. %
For a finite set $A$, $\#A$ denotes the cardinality of $A$.
We write $\#A=\infty$ to mean that the set $A$ is infinite.
For a non-empty finite set $J$ and a positive integer $r$, $\binom J r$ denotes the set $\{e\in2^J : \#e=r\}$ of subsets with cardinality $r$.
For a mapping $f$, $\Im(f)$ denotes the image of $f$.
For a set $A$, $\ichi_A$ denotes the indicator function of $A$.
For a function $f\colon X\to\RR$ on a set $X$ and a non-empty finite subset $A\subseteq X$, we use the expectation symbol to denote the average of $f$ over $A$:
\[
\EE(f\mid A)=\EE(f(a)\mid a\in A)\coloneqq\frac{1}{\#A}\sum_{a\in A}f(a).
\]
For functions $f,g\colon X\to\RR$, if $f(x)\leq g(x)$ holds for all $x\in X$, then we write $f\leq g$.
For a $\ZZ$-module $\calZ$, $W\in\ZZ$ and $b\in\calZ$, the affine transformation $\Aff_{W,b}$ is defined by
\[
\Aff_{W,b}\colon \calZ\to \calZ;\quad \beta\mapsto W\beta+b.
\]

We use big-$O$ and little-$o$ notation in the following sense for statements that take into account some parameters that are not necessarily `numbers.'
Let $x$ be a (natural, real or complex) numerical parameter and $t_1,\dots, t_k$ a part of parameters under consideration.
Let $a\in\CC$.
Let $f$ and $g$ be functions with parameters under consideration, where $g$ is non-negative.
If there exists a positive-valued function $C_{t_1,\dots,t_k}$ depending only on $t_1,\dots, t_k$ such that $|f|\leq C_{t_1,\dots,t_k}\cdot g$, then we write $f=O_{t_1,\dots,t_k}(g)$.
If the inequality is valid only on a certain neighborhood of $a$, then we write $O_{x\to a; t_1,\dots,t_k}(g)$.
If there exists a positive-valued function $c_{t_1,\dots,t_k}(x)$ depending only on $x, t_1,\dots, t_k$ and satisfying $\lim\limits_{x\to a}c_{t_1,\dots,t_k}(x)=0$ such that $|f|\leq c_{t_1,\dots,t_k}\cdot g$ on a certain neighborhood of $a$, then we write $f=o_{x\to a; t_1,\dots,t_k}(g)$.
The convergence of $c$ need not be uniform for $t_1,\dots,t_k$.
We use similar expressions for $a=+\infty$; in this case we only use a natural or positive real numerical parameter, and the symbol $+\infty$ is simply written as $\infty$ in this paper.

%% file: chapter3.tex
\section{Preliminaries on algebraic number theory}\label{section=preliminarynumbertheory}
In this section, we summarize necessary materials from algebraic number theory.
All results in this section are known; see \cite{Neukirch,Hecke,Hardy-Wright}
for details.
\begin{setting}\label{setting=section3}
Throughout this section, with the exception of 
Subsection~\ref{subsection=rationalprimes}, $K$ will denote a \emph{number
field} of \emph{degree} $n$, that is, a finite extension of the rational number field $\QQ$
with $n=[K:\QQ]$.
\end{setting}
\subsection{The ring of integers and its ideals}\label{subsection=OK}
The subset of $K$ consisting of elements which are integral over $\ZZ$ forms
a subring called the \emph{ring of integers} of $K$, and we denote it by $\OK$.
By \cite[Chapter~I, Proposition~2.10]{Neukirch},
$\OK$ is a free $\ZZ$-module of rank $n$; a $\ZZ$-basis of $\OK$ is called
an \emph{integral basis} of $K$.
We also denote by 
$\Ideals_K$ the set of non-zero ideals, and $|\Spec(\OO_K)|$ the set of non-zero prime ideals, of $\OO_K$.
\begin{theorem}[Prime ideal decomposition, 
see {\cite[Chapter~I, Theorem 3.3]{Neukirch}}]
\label{theorem=primeideals}
The mapping
\begin{equation}\label{Eq:prime-decomp}
\bigoplus_{|\Spec(\OK)|}\ZZ_{\geq 0}\to\Ideals_K;\quad (e_{\idealp})_{\idealp}\mapsto\prod_{\idealp\in |\Spec (\OK)|}\idealp^{e_{\idealp}}
\end{equation}
is an isomorphism of commutative monoids.
\end{theorem}
For two ideals 
$\ideala\in\Ideals_K\cup\{ (0)\}$ and 
$\idealb\in\Ideals_K$, we write $\idealb\mid\ideala$ if
$\idealb\supseteq\ideala$ holds. 
If $\ideala \neq (0)$, %
then this is equivalent to saying that
the exponent of each $\idealp$ appearing in the prime decomposition of $\idealb$
is at most that of $\ideala$. 

In Section~\ref{section=quadraticform}, we will treat fractional ideals of $\OK $, which generalize ideals. A \emph{fractional ideal} $\ideala$ of $\OK $ is a finitely generated $\OK$-submodule of $K$.  For a non-zero fractional ideal $\ideala$ of $\OK $, the set $\ideala^{-1}\coloneqq \{x\in K:x\ideala\subseteq \OK\}$ is again a non-zero fractional ideal, called the \emph{inverse fractional ideal} of $\ideala $; we have $\ideala \ideala^{-1}=\OK$; see \cite[Chapter~I, Proposition~3.8]{Neukirch} for the proof. The following generalization of  Theorem~\ref{theorem=primeideals} will be employed in Section~\ref{section=quadraticform}. Up to Section~\ref{section=slidetrick}, fractional ideals will not show up.
\begin{theorem}[Prime ideal decomposition of fractional ideals, 
see {\cite[Chapter~I, Corollary~3.9]{Neukirch}}]
\label{theorem=primeideals_frac}
The mapping from $\bigoplus\limits_{|\Spec(\OK)|}\ZZ$ to the group of non-zero fractional ideals of $\OK $ defined by
\[
(v_{\idealp})_{\idealp}\mapsto\prod_{\idealp\in |\Spec (\OK)|}\idealp^{v_{\idealp}}
\]
is an isomorphism of commutative groups.
\end{theorem}
We define the \emph{ideal norm} of an ideal $\ideala\in\Ideals_K$ by
$\Nrm(\ideala)\coloneqq\#(\OK/\ideala)$.
If $\ideala$ is a principal ideal, that is, 
$\ideala=\alpha\OK$ for some $\alpha\in\OKnz$,
then $\Nrm(\alpha \OK)$ coincides with the ideal norm 
$\Nrm (\alpha)$ of $\alpha$ defined in 
Definition~\ref{definition=norm}.

Let $\sigma\colon K\to\CC$ be a homomorphism of fields.
If the image of $\sigma$ is contained in $\RR$, then we call
$\sigma$ a \emph{real embedding}, and a \emph{complex embedding} otherwise.
If we denote by $r_1$ the number of real embeddings of $K$, and by $r_2$ that of
conjugate pairs of complex embeddings of $K$,
then $n=r_1+2r_2$ holds.
\begin{setting}\label{setting=sigma}
We denote by 
$\sigma_1,\dots,\sigma_{r_1},\sigma_{r_1+1},\dots,\sigma_{r_1+2r_2}$
the embeddings of $K$ into $\CC$. We choose the numbering in such a way that
$\sigma_1,\ldots ,\sigma_{r_1}$ are real embeddings, while conjugate
pairs of complex embeddings are
$(\sigma_{r_1+1},\sigma_{r_1+r_2+1}),(\sigma_{r_1+2},\sigma_{r_1+r_2+2}),\dots,(\sigma_{r_1+r_2},\sigma_{r_1+2r_2})$.
\end{setting}
\begin{lemma}[{see \cite[Chapter~1, Proposition~2.6]{Neukirch}}]\label{lemma=idealnorm}
Under Setting~$\ref{setting=sigma}$,
we have
\[
\Nrm(\alpha)=\prod_{i\in [r_1]}|\sigma_i(\alpha)| 
\prod_{j\in [r_2]}|\sigma_{r_1+j}(\alpha)|^2
\]
for $\alpha\in\OKnz$.
\end{lemma}
In particular, $\Nrm(W)=W^n$ for $W\in\NN$.
\begin{remark}\label{remark=idealnorm}
When $A\to B$ is a homomorphism of commutative unital rings which makes $B$ a free $A$-module of finite rank,
the {\em norm} $N_{B/A}(b )\in A$ of an element $b\in B$ is defined to be the determinant of the $A$-linear map induced by multiplication by $b\colon B\to B$.

In the case of $\QQ \to K$, the norm of $\alpha \in K$ is known to be equal to the product
$\sigma_1(\alpha)\sigma_2(\alpha)\cdots \sigma_{r_1+2r_2}(\alpha)$.
Lemma~\ref{lemma=idealnorm} says its absolute value is equal to the ideal norm $\Nrm (\alpha ) $,
except when $\alpha =0$, for which the ideal norm $\Nrm (0)$ is not defined in Definition~\ref{definition=norm}. 
\end{remark}
\begin{lemma}[{see \cite[Chapter~I, Proposition~6.1]{Neukirch}}]\label{lemma=completemultiplicativity}
Let $\ideala=\prod_{\idealp \in |\Spec (\OO_K) |} \idealp^{e_{\idealp}}$
be the prime ideal decomposition of an ideal $\ideala\in\Ideals_K$. Then
\[
\Nrm(\ideala)=\prod_{\idealp \in |\Spec (\OO_K) |} \Nrm(\idealp)^{e_{\idealp}}.
\]
\end{lemma}
Next we introduce two number theoretic functions.
\begin{definition}\label{def=totient}
The \emph{totient function} $\vph_K$ of $K$ is defined by
\[\vph_K\colon\Ideals_K\to\NN;
\quad \ideala\mapsto\#\left((\OK/\ideala)^{\times}\right).\]
For 
$\alpha\in\OKnz$, we write $\vph_K(\alpha)\coloneqq\vph_K(\alpha \OK)$.
\end{definition}
\begin{proposition}[{see \cite[Theorem 80 in \S 27]{Hecke}}]\label{prop=totient}
For $\ideala\in \Ideals_K$, we have
\[
\vph_K(\ideala)=\Nrm(\ideala)\prod_{\idealp\in|\Spec(\OK)|, \ \idealp\mid\ideala}
(1-\Nrm (\idealp)^{-1}).
\]
\end{proposition}
\begin{definition}\label{dfn:3.9}
We define the \emph{\Mobius function}
$\mu_K\colon \Ideals_K\to\{0,\pm1\}$ by
\[
\mu_K(\ideala)=\begin{cases}
(-1)^r&\text{if $\ideala$ is a product of $r$ ($\ge 0$) distinct prime ideals,}\\
0&\text{otherwise.}
\end{cases}
\]
\end{definition}
It follows from the definition that the \Mobius function is \emph{multiplicative},
that is, $\mu_K(\ideala\idealb)=\mu_K(\ideala)\mu_K(\idealb)$
for ideals $\ideala$ and $\idealb$ relatively prime to each other.
\begin{proposition}\label{proposition=Moebius}
Given a function $f\colon \Ideals _K \to \CC$,
define $g\colon \Ideals _K \to \CC $ by
\[
g(\ideala)\coloneqq\sum_{\idealb\in\Ideals_K, \ \idealb\mid\ideala}f(\idealb).
\]
Then
\[ 
f(\ideala)=\sum_{\substack{\idealb,\idealc\in\Ideals_K \\ \idealb\cdot\idealc=\ideala}}\mu_K(\idealb)\cdot g(\idealc).
\]
\end{proposition}
\begin{proof}
Since $\Ideals_K$ is isomorphic to $\Ideals_{\QQ}$ as monoids
by Theorem~\ref{theorem=primeideals},
the result follows from the property of the standard \Mobius function
$\mu_{\QQ}$.
\end{proof}
The next lemma may be regarded as a refinement of complete multiplicativity of the ideal norm (Lemma~\ref{lemma=completemultiplicativity}), and will be employed in Section~\ref{section=maintheoremfull}.
We will present a proof of a more general statement of this lemma  in Section~\ref{section=quadraticform}; see Proposition~\ref{prop:monogenic-modulo-b}.
\begin{lemma}\label{lemma=a/Wa}
For $\ideala,\idealb\in \Ideals_K$, we have an isomorphism
$\OK/\idealb\simeq\ideala/\ideala\idealb$
of $\OK$-modules.
\end{lemma}
\subsection{The unit group and ideal class group}\label{subsection=Ot-Cl}
The multiplicative group $\OKt$ of
$\OK$ is called the \emph{unit group} of $K$.
The subgroup of 
$\OKt$ consisting of torsion elements is denoted by 
$\mu(K)$ (not to be confused with the \Mobius function $\mu_K$).
We continue to assume Settings~\ref{setting=section3} and
\ref{setting=sigma}.
\begin{lemma}[{see \cite[Chapter~I, Proposition 7.1]{Neukirch}}]\label{lemma=muK}
The group $\mu(K)$ is finite. Moreover, an element
$\alpha\in\OKnz$ is in $\mu (K)$ if and only if
$|\sigma_{i}(\alpha)|=1$ for all $i\in [r_1+r_2]$.
\end{lemma}
We define
$\OKbar\coloneqq\OKt/\mu(K)$. 
The group $\OKbar$ is torsion-free with rank $r_1+r_2-1$ by
Dirichlet's unit theorem
(\cite[Chapter~I, Theorem~7.4]{Neukirch}).
See also Theorem~\ref{theorem=Dirichlet}.
A sequence $\ee=(\varepsilon_{1},\varepsilon_{2},\dots,\varepsilon_{r_1+r_2-1})$
which gives a basis $(\overline{\varepsilon_1},\overline{\varepsilon_2},\dots,\overline{\varepsilon_{r_1+r_2-1}})$ of $\OKbar $ is called \emph{fundamental units} of $K$.

There is an invariant of $K$ called the \emph{class number}, which is a positive integer; this will show up in the Chebotarev density theorem.
For the sake of completeness,
let us give a quick definition. Define an equivalence relation $\sim$
on $\Ideals_K$ as follows. We declare $\ideala\sim\idealb$ if there exist
$\gamma,\delta\in\OK\setminus\{0\}$ such that the equality of ideals
$(\gamma)\cdot\ideala=(\delta)\cdot\idealb$ holds.
The set of equivalence classes
$\Ideals_K/\sim$ inherits the monoid structure from $\Ideals_K$.
It is known that this is in fact a group (see \cite[Chapter~I, Definition~3.7 to Proposition~3.8]{Neukirch}), 
called the \emph{ideal class group}.
It is furthermore known that this is a finite group (see \cite[Chapter~I, Theorem~6.3]{Neukirch}),
and its order, written $h=h_K$, is the \emph{class number}.
\subsection{$p$-Ideals}\label{subsection=pideal}
Recall from Notation in 
Subsection~\ref{subsection=ideasofproof}
that the set of rational primes is denoted by
$\PP=\{2,3,5,\ldots\}$.
For
$\idealp\in|\Spec(\OK)|$, 
the intersection $\idealp\cap\ZZ$ is a non-zero prime ideal of $\ZZ$.
Hence, there exists a unique $p\in \PP$ such that $\idealp\cap\ZZ=p\ZZ$.
In this case, we call $\idealp$ a \emph{prime $p$-ideal}, 
and the set of
prime $p$-ideals is denoted by $|\Spec(\OK)|\ppart$.
For $\idealp\in|\Spec(\OK)|\ppart$, the quotient $\OK/\idealp$ is a finite 
extension field of the finite prime field $\FF_p$.
The extension degree 
$f_{\idealp}\coloneqq[\OK/\idealp:\FF_p]$
is called the \emph{degree} of $\idealp$.
Then $\Nrm(\idealp)=p^{f_{\idealp}}$ holds.

For the prime ideal decomposition $\ideala=\prod_{\idealp\in|\Spec(\OK)|}\idealp^{e_{\idealp}}$ of $\ideala \in \Ideals_K$, for each $p\in \PP$, we define
\[
\ideala\ppart\coloneqq\prod_{\idealp\in|\Spec(\OK)|\ppart}\idealp^{e_{\idealp}};
\]
it is called the \emph{$p$-part} of $\ideala$.
Then we have
\[
\ideala=\prod_{p\in \PP}\ideala\ppart.
\]
An ideal
$\ideala\in\Ideals_K$ is called a \emph{$p$-ideal} if
$\ideala ^{(p)}=\ideala $,
or equivalently,  $\Nrm(\ideala)$ is a power of $p$.
Observe that $\OK$ is a $p$-ideal for every $p\in\PP$.
We denote the set of $p$-ideals of $\OK$ by $\Ideals_K\ppart$.
For $\ideala,\idealb\in \Ideals_K$ and $p\in \PP$, we have
$(\ideala\cap\idealb)\ppart=\ideala\ppart\cap\idealb\ppart$.

When $K=\QQ$, we have $\OO_K=\ZZ$, and the positive generator of
the $p$-part of an ideal $D\ZZ$ ($D\in\NN$)
is denoted by $D\ppart$.
We then have
\[
D=\prod_{p\in \PP}D\ppart,
\]
which is the prime factorization of $D$.

We now exhibit two lemmas needed in
Section~\ref{section=GoldstonYildirim}.
Let $Z$ be a finite abelian group. Then $Z$ admits a unique
decomposition
\[
Z=\bigoplus _{p\in \PP}Z^{(p)},
\]
where $Z\ppart$ is a $p$-group for each $p$. More explicitly,
$Z\ppart=\{ z\in Z : \exists e\ge 0, \ p^e\cdot z =0\}$.
We call $Z\ppart$ the \emph{$p$-part} of $Z$.
\begin{lemma}\label{lemma=chineseremainder}
Let $p\in\PP$.
\begin{enumerate}[$(1)$]
\item\label{en:CRT1}
Let $\idealc \in \Ideals_{K}$.
The $p$-part $(\OO_K/\idealc)\ppart$ of the finite abelian group
$\OK/\idealc$ is isomorphic 
to $\OK/(\idealc\ppart)$
by the composition
\[
(\OK/\idealc)\ppart\hookrightarrow \OK/\idealc \twoheadrightarrow \OK /(\idealc\ppart)
\]
of the inclusion followed by
the canonical surjection. 
\item\label{en:CRT2}
Let $Z$ and $W$ be finite abelian groups.
Then the following map gives a bijection
\[
\Hom(Z,W)\xrightarrow{\simeq}\bigoplus_{p\in\PP}\Hom(Z\ppart,W\ppart);\quad \psi\mapsto(\psi\ppart )_{p\in \PP},
\]
where $\psi\ppart $ is defined as
the restriction of $\psi$ to
$Z\ppart$.
\end{enumerate}
\end{lemma}
\begin{proof}
To prove \eqref{en:CRT1}, simply take the $p$-part of both sides of
the isomorphism by the Chinese remainder theorem:
\[
\OK /\idealc \xrightarrow\simeq \prod _{p\in \PP } \OK /\idealc\ppart.
\]

Next we prove \eqref{en:CRT2}. Since $Z$ and $W$ are finite abelian groups,
we have decompositions into finite direct products
$Z\simeq \bigoplus _{p\in\PP } Z\ppart $ 
and $W\simeq \bigoplus _{\ell\in\PP }W\lpart$.
This implies the direct sum decomposition
\[
\Hom (Z,W) \xrightarrow\simeq \bigoplus_{p,\ell \in \PP }\Hom (Z\ppart ,W\lpart)=\bigoplus_{p\in\PP}\Hom(Z\ppart, W\ppart),
\]
where the last equality follows by observing that
$\Hom (Z\ppart ,W\lpart)=0$ for $p\neq\ell$.
Thus, a homomorphism $\psi:Z\to W$ is determined by its $p$-components
$\psi\ppart:Z\ppart \to W\ppart$.
Since $\psi\ppart$ is the composition of the three homomorphisms
$Z\ppart \xrightarrow{\psi |_{Z\ppart}}W\ppart 
\hookrightarrow W\twoheadrightarrow W\ppart$
in which the composition of the last two is the identity,
the claimed correspondence follows.
\end{proof}
An ideal $\ideala\in\Ideals_K$ is said to be \emph{square-free} if,
in the prime ideal decomposition 
$\ideala=\prod_{\idealp \in |\Spec (\OO_K)|}\idealp^{e_{\idealp}}$,
the condition $e_{\idealp}\in \{0,1\}$ holds for all $\idealp \in |\Spec (\OO_K)|$.
\begin{lemma}\label{lemma=squarefree}
For $p\in\PP$, 
we have $\#(|\Spec (\OO_K)|\ppart)\leq n$.
In particular, the number of square free $p$-ideals is at most $2^n$.
\end{lemma}
\begin{proof}
Every
$\idealp\in|\Spec (\OK)|\ppart$ appears in the prime ideal 
decomposition of $p\OO_K$. 
Thus
\[p\OK=\prod_{\idealp\in |\Spec (\OK)|\ppart}\idealp^{e_{\idealp}},
\quad e_{\idealp}\in \NN.\]
Taking the norm of both sides using 
Lemma~\ref{lemma=completemultiplicativity}, we find
\[
n=\sum_{\idealp\in|\Spec (\OK)|\ppart}f_{\idealp}e_{\idealp}\geq 
\#(|\Spec (\OK)|\ppart).
\]
This proves the first statement. The second statement follows from the first
by the definition of square-freeness.
\end{proof}
Let us record the following corollary to the Chinese remainder theorem. 
Lemma~\ref{lem:Chinese} will be employed in Subsection~\ref{subsection=GYforideals} and Section~\ref{section=quadraticform}. 
\begin{lemma}[Chinese Remainder Theorem for modules]\label{lem:Chinese}
    Let $M$ be a module over a \textup{(}commutative unital\textup{)} ring $A$ and $\ideala _1,\dots ,\ideala _s $ mutually coprime ideals of $A$.
    Then we have an equality of ideals
    \begin{equation}\label{eq:Chinese-ideals}
        \bigcap _{i\in[s]} \ideala _i  = \prod _{i\in[s]} \ideala _i 
    \end{equation}
    and the next natural maps of rings and $A$-modules are isomorphisms:
    \begin{align}
        A/\bigl( \prod _{i\in[s]}\ideala _i\bigr)  &\xrightarrow \simeq \prod _{i\in[s]} A/\ideala_i ,\label{eq:Chinese-rings}
        \\
        M/\bigl(\prod _{i\in[s]}\ideala _i\bigr) M &\xrightarrow \simeq \prod _{i\in[s]} M/\ideala _i M . \label{eq:Chinese-modules}
    \end{align}
\end{lemma}
\begin{proof}
The Chinese remainder theorem provides   \eqref{eq:Chinese-ideals} and \eqref{eq:Chinese-rings}.
    For \eqref{eq:Chinese-modules}, consider the tensor product of \eqref{eq:Chinese-rings} and $M$ over $A$ and apply
    \cite[Chapter 2, Exercise 2]{Atiyah-Macdonald}.
    (The usual proof of the Chinese remainder theorem equally works to prove \eqref{eq:Chinese-modules}.)
\end{proof}
\subsection{The Dedekind zeta function and density of ideals}
\label{subsection=densityofideals}
In this subsection, we present some results on the Dedekind zeta function
and density of ideals for a number field $K$.
\begin{definition}[{Dedekind zeta function; see \cite[Chapter~VII, Definition~5.1]{Neukirch}}]\label{definition=Dedekindzeta}
The sum
\[
\sum_{\ideala \in \Ideals_K}\frac{1}{\Nrm(\ideala)^s}
\]
converges absolutely and uniformly on every compact subset of
the domain $\Re(s)>1$ in the complex plane.
Here, the power is defined by $\Nrm(\ideala)^s= \exp ({s\log \Nrm(\ideala)})$
where $\log \Nrm (\ideala )\in \RR $ (among other branches of $\log $). 
We call the analytic function defined by this sum 
the \emph{Dedekind zeta function} of $K$, and denote it by $\zeta^{}_K$.
\end{definition}
\begin{proposition}[{see \cite[Chapter~VII, Proposition~5.2]{Neukirch}}]\label{proposition=Eulerproduct}
The infinite product
\[
\prod_{\idealp\in|\Spec(\OK)|}\left(1-\frac{1}{\Nrm (\idealp)^s}\right)^{-1}
\]
converges absolutely on the domain $\Re(s)>1$, and coincides with  $\zeta^{}_{K}(s)$.
\end{proposition}
It is known that the Dedekind zeta function $\zeta^{}_K$ has analytic
continuation to a meromorphic function on $\CC$.
\begin{theorem}[{see~\cite[Chapter~VII, Corollary~5.11]{Neukirch}}]
\label{theorem=zeta_K}
The Dedekind zeta function $\zeta^{}_K$ has a  pole of order $1$ at $s=1$, and this is the unique pole.
The residue $\kappa=\kappa_K$ of $\zeta^{}_K$ at $s=1$ is a positive real number.
\end{theorem}
The residue $\kappa=\kappa_K$ can be expressed in terms of the class number $h=h_K$ (Subsection~\ref{subsection=Ot-Cl})
and an invariant called the regulator of $K$. 
This expression is known as the class number formula.
In this paper, however, we do not need the explicit form of $\kappa$.
\begin{proposition}[Density of ideals; see
{\cite[Theorem 121 in \S 40]{Hecke}}]
\label{proposition=idealdensity}
The residue $\kappa$ of $\zeta^{}_K$ at $s=1$ coincides with the 
limit of the density of ideals in the following sense:
\[
\lim_{L\to \infty}\frac{\#\{\ideala \in \Ideals _K:\Nrm (\ideala)\leq L\}}{L}=\kappa>0.
\]
\end{proposition}
Next we turn to prime ideals.
Item~\eqref{Landau} of Theorem~\ref{theorem=Chebotarev} will be used to bound from above the number of
certain elements related to prime ideals in Proposition~\ref{proposition=Masani_Landau}, Proposition~\ref{proposition=PK_kouri_a} and
Subsection~\ref{subsection=proof_orders}.
Item~\eqref{Chebotarev} of Theorem~\ref{theorem=Chebotarev} will be used to bound from below the number of prime elements.

\begin{theorem}	
	\begin{enumerate}[$(1)$]
		\item \textup{(Landau's prime ideal theorem, see
\cite[Theorem~3]{Cassels-Froehlich}).} We have
		\begin{align*}
			\#\{\idealp \in | \Spec (\OK) | : \Nrm (\idealp)\leq L\}=(1+o_{L\to \infty;K}(1))\cdot \frac{L}{\log L}.
		\end{align*} \label{Landau}
		\item \textup{(Chebotarev density theorem for principal prime ideals).} Denote by $|\Spec(\OK)|^{\PI}$ the set of non-zero principal prime ideals of $\OK$.
    		Let $h$ be the class number of $K$.
		Then
		\begin{align*}
			\#\{\idealp\in |\Spec (\OK)| ^{\PI} : \Nrm (\idealp)\leq L\}=(1+o_{L\to \infty;K}(1))\cdot \frac{1}{h}\cdot\frac{L}{\log L}.
		\end{align*}	\label{Chebotarev}
	\end{enumerate}
	\label{theorem=Chebotarev}
\end{theorem}

\begin{proof}[Proof of~\eqref{Chebotarev}]
This is a special case of 
the \Cheb\ density theorem
\cite[Theorem 4]{Cassels-Froehlich},
a reformulation of which is stated as Thoerem~\ref{theorem=Chebotarev_narrow} below.
See the paragraph after Theorem~\ref{theorem=Chebotarev_narrow} for how to deduce our current statement.
\end{proof}
\begin{remark}\label{remark=Landau}    
Some readers might be more familiar with the Chebotarev density theorem
in the analytic density version
(e.g.\ \cite[Chapter~VII, Theorem 13.2]{Neukirch}):
\begin{equation}\label{eq:anal-density}
\lim_{s\to 1+0}\frac{\prod\limits _{\idealp \in |\Spec (\OO_K)|^{\PI}} \left(1-\frac 1 {\Nrm (\idealp )^{s}} \right)^{-1}}{\log \left( \frac{1}{s-1} \right) }= \frac{1}{h} .
\end{equation}
In general, if $f(s)=\sum_{n\geq 1}\frac{a_n}{n^s}$ is a Dirichlet series with real coefficients
convergent in the domain $\Re (s)>1$, then the following inequalities are known:
\begin{equation}\label{eq:Landau-book}
\limsup _{x\to +\infty } \frac{\sum\limits _{1\le n\le x} a_n}{\log \left( \frac{x}{\log x} \right)}\geq
\limsup _{s\to 1+0} \frac{f(s)}{\log \left( \frac{1}{s-1} \right)},
\quad 
\liminf _{s\to 1+0} \frac{f(s)}{\log \left( \frac{1}{s-1} \right)}
\geq 
\liminf _{x\to +\infty } \frac{\sum\limits _{1\le n\le x} a_n}{\log \left( \frac{x}{\log x} \right)} 
.
\end{equation}
The proof is straightforward with Abel's summation method. See for example
\cite[Definition on p.~103 and Theorem on p.~118]{Landau}.

By setting $a_n \coloneqq \# \{ \idealp \in |\Spec(\OK)|^{\PI} : \Nrm (\idealp)= n\} $, we can see that Theorem~\ref{theorem=Chebotarev}~\eqref{Chebotarev} implies \eqref{eq:anal-density}.
Conversely, 
given 
\eqref{eq:anal-density}, by the first half of \eqref{eq:Landau-book}
we conclude that the following
estimate for the number of principal prime ideals holds for
\emph{infinitely many} $x\in \NN$:
\[
\# \{\idealp \in |\Spec(\OK)|^{\PI} : \Nrm (\idealp)\leq x\}\geq \frac{1}{2h} \cdot \frac{x}{\log x}.
\]
Whereas the assertion above is weaker than that of Theorem~\ref{theorem=Chebotarev}~\eqref{Chebotarev}, this suffices for the proofs of infinitary statements such as Theorem~\ref{theorem=primeconstellationssimple} and Corollary~\ref{corollary=primeconstellationssimple}.
\end{remark}
\subsection{Distribution of rational primes}\label{subsection=rationalprimes}
In this subsection we list three results on the distribution of
primes in $K=\QQ$. 
We use the notation introduced in
Subsection~\ref{subsection=ideasofproof}.
As in Subsection~\ref{subsection=densityofideals}, 
$L$ will denote a real parameter greater than $1$.
The number of rational primes not exceeding
$L$ is denoted by $\pi(L)$.

The first result, Lemma~\ref{lemma=Mertenssecond} below,
will be used in the proof of Lemma~\ref{lem:prod_Ep_Ep'}.
\begin{lemma}\label{lemma=Mertenssecond}
For $L>1$, we have
\[
\sum_{p\in \PP_{>L}}\frac{1}{p^2}=O\left(\frac{1}{L\log L}\right).
\]
\end{lemma}
\begin{proof}
It follows from inequality
\eqref{eq:Chebyshev-ineq} below and
\cite[(22.4.2)]{Hardy-Wright} that
\[
\pi(L)\leq L^{\frac{3}{5}}+\frac{5}{3}\cdot\frac{\vartheta(L)}{\log L}<\left(1+\frac{10}{3}\log 2\right)\cdot\frac{L}{\log L}<\frac{3.5L}{\log L}.
\]
By Abel's summation formula (see \cite[Theorem~421]{Hardy-Wright}),
we obtain
\[
\sum_{p\in \PP_{>L}}\frac{1}{p^2}=-\frac{\pi(L)}{L^2}+2\int_L^{\infty}\frac{\pi(t)}{t^3}\rd t\leq 7\int_L^{\infty}\frac{\rd t}{t^2\log t}\leq\frac{7}{L\log L},
\]
as desired.
\end{proof}
The second result is known as Mertens's first theorem. It will be
used in the proof of Lemma~\ref{lem:Ep'(p_leq_w)}.
\begin{proposition}[Mertens's first theorem; see
{\cite[Theorem~425]{Hardy-Wright}}]
\label{proposition=Mertens}
For $L\geq 2$,
\[\sum_{p\in \PP_{\leq L}}\frac{\log p}{p}=\log L+O(1).\]
\end{proposition}
The third result is on %
the
\emph{first Chebyshev function}:
$
\vartheta(L)\coloneqq \sum_{p\in\PP_{\leq L}}\log p.
$
The prime number theorem 
is equivalently formulated as the asymptotic
$\vartheta (L)=(1+o_{L\to \infty}(1))L$.
For us,
the following linear bound, whose proof is considerably easier, suffices:
\begin{proposition}[{Chebyshev's theorem; see
\cite[Theorem~415]{Hardy-Wright}}]\label{prop=elementaryChebyshev}
For $L\geq1$, we have
\begin{equation}\label{eq:Chebyshev-ineq}
\vartheta (L)\leq 2(\log 2)L.
\end{equation}
\end{proposition}
This will be employed to determine the choice of $w=w(M)$ in Section~\ref{section=positiveweighteddensity}.

%% file: chapter4.tex
\section{Norm-length compatibility and geometry of numbers}\label{section=NLC}
In this section, we study the NL-compatibility of subsets of $\OKnz$, which was introduced in Definition~\ref{definition=normrespecting}. 
In Subsection~\ref{subsection=Miknowski}, we recall some definitions from the \emph{geometry of numbers}, including Minkowski embeddings. 
In Subsection~\ref{subsection=NLCcondition}, we characterize the NL-compatibility in terms of the (weighted) 
multiplicative Minkowski embedding (Theorem~\ref{theorem=NLcompatible}). 
Then, in Subsection~\ref{subsection=DDKee}, we  provide a way of constructing an NL-compatible $\OKt$-fundamental domain (Definition~\ref{definition=fundamentaldomainDDee}, Proposition~\ref{proposition=normrespectingfundamentaldomain}). 
The existence of an NL-compatible fundamental domain plays a key role throughout the present paper. 
We remark that these ideas have already been used essentially in~\cite[Lemma 4.2]{maynard_2020}.

In Subsection~\ref{subsection=OKt_orbit}, we estimate the size of 
subsets of $\OKt $-orbits truncated by bounding the $\lmugen $-length.
The results, Lemma~\ref{lemma=OKt_orbit} and Corollary~\ref{corollary=OKt_orbit_ideal}, enable us to switch from counting ideals to counting elements in Sections~\ref{section=maintheoremfull}--\ref{section=quadraticform}. Here is the setting of this section.

\begin{setting}\label{setting=section4}
	Let $K$ be a number field of degree $n$, and let
	$\omom=(\omega_1,\dots,\omega_ n)$ be an integral basis of $K$.
	Let $\ee$ be fundamental units of $K$.
	We also use the notation defined in 
	Setting~\ref{setting=sigma} for embeddings.
	We say that a subset of a finite-dimensional real vector space is
	\emph{bounded} if it is bounded with respect to some norm.
\end{setting}
The notion of boundedness is independent of the choice of a norm.
Indeed, it is equivalent to relative compactness in the natural topology.

\subsection{Weighted multiplicative Minkowski embedding}\label{subsection=Miknowski}

Throughout this subsection, we use 
Setting~\ref{setting=section4}.
In this section, we introduce the additive Minkowski embedding
and (weighted) multiplicative Minkowski embedding. 
We write 
\begin{align*}
\sigma_{i,\RR}&\colon K\otimes_{\QQ} \RR\to \RR\quad(i\in [r_1]), \\
\sigma_{r_1+j,\RR}&\colon K\otimes_{\QQ} \RR\to \CC\quad(j\in [2r_2])
\end{align*}
for the $\RR $-linear extensions of $\sigma_i$ ($i\in[r_1]$) and $\sigma_{r_1+j}$ ($j\in[2r_2]$). %
Then we define $\Nrm_{\RR}\colon K\otimes_{\QQ} \RR\to\RR$ by
\begin{equation}\label{eq=norm_extended}
\Nrm_{\RR}(x)\coloneqq\prod_{i\in [r_1]}|\sigma_{i,\RR}(x)| 
\prod_{j\in [r_2]}|\sigma_{r_1+j,\RR}(x)|^2 .
\end{equation}
This is the composite of the ring-theoretic norm (Remark~\ref{remark=idealnorm})
$N_{K\otimes _{\QQ }\RR / \RR} \colon K\otimes _{\QQ }\RR \to \RR $
and the absolute value $x\mapsto |x|$.
\begin{definition}\label{definition=Minkowski}
	We define the \emph{additive Minkowski embedding} 
	\[\calM_{\RR}\colon K\otimes_{\QQ}\RR\to \RR^{r_1}\times \CC^{r_2}\]
	by
	\[
	\calM_{\RR}(x)\coloneqq
	(\sigma_{1,\RR}(x),\ldots ,\sigma_{r_1,\RR}(x), \sigma_{r_1+1,\RR}(x),\ldots ,\sigma_{r_1+r_2,\RR}(x))
	\quad (x\in K\otimes_{\QQ}\RR).
	\]
	The restriction of $\calM_{\RR}$ to $K$ will be denoted by $\calM$.
	We also define the \emph{weighted multiplicative Minkowski embedding} 
	\[\calL_{\RR}\colon 
	(K\otimes_{\QQ}\RR)\setminus \{x\in K\otimes_{\QQ}\RR:\Nrm_{\RR}(x)=0 \}
	\to \RR^{r_1+r_2}\]
	by
	\begin{equation}\label{eq=calMcalL}
	\calL_{\RR}(x)\coloneqq\bflog(\calM_{\RR}(x)),
	\end{equation}
	where
	$\bflog\colon(\RR^{\times})^{r_1}\times(\CC^{\times})^{r_2}\to\RR^{r_1+r_2}$
	is defined by
	\[
	\bflog(x_1,\dots ,x_{r_1},z_1,\dots ,z_{r_2})=
	(\log |x_1|,\dots ,\log |x_{r_1}|,
	\sqrt{2}\log |z_1|,\dots ,\sqrt{2}\log |z_{r_2}|).
	\]
	In other words, the first $r_1$ coordinates of $\calL_{\RR }(x)$ are $\log |\sigma_{i,\RR}(x)|$ ($i\in [r_1]$)
	and the latter $r_2$ are $\sqrt 2 \log |\sigma_{r_1+j,\RR}(x)|$ ($j\in [r_2]$).
	See Remark~\ref{remark=standardMinkowski} for the motivation of this specific weight convention.
	The restriction of $\calL_{\RR}$ to $K^\times$ will be denoted by $\calL$.
\end{definition}
\begin{lemma}\label{lemma=Minkowski}
	The additive Minkowski embedding $\calM_{\RR}$
	is an isomorphism of $\RR$-algebras.
\end{lemma}
\begin{proof}
	This follows from \cite[Chapter~I, Proposition~5.2]{Neukirch}.
\end{proof}
By definition we see that $\calL:K^\times\to\RR^{r_1+r_2}$ is a homomorphism of groups
and that $\mu(K)\subseteq\ker\calL$.
By Lemma~\ref{lemma=muK}, we also have an inclusion $\OK \cap \ker\calL \subseteq \mu (K)$ so that $\calL $ induces an injection
$\overline{\calL}\colon\overline{\OKt}\coloneqq \OKt / \mu (K)\to\RR^{r_1+r_2}$.

\begin{definition}[The hyperplane $\HH$ and vector $u_0$]\label{definition=hyperplaneHH}
	We define a vector $u_0\in\RR^{r_1+r_2}$ by
	\[u_{0}\coloneqq(\underbrace{1,1,\dots,1}_{r_1},\underbrace{\sqrt{2},\sqrt{2},\dots,\sqrt{2}}_{r_2}).\]
	Then we define the hyperplane $\HH$ in $\RR^{r_1+r_2}$ by
	\[\HH\coloneqq\{ x \in \RR^{r_1+r_2} : \langle x, u_0 \rangle = 0\}.\]
	Here, $\langle \cdot, \cdot \rangle$ denotes the standard inner product on $\RR^{r_1+r_2}$.
	We denote by $\mathbf{P}_{\HH}$ the orthogonal projection
	from $\RR^{r_1+r_2}$ onto the hyperplane $\HH$.
\end{definition}
\begin{theorem}\label{theorem=Dirichlet}
	The image $\overline{\calL}(\overline{\OKt})$ is a lattice of full rank in
	the hyperplane $\HH$.
\end{theorem}
\begin{proof}
	This is essentially Dirichlet's unit theorem: %
	\cite[Chapter~I, Theorem~7.3]{Neukirch}.
\end{proof}
\subsection{Geometric characterization of NL-compatibility}
\label{subsection=NLCcondition}

We continue to use Setting~\ref{setting=section4}.

\begin{lemma}\label{lemma=bounded}
	Let $Z$ be a subset of $K\otimes_{\QQ}\RR$. Then 
	the following statements hold.
	\begin{enumerate}[$(1)$]
		\item\label{en:B1} The set $Z$ is bounded if and only if 
		$\calM_{\RR}(Z)\subseteq \RR ^{r_1}\times \CC ^{r_2}$ is bounded.
		\item\label{en:B2} 
		Let $Z\subseteq (K\otimes_{\QQ}\RR)\setminus 
		\{x: \Nrm_{K,\RR}(x)=0\}$. If
		$\calL_{\RR}(Z)$ is bounded, then so is $Z$.
		\item\label{en:B3} Assume that 
		$Z\subseteq (K\otimes_{\QQ}\RR)\setminus 
		\{x: \Nrm_{K,\RR}(x)=0\}$ 
		is bounded and 
		$\inf\{\Nrm_{\RR}(x) : x\in Z\}>0$. Then $\calL_\RR(Z)$ is bounded.
	\end{enumerate}
\end{lemma}
\begin{proof}
	Item \eqref{en:B1} is obvious by Lemma~\ref{lemma=Minkowski}.	
	It follows that \eqref{en:B2} and \eqref{en:B3} are equivalent to 
	the corresponding statements for subsets of $\RR ^{r_1}\times \CC ^{r_2}$.
	For $x = (x_1,\dots, x_{r_1},z_1,\dots ,z_{r_2})\in \RR ^{r_1}\times \CC ^{r_2}$, denote by $N(x)=N_{\RR ^{r_1}\times \CC ^{r_2}/\RR} (x)\in \RR $ the ring-theoretic norm $N(x)= \prod _{i\in [r_1]} x_i \prod _{j\in [r_2]} |z_j|^2 $.
	Now we have to prove:
	\begin{itemize}
		\item [\eqref{en:B2}{}$'$]
		Let $Z\subseteq (\RR^\times) ^{r_1}\times (\CC^\times)^{r_2}$. 
		If $\bflog (Z)\subseteq \RR ^{r_1+r_2}$ is bounded, then so is $Z$.
		\item [\eqref{en:B3}{}$'$]\label{en:B3'}
		Assume that 
		$Z\subseteq (\RR ^\times)^{r_1}\times (\CC ^\times)^{r_2}$ 
		is bounded and 
		$\inf\{ N(x) : x\in Z\}>0$. Then $\bflog (Z)$ is bounded.
	\end{itemize}
Let us prove \eqref{en:B2}{}$'$.
In general, if $X\subseteq \RR $ is a bounded set, then $
\{ z\in \CC ^\times : \log |z|\in X  \} $
is bounded because this set can be written as $\{ ue^{x} : u\in \CC ,\ |u|=1,\ x\in X \} $.
Hence \eqref{en:B2}{}$'$ follows.

Next we show \eqref{en:B3}{}$'$.
Write $x=(x_1,\dots ,x_{r_1+r_2})$ for elements of $\RRCC $. 
By the boundedness assumption on $Z$, there exists $C>0$ such that the following holds for all $x\in Z$:
\[|x_j| < C \quad (j\in [r_1]),  \quad |x_j|^2 <C \quad (j\in [r_1+1,r_2]). \]
Also, set $c= \inf\{ N(x) : x\in Z\}$, which is positive by assumption.
Then for every $i\in [r_1+r_2]$, if we set $a=1$ or $2$ depending on whether $i\le r_1$ or $i>r_1$, the following inequality holds for all $x\in Z $:
\[ |x_i|^a = \frac{N (x)}{\prod\limits _{j\in [r_1]\setminus \{ i\} }|x_j| \prod\limits _{j\in [r_1+1,r_2]\setminus \{ i \} } |x_j|^2  } > \frac{c}{C^{r_1+r_2-1}} .\]
Therefore there are uniform bounds on the values $\log |x_i|$ ($x\in Z$, $i\in [r_1+r_2]$) from above and below, which implies that $\bflog (Z)\subseteq \RR ^{r_1+r_2}$ is bounded.
\end{proof}
\begin{theorem}[Geometric characterization of the NL-compatibility]
	\label{theorem=NLcompatible}
	With reference to Setting~$\ref{setting=section4}$, the following conditions %
	are equivalent for $X\subseteq \OKnz$.
	\begin{enumerate}[$(i)$] 
		\item $X$ is NL-compatible;
		\item $(\mathbf{P}_{\HH}\circ \calL)(X)(\subseteq \HH)$ is
		bounded. 
		Here $\HH$ and $\mathbf{P}_{\HH}$ are defined in
		Definition~$\ref{definition=hyperplaneHH}$.
	\end{enumerate}
\end{theorem}
\begin{proof}
	Note that the $\ell_\infty$-length $\|\cdot \|_{\infty,\omom}$
	defined in Definition~\ref{definition=lmugenlength}
	can be naturally extended to $K$ with values in $\QQ$,
	and then to $K\otimes_{\QQ}\RR$ with values in $\RR$
	(as the $\lmugen $-length with respect to the basis $\omom $).
	$
	$
	This makes $K\otimes_{\QQ}\RR$ a normed vector space
	over $\RR$.
	Let
	\[
	\tilde{X}\coloneqq
	\{\Nrm (\alpha)^{-1/n}\alpha 
	: \alpha\in X\}\subseteq K\otimes_{\QQ}\RR.
	\]
	It follows from Definition~\ref{definition=normrespecting}
	that (i) is equivalent to the boundedness of 
	$\tilde{X}$.
	Since
	\begin{equation}\label{eq=tildeXhasN1}
	\Nrm_{K,\RR}(\tilde{\alpha})=1\quad(\tilde{\alpha}
	\in\tilde{X}),
	\end{equation}
	we can use Lemma~\ref{lemma=bounded}~\eqref{en:B2}
	and \eqref{en:B3} to conclude that
	$\tilde{X}$ is bounded if and only if 
	$\calL_{\RR}(\tilde{X})$ is bounded.
	Therefore, it only remains to establish
	\[
	\calL_{\RR}(\tilde{X})=
	(\mathbf{P}_{\HH}\circ\calL)(X).
	\]
	
	Observe that, for
	$t>0$ and $x\in K\otimes_{\QQ}\RR$ with $\Nrm_{K,\RR}(x)\neq0$,
	we have
	\begin{equation}\label{eq=calLlog}
	\calL_{\RR}(t\cdot x)=\calL_{\RR}(x)+(\log t)\cdot u_0. 
	\end{equation}
	By \eqref{eq=tildeXhasN1}, we have
	$\calL_{\RR}(\tilde{X})\subseteq\HH$,
	and hence
	\begin{align*}
	\calL_{\RR}(\tilde{X})&=
	\mathbf{P}_{\HH}(\calL_{\RR}(\tilde{X}))
	\\&=
	\{\mathbf{P}_{\HH}(\calL_{\RR}(\Nrm (\alpha)^{-1/n}\alpha ))
	: \alpha\in X\}
	\\&=
	\{\mathbf{P}_{\HH}(
	\calL_{\RR}(\alpha)+(\log\Nrm (\alpha)^{-1/n})u_0)
	: \alpha\in X\}
	&&\text{(by \eqref{eq=calLlog})}
	\\&=
	\{\mathbf{P}_{\HH}(\calL(\alpha)): \alpha\in X\}
	\\&=
	(\mathbf{P}_{\HH}\circ\calL)(X).
	\end{align*}
	This completes the proof.
\end{proof}

We regard $\RR^{r_1}\times\CC^{r_2}$ as a normed real vector space by 
introducing the norm 
$\|\cdot\|^{}_{\infty}$ defined by
\begin{equation}\label{eq:lmugen-length}
\|(x_1,\dots,x_{r_1},z_1,\dots,z_{r_2})\|^{}_{\infty}=
\max\{|x_1|,\ldots ,|x_{r_1}|,|z_1|,\ldots ,|z_{r_2}|\}.
\end{equation}

\begin{lemma}\label{lemma=NLCreversed}
	Define
	\begin{equation}
	\Theta\coloneqq
	\max_{i\in[n]}
	\sum_{j\in[n]}|\sigma_i(\omega_j)|.
	\label{36e}
	\end{equation}
	Then,
	for all $\alpha\in \OKnz$, we have inequalities
	\begin{align}
	\|\calM(\alpha)\|^{}_{\infty}&\leq \Theta\|\alpha\|_{\infty,\omom},
	\label{28a}\\
	\Nrm (\alpha)&\leq 
	\Theta^n\|\alpha\|_{\infty,\omom}^n.
	\label{28b}
	\end{align}
\end{lemma}
\begin{proof}
	For $\alpha\in\OK$,
	we have $|\sigma_i(\alpha)|\leq\Theta\|\alpha\|_{\infty,\omom}$
	for all $i\in[n]$. This implies \eqref{28a}.
	Then \eqref{28b} follows from \eqref{28a} and
	Lemma~\ref{lemma=idealnorm}.
\end{proof}

Let $\DD$ be an NL-compatible
$\OKt$-fundamental domain on $\OKnz$.
Then there exist constants
$C=C(\omom,\DD)>0$ and $C'=C'(\omom)>0$
such that
\begin{equation}\label{NLC}
C\|\alpha\|_{\infty,\omom}^n\leq \Nrm (\alpha)\leq C' \|\alpha\|_{\infty,\omom}^n\tag{NLC}
\qquad(\alpha\in \DD).
\end{equation}
\renewcommand*{\theHequation}{notag.\theequation} %
Indeed, the existence of $C$ follows from 
Definition~\ref{definition=normrespecting},
while that of $C'$ is ensured by Lemma~\ref{lemma=NLCreversed}.
The inequality \eqref{NLC} will be used frequently in
Sections~\ref{section=positiveweighteddensity} 
and \ref{section=slidetrick}.

\begin{remark}\label{remark=standardMinkowski}
	In a conventional definition of the multiplicative Minkowski embedding,
	the coefficients of the logarithm of the imaginary embeddings are
	$2$, instead of $\sqrt{2}$ in our Definition~\ref{definition=Minkowski}.
	The motivation for the coefficients $\sqrt 2$
	is as follows.
	With any coefficients, if we define $\HH $ such that we have $\calL (\OKt )\subseteq \HH $, and $u_0$ such that \eqref{eq=calLlog}
	holds, 
	then our discussions so far work as well.
	In this case, we define $\mathbf{P}_{\HH} \colon \RR ^{r_1+r_2}\to \HH $ as the projection associated with the decomposition $\RR^{r_1+r_2}=\HH \oplus \RR u_0$.
	The advantage of our convention is that $u_0$ is orthogonal to $\HH $ with respect to the standard inner product of $\RR ^{r_1+r_2}$. 
	This makes $\mathbf{P}_{\HH}$ an orthogonal projection and makes our treatment in Subsection~\ref{subsection=OKt_orbit} slightly easier to write down.
\end{remark}
\subsection{Construction of the domain $\DD_K(\ee,\sigma)$}
\label{subsection=DDKee}
We continue to use Setting~\ref{setting=section4}.
We fix fundamental units
$\ee=(\varepsilon_1,\dots ,\varepsilon_{r_1+r_2-1})$, and 
for each $i\in[r_1+r_2-1]$, 
define
$u_i\coloneqq\overline{\calL}(\overline{\varepsilon_i})$.
Then $u_1,u_2,\dots,u_{r_1+r_2-1}$ form a basis of
$\HH$, and hence 
$u_0,u_1,\dots,u_{r_1+r_2-1}$ form a basis of
$\RR^{r_1+r_2}$.

In this subsection, we construct an example of an NL-compatible $\OKt$-fundamental domain.
Its construction is classical (see for instance \cite[\S 40]{Hecke})
and in fact related to Proposition \ref{proposition=idealdensity} on the density of ideals.

First, using the chosen fundamental units $\ee $, we construct a preliminary domain $\widetilde{\DD}_K(\ee )\subseteq \OKnz $, which is almost an $\OKt$-fundamental domain except that it is stable under the action of $\mu (K)$.
Given an embedding
$\sigma\colon K\hookrightarrow \CC$, we then
construct a fundamental domain $\DD_K(\ee,\sigma)$
for the action 
$\mu(K)\curvearrowright \widetilde{\DD}_K(\ee)$.
The $\OKt$-fundamental domain $\DD_K(\ee,\sigma)$ will turn out to be
NL-compatible in Proposition~\ref{proposition=normrespectingfundamentaldomain}.

%
\begin{definition}[The fundamental domain $\DD(\ee,\sigma)$]
	\label{definition=fundamentaldomainDDee}
	We define
	\[
	\calC_K(\ee)=\left\{\Biggl(\sum_{i\in[r_1+r_2-1]}y_iu_i\Biggr)+y_0u_0 : y_i\in [0,1) \ (i\in[r_1+r_2-1]), \ y_0\in\RR\right\}.
	\]
	\begin{enumerate}[(1)] 
		\item\label{en:tildeD} 
		We define a subset $\widetilde{\DD}_K(\ee)$
		of $\OKnz$ as
		\[
		\widetilde{\DD}_K(\ee)\coloneqq\calL_{\RR}^{-1}(\calC_K(\ee))\cap( \OKnz) .
		\]
		\item\label{en:tildeDtoD} 
		For an embedding  
		$\sigma\colon K\hookrightarrow \CC$, 
		we define a subset
		$\DD_K(\ee,\sigma)$ of $\OKnz$ as
		\[
		\DD_K(\ee,\sigma)\coloneqq\left\{\alpha\in \widetilde{\DD}_K(\ee) : 0\leq \mathrm{arg}(\sigma(\alpha))<\frac{2\pi}{\# \mu(K)}\right\},
		\]
		where `$\mathrm{arg}$' denotes the argument of a non-zero complex number.
	\end{enumerate}
\end{definition}
\begin{proposition}\label{proposition=normrespectingfundamentaldomain}
	For an embedding $\sigma\colon K\hookrightarrow \CC$,
	the set $\DD_K(\ee,\sigma)\subseteq\OKnz$ is an
	$\OKt$-fundamental domain
	which is NL-compatible.
\end{proposition}
\begin{proof}
	By Definition~\ref{definition=fundamentaldomainDDee}~\eqref{en:tildeD},
	we have
	\begin{align*}
	(\mathbf{P}_{\HH}\circ \calL)(\widetilde{\DD}_K(\ee))
	&\subseteq 
	\calC_K(\ee)\cap \HH.
	\end{align*}
	Since $\calC_K(\ee)$ is bounded,
	this implies that
	$(\mathbf{P}_{\HH}\circ \calL)(\widetilde{\DD}_K(\ee))$
	is bounded. Then by Theorem~\ref{theorem=NLcompatible},
	$\widetilde{\DD}_K(\ee)$ is NL-compatible.
	Since 
	$\DD_K(\ee,\sigma)\subseteq\widetilde{\DD}_K(\ee)$,
	$\DD_K(\ee,\sigma)$ is also NL-compatible.
	
	Next we show that $\DD_K(\ee,\sigma)$ is an 
	$\OKt$-fundamental domain.
	It is clear from 
	Definition~\ref{definition=fundamentaldomainDDee} that 
	$\DD_K(\ee,\sigma)\subseteq\OKnz$.
	Since $\mu(K)$ is a cyclic group,
	we can write $\sigma(\mu(K))=\langle\zeta\rangle$, where
	$\zeta=e^{2\pi\sqrt{-1}/(\#\mu(K))}\in\CC$.
	Thus
	\begin{align}
	\sigma(\widetilde{\DD}_K(\ee))&=
	\bigsqcup_{\xi\in\mu(K)}\sigma(\xi\DD_K(\ee,\sigma)).
	\label{4.3a2}
	\end{align}
	Since $\RR^{r_1+r_2}=\bigsqcup_{\overline{\eta}\in\OKbar}
	(\overline{\calL}(\overline{\eta})+\calC_K(\ee))$,
	we have
	\[
	\OO_K\setminus\{0\}=
	\bigsqcup_{\overline{\eta}\in\OKbar}
	\calL^{-1}(\overline{\calL}(\overline{\eta})+\calC_K(\ee))
	\cap(\OO_K\setminus\{0\})
	=\bigsqcup_{\overline{\eta}\in\OKbar}
	\eta \bigsqcup_{\xi\in\mu(K)}\xi\DD_K(\ee,\sigma)=
	\bigsqcup_{\eta\in\OKt}
	\eta\DD_K(\ee,\sigma)
	\]
	by \eqref{4.3a2}.
	Therefore, $\DD_K(\ee,\sigma)$ is an
	$\OKt$-fundamental domain.
\end{proof}

Note that $\calL_{\RR}(\DD_K(\ee,\sigma))=\calC_K(\ee)$.

\begin{example}\label{example=sqrt10}
	For $K=\QQ(\sqrt{2})$, the ring of integers is
	$\OK=\ZZ[\sqrt{2}]$, and we may take fundamental units to be
	$\ee=(1+\sqrt{2})$. 
Then,
	\begin{align*}
	\calC_K(\ee)&=
	\{y_0(1,1)+y_1(\log(\sqrt{2}+1),\log(\sqrt{2}-1)):y_0\in\RR,\;y_1\in[0,1)\}
	\\&=
	\{(x,y):x,y\in\RR,\;0\leq x-y<\log(3+2\sqrt{2})\},\\
	\intertext{and}
	\calL_{\RR}^{-1}(\calC_K(\ee))
	&=
	\{a+b\sqrt{2}:a,b\in\RR,\;a>2b\geq0\text{ or }b>a\geq0\}
	\\&\quad\cup
	\{a+b\sqrt{2}:a,b\in\RR,\;a<2b\leq0\text{ or }b<a\leq0\}.
\end{align*}
Define an embedding
	$\sigma\colon K\hookrightarrow \CC$ by
	$\sigma(x+y\sqrt{2})=x+\sqrt{2}y$ ($x,y\in\QQ$).
Then, we have that $\DD_K(\ee,\sigma)=\{a+b\sqrt{2}:a,b\in\ZZ,\;a>2b\geq0\text{ or }b>a\geq0\}$.
In  Figure~\ref{figure=1}, we illustrate
	the $\OKt$-fundamental domain $\DD_{K}(\ee,\sigma)$ and its image $\calC_K(\ee)$ under $\calL_{\RR}$.
\end{example}
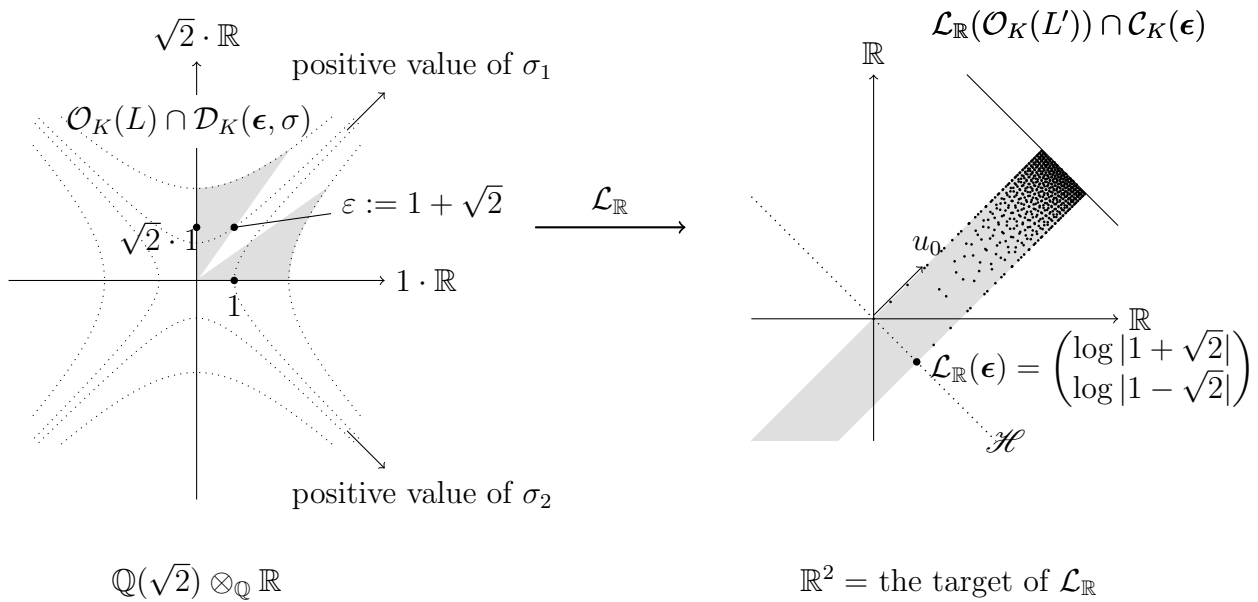
\begin{figure}[htbp]
	\centering
	\input{Fig1_12_31.tex}
	\caption{The case $K=\mathbb{Q}(\sqrt{2})$ with
		$L=6$, $L'=1000$.}
	\label{figure=1}
\end{figure}

The following statement explains why we need the notion of NL-compatibility.
\begin{proposition}\label{proposition=badchoicedomain}
	Assume the unit group of $K$ is infinite, equivalently,
	$r_1+r_2\geq 2$. 
	Let $S\subseteq\OK$ be a finite subset with
	$\#S\geq  3$. Then there exists an
	$\OKt$-fundamental domain 
	$\DD_S\subseteq \OKnz$ such that
	$\DD_S$ contains no $S$-constellation.
\end{proposition}
\begin{proof}
	By Lemma~\ref{lemma=muK}, we have
	$(\OKnz)\cap\ker\calL=\mu(K)$.
	Since $\OKt$ is infinite, 
	there exists 
	$\varepsilon\in\OKt\setminus\ker\calL$.
	We claim, for all $\alpha\in\OKnz$,
	\begin{equation}\label{4.3b1}
	\lim_{m\to\infty}\|\calM_{\RR}(\varepsilon^m\alpha)\|^{}_{\infty}=\infty.
	\end{equation}
	Indeed, since $\varepsilon\in\OKt$, we have
	$1=\Nrm(\varepsilon)$,
	while $\calL(\varepsilon)\neq0$ implies that 
	$|\sigma_i(\varepsilon)|\neq1$ for some $i\in[n]$.
	Thus, there exists $i_0\in[r_1+r_2]$ such that
	$|\sigma_{i_0}(\varepsilon)|>1$. Then for $m\in\NN$,
	\[\|\calM_{\RR}(\varepsilon^m\alpha)\|^{}_{\infty}
	\geq
	|\sigma_{i_0}(\varepsilon^m\alpha)|
	=
	|\sigma_{i_0}(\varepsilon)|^m|\sigma_{i_0}(\alpha)|
	\to\infty\quad(m\to\infty).\]
	Define
	\[
	\calR\coloneqq\max_{\{s_1,s_2,s_3\}\in\binom{S}{3}}
	\frac{\|\calM_{\RR}(s_3)-\calM_{\RR}(s_2)\|^{}_{\infty}}{\|\calM_{\RR}(s_2)-\calM_{\RR}(s_1)\|^{}_{\infty}},
	\]
	where $\binom{S}{3}$ is the family of all three-element subsets of $S$. 	Note that the denominator above is never zero by Lemma~\ref{lemma=Minkowski}.

	We enumerate the elements of the countable set 
	$(\OKnz)/\OKt$ as $\{\upsilon_m : m\in\NN\}$.
	Fix $\alpha_1$ for a representative for $\upsilon_1$.
	By \eqref{4.3b1}, the image of the equivalence class $\upsilon_m\subseteq \OKnz $
	under the embedding $\calM_{\RR}$ is unbounded.
	Thus, we may take a sequence $(\alpha_m)_{m \in \NN}$, chosen inductively on $m$, such that for all $m\in \NN$, $\alpha_m$ is a representative of $\upsilon_m$ and 
	\begin{equation}\label{eq=alpham}
	\|\calM_{\RR}(\alpha_{m+1})\|^{}_{\infty}\geq
	(2\mathcal{R}+2) \|\calM_{\RR} (\alpha_{m})\|^{}_{\infty}
	\end{equation}
	holds.
	Let $\DD_S\coloneqq\{\alpha_m:m\in\NN\}$.
	By construction, $\DD_S$ is an $\OKt$-fundamental domain.
	
	We will show that $\DD_S$ contains no $S$-constellation.
	Suppose, by way of contradiction, $\DD_S$ contains an $S$-constellation
	$\eS$. 
	By Lemma~\ref{lemma=Minkowski}, we have
	\begin{equation}\label{4.3b3}
	\calR=\max_{\{s'_1,s'_2,s'_3\}\in\binom{\eS}{3}}\frac{\|\calM_{\RR}(s'_3)-\calM_{\RR}(s'_2)\|^{}_{\infty}}{\|\calM_{\RR}(s'_2)-\calM_{\RR}(s'_1)\|^{}_{\infty}}.
	\end{equation}
	Let $\beta_1,\beta_2,\beta_3\in\eS$ be distinct.
	Then, there exist distinct $m_1,m_2,m_3\in\NN$ such that
	$\beta_j=\alpha_{m_j}$ for $j=1,2,3$. We may assume
	without loss of generality that $m_1<m_2<m_3$.
	Since
	\begin{align*}
	\|\calM_{\RR}(\beta_3)-\calM_{\RR}(\beta_2)\|^{}_{\infty}
	&\geq
	\big|\|\calM_{\RR}(\alpha_{m_3})\|^{}_{\infty}-\|\calM_{\RR}(\alpha_{m_2})\|^{}_{\infty}\bigr|
	\\&\geq
	\bigl((2\calR+2)^{m_3-m_2}-1\bigr)\|\calM_{\RR}(\beta_2)\|^{}_{\infty}
	&&\text{(by \eqref{eq=alpham})}
	\\&\geq
	(2\calR+1)\|\calM_{\RR}(\beta_2)\|^{}_{\infty},
	\end{align*}
	and
	\begin{align*}
	\|\calM_{\RR}(\beta_2)-\calM_{\RR}(\beta_1)\|^{}_{\infty}
	&\leq 
	\|\calM_{\RR}(\beta_2)\|^{}_{\infty}
	+\|\calM_{\RR}(\beta_1)\|^{}_{\infty}
	\\&\leq 
	\|\calM_{\RR}(\beta_2)\|^{}_{\infty}
	+\frac{1}{2\calR+2}\|\calM_{\RR}(\beta_2)\|^{}_{\infty}
	&&\text{(by \eqref{eq=alpham})}
	\\&<
	2\|\calM_{\RR}(\beta_2)\|^{}_{\infty},
	\end{align*}
	we have
	\[
	\frac{\|\calM_{\RR}(\beta_3)-\calM_{\RR}(\beta_2)\|^{}_{\infty}}{\|\calM_{\RR}(\beta_2)-\calM_{\RR}(\beta_1)\|^{}_{\infty}}
	>\frac{2\calR+1}{2}
	>\calR.
	\]
	This contradicts \eqref{4.3b3}.
\end{proof}
\subsection{Counting elements in $\OKt$-orbits with respect to the
	$\lmugen$-length}
\label{subsection=OKt_orbit}

In this subsection,
we give an estimate on the size of 
subsets of the orbit
$\OKt\cdot\alpha$ %
truncated by $\lmugen$-length.
More precisely,
for $M\geq1$, 
in Lemma~\ref{lemma=OKt_orbit}
we give an upper bound
on $\#( (\OKt\cdot \alpha )\cap \OK(\omom,M))$.

The results in this subsection are not needed to
prove Theorem~\ref{theorem=primeconstellationsfinite} because
the mapping $\OKnz\ni \alpha \mapsto \alpha\OK\in \Ideals_K$
restricted to an $\OKt$-fundamental domain is injective.
However, to prove Theorem~\ref{mtheorem=primeconstellationsfinite}, where we are no longer in an $\OKt$-fundamental domain,
an estimate as above is in addition required.

We continue to use Setting~\ref{setting=section4}. 
Recall from Definition~\ref{definition=Minkowski}
the additive and weighted multiplicative
Minkowski embeddings, and from Definition~\ref{definition=hyperplaneHH}
the hyperplane $\HH$ and vector $u_0$.
Further, recall the $\ell_\infty$-length on 
$\RR^{r_1}\times\CC^{r_2}$ defined in \eqref{eq:lmugen-length},
and the constant $\Theta$ defined in \eqref{36e}.
We define the following sets:
\[
\scrQ\coloneqq(-\infty,0]^{r_1+r_2},\quad\calT\coloneqq\frac{1}{n}(u_0+\scrQ)\cap\HH.
\]

The following lemma is a key to counting the number of
associates of a given element $\alpha\in\OKnz$
in $\OK(\omom,M)$. The assumption 
$\Nrm(\alpha)\leq\Xi M^n$ in 
Lemma~\ref{lemma=OKt_orbit}~\eqref{en:precisecounting}
is not essential, as is seen from Lemma~\ref{lemma=NLCreversed}.
We write $k\coloneqq r_1+r_2-1$.

\begin{lemma}
	\begin{enumerate}[$(1)$]
		\item\label{en:precisecounting}
		There exists $\Xi>0$ depending only on $\omom$, such that 
		for all $M\geq1$ and $\alpha\in \OKnz$ with
		$\Nrm(\alpha)\leq {\Xi} M^n$,
		\[
		\#(\OKt\cdot \alpha \cap \OK(\omom,M))\leq \Xi\cdot 
		\left(\log\frac{\Xi M^n}{\Nrm(\alpha)}\right)^{k}
		\]
		holds.
		\item\label{en:roughcounting}
		There exists $\Xi'>0$ depending only on $\omom$, such that 
		for all $M\geq2$ and $\alpha\in \OKnz$,
		\[
		\#(\OKt\cdot \alpha \cap \OK(\omom,M))\leq \Xi'\cdot (\log M)^{k}
		\]
		holds.
	\end{enumerate}
\label{lemma=OKt_orbit}
\end{lemma}

Lemma~\ref{lemma=OKt_orbit} follows from the classical fact
that the number of lattice poins in a large enough, well-behaved bounded subset of a Euclidean space such as a convex body can be approximated by its volume, and in particular bounded by a multiple of the volume.
For the convenience of the reader
we explain the proof of Lemma~\ref{lemma=OKt_orbit} in a self-contained manner.
The proof of Lemma~\ref{lemma=OKt_orbit} will be provided
after the following auxiliary lemma.
In these proof, see Figure~\ref{figure:tohu}.
\begin{figure}[htbp]
	\centering
	\input{Tohu_12_31.tex} 
	\caption{Images under $\mathcal{L}$ and areas in Lemmas~\ref{lemma=OKt_orbit} and \ref{lem:36a}}
\label{figure:tohu}
\end{figure}
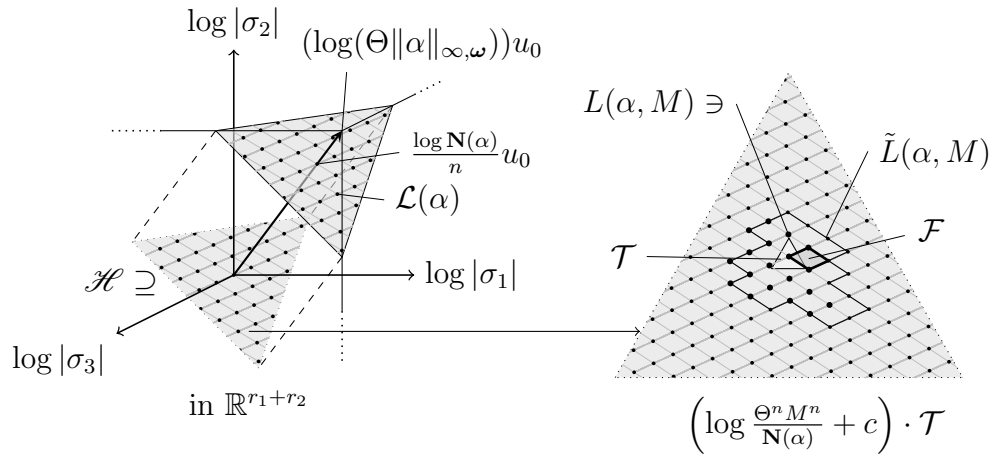

\begin{lemma}\label{lem:36a}
	For $\alpha\in\OKnz$, the following statements hold.
	\begin{enumerate}[$(1)$]
		\item\label{36a-2} 
		$\calL(\alpha)\in(\log(\Theta\|\alpha\|_{\infty,\omom}))
		u_0+\scrQ$,
		\item\label{36a-3} 
		Let $M>0$ and $\alpha\in\OK(\Theta^nM^n)$.
		Then for every $\beta\in\OKt\cdot\alpha\cap\OK(\omom,M)$, we have
		\[\calL(\beta)-\frac{\log\Nrm(\alpha)}{n}u_0
		\in\left(\log\frac{\Theta^n M^n}{\Nrm(\alpha)}\right)
		\cdot\calT.\]
	\end{enumerate}
\end{lemma}
\begin{proof}
	Item \eqref{36a-2} follows from \eqref{28a}.
	Since $\alpha\in\OK(\Theta^nM^n)$, we have
	\[(\log(\Theta^n M^n)-\log\Nrm(\alpha))\scrQ=\scrQ.\]
	Thus, by \eqref{36a-2}, we have
	\begin{align*}
	\calL(\beta)-\frac{\log\Nrm(\alpha)}{n}u_0&
	\in
	\left(\frac{1}{n}(\log(\Theta^n M^n)-{\log\Nrm(\alpha)})u_0+\scrQ\right)\cap\HH
	\\&=
	\frac{1}{n}(\log(\Theta^n M^n)-{\log\Nrm(\alpha)})
	((u_0+\scrQ)\cap\HH)
	\\&=
	\left(\log\frac{\Theta^n M^n}{\Nrm(\alpha)}\right)
	\cdot\calT.
	\end{align*}
	This completes the proof.
\end{proof}

\begin{proof}[Proof of Lemma~$\ref{lemma=OKt_orbit}$]
	Pick a relatively compact Borel measurable
	fundamental domain $\mathcal{F}$ of the lattice
	$\calL(\OKt)$ in $\HH$. 
	Since $0$ is an interior point of $\calT$ in $\HH$,
	we may choose $c>0$ in such a way that
	$\mathcal{F}\subseteq c\cdot\calT$.
	Let
	\begin{align*}
	\Xi&\coloneqq \max
	\left\{\#\mu(K)\cdot\frac{\lebesgue^{(k)}(\calT)}{\lebesgue^{(k)}(\mathcal{F})},\ e^c\Theta^n\right\},
	\end{align*}
	where $\lebesgue^{(k)}$ denotes the 
	$k$-dimensional Lebesgue measure.
	
	If $\Nrm(\alpha)>\Theta^n M^n$, then 
	$\OKt\cdot \alpha \cap \OK(\omom,M)=\varnothing$
	by Lemma~\ref{lemma=NLCreversed}.
	Thus, we assume $\Nrm(\alpha)\leq\Theta^n M^n$.
	
	We define $L(\alpha,M)$ and $\tilde{L}(\alpha,M)$ by 
	\begin{align}
	L(\alpha,M)&\coloneqq \calL(\OKt\cdot \alpha \cap \OK(\omom,M))
	-\frac{\log \Nrm(\alpha)}{n}u_0,\notag\\
	\tilde{L}(\alpha,M)
	&\coloneqq 
	\bigsqcup_{v\in L(\alpha,M)}
	(v+\mathcal{F}).
	\label{38d1}
	\end{align}
	Since
	\[L(\alpha,M)\subseteq\calL(\OKt)+\calL(\alpha)-
	\frac{\log \Nrm(\alpha)}{n}u_0\subseteq\HH \]
	and $\mathcal{F}$ is a fundamental domain of the lattice
	$\calL(\OKt)$ in $\HH$, the union in the right-hand side
	of \eqref{38d1} is indeed disjoint. 
	Since the restriction of $\calL$ to $\OKnz$ has kernel
	$\mu(K)$, we have
	\begin{equation}\label{eq:countingcalL}
	\#(\OKt\cdot \alpha \cap \OK(\omom,M))
	=\#\mu(K)\cdot\#L(\alpha, M).
	\end{equation}
	By Lemma~\ref{lem:36a}~\eqref{36a-3}, we have
	\[
	L(\alpha,M)\subseteq
	\left(\log\frac{\Theta^n M^n}{\Nrm(\alpha)}\right)
	\cdot\calT.
	\]
	Hence
	\[
	\tilde{L}(\alpha,M)=L(\alpha,M)+\mathcal{F}
	\subseteq
	\left(\log\frac{\Theta^n M^n}{\Nrm(\alpha)}+c\right)
	\cdot\calT,
	\]
	since $\calT$ is convex.
	Therefore, by \eqref{38d1}, we obtain
	\[
	\#L(\alpha,M)\cdot\lebesgue^{(k)}(\mathcal{F})=
	\lebesgue^{(k)}(\tilde{L}(\alpha,M))
	\leq
	\lebesgue^{(k)}(\calT)
	\left(\log\frac{\Theta^n M^n}{\Nrm(\alpha)}+c\right)^{k}.
	\]
	This, together with
	\eqref{eq:countingcalL} implies the inequality in \eqref{en:precisecounting}.
	
	To prove \eqref{en:roughcounting}, let $\Xi'\coloneqq
	(\log_2 \Xi+n)^{k} \cdot \Xi$.
	Assuming $M\geq2$, we have
	$\log(\Xi M^n)\leq(\log_2\Xi+n)\log M$.
	Since $\Nrm(\alpha)\geq1$, the desired inequality follows from
	\eqref{en:precisecounting}.
\end{proof}
The following variant, Corollary~\ref{corollary=OKt_orbit_ideal}, will be used to prove our results on quadratic forms in Section~\ref{section=quadraticform}
via the machinery in Sections~\ref{section=maintheoremfull} and \ref{section=slidetrick}.

For a $\ZZ$-basis $\bv=(v_1,\dots,v_n)$ 
of an ideal $\ideala\in\Ideals_K$,
write
\[
v_j=\sum_{i\in[n]}c_{ij}\omega_i\quad(j\in[n]),
\]
where $c_{ij}\in\ZZ$ for $i,j\in[n]$.
Define
\[C_{\bv}
\coloneqq\max_{i\in[n]}
\sum_{j\in[n]}|c_{ij}|.\]
Recall from Definition~\ref{definition=lmugenlength}
the $\lmugen$-length $\|\cdot\|_{\infty,\bv}$ and the set
$\ideala(\bv,M)$ for $M\geq0$.

\begin{corollary}\label{corollary=OKt_orbit_ideal}
	Let $\ideala\in \Ideals_K$, and let 
	$\bv$ be a $\ZZ$-basis of $\ideala$.
	\begin{enumerate}[$(1)$]
		\item\label{en:precisecounting_ideal}
		There exists a constant $\Xi(\bv)>0$
		depending on $\bv$ such that
		for all $M\in\RR_{\geq 2}$ and for all
		$\alpha\in \ideala \cap \OK(\Xi(\bv) M^n)$, the inequality
		\[
		\#(\OKt\cdot \alpha \cap \ideala(\bv,M))\leq 
		\Xi(\bv)\cdot 
		\left(\log \frac{\Xi(\bv) M^n}{\Nrm(\alpha)}\right)^{k}
		\]
		holds. 
		\item\label{en:roughcounting_ideal} 
		There exists a constant $\Xi'(\bv)>0$
		depending on $\bv$ such that,
		for all $M\in\RR_{\geq 2}$ and for all
		$\alpha\in \ideala\setminus\{0\}$, the inequality
		\[
		\#(\OKt\cdot \alpha \cap \ideala(\bv,M))
		\leq \Xi'(\bv)\cdot (\log M)^{k}
		\]
		holds.
	\end{enumerate}
\end{corollary}
\begin{proof}
	Let $\Xi(\bv)\coloneqq\Xi C_{\bv}^n$.
	Let $M\geq1$ and $\alpha\in\ideala\setminus\{0\}$.
	Since
	$\ideala(\bv,M)\subseteq \OK(\omom,C_{\bv}M)$,
	we have
	\begin{equation}
	\#(\OKt\cdot \alpha \cap \ideala(\bv,M))\leq 
	\#(\OKt\cdot \alpha \cap \OK(\omom,C_{\bv}M)).
	\label{40g3}
	\end{equation}
	To prove \eqref{en:precisecounting_ideal}, suppose further that
	$\alpha\in\OK(\Xi(\bv)M^n)=\OK(\Xi (C_{\bv} M)^n)$.
	Since $C_{\bv}\geq1$, the desired inequality follows from
	Lemma~\ref{lemma=OKt_orbit}~\eqref{en:precisecounting}
	and \eqref{40g3}.
	
	To prove \eqref{en:roughcounting_ideal}, let $\Xi'(\bv)\coloneqq\Xi'\cdot(\log_2 C_{\bv}+1)^{k}$ and assume $M\geq2$.
	Then we have
	$\Xi'\cdot(\log(C_{\bv}M))^{k}\leq
	\Xi'(\bv)(\log M)^{k}$.
	Now the desired inequality follows from
	Lemma~\ref{lemma=OKt_orbit}~\eqref{en:roughcounting}
	and \eqref{40g3}.
\end{proof}

%% file: Fig1_12_31.tex
\newcommand{\myD}{2} 
\newcommand{\mynorm}{6}
\newcommand{\mySize}{3}
\def\orenoscale{0.5}
    \begin{tikzpicture}
        \begin{scope}[xscale=\orenoscale, yscale=\orenoscale*sqrt(\myD)]
        
        \newcommand{\horiRate}{1.1} 
        \newcommand{\horiRateBis}{1.2} 
        \newcommand{\vertiRate}{0.85}
        \newcommand{\vertiRateBis}{1}
        
        \newcommand{\newi}{sqrt(\mynorm)*\i/20}
        \filldraw [lightgray!50] (0,0) \foreach \i in {0,1,...,20} { -- ( {\newi }, {sqrt(\mynorm +\newi*\newi)/sqrt(\myD)} ) }-- (0,0);
        \filldraw [lightgray!50] (0,0) \foreach \i in {0,1,...,20} { -- ( {sqrt(\mynorm +\newi*\newi)}, {\newi /sqrt(\myD)} )} -- (0,0);

        \coordinate (eps) at (1,1); 
        \node [fill=black,inner sep=1pt,shape=circle] at (eps) {};
        \node[fill=white] (label-eps) at (6,1.5) {$\varepsilon :=1+\sqrt{2}$};
        \draw (eps)--(label-eps);

        \node [fill=black,inner sep=1pt,shape=circle, label=below:$1$] at (1,0) {};
        \node [fill=black,inner sep=1pt,shape=circle](root-d) at (0,1) {};
        \node [] at ($(root-d)-(1,0.2)$) {$\sqrt{\myD}\hspace{-0pt}\cdot\hspace{0pt}1$};

        \coordinate (0) at (0,,0); 
        \coordinate [label=right:$ 1\cdot \mathbb R$] (1) at (\mySize +2,0);
        \coordinate [label=above:$\sqrt{\myD} \cdot \mathbb R$] (2) at (0,{\mySize /sqrt(\myD)+2}); 
         \coordinate (3) at (-\mySize -2,0); 
        \coordinate (4) at (0,{-\mySize / sqrt(\myD)-2});
        
        \draw[->] (3) -- (1);
        \draw[->] (4) -- (2);
        
        \node[fill=white] at (\mySize/2-1.7,\mySize) {$\mathcal O_K(L) \cap \DD_K(\ee,\sigma)$};

        \newcommand{\drawhyperbolas}[3]{ 
        \draw[samples=100, domain=-\mySize/sqrt(\myD)*#3:\mySize/sqrt(\myD)*#3, variable=\y #2] plot( {sqrt((#1 +\myD*((\y)^2) ) ) },\y );
        \draw[samples=100, domain=-\mySize/sqrt(\myD)*#3:\mySize/sqrt(\myD)*#3, variable=\y #2] plot( -{sqrt((#1 +\myD*((\y)^2) ) ) },\y ); 
        \draw[samples=100, domain=-\mySize*#3:\mySize*#3 #2] plot( \x, {sqrt((#1 +(\x)^2 )/ \myD ) } ); 
        \draw[samples=100, domain=-\mySize*#3:\mySize*#3 #2] plot( \x, -{sqrt((#1 +(\x)^2 )/\myD ) } ); 
        }

        \drawhyperbolas{1}{,dotted}{1.4};
        \drawhyperbolas{\mynorm}{,dotted}{1.2};
        \draw[->](\mySize+1,{(\mySize+1)/sqrt(\myD)} )--++(1,{1/sqrt(\myD)}) node at ++(1,0.5)    {positive value of $\sigma _1$};
        \draw[->] (\mySize+1,{(-\mySize-1)/sqrt(\myD)})--++(1,{-1/sqrt(\myD)}) node at ++(1,-0.5) {positive value of $\sigma _2$};;

		\draw[thick, ->] (\mySize +6,1)--(\mySize +10,1);        
		\node at (\mySize +8,1.5) {$\mathcal L_{\mathbb R}$};
	\end{scope}

    \begin{scope}[xshift=9cm, yshift=-0.51cm, scale=0.65]
	    \node at (4,6) {$\mathcal L_{\mathbb R}(\mathcal O_K (L')) \cap \mathcal C_K (\boldsymbol{\epsilon})$};
        \newcommand{\upperr}{5}
        \newcommand{\lowerr}{-2.5}
        \input{root_2_points_new2}
        
		\coordinate [label=right:$\mathbb R$] (1) at (\upperr ,0);
        \coordinate [label=above:$\mathbb R$] (2) at (0,\upperr );
        \coordinate (3) at (\lowerr ,0);
        \coordinate (4) at (0,\lowerr );

        \draw[->] (3)--(1); 
        \draw[->] (4)--(2); 

        \draw[dotted,line width=0.51pt] (\lowerr ,-\lowerr )--(-\lowerr ,\lowerr ) node at (-\lowerr +0.2,\lowerr ) {$\mathscr H$};

        \def\zurashi{0.07}
        \draw[->] (0,0+\zurashi)--(1,1+\zurashi) node at (1.1,1.5) {$u_0$}; 

        \coordinate (log) at ({ln(1+sqrt(2))},{ln(sqrt(2)-1)}); 
		\node at ($(log)+(3.6,-0.15)$) {$\mathcal L_\mathbb{R}(\ee)=\begin{pmatrix}\log |1+\sqrt{2}|\\ \log |1-\sqrt{2} | \end{pmatrix}$}; 
             \node [fill=black,shape=circle,inner sep=1pt] at (log) {};               
            \end{scope}

            \node at (0,-4) {$\mathbb Q(\sqrt{2})\otimes _{\mathbb Q}\mathbb R$};
            \node at (10,-4) {$\mathbb R^2=$ the target of $\mathcal L_{\mathbb R}$};
    \end{tikzpicture}

%% file: Tohu_12_31.tex
\begin{tikzpicture}
\begin{scope}[xscale=0.7,yscale=0.7, xshift = 300 , yshift = 10pt]
	\coordinate (O) at (0,0) ;
	\coordinate (N) at (0,3.5) ;
	\coordinate (SE) at (-3.3,-2.3) ;
	\coordinate (SW) at (3.3,-2.3) ;
	\def\Tscale{0.1}
	\coordinate (n) at ($\Tscale*(N)$) ;
	\coordinate (se) at ($\Tscale*(SE)$) ;
	\coordinate (sw) at ($\Tscale*(SW)$) ;
	
	\def\FirstVector{0.37,0.17}
	\def\SecondVector{0.38,-0.25}
	\def\LattticeOrigin{O} 
				
	\begin{scope}
		\clip (N)--(SE)--(SW)--cycle;
	
		\foreach \x in {-10,-9,...,10}{
		\foreach \y in {-10,-9,...,10}{
			\draw[very thin,color=gray!50] ($10 *(\FirstVector )+\y *(\SecondVector )+(\LattticeOrigin )$)--($-10 *(\FirstVector )+\y *(\SecondVector )+(\LattticeOrigin )$);
		}
			\draw[very thin,color=gray!50] ($\x *(\FirstVector )-10*(\SecondVector )+(\LattticeOrigin )$)--($\x *(\FirstVector )+10*(\SecondVector )+(\LattticeOrigin )$);
		}
			
		\foreach \x in {-10,-9,...,10}{
		\foreach \y in {-10,-9,...,10}{
			\node[circle,fill=black,inner sep=0.5pt] at ($\x *(\FirstVector ) + \y *(\SecondVector ) +(\LattticeOrigin )$) {};
		}
		}
	\end{scope}	
	
	\draw ($-2*(\FirstVector)+3*(\SecondVector)$)
	--($-2*(\FirstVector)+4*(\SecondVector)$)
	--($-1*(\FirstVector)+4*(\SecondVector)$)
	--($0*(\FirstVector)+4*(\SecondVector)$)
	--($0*(\FirstVector)+3*(\SecondVector)$)
	--($0*(\FirstVector)+2*(\SecondVector)$)
	--($1*(\FirstVector)+2*(\SecondVector)$)
	--($1*(\FirstVector)+1*(\SecondVector)$)
	--($2*(\FirstVector)+1*(\SecondVector)$)
	--($2*(\FirstVector)+0*(\SecondVector)$)
	--($2*(\FirstVector)-1*(\SecondVector)$)
	--($2*(\FirstVector)-2*(\SecondVector)$)
	--($1*(\FirstVector)-2*(\SecondVector)$)
	--($1*(\FirstVector)-2*(\SecondVector)$)
	--($0*(\FirstVector)-2*(\SecondVector)$)
	--($0*(\FirstVector)-1*(\SecondVector)$)
	--($-1*(\FirstVector)-1*(\SecondVector)$)
	--($-2*(\FirstVector)-1*(\SecondVector)$)
	--($-2*(\FirstVector)+0*(\SecondVector)$)
	--($-3*(\FirstVector)+0*(\SecondVector)$)
	--($-3*(\FirstVector)+1*(\SecondVector)$)
	--($-3*(\FirstVector)+2*(\SecondVector)$)
	--($-3*(\FirstVector)+2*(\SecondVector)$)
	--($-2*(\FirstVector)+2*(\SecondVector)$)--cycle;
	
	\foreach \x /\y in {-3/0,-2/0,-1/0,0/0,1/0,-1/-1,0/-1,1/-1, 0/-2, 1/-2,  -3/1,-2/1,-1/1,0/1, -2/2,-1/2, -2/3,-1/3,-2/-1}{
			\node[circle,fill=black,inner sep=0.8pt] at ($\x *(\FirstVector ) + \y *(\SecondVector ) +(\LattticeOrigin )$) {};
	}
	
	\draw[line width = 1pt] (O)
		--($0*(\FirstVector)+1*(\SecondVector)$)
		--($1*(\FirstVector)+1*(\SecondVector)$)
		--($1*(\FirstVector)+0*(\SecondVector)$)--cycle;
	\node at (2.7,0.5) {$\mathcal{F}$};
	\draw[ultra thin] (2.2,0.4)--($0.5*(\FirstVector)+0.5*(\SecondVector)$);
	\coordinate (T) at (-2.7,0);
	\node[left] at (T) {$\mathcal{T}$};
	\draw[ultra thin] (T)--(-0.15,-0.05);
	
	\node[right] at (1.5,2) {$\tilde{L}(\alpha,M)$};
	\draw[ultra thin] (1.5,2)--(0.7,0.35);
	
	\coordinate (LaM) at (-1,2.9); 
	\node[left] at (LaM) {$L(\alpha,M) \ni$};
	\draw[ultra thin] (LaM)--($1*(\FirstVector)-1*(\SecondVector)$);
	
	\begin{scope}[on background layer]
		\filldraw[fill=lightgray!30,dotted] (SE)--(N)--(SW)--(SE);
		\filldraw[fill=white!80!lightgray] (se)--(sw)--(n)--(se);
		\filldraw[fill=white!10!lightgray,opacity=.5] 	(O)--($0*(\FirstVector)+1*(\SecondVector)$)--($1*(\FirstVector)+1*(\SecondVector)$)--($1*(\FirstVector)+0*(\SecondVector)$)--cycle;
	\end{scope}
	
	\node[] at (0.5,-3.2) {$\left( \log \frac{ \Theta^n M^n }{\Nrm(\alpha)} + c \right) \cdot \mathcal{T}$};
\end{scope}

	\def\vLength{0.8}
	\def\TohuScale{0.7}
	\def\FirstVector{0.23,0.1}
	\def\SecondVector{0.25,-0.15}
	\coordinate (O) at (0,0);
	\coordinate (R) at (3,0);
	\coordinate (U) at (0,3);
	\coordinate (LD) at (-1.2,-0.6);
	
	\coordinate (vR) at ($\vLength*(R)$);
	\coordinate (vU) at ($\vLength*(U)$);
	\coordinate (vLD) at ($\vLength*(LD)$);
	\coordinate (v) at ($(vR)+(vU)+(vLD)$);
		
	\coordinate (l) at ($(v)-\TohuScale*(vR)$);
	\coordinate (d) at ($(v)-\TohuScale*(vU)$);		
	\coordinate (ru) at ($(v)-\TohuScale*(vLD)$);

	\coordinate (TI) at (${1-0.33*\TohuScale}*(v)$);

	\coordinate (l2) at ($(l)-(TI)$);
	\coordinate (d2) at ($(d)-(TI)$);		
	\coordinate (ru2) at ($(ru)-(TI)$);

	\draw[dotted] (l2)--(d2)--(ru2)--cycle;
	\fill[color=lightgray!30] (l2)--(d2)--(ru2)--cycle;
	
	\begin{scope}
    	\def\LattticeOrigin{O}
        
		\clip (l2)--(d2)--(ru2)--cycle;
		\foreach \x in {-10,-9,...,10}{
		\foreach \y in {-10,-9,...,10}{
			\draw[very thin,color=gray!50] ($10 *(\FirstVector )+\y *(\SecondVector )+(\LattticeOrigin )$)--($-10 *(\FirstVector )+\y *(\SecondVector )+(\LattticeOrigin )$);
		}
			\draw[very thin,color=gray!50] ($\x *(\FirstVector )-10*(\SecondVector )+(\LattticeOrigin )$)--($\x *(\FirstVector )+10*(\SecondVector )+(\LattticeOrigin )$);
		}
        
		\foreach \x in {-10,-9,...,10}{
		\foreach \y in {-10,-9,...,10}{
            \node[circle,fill=black,inner sep=0.5pt] at ($\x *(\FirstVector ) + \y *(\SecondVector ) +(\LattticeOrigin)$) {};
        }
        }
    \end{scope}	
	\draw[semithick,->] (O)--($0.8*(R)$);	
	\draw[semithick,->] (O)--(U);	
	\draw[semithick,->] (O)--($1.3*(LD)$);	
	\draw ($0.8*(R)$)node[right]{$\log | \sigma_1|$};	
	\draw (U)node[above]{$\log | \sigma_2|$};	
	\draw ($1.3*(LD)$)node[below left]{$\log | \sigma_3|$};	

	\draw[dashed] (l)--(l2) (d)--(d2) (ru)--(ru2);
	\draw[thick] (O)--(TI);	

	\fill[color=lightgray!50,opacity=0.6] (l)--(d)--(ru)--cycle;
	\draw[ultra thin] (l)--(d)--(ru)--cycle;
	\begin{scope}
    	\def\LattticeOrigin{TI}
        
		\clip (l)--(d)--(ru)--cycle;
		\foreach \x in {-10,-9,...,10}{
		\foreach \y in {-10,-9,...,10}{
			\draw[very thin,color=gray!50] ($10 *(\FirstVector )+\y *(\SecondVector )+(\LattticeOrigin )$)--($-10 *(\FirstVector )+\y *(\SecondVector )+(\LattticeOrigin )$);
		}			\draw[very thin,color=gray!50] ($\x *(\FirstVector )-10*(\SecondVector )+(\LattticeOrigin )$)--($\x *(\FirstVector )+10*(\SecondVector )+(\LattticeOrigin )$);
		}
		\foreach \x in {-10,-9,...,10}{
		\foreach \y in {-10,-9,...,10}{
            \node[circle,fill=black,inner sep=0.5pt] at ($\x *(\FirstVector ) + \y *(\SecondVector ) +(\LattticeOrigin)$) {};
        }
        }
    \end{scope}	
    
    \coordinate (label_TI) at (2.2,1.63);
    \node[right] at (label_TI) {$\frac{\log \Nrm(\alpha)}{n} u_0$};
    \draw[ultra thin] (label_TI) -- (TI);
    
    \coordinate (label_v) at (1.5,2.7);
    \node[above] at ($(label_v)+(1,0)$) {$(\log (\Theta \|\alpha\|_{\infty,\omom})) u_0$};
    \draw[ultra thin] (label_v) -- (v);

    \coordinate (alpha) at ($-1*(\FirstVector)+2*(\SecondVector)+(TI)$);
    \coordinate (label_alpha) at (2,1);
    \node[right] at (label_alpha) {$\calL(\alpha)$};
    \draw[ultra thin] (label_alpha) -- (alpha);
    
    \node[] at (-1.5,-0.1) {$\HH \supseteq$};
    \node[] at (0.2,-1.7) {in $\RR^{r_1+r_2}$};
    
	\coordinate (vru) at ($(vR)+(vU)$);
	\coordinate (vl) at ($(vU)+(vLD)$);
	\coordinate (vd) at ($(vLD)+(vR)$);
	\draw[thin] (v) -- (vru) (v) -- (vl) (v) -- (vd);
	\draw[dotted, semithick] (vru) -- ($(vru)-0.4*(vLD)$) (vl) -- ($(vl)-0.3*(vR)$) (vd) -- ($(vd)-0.3*(vU)$);
	\draw[->,>=stealth,thick] (TI) -- (v);
	\node[inner sep = 1pt] at (TI){};
	
	\draw[->] (0.2,-0.75) -- (5.4,-0.75);
\end{tikzpicture}

%% file: chapter5.tex
\section{Relative multidimensional \Sz\ theorem}\label{section=relativeSzemeredi} 
We develop an axiomatic framework that enables us to carry out Steps~5--7
in Subsection~\ref{subsection=ideasofproof}.
More specifically, the goal of this section is to prove the relative multidimensional Szemer\'edi theorem (Theorem~\ref{thm:RMST}) below.
\begin{setting}\label{setting=section5}
    Let $\calZ$ be a free $\ZZ$-module of rank $n$.
    Let $\bv=(v_i)_{i\in[n]}$ be a basis for $\calZ$.
    Fix a finite subset $S\subseteq\calZ$ which is the shape of constellations we are looking for and assume it is a standard shape (Definition~\ref{definition=standardshape}).
    Namely, $S$ generates $\calZ $ as a $\ZZ $-module and satisfies $0\in S$ and $S=-S$.
    Let $r$ be the positive integer defined by $\#S=r+1$ and write $S=\{s_1, \dots, s_r, s_{r+1}=0\}$.
    For $j\in[r+1]$, set $e_j:=[r+1]\setminus\{j\}$. 
\end{setting}
    Recall that for a given positive integer $N$, we define $\calZ(\boldsymbol{v},N)=\{\sum_{i\in[n]}a_iv_i\in\calZ : a_i\in [-N,N] \text{ for all } i \in [n] \}$.
    We often denote elements of direct products $(a_i)_{i\in I}\in \prod _{i\in I}A_i$ by $a_I$ for short.
\subsection{The statement of the relative multidimensional \Sz\ theorem}
The following is a multidimensional generalization of the celebrated theorem of \Sz.
It was first established by Furstenberg--Katznelson \cite{Furstenberg-Katznelson78}, whose proof is based on ergodic theory.
See Gowers \cite[Theorems~10.2, 10.3]{Gowers07} and R\"odl--Schacht--Tengan--Tokushige \cite[Section~2]{Rodl-Schacht-Tengan-Tokushige06}
for a proof using the hypergraph removal lemma.
\begin{theorem}[Multidimensional \Sz\ theorem]\label{theorem=multiSzemeredi}
Let $n$ be a positive integer, $\delta$ a positive real number and $S$ a finite subset of $\ZZ^n$.
Then, there exists a positive integer $N_{\mathrm{MS}}(\delta,S)$ such that 
for every $N\geq N_{\mathrm{MS}}(\delta,S)$ and every subset $B\subseteq[-N,N]^n$ with
\[
\EE(\mathbf{1}_B\mid[-N,N]^n)\geq\delta ,
\]
there exist $S$-constellations in $B$.
\end{theorem}
In this article, we prove a relative version (in terms of weight) of the above theorem, Theorem~\ref{thm:RMST}, and use it in the proof of the main theorems.
To state Theorem~\ref{thm:RMST}, we now introduce a condition on weight functions
$\lambda\colon\calZ \to\RR_{\geq0}$.
\begin{definition}[$(\rho,N,S)$-linear forms condition]\label{definition=S-linearform}
Assume Setting~\ref{setting=section5}.
For each $\omega\in\bigsqcup_{j\in [r+1]}\{0,1\}^{e_j}$, we define a $\ZZ $-linear map $\psi\super{\omega}_{S}\colon\ZZ^{2r+2}\to\calZ$ as follows:
if $j\in[r]$, then the map associated to $\omega=(\omega_i)_{i\in e_j}\in\{0,1\}^{e_j}$ is
\begin{equation}\label{Eq:the-linear-maps-1}
\psi\super{\omega}_S(a_{1}\super{0},\ldots ,a_{r+1}\super{0},a_{1}\super{1},\ldots ,a_{r+1}\super{1})\coloneqq\Biggl(\sum_{i\in [r]\setminus\{j\}}(s_i-s_j)a_i\super{\omega_i}\Biggr)+s_{j}a_{r+1}\super{\omega_{r+1}}
\end{equation}
and if $\omega=(\omega_i)_{i\in e_{r+1}}\in\{0,1\}^{e_{r+1}}$, then we define
\begin{equation}\label{Eq:the-linear-maps-2}
\psi\super{\omega}_{S}(a_{1}\super{0},\ldots ,a_{r+1}\super{0},a_{1}\super{1},\ldots ,a_{r+1}\super{1})\coloneqq\sum_{i\in[r]}s_ia_i\super{\omega_i}.
\end{equation}

Let $0<\rho<1$ be a real number and $N$ a positive integer.
A function $\lambda\colon\calZ\to\RR_{\geq0}$ is said to satisfy the \emph{$(\rho,N,S)$-linear forms condition}
if for every subset $\calB \subseteq \ZZ ^{r+1}$ that is the product of intervals of lengths $\ge N$
and every $(n_{\omega})_{\omega }\in\{0,1\} ^{\bigsqcup_{j\in [r+1]}\{0,1\}^{e_j}}$, we have
\begin{equation}
\left|\EE\left(\prod_{j\in [r+1]}\prod_{\omega\in\{0,1\}^{e_j}}(\lambda\circ\psi\super{\omega}_{S})^{n_{\omega}}\relmiddle|\calB\times \calB\right)-1\right|\leq\rho
\label{eq:linearform-condition}.\end{equation}

A function $\lambda \colon \calZ \to \RR _{\ge 0}$ satisfying the $(\rho,N,S)$-linear forms condition is also called a \emph{$(\rho,N,S)$-pseudorandom measure} on $\calZ $.
\end{definition}
We note that this terminology differs from the usage in preceding work of \cite{Green-Tao08} and \cite{Conlon-Fox-Zhao15}.
The role of this definition will be clear in the proof of Proposition~\ref{proposition=pseudorandom}. 

The goal in this section is the following theorem, whose proof will be completed in Subsection~\ref{subsection=proofofRMST}.
\begin{theorem}[Relative multidimensional Szemer\'edi theorem]\label{thm:RMST}
    Assume Setting~$\ref{setting=section5}$.
    Then for every $\delta>0$, there exist positive real numbers $\gamma=\gamma_{\mathrm{RMS}}(\bv,\delta,S)$ and $\rho=\rho_{\mathrm{RMS}}(\bv,\delta,S)$
    such that the following holds:
    let $N$ be a positive integer and $\lambda$ a $(\rho,N,S)$-pseudorandom measure on $\calZ $. Let $B\subseteq \calZ(\boldsymbol{v},N)$ be a subset 
    satisfying the next two conditions:
    \begin{enumerate}[$(i)$]
    \item\label{en:weighted density}
     \emph{(Weighted density)}\quad $\EE(\mathbf{1}_B\cdot\lambda\mid\calZ(\boldsymbol{v},N))\geq\delta$,
    \item\label{en:smallness}
    \emph{(Smallness)}\quad $\EE(\mathbf{1}_B\cdot\lambda^{r+1}\mid\calZ(\boldsymbol{v},N))\leq \gamma N$.
    \end{enumerate}
    Then $B$ contains an $S$-constellation.
    \end{theorem}
This covers Theorem~\ref{theorem=multiSzemeredi}.
Indeed, consider $\calZ \coloneqq \ZZ^n$ with the standard basis $\bv$ and $\lambda\coloneqq\mathbf{1}_{\calZ}$;
then Theorem~\ref{thm:RMST} implies Theorem~\ref{theorem=multiSzemeredi} with threshold $N_{\mathrm{MS}}(\delta,S)=(\gamma_{\mathrm{RMS}}(\bv,\delta,S))^{-1}$.

In some \Sz -type theorems, it is possible to obtain lower bounds of the asymptotic number of $S$-constellations with respect to the parameter $N$.
Results of this sort date back to Varnavides's \cite{Varnavides59} work on Roth's theorem.
This is also the case for our Theorem~\ref{thm:RMST} in a weighted sense we now state.
\begin{theorem}\label{theorem=weighted-counting}
In the setting of Theorem~$\ref{thm:RMST}$, in addition to the existence of $S$-constellations in $B$,
we have the following inequality:
\[
\EE\left(\prod_{s\in S}(\ichi_B\cdot\lambda)(\alpha+ks) \relmiddle| (\alpha,k)\in \calZ(\bv,N)\times[N]\right) > \gamma.
\]
\end{theorem}
\begin{remark}\label{remark=smallness}
    The smallness condition guarantees that the contribution of the `trivial' $S$-constellations (those of the form $a+0\cdot S =\{ a\} $) is small.
    This condition is necessary to prove the existence of `non-trivial' $S$-constellations.
    The classical \Sz\ theorem as in \cite[Theorem~3.5]{Green-Tao08}
    and the relative multidimensional \Sz\ theorems in \cite[Theorem~2.18]{Tao06Gaussian}, \cite[Theorem~3.1]{Conlon-Fox-Zhao15} do not impose the smallness condition on their formulations of the relative multidimensional \Sz\ theorem; instead, in the step of applications, they argue that the smallness condition is satisfied if the parameter $N$ is sufficiently large.
\end{remark}
\begin{remark}\label{remark=affine}
    Our main problem is to show the existence of $S$-constellations in a subset $A\subseteq \calZ $.
    In this article, the subset $A$ is the set of prime elements of the ring of integers in a number field $K$.
    In this case, it seems difficult to construct directly a pseudorandom measure on $\calZ =\OK $ to make the relative multidimensional \Sz\ theorem applicable.
    We will instead employ the so-called $W$-trick,
    where we choose suitable $W\in \NN $ and $b\in \OK $ and try to apply the relative multidimensional \Sz\ theorem to the inverse image of $A$ by the affine transformation $\Aff _{W,b}\colon \OK \to \OK $.
    If we can find an $S$-constellation in the inverse image $\Aff ^{-1}_{W,b}(A)$,
    then we may send it back by $\Aff _{W,b}$ to obtain one in $A$.
    This approach was also used effectively by Green--Tao~\cite{Green-Tao08}.
    \end{remark}
    Before ending this subsection, let us prove the following elementary fact.
    This implies that the family of maps $(\psi _S\super{\omega })_{\omega }$ in Definition~\ref{definition=S-linearform} satisfies the conditions on kernels in Theorem~\ref{Th:Goldston_Yildirim} in Section~\ref{section=GoldstonYildirim}.
    \begin{lemma}
        \begin{enumerate}[$(1)$]
        \item\label{item:A-omega}
        For a given $j$ and $\omega \in \{ 0,1\} ^{e_j}$,
        consider the defining formulas \eqref{Eq:the-linear-maps-1} and \eqref{Eq:the-linear-maps-2} for the map $\psi_S\super{\omega}$.
        Then the indices $(i,\sign )\in [r+1]\times \{ 0,1\} $ 
        where the variable 
        $a_i\super{\sign} $ has a non-trivial coefficient are precisely those in the following set:
        \[
        A\super{\omega}\coloneqq\{(i,\omega_i)\}_{i\in e_j}.
        \]
        \item\label{item:A-omega-indep}
        For distinct $\omega, \omega' \in\bigsqcup_{j\in[r+1]}\{0,1\}^{e_j}$, we have $A\super{\omega} \not\subseteq A\super{\omega'}$.
        \item\label{item:ker-indep}
        For distinct $\omega, \omega' \in\bigsqcup_{j\in[r+1]}\{0,1\}^{e_j}$, 
        we have $\ker (\psi _S\super{\omega }) \not\subseteq \ker (\psi _S\super{\omega'})$.\\
        \end{enumerate}
\label{lem:indep}	
        \end{lemma}
        \begin{proof}
        Item~\eqref{item:A-omega} is clear from the definition.
        
        Now we show \eqref{item:A-omega-indep}.
        For the notational convenience, we show the opposite inclusion $A\super{\omega'} \not\subseteq A\super{\omega}$. 
        If $j\neq j'$, then we have $(j,\omega'_j)\in A\super{\omega'}\setminus A\super{\omega}$.
        If $j=j'$, there exists an $i\in e_j$ with $\omega_i\neq \omega'_i$.
        For this $i$, we have $(i,w'_i)\in A\super{\omega'}\setminus A\super{\omega}$.
        
        The assertion \eqref{item:ker-indep} follows from \eqref{item:A-omega-indep}.
        Indeed, we have that
        there exists an element $(i,\sign)\in A\super{\omega'}\setminus A\super{\omega}$
        by assertion \eqref{item:A-omega-indep}.
        The vector whose $(i,\sign)$-entry is $1$ and the others are $0$
        belongs to $\ker(\psi_S\super{\omega})\setminus\ker(\psi_S\super{\omega'})$.
        This completes the proof.
    \end{proof}
\subsection{Relative hypergraph removal lemma}\label{subsec:RHRL}
As is the case for the other versions \cite[Theorem~2.18]{Tao06Gaussian}, \cite[Theorem~3.1]{Conlon-Fox-Zhao15},
Theorem~\ref{thm:RMST} is derived from the \emph{relative hypergraph removal lemma}.
We use the version of Conlon--Fox--Zhao \cite[Theorem~2.12]{Conlon-Fox-Zhao15}
which is a refinement of Tao's \cite[Theorem~2.17]{Tao06Gaussian},
which in tern is a relative generalization of the hypergraph removal lemma of  
Gowers~\cite{Gowers07} and Nagle--R\"odl--Schacht--Skokan~\cite{Rodl-Skokan04,Nagle-Rodl-Schacht06,Rodl-Schacht071,Rodl-Schacht072}.
Cook--Magyar--Titichetrakun have proved a further strengthening of
the relative hypergraph removal lemma \cite[Theorem~1.4]{Cook-Magyar-Titichetrakun18}
to prove the multidimensional \Sz\ theorem in the primes, although we do not need this in this paper.

Let $J$ be a non-empty finite set and $r$ a positive integer.
The pair of $J$ and a subset $E\subseteq\binom{J}{r}$ is called an {\em $r$-uniform hypergraph}
({\em $r$-graph} for short).
Here, recall $\binom{J}{r}=\{e\in2^J : \#e=r\}$ from our notation.
If further for each $j\in J$ a finite set $V_j$ is given, then we call the tuple $((J,E); (V_j)_{j\in J})$ an {\em $r$-graph system}.
For a subset $e\subseteq J$ of indices, we write $V_e\coloneqq\prod_{j\in e}V_j$ for short and 
$x_e \coloneqq (x_j)_{j\in e}\in V_e$.
Let 
$V=((J,E); (V_j)_{j\in J})$ be an $r$-graph system.
We say $g$ is a {\em weighted hypergraph} on $V$ if $g$ is a tuple $g=(g_e)_{e\in E}$
of functions $g_e\colon V_e\to \RR _{\geq 0}$.
For two weighted hypergraphs $g=(g_e)_{e\in E}$ and $g'=(g'_e)_{e\in E}$ on $V$,
we write $g\leq g' $ if for all $e\in E$, we have $g_e\leq g'_e$ pointwise.
\begin{definition}\label{def:pseudorandomsystem}
    Let $\rho$ be a positive real number and $V=((J,E);(V_j)_{j\in J})$ an $r$-graph system.
    We say a weighted hypergraph $\nu $ on $V$ is
    {\em $\rho$-pseudorandom} if the following inequality holds for every 
    tuple $(n_{\omega})_{\omega }\in\{0,1\} ^{\bigsqcup_{e\in E}\{0,1\}^e}$:
    \begin{equation}
    \left|\EE\left(\prod_{e\in E}\prod_{\omega\in\{0,1\}^e}\nu_e(x_e\super{\omega})^{n_{\omega}}\relmiddle| (x_J\super{0}, x_J\super{1})\in V_J\times V_J\right)-1\right|\leq\rho
    \label{eq:pseudorandomsystem}.\end{equation}
    Here, the symbol $x_e\super{\omega}$ denotes the tuple
    $(x_j^{(\omega_j)})_{j\in e} \in V_e$.
    \end{definition}
Now we are ready to state the relative hypergraph removal lemma.
\begin{theorem}[Relative hypergraph removal lemma\ {\cite[Theorem 2.12]{Conlon-Fox-Zhao15}}]\label{thm:RHRL}
For every positive integer $k$ and every positive $\varepsilon>0$,
there are positive real numbers $\gamma=\gamma_{\mathrm{RHR}}(k,\varepsilon)>0$
and $\rho=\rho_{\mathrm{RHR}}(k,\varepsilon)>0$ such that the following holds:
let $V=((J,E);(V_j)_{j\in J})$ be an $r$-graph system with $r\leq k= \#J $
and $g, \nu$ two weighted hypergraphs on $V$ with $g\leq \nu$.
If $\nu $ is $\rho $-pseudorandom and the estimate
\begin{equation}\label{eq:RHRLhyp}
\EE\left(\prod_{e\in E}g_e(x_e)\relmiddle|x_J\in V_J\right)\leq\gamma
\end{equation}
holds, then there exists a family of subsets $E_e\subseteq V_e$ for $e\in E$
such that the following hold:
\[
\bigcap_{e\in E}(E_e\times V_{J\setminus e})=\varnothing
\]
and for all $e\in E$,
\[
\EE\left(g_e\cdot \mathbf{1}_{V_e\setminus E_e} \mid V_e\right)\leq\varepsilon.
\]
\end{theorem}
\begin{remark}\label{remark=refinedCFZ}
    It will be useful later that the only requirement on the weighted hypergraph $\nu $
    is $\rho $-randomness, with $\rho $ depending only on $k$ and $\varepsilon $.
    This fact might not be clear from the reference \cite{Conlon-Fox-Zhao15},
    where they consider a family $(\nu \super{N})_{N\in\NN}$ and state the above result for $N$ large enough without explicitly mentioning how the threshold for $N$ is determined.

    Theorem~\ref{thm:RHRL} can be verified by examining the proof of Conlon--Fox--Zhao \cite[Theorem~2.12]{Conlon-Fox-Zhao15} as follows:
    first, note that the dependence on $N$ stems solely from the \RCL\ \cite[Theorem~2.17]{Conlon-Fox-Zhao15}.
    Also note that their arguments 
    actually show Theorem~\ref{theorem=RCL} below by examining each $o(1)$ related to $\nu $ in the proof.
    This verifies Theorem~\ref{thm:RHRL} above.
    Such a family-free argument has already appeared in Romani\'c--Wolf \cite{RW} for the study of a quantitative version of the $1$-dimensional relative \Sz\ theorem.
 \end{remark}
\begin{theorem}[Relative counting lemma]\label{theorem=RCL}
    Let $k$ be a positive integer, and 
    $\varepsilon$ and $\rho$ sufficiently small positive real numbers  depending on $k$.
    Then there exist positive real numbers
    \footnote{For instance, we may take $a_k=2^{1-2^{2^{k-1}-1}}$ and $b_k=2^{1-(k+2)2^{2^{k-1}-2}}$.}
    $a_k$ and $b_k$ depending only on $k$ such that the following holds.
    Let $V=((J,E);(V_j)_{j\in J})$ be an $r$-graph system with $r \leq k = \# J$ and $\nu$ a $\rho$-pseudorandom weighted hypergraph on $V$.
    Let $g$ and $\tilde{g}$ be two weighted hypergraphs on $V$ with $g\leq\nu$ and $\tilde{g}\leq 1$ such that $(g,\tilde g)$ is a \emph{$\varepsilon$-discrepancy pair}; see {\rm \cite[Definition~2.13]{Conlon-Fox-Zhao15}} for this notion.
    Then we have 
    \[
    \left|\EE\left(\prod_{e\in E}g_e(x_e)\relmiddle| x_J\in V_J\right)-\EE\left(\prod_{e\in E}\tilde{g}_e(x_e)\relmiddle| x_J\in V_J\right)\right|=O_k(\varepsilon^{a_k}+\rho^{b_k}).
    \]
    \end{theorem}
\subsection{Construction of pseudorandom weighted hypergraphs}\label{subsec:psuedo-random-measure}
Assume Setting~\ref{setting=section5}.
Let $\phi_S\colon\ZZ^r\to\calZ$ be the homomorphism of $\ZZ $-modules sending the $i$-th standard vector $\epsilon _i \in \ZZ ^r $ to $s_i\in \calZ $ for each $i\in [r]$.
Let $\epsilon _{r+1}\in \ZZ ^{r}$ denote the zero vector for a notational purpose.
Since $S$ generates $\calZ$ by assumption, we have the following exact sequence
\[
0\longrightarrow\ker(\phi_S)\longrightarrow\ZZ^r\stackrel{\phi_S}{\longrightarrow}\calZ\longrightarrow 0,
\]
which splits because $\calZ$ is a free $\ZZ$-module.
\begin{lemma}\label{lem:phi^(-1)-bound}
    There exists a positive integer $U=U(\bv,S)$ such that the following holds:
    for every positive integer $N$ and every $\alpha \in \calZ (\bv ,N)$, we have the inequality
    \[
    (2N+1)^{r-n}\leq\#(\phi_S^{-1}(\alpha)\cap[-UN,UN]^r)\leq(2UN+1)^{r-n}.
    \]
\end{lemma}
\begin{proof}
Choose a $\ZZ $-linear section $\sigma\colon\calZ\to\ZZ^r$ to $\phi_S$.
Choose a basis $w_1,\dots,w_{r-n}$ for the rank $r-n$ free $\ZZ $-module $\ker(\phi_S)$.
Then the vectors 
$\sigma(v_1),\dots,\sigma(v_n), w_1,\dots, w_{r-n}$ form a basis for $\ZZ^r$.
Let $U$ be $r$ times the maximum of the absolute values of the entries of this combined basis.

The second inequality easily follows from the fact that $\ker (\phi _S)$ is has rank $r-n$.
Next, by the choice of $U$, we have an inclusion of sets
\[
\left\{\sigma(\alpha)+\sum_{i\in[r-n]}b_iw_i : b_i\in[-N,N] \ (i\in[r-n])\right\}\subseteq\phi_S^{-1}(\alpha)\cap[-UN,UN]^r,
\]
which shows the first inequality.
\end{proof}
    Let $U$ be an integer given by Lemma~\ref{lem:phi^(-1)-bound}.
    Let $N$ be a positive integer and $\lambda\colon\calZ\to\RR_{\geq0}$ a function.
    By following Solymosi~\cite{Solymosi03}, we construct a weighted hypergraph $\nu=\nu(\lambda,N,\bv,S)$ as follows.
    Denote by $K_{r+1}^{(r)}=([r+1], \binom{[r+1]}{r})$ the complete $r$-hypergraph with $r+1$ vertices.
    By our notation $e_j=[r+1]\setminus\{j\}$, we have $\binom{[r+1]}{r}=\{e_j : j\in[r+1]\}$.
    For each integer $a$ and index $j\in[r+1]$, we define a hyperplane $H_j(a)$ of $\ZZ^r$ by
    \[
    H_j(a)\coloneqq
    \begin{cases}
    \{(x_1,\dots, x_r)\in\ZZ^r : x_j=a\} & \text{if} \ j\in[r], \\
    \{(x_1,\dots, x_r)\in\ZZ^r : \sum_{i\in[r]}x_i=a\} & \text{if} \ j=r+1
    \end{cases}
    \]
    and a set $V_j$ by
    \begin{equation}\label{eq:def-of-V_j}
    V_j\coloneqq
    \begin{cases}
    \{H_j(a) : a\in[-UN,UN]\} & \text{if} \ j\in [r], \\
    \{H_{r+1}(a) : a\in[-rUN,rUN]\} & \text{if} \ j=r+1.
    \end{cases}
    \end{equation}
    We define an $r$-graph system $V$ by $V\coloneqq(K_{r+1}^{(r)}; (V_j)_{j\in[r+1]})$.
        Let $j\in[r+1]$.
    For every tuple $(H_i)_i \in \prod _{i\in e_j} V_i$, the intersection $\bigcap_{i\in e_j}H_i$ consists of a single point.
    Let $T_j\colon V_{e_j}\to\ZZ^r$ be the map sending the given tuple to the point.
    We define a weighted hypergraph $\nu=(\nu_{e_j})_{j\in[r+1]}$ on $V$ by the following composition:
    \begin{equation}\label{eq:def-of-nu}
        \nu_{e_j}\colon V_{e_j}\stackrel{T_j}{\longrightarrow}\ZZ^r\stackrel{\phi_S}{\longrightarrow}\calZ\stackrel{\lambda}{\longrightarrow}\RR_{\geq 0}.
    \end{equation}
        Here, we exhibit an explicit form of the map $T_j$.
    Set $\calB=\calB(N,\bv,S)\coloneqq[-UN,UN]^r\times[-rUN,rUN]$.
    Then, for every point $a_{[r+1]}\in\calB$, we have
    \begin{equation}\label{equation=T_j-cal}
    T_j((H_i(a_i))_{i\in e_j})=
    \begin{cases}
    \displaystyle \biggl(a_1,\dots,a_{j-1},a_{r+1}-\sum_{i\in [r]\setminus\{j\}}a_i,a_{j+1},\dots,a_r\biggr) & \text{if} \ j\in[r], \\
    (a_1,\dots,a_r) & \text{if} \ j=r+1.
    \end{cases}
    \end{equation}
\begin{proposition}\label{proposition=pseudorandom}
    Let $\rho$ be a positive real number.
    If the function $\lambda \colon \calZ \to \RR _{\geq 0}$ satisfies the $(\rho,N,S)$-linear forms condition,
    then the weighted hypergraph $\nu $ constructed in \eqref{eq:def-of-V_j} and \eqref{eq:def-of-nu} is $\rho $-pseudorandom.
    \end{proposition}
\begin{proof}
    In the current situation, the expectation in \eqref{eq:pseudorandomsystem} is equal to
    \begin{align*}
    &\EE\left(\prod_{j\in[r+1]}\prod_{\omega\in\{0,1\}^{e_j}}(\lambda\circ\phi_S\circ T_j)(H_{e_j}\super{\omega})^{n_{\omega}} \relmiddle| (H_{[r+1]}\super{0}, H_{[r+1]}\super{1})\in V_{[r+1]}\times V_{[r+1]}\right) \\
    &=\EE\left(\prod_{j\in[r+1]}\prod_{\omega\in\{0,1\}^{e_j}}(\lambda\circ\phi_S\circ T_j)((H_i(a_i\super{\omega}))_{i\in e_j})^{n_{\omega}}\relmiddle|(a_{[r+1]}\super{0}, a_{[r+1]}\super{1})\in \calB\times\calB\right).
    \end{align*}
    By \eqref{equation=T_j-cal} and defining formulas \eqref{Eq:the-linear-maps-1}, \eqref{Eq:the-linear-maps-2}, we have for each $\omega\in\bigsqcup_{j\in[r+1]}\{0,1\}^{e_j}$,
    \[
    (\phi_S\circ T_j)((H_i(a_i\super{\omega}))_{i\in e_j})=\psi_S\super{\omega}(a_{[r+1]}\super{0},a_{[r+1]}\super{1})
    \]
    and hence the expectation above equals
    \[
    \EE\left(\prod_{j\in[r+1]}\prod_{\omega\in\{0,1\}^{e_j}}(\lambda\circ\psi_S\super{\omega})^{n_{\omega}}\relmiddle|\calB\times\calB \right).
    \]
    Since $\lambda $ satisfies the $(\rho ,N,S)$-linear forms condition, the difference of the above value and $1$ is at most $\rho $.
    This verifies the desired $\rho $-pseudorandomness.
    \end{proof}
\subsection{The proof of \RMST}\label{subsection=proofofRMST}
We use notation in Theorem~\ref{thm:RMST}.
We prove Theorems~\ref{thm:RMST} and \ref{theorem=weighted-counting} simultaneously.
\begin{proof}[Proof of Theorems~$\ref{thm:RMST}$ and $\ref{theorem=weighted-counting}$]
For a positive real number $\delta>0$, set 
\[
\varepsilon=\varepsilon(\bv,\delta,S)\coloneqq\frac{\delta}{(r+1)(rU)^r} ,
\]
where $U$ is the integer given by Lemma~\ref{lem:phi^(-1)-bound}. 
Also, using Theorem~\ref{thm:RHRL}, set 
\[
\gamma=\gamma_{\mathrm{RMS}}(\bv,\delta,S)\coloneqq\frac{2}{3}\gamma_{\mathrm{RHR}}(r+1,\varepsilon),\quad \rho=\rho_{\mathrm{RMS}}(\bv,\delta,S)\coloneqq\rho_{\mathrm{RHR}}(r+1,\varepsilon).
\]
Let $N$ be a positive integer, $\lambda$ a $(\rho,N,S)$-pseudorandom measure on $\calZ$ and $B\subseteq\calZ(\bv, N)$ a subset satisfying the weighted density and the smallness conditions in Theorem~\ref{thm:RMST}.
Let $\nu=\nu(\lambda,N,\bv,S)$ be the weighted hypergraph defined in \eqref{eq:def-of-nu}, which is $\rho $-pseudorandom by Proposition~\ref{proposition=pseudorandom}.
Define a subset $E_{e_j}\subseteq V_{e_j}$ for each $j\in[r+1]$ by $E_{e_j}\coloneqq(\phi_S\circ T_j)^{-1}(B)$ and a weighted hypergraph $g=(g_{e_j})_{j\in[r+1]}$ on $V$
by $g_{e_j}\coloneqq\mathbf{1}_{E_{e_j}}\cdot\nu_{e_j}$.
Then we have $g\leq\nu$ in the sense discussed at the beginning of Subsection \ref{subsec:RHRL}.

By the definitions and \eqref{equation=T_j-cal}, we have
\begin{align}
&\EE\left(\prod_{j\in[r+1]}g_{e_j}(H_{e_j}) \relmiddle| H_{[r+1]}\in V_{[r+1]}\right)\notag\\
&=\EE\left(\prod_{j\in[r+1]}(\mathbf{1}_{E_{e_j}}\cdot\nu_{e_j})((H_i(a_i))_{i\in e_j}) \relmiddle| a_{[r+1]}\in \calB\right)\notag\\
&=\EE\left(\prod_{j\in[r+1]}(\mathbf{1}_{\phi_S^{-1}(B)}\cdot(\lambda\circ\phi_S))(a_{[r]}+k\epsilon_j) \relmiddle| a_{[r+1]}\in \calB\right),\label{eq:g-nu-lambda}
\end{align}
where $k\coloneqq a_{r+1}-\sum_{i\in[r]}a_i$.
By the upper bound in Lemma~\ref{lem:phi^(-1)-bound} and the smallness condition,
the contribution of those $a_{[r+1]}$'s with $k=0$ in the above expectation
is bounded from above by
\begin{equation}\label{eq:contribution_k=0}
\begin{split}
&\frac{1}{2rUN+1}\EE(\mathbf{1}_{\phi_S^{-1}(B)}\cdot(\lambda\circ\phi_S)^{r+1}\mid[-UN,UN]^r)\\
&\leq\frac{(2UN+1)^{r-n}}{2rUN+1}\cdot\frac{(2N+1)^n}{(2UN+1)^r}\cdot\EE(\mathbf{1}_B\cdot\lambda^{r+1}\mid\calZ(\bv,N) )\leq\frac{\gamma}{2}.
\end{split}
\end{equation}

By way of contradiction, suppose that the contribution of those $a_{[r+1]}$'s with $k\neq 0$ in \eqref{eq:g-nu-lambda} is at most $\gamma$, that is
\begin{equation}\label{eq:supposition}
    \EE\left(\prod_{j\in[r+1]}(\mathbf{1}_{\phi_S^{-1}(B)}\cdot(\lambda\circ\phi_S))(a_{[r]}+k\epsilon_j) \relmiddle| \underset{\text{with } k\neq 0}{a_{[r+1]}\in \calB}\right)
    \le \gamma .
\end{equation}
Combined with \eqref{eq:contribution_k=0}, it implies that the value \eqref{eq:g-nu-lambda} does not exceed $\frac 3 2 \gamma =\gamma _{\mathrm{RHR}} (r+1,\varepsilon )$.
Therefore, Theorem~\ref{thm:RHRL} applies, and there exists a family $(E'_{e_j})_{j\in[r+1]}$ of subsets $E'_{e_j}\subseteq V_{e_j}$ for $j\in [r+1]$ such that
\begin{equation}
\bigcap_{j\in[r+1]} (E'_{e_j}\times V_{j})=\varnothing
\label{equation=RHRresult1}\end{equation}
and 
\begin{equation}
\EE(\mathbf{1}_{E_{e_j}\setminus E'_{e_j}}\cdot\nu_{e_j}\mid V_{e_j})\leq\varepsilon
\quad \text{ for all $j\in [r+1]$.}
\label{equation=RHRresult2}\end{equation}
We will argue that this contradicts the assumption of the weighted density condition.

Define a map $\iota_0\colon\phi_S^{-1}(B)\cap[-UN,UN]^r\to V_{[r+1]}$ by
\[
\iota_0(a_{[r]})\coloneqq\biggl(H_1(a_1),\dots,H_r(a_r),H_{r+1}\biggl(\sum_{i\in[r]}a_i\biggr)\biggr)
.\]
Denote by $\pr _{e_j} \colon V_{[r+1]} \to V_{e_j}$ the projection which forgets the $j$-th entry.
Then, for all $j\in[r+1]$, the map $T_j \circ \pr _{e_j} \circ \iota _0$ coincides with the inclusion map from $\phi ^{-1}(B)\cap [-UN,UN]^r$ to $\ZZ^r$.
Hence by the definition of $E_{e_j}$, it follows that 
$\iota _0$ maps into $\bigcap_{j\in[r+1]} (E_{e_j}\times V_j)$.

Let us write $\widetilde{E_{e_j}}\coloneqq E_{e_j}\times V_{j}$ and $\widetilde{E'_{e_j}}\coloneqq E'_{e_j}\times V_{j}$ for short.
Consider the following decreasing sequence of subsets
\[
\bigcap_{j\in[r+1]}\widetilde{E_{e_j}}
 \: \supseteq \: 
\Biggl(\bigcap_{j\in[r]}\widetilde{E_{e_j}}\cap\widetilde{E_{e_{r+1}}'}\Biggr)
 \: \supseteq \: 
\Biggl(\bigcap_{j\in[r-1]}\widetilde{E_{e_j}}\cap\bigcap_{j\in[r,r+1]}\widetilde{E_{e_j}'}\Biggr)
 \: \supseteq \: \cdots \: \supseteq \: \bigcap_{j\in[r+1]}\widetilde{E_{e_j}'} =\varnothing,
\]
where the last equality is \eqref{equation=RHRresult1}.
Consider also the associated partition
\[
\bigcap_{j\in[r+1]}\widetilde{E_{e_j}}=\bigsqcup _{j\in [r+1]} 
\left( \widetilde {E_{e_1}}\cap \dots \cap \left( \widetilde{E_{e_j}}\setminus \widetilde{E'_{e_j}}\right) \cap \dots \cap \widetilde{E'_{e_{r+1}} }\right).
\]
Composing the projection $\pr _{e_j}\colon \widetilde{E_{e_j}}\setminus \widetilde{E'_{e_j}} \to E_{e_j}\setminus E'_{e_j}$,
we obtain the following composite map
\[
\iota\colon\phi_S^{-1}(B)\cap[-UN,UN]^r\xrightarrow{\iota_0}\bigcap_{j\in[r+1]}\widetilde{E_{e_j}}\to\bigsqcup_{j\in[r+1]}E_{e_j}\setminus E'_{e_j}
,\]
which is injective since $T_j \circ \pr _{e_j} \circ \iota _0$ is an inclusion map.

If $a_{[r]}\in\phi_S^{-1}(B)\cap[-UN,UN]^r$ is mapped by $\iota $ into $ E_{e_j}\setminus E_{e_j}'$, then we have by definition of $\nu _{e_j}$,
\[
(\lambda\circ\phi_S)(a_{[r]})=\nu_{e_j}(\iota(a_{[r]})).
\]
Therefore by the injectivity of $\iota$ and the lower bound in Lemma~\ref{lem:phi^(-1)-bound} we have that 
\begin{align*}
(2N+1)^{r-n}\sum_{\alpha\in\calZ(\bv,N)}(\mathbf{1}_{B}\cdot\lambda)(\alpha)
&\leq \sum_{a_{[r]}\in[-UN,UN]^r}(\mathbf{1}_{\phi^{-1}(B)}\cdot(\lambda\circ\phi_S))(a_{[r]})\\
&\leq\sum_{\substack{j\in[r+1] \\ H_{e_j}\in V_{e_j} }}(\mathbf{1}_{E_{e_j}\setminus E'_{e_j}}\cdot\nu_{e_j})(H_{e_j})
.\end{align*}
We divide this formula by $(2N+1)^r$ and apply \eqref{equation=RHRresult2}.
Since $\# V_{e_j} \leq (2rUN+1)^r $, we have
\begin{align*}
\EE(\mathbf{1}_B\cdot\lambda\mid\calZ(\bv,N))
&\leq\frac{1}{(2N+1)^r}\sum_{j\in[r+1]}\#V_{e_j}\cdot \EE(\mathbf{1}_{E_{e_j}\setminus E'_{e_j}}\cdot\nu_{e_j}\mid V_{e_j})\\
&< (r+1)(rU)^r\varepsilon=\delta.
\end{align*}
This contradicts the assumption of the weighted density condition.
Therefore our supposition \eqref{eq:supposition} turns out false and we conclude:
\begin{equation}
\sum_{\substack{a_{[r]}\in[-UN,UN]^r \\ a_{r+1}\in[-rUN,rUN] \\ \text{with }k\neq 0 }}\prod_{j\in[r+1]}(\ichi_{\phi_S^{-1}(B)}\cdot(\lambda\circ\phi_S))(a_{[r]}+k\epsilon_j) > \gamma\cdot\#\calB
\label{equation=k-part},\end{equation}
where we recall $k=a_{r+1}-\sum_{i\in[r]}a_i$.
By the upper bound in Lemma~\ref{lem:phi^(-1)-bound} and the assumptions $0\in S$ and $S=-S$, we have
\begin{align*}
\text{L.H.S. of \eqref{equation=k-part}}
&\leq\sum_{\substack{a_{[r]}\in[-UN,UN]^r \\ k\in[-2rUN,2rUN] \setminus \{0\} }}\prod_{j\in[r+1]}(\ichi_{\phi_S^{-1}(B)}\cdot(\lambda\circ\phi_S))(a_{[r]}+k\epsilon_j)\\
&\leq2(2UN+1)^{r-n}\sum_{\substack{\alpha\in\calZ(\bv,N) \\ k\in[2rUN]}}\prod_{s\in S}(\ichi_B\cdot\lambda)(\alpha+ks)\\
&=2(2UN+1)^{r-n}\sum_{\substack{\alpha\in\calZ(\bv,N) \\ k\in[N]}}\prod_{s\in S}(\ichi_B\cdot\lambda)(\alpha+ks)
.\end{align*}
The last equality holds because the summand can be non-zero only when $\alpha+kS\subseteq B\subseteq\calZ(\bv,N)$ and because $S=-S$.
Therefore we conclude
\begin{align*}
\EE\left(\prod_{s\in S}(\ichi_B\cdot\lambda)(\alpha+ks) \relmiddle| (\alpha,k)\in \calZ(\bv,N)\times[N]\right)
&>\left(\frac{2UN+1}{2N+1}\right)^n\cdot\frac{2rUN+1}{N}\cdot\frac{\gamma}{2}
\geq\gamma.
\end{align*}
This completes the proof of Theorems~\ref{thm:RMST} and \ref{theorem=weighted-counting}.
\end{proof}
A slight modification of  the argument in the proof above yields the following variant, Theorem~\ref{theorem=weighted-counting2}. It estimates a weighted expectation in the situation where
the shape $S$ is not assumed to satisfy the condition `$S=-S$' but instead the scaling factor is allowed to be negative. 
This theorem enables us to 
allow the length $k$ of arithmetic progression in Theorem~\ref{theorem=BanachGreenTao} to be even as well.
\begin{theorem}\label{theorem=weighted-counting2}
	In the setting of Theorem~$\ref{thm:RMST}$, instead of assuming $S$ is a standard shape, we only assume that a finite subset $S\subseteq \calZ$ generates $\calZ $ as a $\ZZ $-module, and that $0\in S$. Then we have that
	\[
	\EE\left(\prod_{s\in S}(\ichi_B\cdot\lambda)(\alpha+ks) \relmiddle| (\alpha,k)\in \calZ(\bv,N)\times([-2N,2N]\setminus\{0\})\right) > \frac{\gamma}{2}.
	\]
\end{theorem}
\begin{proof}
Almost all parts of the proof of Theorem~\ref{theorem=weighted-counting} remain to work under the current weaker assumptions on $S$. However, since the condition `$S=-S$' is  dropped, the estimate of the left-hand side of \eqref{equation=k-part} may not hold in the original form. This is the only point to be modified in the present proof; the modification can be done in the following manner.
\begin{align*}
\text{L.H.S. of \eqref{equation=k-part}}
&\leq\sum_{\substack{a_{[r]}\in[-UN,UN]^r \\ k\in[-2rUN,2rUN] \setminus \{0\} }}\prod_{j\in[r+1]}(\ichi_{\phi_S^{-1}(B)}\cdot(\lambda\circ\phi_S))(a_{[r]}+k\epsilon_j)\\
&\leq(2UN+1)^{r-n}\sum_{\substack{\alpha\in\calZ(\bv,N) \\ k\in[-2rUN,2rUN]\setminus\{0\}}}\prod_{s\in S}(\ichi_B\cdot\lambda)(\alpha+ks)\\
&=(2UN+1)^{r-n}\sum_{\substack{\alpha\in\calZ(\bv,N) \\ k\in[-2N,2N]\setminus\{0\}}}\prod_{s\in S}(\ichi_B\cdot\lambda)(\alpha+ks).
\end{align*}
This provides the desired estimate.
\end{proof}

%% file: chapter6.tex
\section{Goldston--Y\i ld\i r\i m type asymptotic formula}\label{section=GoldstonYildirim}

The main result in this section is a Goldston--Y\i ld\i r\i m type asymptotic formula (Theorem~\ref{Th:Goldston_Yildirim}).
It essentially shows the pseudorandomness of the weight function $\tilde{\lambda}$ in the `$N$-world,'
which is required in Step~$4$ of the strategy %
in Subsection~\ref{subsection=ideasofproof}.
Recall that the notion of pseudorandomness of a weight function is formulated in Definition~\ref{definition=S-linearform},
and $\tilde{\lambda}$ will be constructed in Definition~\ref{def=our-measure}.

The presence of the factor $W$ in our formulation of Theorem~\ref{Th:Goldston_Yildirim} is related to our use of the $W$-trick.
If we tried to use the relative multidimensional \Sz\ theorem (Theorem~\ref{thm:RMST}) directly in the `$M$-world,' then we would need to prove an asymptotic formula without $W$-trick.
Unfortunately, this seems to be an extremely hard task.
For this reason, we need to transfer the setting in the `$M$-world' to that in the `$N$-world.' 
In this transition, the fact that the choices of $b_1,\ldots,b_m$ are rather flexible in Theorem~\ref{Th:Goldston_Yildirim} will play a significant role.

\subsection{Von Mangoldt function and its variants}\label{subsection=vonMangoldt}
For a number field $K$, recall symbols such as $\Ideals_K$ and $|\Spec(\OO_K)|$ from Subsection~\ref{subsection=OK}.
Let us recall that %
the von Mangoldt function $\Lambda$ ($=\Lambda_K$) for $K$ is defined by
\[ 
\Lambda(\ideala)\coloneqq
\begin{cases}
\log \Nrm (\idealp ) & (\text{if $\ideala$ is a non-trivial power of a prime ideal $\idealp$}),\\
0	& (\text{otherwise})
\end{cases}
\]
for each non-zero ideal $\ideala \in \Ideals_K$.

Recall that the norm $\Nrm$ is multiplicative by Lemma~\ref{lemma=completemultiplicativity}.
Hence, by considering the decomposition of an ideal $\ideala$ into prime ideals 
we have 
$\sum_{\idealb | \ideala} \Lambda(\idealb) = \log \Nrm(\ideala)$.
The M\"obius inversion formula (Proposition~\ref{proposition=Moebius}) implies 
\[
\Lambda(\ideala)=\sum\limits_{\idealb\mid\ideala}\mu(\idealb)\log\left(\frac{\Nrm(\ideala)}{\Nrm(\idealb)}\right)=\log\Nrm(\ideala)\cdot\sum\limits_{\idealb\mid\ideala}\mu(\idealb)\left(1-\frac{\log\Nrm(\idealb)}{\log\Nrm(\ideala)}\right),
\]
where $\mu \coloneqq\mu_K$.
Next we consider an analog 
of the Goldston--Y\i ld\i r\i m truncated divisor sum.
More precisely we take a sufficiently large $R > 0$ %
and consider the sum over the ideals $\idealb$ with $\Nrm(\idealb) \leq R$.
Moreover we replace $\log \Nrm(\ideala)$ with $\log R$:
\[
\log R\cdot\sum_{\idealb\mid\ideala, \ \Nrm(\idealb)\leq R}\mu(\idealb)\left(1-\frac{\log\Nrm(\idealb)}{\log R}\right).
\]
This is a version of the Goldston--Y\i ld\i r\i m truncated divisor sum for $K$. 
Following the method of Tao~\cite{Tao06Gaussian}, we first regard $1-\log \Nrm (\idealb )/\log R$ as the function obtained by substituting $\log\Nrm(\idealb) / \log R$ for $x$ in $\max\{1-|x|,0\}$.
This makes the condition `$\Nrm(\idealb)\leq R$' under the summation symbol redundant.
Secondly, we replace the function $\max\{1-|x|,0\}$ with a non-negative $C^\infty$-function $\chi$ whose support is contained in $[-1,1]_\RR$.
Note nonetheless that the values of $\chi$ on $[-1,0)_\RR$ do not influence the sum.
Thus we are led to the following definition.

\begin{definition}\label{def:chi}
Let $R$ be a real number greater than $1$,
and fix a non-negative $C^\infty$-function $\chi$ whose support is contained in $[-1,1]_\RR$.
For convenience of later calculation, we assume that $\chi(0)=1$ and $\chi(x) \leq 1$ for all  $x$.
Then the \emph{$(R,\chi)$-von Mangoldt function} $\Lambda_{R,\chi}\colon\Ideals_K\cup\{(0)\}\to\RR$ is defined to be
\begin{equation}\label{Eq:def-of-von-M}
\Lambda_{R,\chi}(\ideala)\coloneqq\log R\cdot\sum_{\idealb                    \in \Ideals_K \text{ with } \idealb\mid\ideala}\mu(\idealb)\chi\left(\frac{\log\Nrm(\idealb)}{\log R}\right).
\end{equation} 
By abuse of notation, we write for each $\alpha \in \OK$, 
\[
	\Lambda_{R,\chi }(\alpha) \coloneqq \Lambda_{R,\chi}(\alpha \OK).
\]
Let $c_\chi$ be the positive real number defined by
\begin{equation}\label{Eq:def-c-chi}
c_\chi\coloneqq\int_0^{\infty }\chi'(x)^2\rd x,
\end{equation}
where $\chi'$ is the derivative of $\chi$.
\end{definition}
Note that $\Lambda_{R,\chi}(0)$ is well-defined because 
the right-hand side of \eqref{Eq:def-of-von-M} is a finite sum 
by the assumption on $\supp (\chi)$ and Proposition~\ref{proposition=idealdensity}.
\subsection{Statement of Goldston--Y\i ld\i r\i m type asymptotic formula}
The following asymptotic formula implies that the measure constructed in Definition~\ref{def=our-measure} satisfies the linear forms condition (Definition~\ref{definition=S-linearform}).
The main challenge here is finding an appropriate formula for a general number field. Once it is successfully done, the proof can be carried out %
by following strategies in \cite[Section 9]{Tao06Gaussian} and \cite[Section 10]{Conlon-Fox-Zhao14}.
Recall that $\kappa$ is the positive real number in Theorem~\ref{theorem=zeta_K},
$c_\chi$ is the positive real number in \eqref{Eq:def-c-chi} determined by $\chi$,
and $\vph_K$ is the totient function (Definition~\ref{def=totient}). 
\begin{theorem}[Goldston--Y\i ld\i r\i m type asymptotic formula]\label{Th:Goldston_Yildirim}
Let $K$ be a number field of degree $n$.
Let $m$ and $t$ be positive integers, and $\psi_1,\dots,\psi_m\colon\ZZ ^t\to\OK$ be $\ZZ $-module homomorphisms.
Let $w$ be a positive real number, and $W$ a positive integer of which the set of prime divisors is $\PP_{\leq w} = \{ p \in \PP : p \leq w \}$.
Let $b_1,\dots ,b_m$ be elements in $\OK$ each of which is prime to $W$, and define affine transformations $\theta _1,\dots ,\theta _m\colon\ZZ^t \to \OK$ as
\[
\theta_j(x)\coloneqq\Aff _{W,b_j}(\psi_j(x))=W\psi _j(x)+b_j.
\]
Let $R$ be a positive real number,
and $I_1, \ldots, I_t\subseteq \mathbb Z$ intervals of lengths at least $R^{4m+1}$.
Set $\calB\coloneqq I_1\times\dots\times I_t \subseteq\ZZ^t$. Fix a $C^{\infty}$-function $\chi\colon\mathbb{R}\to[0,1]_{\RR}$ which satisfies $\chi(0)=1$ and $\supp (\chi)\subseteq [-1,1]_{\RR}$.  
Assume that
\begin{equation}\label{Eq:no-inclusion}
\text{all $\coker(\psi_j)$ are finite, and for all $i,j \in [m]$, $\ker(\psi_j) \subseteq \ker(\psi_i)$ implies $i=j$.}
\end{equation}
Then there exist positive real numbers $R_0=R_0(m,K)$, $F_0=F_0(m,n)$ and $w_0=w_0((\psi_j)_{j\in [m]})$ such that  if $R \geq R_0$, $ w \geq w_0$ and $\log w \leq F_0\cdot\sqrt{\log R}$, then
\begin{multline}\label{Eq:formula_to_show}
\EE(\Lambda_{R,\chi }(\theta_1(x))^2\cdots\Lambda_{R,\chi}(\theta_m(x))^2\mid x\in\calB) \\
=\left(1+O_{\chi,m,n}\left(\frac{1}{w\log w}\right)+O_{\chi ,m,t,K }\left(\frac{\log w}{\sqrt{\log R}}\right)\right)\cdot \left(\frac{W^nc_\chi \log R}{\vph_K(W)\cdot\kappa}\right)^m
\end{multline}
holds true.
In particular, the error terms $O_{\chi,m,n}\left(\frac{1}{w\log w}\right)$ and $O_{\chi ,m,t,K }\left(\frac{\log w}{\sqrt{\log R}}\right)$ are bounded uniformly on $W$ and $(b_j)_{j\in[m]}$.
\end{theorem}
The rest of this section is devoted to the proof of this theorem.

In this section, we use `Propositions' to rewrite the expectation on the 
left-hand side of \eqref{Eq:formula_to_show}; during the rewriting processes, we describe partial estimates as `Lemmas.'
We divide this section into subsections according to the particular aspect of estimate we focus on, such as $p$-parts of ideals.
We clarify our setting as `Setting' at the beginning of each subsection.

Recall the symbols and definitions given in `Notation' at the end of Subsections~\ref{subsection=ideasofproof} and~\ref{subsection=pideal}.
In particular, the following symbols are frequently used:
\begin{itemize}
	\item the set $\PP$ of all prime numbers, and subsets of form $\PP_{\leq x}$ and  $\PP_{> x}$, 
	\item the set $\Ideals_K\ppart$ of all $p$-ideals of $\OK$, 
	\item and the set $|\Spec(\OK)|\ppart$ of all prime $p$-ideals of $\OK$.
\end{itemize}
Under Setting~\ref{setting=claim_GY} below, we transform the following expectation
\begin{equation}\label{eq:main-exp}
\EE\Biggl(\prod_{j\in[m]}(\Lambda_{R,\chi}\circ\theta_j)^2 \ \Bigg| \ \calB\Biggr),
\end{equation}
which appears on the left-hand side of~\eqref{Eq:formula_to_show}.
\begin{setting}\label{setting=claim_GY}
Assume the setting in Theorem~\ref{Th:Goldston_Yildirim}.
In addition, we assume $R > 1$.
\end{setting}

\begin{proposition}\label{prop:average-to-show-2.5}
Expectation~\eqref{eq:main-exp} is equal to
\begin{equation}\label{Eq:average-to-show-2.5}
 (\log R)^{2m}\sum_{(\ideala_j,\idealb_j )_{j \in[m]}\in\Ideals_K^{2m}}\Pi_{R,\chi}\left((\ideala_j,\idealb_j)_{j\in[m]}\right)\cdot\EE\Bigg(\prod_{j \in [m]}(\ichi_{\ideala_j\cap\idealb_j}\circ\theta_j) \ \Bigg| \ \calB\Biggr),
\end{equation}
where
 \[
 \Pi_{R,\chi}\left((\ideala_j,\idealb_j)_{j\in[m]}\right)\coloneqq\prod_{j \in[m]}\mu(\ideala_j)\mu(\idealb_j) \chi\left(\frac{\log\Nrm(\ideala_j)}{\log R}\right)\chi\left(\frac{\log\Nrm(\idealb_j)}{\log R}\right).
 \]
\end{proposition}
\begin{proof}
Let $x\in\calB$.
Substituting \eqref{Eq:def-of-von-M} for $\Lambda_{R,\chi}(\theta_j(x))$ in~\eqref{eq:main-exp} and expanding it, we \havethat\
\begin{align*}
&(\log R)^{-2m}\prod_{j\in[m]}\Lambda_{R,\chi}(\theta_j(x))^2\\
&=\sum_{\substack{(\ideala_j,\idealb_j)_{j\in[m]}\in\Ideals_K^{2m}\\ \theta_j(x)\in\ideala_j\cap\idealb_j \ (\forall j\in[m])}}\prod_{j\in[m]}\mu(\ideala_j)\mu(\idealb_j)\chi\left(\frac{\log\Nrm(\ideala_j)}{\log R}\right)\chi\left(\frac{\log\Nrm(\idealb_j)}{\log R}\right)\\
&=\sum_{(\ideala_j,\idealb_j)_{j\in[m]}\in\Ideals_K^{2m}}\left(\prod_{j\in[m]}\ichi_{\ideala_j\cap\idealb_j}(\theta_j(x))\right)\cdot\Pi_{R,\chi}\left((\ideala_j,\idealb_j)_{j\in[m]}\right).
\end{align*}
Only the characteristic functions $\prod_{j\in[m]}\ichi_{\ideala_j\cap\idealb_j}(\theta_j(x))$ depend on $x \in \calB$, and hence the desired result holds.
\end{proof}
Next we focus on the expectation in~\eqref{Eq:average-to-show-2.5}.
Although a statement similar to the following lemma might be standard in this research area, we write down a proof for the convenience of the reader.
\begin{lemma}\label{lem:E(D)}
Let $(\ideala_j,\idealb_j)_{j\in[m]}\in\Ideals_K^{2m}$,
and define the positive integer $D=D\left((\ideala_j,\idealb_j)_{j\in[m]}\right)$ by
\begin{align}\label{lem:E(D):1}
D\ZZ=\ZZ\cap\Biggl(\bigcap_{j\in[m]}(\ideala_j\cap\idealb_j)\Biggr).
\end{align}
If $\Nrm(\ideala_j),\Nrm(\idealb_j)\leq R$ holds for every $j\in[m]$, then the following hold.
\begin{enumerate}[$(1)$]
\item\label{en:D-R}
$D\le R^{2m}$.
\item\label{en:B-D}
We have
\[
\EE\Biggl(\prod_{j\in[m]}(\ichi_{\ideala_j\cap\idealb_j}\circ\theta_j) \ \Bigg| \ \calB\Biggr)\\
=\EE\Biggl(\prod_{j\in[m]}(\ichi_{\ideala_j\cap\idealb_j}\circ\theta_j) \ \Bigg| \ (\ZZ/D\ZZ)^t\Biggr)+O_t(R^{-2m-1}).	
\]
\end{enumerate}
\end{lemma}
Note that for every $x\in(\ZZ/D\ZZ)^t$, the value $\ichi_{\ideala_j\cap\idealb_j}(\theta_j(x))\in\{0,1\}$ is well-defined.
\begin{proof}
First we prove \eqref{en:D-R}.
For each $\ideala\in\{\ideala_j,\idealb_j\}_{j\in [m]}$,
the injection $\ZZ/(\ZZ\cap\ideala)\hookrightarrow\OK/\ideala$ shows $\#(\ZZ/(\ZZ\cap\ideala))\leq\#(\OK/\ideala)=\Nrm(\ideala)\leq R$.
This together with
\[
D\ZZ=\bigcap_{j\in[m]}\left((\ZZ\cap\ideala_j)\cap(\ZZ\cap\idealb_j)\right)\supseteq\prod_{j\in [m]}\left((\ZZ\cap \ideala_j)\cdot(\ZZ\cap\idealb_j)\right) 
\]
implies that
\[
D=\# (\ZZ /D\ZZ )\leq\prod_{j\in[m]}(\#(\ZZ /\ZZ \cap\ideala_j)\cdot\#(\ZZ/\ZZ\cap\idealb_j))\leq R^{2m}.
\]

Next we prove \eqref{en:B-D}.
Since $\calB$ contains pairwise disjoint $\prod_{i\in[t]}\left\lfloor\frac{\#I_i}{D}\right\rfloor$  translates of $[D]^t$ in $\ZZ^t$,
we see that
\begin{multline}
\#\calB\cdot\EE\Biggl(\prod_{j\in[m]}(\ichi_{\ideala_j\cap\idealb_j}\circ\theta_j) \ \Bigg| \ \calB\Biggr)
=\Biggl(\prod_{i\in[t]}\left\lfloor\frac{\#I_i}{D}\right\rfloor\Biggr)\cdot D^t\cdot\EE\Biggl(\prod_{j\in[m]}(\ichi_{\ideala_j \cap\idealb_j}\circ\theta_j) \ \Bigg| \ (\ZZ/D\ZZ)^t\Biggr)\\
+O\left(\#\calB-\Biggl(\prod_{i\in[t]}\left\lfloor\frac{\#I_i}{D}\right\rfloor\Biggr)\cdot D^t\right).
\end{multline}
Since $\left\lfloor\frac{\#I_i}{D}\right\rfloor\cdot D>\#I_i-D$, we have
\[
\frac{1}{\#\calB}\left(\#\calB-\Biggl(\prod_{i\in[t]}\left\lfloor\frac{\#I_i}{D}\right\rfloor\Biggr)\cdot D^t\right)<1-\frac{1}{\#\calB}\prod_{i\in[t]}(\#I_i-D)=1-\prod_{i\in[t]}\left(1-\frac{D}{\#I_i}\right).
\]
Note that $D/\#I_i\leq R^{2m}/R^{4m+1}=R^{-2m-1}$ follows from \eqref{en:D-R}.
By Bernoulli's inequality $(1-x)^t\ge 1-tx $ for every $x \leq 1$, we conclude that
\[
1-\prod_{i\in[t]}\left(1-\frac{D}{\#I_i}\right)\leq1-(1-R^{-2m-1})^t\leq tR^{-2m-1}.
\]
Furthermore, 
\[
\frac{1}{\#\calB}\cdot\Biggl(\prod_{i\in[t]}\left\lfloor\frac{\#I_i}{D}\right\rfloor\Biggr)\cdot D^t=1+O_t(R^{-2m-1})
\]
follows, and we derive the desired equality.
\end{proof}
We write for short the following expectation
\begin{equation}\label{Eq:E}
\E\left((\ideala_j,\idealb_j)_{j \in [m]}\right)=\E\left((\ideala_j,\idealb_j)_{j\in[m]} ; (\theta_j)_{j \in [m]}\right)\coloneqq\EE\Biggl(\prod _{j\in[m]}(\ichi_{\ideala_j\cap\idealb_j}\circ\theta_j) \ \Bigg| \ (\ZZ/D\ZZ)^t\Biggr),
\end{equation}
which depends on $(\theta_j)_{j \in [m]}$ and $(\ideala_j,\idealb_j)_{j\in [m]}$.
Here the positive integer $D=D\left((\ideala_j,\idealb_j)_{j\in [m]}\right)$ is taken as in~Lemma~\ref{lem:E(D)}.
In what follows, we prove the multiplicativity of $\E$ in Lemma~\ref{lem:p-typical-E} and the estimates in Lemma~\ref{Lem:value_of_E}. %
Once these two properties are established, we will no longer need the definition of $\E$ for the proof of Theorem~\ref{Th:Goldston_Yildirim}.

\begin{proposition}\label{prop:replace_B_by_D}
Expectation~\eqref{Eq:average-to-show-2.5} equals
\begin{equation}\label{Eq:formula_to_show_3}
(\log R)^{2m}\sum_{(\ideala_j,\idealb_j)_{j\in[m]}\in(\Ideals_K)^{2m}}\Pi_{R,\chi}\left((\ideala_j,\idealb_j)_{j\in[m]}\right)\cdot\E\left((\ideala_j,\idealb_j)_{j \in [m]}\right)
\end{equation}
with an additive error term $O_{m,t,K}\left(\frac{(\log R)^{2m}}{R}\right)$, where $\Pi_{R,\chi}\left((\ideala_j,\idealb_j)_{j\in[m]}\right)$ is defined in Proposition~$\ref{prop:average-to-show-2.5}$.
\end{proposition}
\begin{proof}
By Lemma~\ref{lem:E(D)}~\eqref{en:B-D}, $\supp(\chi) \subseteq [-1,1]_\RR$ and $|\Pi_{R,\chi}((\ideala_j,\idealb_j)_{j\in[m]})|\leq 1$, the absolute value of
 the difference between \eqref{Eq:average-to-show-2.5} and \eqref{Eq:formula_to_show_3} is estimated as
\begin{align*}
	&\left\vert 
		(\log R)^{2m}\sum_{(\ideala_j,\idealb_j)_{j\in[m]}\in(\Ideals_K)^{2m}}\Pi_{R,\chi}\left((\ideala_j,\idealb_j)_{j\in[m]}\right)\cdot O_t(R^{-2m-1}) 
	\right\vert
	\\
	&\le 
	(\log R)^{2m}\sum_{
		\substack{
			(\ideala_j,\idealb_j)_{j\in[m]}\in(\Ideals_K)^{2m} \\
			\Nrm(\ideala_j),\Nrm(\idealb_j)\leq R \text{ for all } j \in [m]
		}
	}
	O_t(R^{-2m-1}) .
\end{align*}
Since Proposition~\ref{proposition=idealdensity} implies that the number of summands in the right-hand side is $O_{m,K}(R^{2m})$, the assertion follows.
\end{proof}
It will turn out that %
the contribution of the error term in 
Proposition~\ref{prop:replace_B_by_D}
is permissible in the proof of Theorem~\ref{Th:Goldston_Yildirim} as well as all the other error terms that appear below.
Next we focus on the expectation $\E\left((\ideala_j,\idealb_j)_{j \in [m]}\right)$ in~\eqref{Eq:formula_to_show_3}.
\subsection{The expectation of the characteristic function}\label{subsection=ichi-exp}
In this subsection, we assume the following setting:
\begin{setting}\label{Setting:kitaichi-hyoka}
Assume Setting~\ref{setting=claim_GY}.
Fix a tuple $(\ideala_j,\idealb_j)_{j\in [m]}$ of arbitrary non-zero ideals of $\OK$ each of which is not necessary of norm at most $R$,
and write $\idealc _j\coloneqq\ideala_j\cap\idealb_j$ for short.
Note that $\idealc_j\ppart=\ideala_j\ppart\cap\idealb_j\ppart $ (see Subsection~\ref{subsection=pideal} where $p$-parts are defined). 
Let $D$ be the positive integer defined in Lemma~\ref{lem:E(D)}.
Then note that $D\ZZ=\ZZ\cap(\bigcap_{j\in[m]}\idealc_j)$ and $D\ppart\ZZ=\ZZ\cap(\bigcap_{j\in[m]}\idealc_j\ppart )$.
\end{setting}
The following symbols and easy equality are helpful to estimate \eqref{Eq:E}.
We consider the $\ZZ$-module homomorphisms and affine transformations
\[
\overline{\psi_j}, \overline{\theta_j}\colon(\ZZ/D\ZZ)^t\to\OK/\idealc_j
\]
induced by $\psi_j$ and $\theta_j$, respectively.
Let
\[
\overline{\psi}, \overline{\theta}\colon(\ZZ/D\ZZ)^{t}\to\prod_{j\in[m]}\OK /\idealc_j
\]
be the two maps defined by $\overline{\psi}(x) = (\overline{\psi_1}(x),\ldots,\overline{\psi_m}(x))$ and $\overline{\theta}(x) = (\overline{\theta_1}(x),\ldots,\overline{\theta_m}(x))$.
Then we see that
\begin{equation}\label{Eq:easy-paraphrase}
\E\left((\ideala_j,\idealb_j)_{j \in[m]}\right)=\EE\Biggl(\prod_{j\in[m]}(\ichi_{\idealc_j}\circ\theta_j) \ \Bigg| \ (\ZZ/D\ZZ)^t\Biggr)=\EE\left(\ichi_{\{0\}}\circ\overline{\theta}\mid (\ZZ/D\ZZ)^t\right).
\end{equation}
\begin{lemma}\label{lem:p-typical-E}
The expectation $\E\left((\ideala_j,\idealb_j)_{j\in[m]}\right)$ is decomposed into its $p$-parts.
Namely,
\[
\E\left((\ideala_j,\idealb_j)_{j\in [m]}\right)=\prod_{p \in\PP}\E\left((\ideala^{(p)}_j,\idealb^{(p)}_j)_{j\in[m]}\right).
\]
\end{lemma}
\begin{proof}
Recall the definition of $\E$ from \eqref{Eq:E}  %
or \eqref{Eq:easy-paraphrase}. It suffices to prove that
\[
\EE\Biggl(\prod\limits_{j\in[m]}(\ichi_{\idealc_j}\circ\theta_j) \ \Bigg| \ (\ZZ/D\ZZ)^t \Biggr) = \prod_{p \in\PP}\EE\Biggl(\prod_{j\in[m]}(\ichi_{\idealc_j\ppart}\circ\theta_j) \ \Bigg| \ (\ZZ/D^{(p)}\ZZ)^t \Biggr).
\]

By the Chinese remainder theorem (Lemma~\ref{lemma=chineseremainder}), the $\ZZ $-module homomorphism $\overline{\psi}$ equals the product of its restrictions $\overline{\psi}\ppart$ to $(\ZZ/D\ppart \ZZ)^t$, that means
\[
\overline{\psi}=\prod_{p\in\PP}\overline{\psi}\ppart\colon\prod_{p\in\PP}(\ZZ/D\ppart\ZZ)^t\to\prod_{p\in\PP}\left(\prod_{j\in[m]}\OK/\idealc_j\ppart\right).
\]
Hence the affine transformation $\overline{\theta}$ is the product of the restrictions 
\[
\overline{\theta}\ppart\colon(\ZZ/D\ppart\ZZ)^t\to\prod_{j\in[m]}\OK/\idealc_j\ppart  .
\]
Applying \eqref{Eq:easy-paraphrase} to the ideals $(\idealc_j)_{j\in[m]}$ and $(\idealc_j\ppart)_{j\in[m]}$, we see that it suffices to prove
\begin{equation}\label{Eq:Chinese-equality-to-show}
\EE\left(\ichi_{\{0\}}\circ\overline{\theta}\relmiddle|(\ZZ/D\ZZ)^t\right)=\prod_{p\in\PP}\EE \left(\ichi_{\{0\}}\circ\overline{\theta}\ppart\relmiddle|(\ZZ/D\ppart\ZZ)^t\right).
\end{equation}
Since $\idealc_j=\prod_{p \in \PP} \idealc_j^{(p)}=\bigcap_{p \in \PP} \idealc_j^{(p)}$, for each element
\[
x= (x_p)_{p\in\PP}\in\prod_{p\in\PP}\left(\ZZ/D\ppart\ZZ\right)^t ,
\]
we have $\overline{\theta}(x)=0$ if and only if $\overline{\theta}\ppart(x_p)=0$ holds for every $p\in\PP$.
Hence
\[
\EE\left(\ichi_{\{0\}}(\overline{\theta}(x))\relmiddle|x\in(\ZZ/D\ZZ)^t\right)=\EE\left(\prod_{p\in \PP}\ichi_{\{0\}}(\overline{\theta}\ppart(x_p))\relmiddle| (x_p)_{p\in\PP}\in\prod_{p\in\PP}(\ZZ/D\ppart\ZZ)^t\right).
\]
This coincides with the right-hand side of \eqref{Eq:Chinese-equality-to-show} by a Fubini-type argument.
\end{proof}
By this lemma, in the next subsection, we may restrict our attention to a tuple of $p$-ideals for each rational prime number $p$. 
\subsection{The expectation of the characteristic function for $p$-ideals}
In this subsection, we assume the following setting.
We use Greek symbols such as $\alpha_j$, $\beta_j$ and $\gamma_j$ for $p$-ideals.
\begin{setting}\label{Setting:p-ideals}
Let $K,n,t,m,(\psi_j)_{j\in[m]}$ be as in Theorem~\ref{Th:Goldston_Yildirim}; we assume \eqref{Eq:no-inclusion}.
In addition, we take $w, W, (\theta_j)_{j\in[m]}$ as in Theorem~\ref{Th:Goldston_Yildirim}.
Let $(\pideala_j,\pidealb_j)_{j\in [m]}\in(\Ideals _K^{(p)})^{2m}$ be a tuple of $p$-ideals for some prime number $p$.
Write $\pidealc_j\coloneqq\pideala_j\cap\pidealb _j$.
Let $D$ be the positive integer such that $D\ZZ=\ZZ \cap\left(\bigcap_{j\in[m]}\pidealc_j\right)$.
Then the ideals $\pidealc_j$ are $p$-ideals, and $D$ is a power of $p$.
The two maps $\overline{\psi}$, $\overline{\theta}\colon(\ZZ/D\ZZ)^t\to\prod_{j\in[m]}\OK/\pidealc_j$ are defined in Subsection~\ref{subsection=ichi-exp}.
\end{setting}
\begin{setting}\label{Setting:w_0}
Assume that $w$ is at least $w'_0$ and $w''_0$, which are defined as follows:
\begin{enumerate}[(1)]
\item\label{en:w_1}
For every $j\in[m]$, the cardinality of $\coker(\psi_j)$ is finite by assumption~\eqref{Eq:no-inclusion}.
We let $w'_0$ be the largest prime factor of $\prod_{j\in[m]} \#\coker(\psi_j)$.
\item\label{en:w_2}
By assumption~\eqref{Eq:no-inclusion}, for each $(j,k)\in[m]^2$  with $j\neq k$, we may take $x_{jk}\in\ker(\psi_j)\setminus\ker(\psi_k)$.
Then let
\[
w''_0\coloneqq\max\biggl\{\ell\in\PP : \exists\idealp\in|\Spec\OK|^{(\ell)}, \ \idealp\mid \prod_{(j,k)\in[m]^2, \ j\neq k}\psi_k(x_{jk})\biggr\}.
\]
We take the elements $x_{jk}$ which minimizes $w_0''$, so that $w_0''$ depends only on $(\psi_j)_{j \in [m]}$.
\end{enumerate}
\end{setting}
Below we estimate the expectation
\[
\E\left((\pideala_j,\pidealb_j)_{j\in[m]}\right)=\EE\Biggl(\prod_{j\in[m]}(\ichi_{\pidealc_j}\circ\theta_j) \ \Bigg| \ (\ZZ/D\ZZ)^t\Biggr)=\EE\left(\ichi_{\{0\}}\circ\overline{\theta}\mid (\ZZ/D\ZZ)^t\right).
\]
We prove the following without assuming Setting~\ref{Setting:w_0}.
\begin{lemma}\label{Lem:smaller_or_bigger}
If $p>w$, then
\[
\E\left((\pideala_j,\pidealb_j)_{j\in[m]}\right)=
\begin{cases}
	(\#\Im(\overline{\psi}))^{-1} & \text{if } 0\in\Im(\overline{\theta}),\\ 
	0 & \text{otherwise}.\end{cases}
\]

\end{lemma}
\begin{proof}
In general, for an affine transformation $\theta\colon Z\to Z'$ between two finite abelian groups  $Z$ and $Z'$, we have
\[
\EE\left(\ichi_{\{0\}}\circ\theta\mid Z\right)=\begin{cases} (\#\Im(\theta))^{-1} &\text{if } 0\in\Im(\theta), \\ 0 & \text{otherwise}.\end{cases}
\]
By applying this to $\overline{\theta}\colon(\ZZ/D\ZZ)^t\to\prod_{j\in[m]}\OK/\pidealc_j $,
we see that the expectation $\E\left((\pideala_j,\pidealb_j)_{j\in[m]}\right)$ equals $(\#\Im(\overline{\theta}))^{-1}$ if $0\in\Im(\overline{\theta})$, and $0$ otherwise.
Consider the first case.
By $p>w$, $W$ and $p$ are coprime.
Since the order of each $\OK/\pidealc_j$ is a power of $p$, this implies that the multiplication by $W$ on $\prod_{j\in[m]}\OK/\pidealc_j$ is an automorphism.
Therefore $\#\Im(\overline{\theta})=\#\Im (\overline{\psi})$ follows.
\end{proof}
Only the following lemma and its consequences exploit assumption~\eqref{Eq:no-inclusion}.
\begin{lemma}\label{Lem:value_of_E}
For each prime number $p$ and each tuple $(\pideala_j,\pidealb_j)_{j\in[m]}$ of $p$-ideals, the following hold:
\begin{enumerate}[$(1)$]
\item\label{en:localfactor1}
If $\pidealc_j=\OK$ holds for every $j \in [m]$, then $\E\left((\pideala_j,\pidealb_j)_{j\in[m]}\right)=1$.
\item\label{en:localfactor2}
Suppose that $p\leq w$ and there exists $j_0\in[m]$ such that $\pidealc_{j_0}\subsetneq\OK$.
Then $\E\left((\pideala_j,\pidealb_j)_{j\in[m]}\right)=0$.
\item\label{en:localfactor3}
Suppose that $p>w$ and there exists $j_0 \in [m]$ such that $\pidealc_{j_0}\subsetneq\OK$ and $\pidealc_j=\OK$ for all $j\in[m]\setminus\{j_0\}$.
Then $\E\left((\pideala_j,\pidealb_j)_{j\in[m]}\right)=1/\Nrm(\pidealc_{j_0})$.
\item\label{en:localfactor4}
Suppose that $p>w$ and there exist two distinct $j_1$ and $j_2\in[m]$ such that $\pidealc_{j_1},\pidealc_{j_2}\subsetneq\OK$.
Then $\E\left((\pideala_j,\pidealb_j)_{j\in[m]}\right)\leq1/p^2$.
\end{enumerate}
\end{lemma}
\begin{proof}
First we prove \eqref{en:localfactor1}.
If $\pidealc_j=\OK$ for every $j \in [m]$, then $\prod_{j\in[m]}(\ichi_{\pidealc_j}\circ\theta_j)$ is identical with the constant function $1$, and hence $\E\left((\pideala_j,\pidealb_j)_{j\in[m]}\right)=1$.

Secondly, we prove \eqref{en:localfactor2}.
It suffices to show that for all $x\in(\ZZ/D\ZZ)^t$, $\theta_{j_0}(x)=W\psi_{j_0}(x)+b_{j_0}\not\in\pidealc_{j_0}$.
Let $\idealp$ be an arbitrary prime ideal $\idealp\supseteq\pidealc_{j_0}$.
Since $\pidealc_{j_0}$ is a $p$-ideal, we have $\idealp\cap\ZZ=p\ZZ$.
From $p\leq w$, $p$ divides $W$, and hence $W\in\idealp$ follows.
In addition, the assumption $b_{j_0} \OK+ W \OK=\OK$ implies that  $b_{j_0}\not\in\idealp$.
Hence we see that for all $x \in \ZZ^t$, $\theta_{j_0}(x)=W\psi_{j_0}(x)+b_{j_0}\not\in\idealp$.
This together with $\pidealc_{j_0}\subseteq\idealp$ implies that $\theta_{j_0}(x)\not\in\pidealc_{j_0}$,
as desired.

Thirdly, we prove \eqref{en:localfactor3}.
Set $C_{j_0}\coloneqq\#\coker(\psi_{j_0})$.
Note that for every $x \in \OK$, $C_{j_0} \cdot x$ is contained in the image of $\psi_{j_0}$.
By Setting~\ref{Setting:w_0}~\eqref{en:w_1}, $p$ and $C_{j_0}$ are coprime.
Since the order of $\OK/\pidealc_{j_0}$ is a power of $p$, this implies that the multiplication by $C_{j_0}$ on $\OK/\pidealc_{j_0}$ is an automorphism.
Hence we see that $\overline{\psi_{j_0}}\colon(\ZZ/D\ZZ)^t\to\OK/\pidealc_{j_0}$ is surjective.
Since $W$ and $p$ are coprime, the map $\overline{\theta_{j_0}}=\Aff_{W,b_{j_0}}\circ\overline{\psi_{j_0}}\colon(\ZZ/D\ZZ)^t\to\OK/\pidealc_{j_0}$ is also surjective.
Then Lemma~\ref{Lem:smaller_or_bigger} yields that $\E\left((\pideala_j,\pidealb_j)_{j\in[m]}\right)=(\#\Im(\overline{\psi}))^{-1}$;
recall that $\pidealc_j=\OK$ for all $j \in [m] \setminus \{j_0\}$.
We have
\[
\#\Im(\overline{\psi})=\#\Im(\overline{\psi_{j_0}})=\#(\OK/\pidealc_{j_0})=\Nrm(\pidealc_{j_0}).
\]
This is the desired result.

Finally we prove \eqref{en:localfactor4}.
By Lemma~\ref{Lem:smaller_or_bigger}, it suffices to show that $\#\Im(\overline{\psi})\geq p^2$.
Without loss of generality, we may assume that $\pidealc_1,\pidealc_2 \subsetneq \OK$. 
Recall that two elements $x_{12}$ and $x_{21}$ are chosen in Setting~\ref{Setting:w_0}~\eqref{en:w_2}.
Since $p>w\geq w''_0$, both $\psi_2(x_{12})$ and $\psi_1(x_{21})$ are prime to all $p$-ideals.
We focus on the mapping $(\overline{\psi_1},\overline{\psi_2})$, which is defined as
\begin{equation*}\begin{array}{cccc}
(\overline{\psi_1},\overline{\psi_2})\colon
&(\ZZ/D\ZZ)^t 
&\to
&\OK/\pidealc_1\times\OK/\pidealc_2
\\[10pt]
& x
& \mapsto
& (\overline{\psi_1}(x),\overline{\psi_2}(x)).
\end{array}
\end{equation*}
This maps $x_{12}$ and $x_{21}$ to non-zero elements $(0,\overline{\psi_2}(x_{12}))$ and $(\overline{\psi_1}(x_{21}),0)$, respectively.
The order of the linear span of these two images is at least $p^2$.
Hence
\[
p^2\leq\#\Im(\overline{\psi_1},\overline{\psi_2})\leq\#\Im(\overline{\psi}),
\]
and \eqref{en:localfactor4} follows.
\end{proof}
\subsection{Estimate for the error by a change of domain of integration}
\label{subsec:domain change}
In this subsection, we estimate expectation~\eqref{eq:main-exp}.
Its main term is equal to \eqref{Eq:formula_to_show_3} by Proposition~\ref{prop:replace_B_by_D}.
In this subsection, we prove Proposition~\ref{proposition=third}, which provides an integral representation of the main term.

First, we use  the Fourier transform to derive an integral representation of $\chi$.
Recall that $\chi $ is the smooth compactly supported function which was chosen in the setting of Theorem~\ref{Th:Goldston_Yildirim}.
Let $\chihat $ be the inverse Fourier transform of the function $x\mapsto e^x \chi (x)$.
Here we normalize it to satisfy
\begin{equation}\label{eq:int_rep_of_chi}
e^x\chi(x)=\int_{\RR}\chihat(\xi)e^{-x\xi\kyo}\rd\xi,\quad\text{or equivalently, }\quad\chi(x)=\int_{\RR}\chihat(\xi)e^{-x(1+\xi\kyo)}\rd\xi.
\end{equation}
Then for all $\idealc\in\Ideals_K$ and $R>1$, we \havethat\
\begin{equation}\label{eq:int-chi}
\chi\left(\frac{\log\Nrm(\idealc)}{\log R}\right)=\int_{\RR}\chihat(\xi)\Nrm(\idealc)^{-(1+\xi\kyo)/\log R}\rd\xi.
\end{equation}

Assume the following:
\begin{setting}\label{setting=error}
Assume the setting in Theorem~\ref{Th:Goldston_Yildirim}.
In addition, assume that $R \geq e$ and $w \geq\max\{4^{mn},w'_0,w''_0\}$,
where $w'_0$ and $w''_0$ are defined in Setting~\ref{Setting:w_0}.
Recall that $\E\left((\ideala_j,\idealb_j)_{j\in[m]}\right)$ is defined as \eqref{Eq:E}.
Let $I=I(R)\coloneqq[-\sqrt{\log R},\ +\sqrt{\log R}]_{\RR}$, and $\xi_j$ and $\eta_j$ be variables of integration.
Let us use shorthand symbols %
$\rd\uxi = \rd\xi_1\cdots\rd\xi_m$ and $\rd\ueta = \rd\eta_1 \cdots \rd\eta_m$.
Let
\[
z_j \coloneqq\frac{1+\xi_j\sqrt{-1}}{\log R}, \quad w_j\coloneqq\frac{1+\eta_j\sqrt{-1}}{\log R}.
\]
\end{setting}
The goal in this subsection is to prove the following.
We define $E=E\left((\xi_j,\eta_j)_{j\in[m]};R\right)$, depending also on $(\theta_j)_{j\in[m]}$, by
\begin{equation}\label{eq:def-of-E}
E\coloneqq\sum_{(\ideala_j,\idealb_j)_{j\in[m]}\in\Ideals_K^{2m}}\Biggl(\prod_{j\in[m]}\left\{\mu(\ideala_j)\mu(\idealb_j)\Nrm(\ideala_j)^{-z_j}\Nrm(\idealb_j)^{-w_j}\right\}\Biggr)\E\left((\ideala_j,\idealb_j)_{j \in [m]}\right).
\end{equation}

\begin{proposition}\label{proposition=third}
The series $E=E\left((\xi_j,\eta_j)_{j\in[m]};R\right)$ converges uniformly on $(\xi_j,\eta_j)_{j\in[m]}\in I^{2m}$.
Furthermore, for every positive real number $A$, \eqref{Eq:formula_to_show_3} is equal to
\begin{equation}\label{Eq:formula_to_show_4'}
(\log R)^{2m}\int_{I^{2m}}\rd\uxi\rd\ueta\left[\Biggl(\prod_{j\in[m]}\chihat(\xi_j)\chihat(\eta_j)\Biggr)\cdot E\left((\xi_j,\eta_j)_{j\in[m]};R\right)\right]
\end{equation}
with an additive error term $O_{A,\chi,m,n}\left((\log R)^{-A}\right)$.
\end{proposition}
The proof of this proposition will proceed in several steps:
\begin{itemize}
	\item Substitute for $\chi $ in \eqref{Eq:formula_to_show_3} its integral representation over $\RR$
	to obtain a summation of integrals over $\RR^{2m}$. 
	\item Replace these integrals by ones over ${I^{2m}}$ up to small additive errors.
	\item 
	Show the uniform (and absolute) convergence of $E$
	to
	interchange summation and integration. 
	We end up with an integral over $I^{2m}$ of a certain summation, 
	which turns out to be the desired expression.
\end{itemize} 
Each term of the series~\eqref{Eq:formula_to_show_3} has $\chi(\log\Nrm(\cdot)/ \log R)$ as a factor.
Hence it suffices to consider this series over $(\ideala_j,\idealb_j)_{j \in [m]}$ all of whose entries have norm at most $R$.
Hence it is a finite series.
However, the term represented by the integral in Lemma~\ref{lem:prod_chi} should be considered over all $(\ideala_j,\idealb_j)_{j \in [m]}$,
and hence $E$ defined in \eqref{eq:def-of-E} is an infinite series over all $(\ideala_j,\idealb_j)_{j \in [m]}$.
This causes a subtlety in interchanging summation and integration
unless we prove that $E$ converges absolutely and uniformly.

We begin by estimating the error caused by the change of the domain of integration from $\mathbb{R}^{2m}$ to $I^{2m}$.
First we recall the following estimate in Fourier analysis; we write down a proof for the convenience of the reader.
\begin{lemma}\label{lem:Fourier-easy}
For every real-valued $C^N$-function $f$ with compact support, we have
\[
\Fourier(f)(x)\coloneqq\int_{\RR}f(\xi)e^{x\xi\kyo}\rd\xi=O_{N,f}\left((1+|x|)^{-N}\right),
\]
where $\Fourier(f)$ denotes the inverse Fourier transform of $f$.
\end{lemma}
\begin{proof}
Since the support of $f$ is compact, $f\in L^1(\RR)$ holds, and then $\|\Fourier(f)\|^{}_{\infty}\leq\|f\|^{}_1<\infty$ follows.
For every $k \in [0,N]$, since the $k$th derivative of $f$ has a compact support, there exists a real number $C_{N,f}$ such that $\|\Fourier(f^{(k)})\|^{}_{\infty}\leq C_{N,f}$.
By integration by parts, we have $\Fourier(f')(x)=-\kyo x\Fourier(f)(x)$. 
Applying this repeatedly, we see that for all $k \in [0,N]$, $|\Fourier(f^{(k)})(x)|=|x|^k|\Fourier(f)(x)|$.
Hence
\[
|\Fourier(f)(x)|\sum_{k\in[0,N]}|x|^k=\sum_{k\in[0,N]}|\Fourier(f^{(k)})(x)|\leq(N+1)C_{N,f}\eqqcolon C'_{N,f}.
\]
By the binomial theorem, there exists $c_N>0$ such that $c_N(1+|x|)^N\leq\sum_{k\in[0,N]}|x|^k$.
Therefore we obtain the desired estimate $|\Fourier(f)(x)|\leq\frac{C'_{N,f}}{c_N}(1+|x|)^{-N}$.
\end{proof} 
\begin{corollary}\label{cor:ineq_of_chihat}
The integral $\int_{\RR}|\chihat(\xi)|\rd\xi$ is a finite value.
In other words
\begin{equation}\label{Eq:ineq_of_chihat2pre}
\int_{\RR}|\chihat(\xi)|\rd\xi=O_{\chi}(1).
\end{equation}
In addition, for all positive real numbers $b$ and $A$, we have
\begin{equation}\label{Eq:ineq_of_chihat2}
\int_b^{\infty}|\chihat(\xi)|\rd\xi=O_{A,\chi}(b^{-A})\quad \text{ and } \quad \int_{-\infty}^{-b}|\chihat(\xi)|\rd\xi=O_{A,\chi}(b^{-A}).
\end{equation}
\end{corollary}
\begin{proof}
The $C^\infty$-function $\chi$ has a compact support.
Applying Lemma~\ref{lem:Fourier-easy} with $f(x)=e^x\chi(x)$, we have $\chihat(\xi)=O_{B,\chi}\left((1+|\xi|)^{-B}\right)$ for all $B\geq 0$.
Let $b\geq 0$, $B=A+1$, and $I_b=[b,\infty]_{\RR}$ or $[-\infty,-b]_{\RR}$.
Then we \havethat\
\begin{align*}
\int_{I_b}|\chihat(\xi)|\rd\xi	
&=\int_{I_b}O_{B,\chi}\left((1+|\xi|)^{-B}\right)\rd\xi=O_{B,\chi}\left(\int_{b}^{\infty}(1+\xi)^{-B}\rd\xi \right)\\
&=O_{B,\chi}\left(\frac{1}{B-1}(1+b)^{1-B}\right)
=\begin{cases}
O_{A,\chi}\left(b^{-A}\right) & \text{if } b>0,\\
O_{A,\chi}(1) &\text{otherwise}.\
\end{cases}
\end{align*}
\end{proof}
\begin{lemma}\label{lem:prod_chi}
For every positive real number $A$ and every tuple $(\ideala_j,\idealb_j)_{j\in[m]}\in\Ideals_K^{2m}$, we \havethat\
\begin{multline*}
\prod_{j\in[m]}\chi\left(\frac{\log\Nrm(\ideala_j)}{\log R}\right)\chi\left(\frac{\log\Nrm(\idealb_j)}{\log R}\right)
=\int_{I^{2m}}\rd\uxi\rd\ueta\left[\prod_{j\in[m]}\Nrm(\ideala_j)^{-z_j}\Nrm(\idealb_j)^{-w_j}\chihat(\xi_j)\chihat(\eta_j)\right]\\
+O_{A,\chi,m}\left((\log R)^{-A}\prod_{j\in[m]}\Nrm(\ideala_j)^{-\frac{1}{\log R}}\Nrm(\idealb_j)^{-\frac{1}{\log R}}\right).
\end{multline*}
\end{lemma}
\begin{proof}
Let $\idealc\in\Ideals_K$.
Integral representation~\eqref{eq:int-chi} of $\chi$ is decomposed as
\begin{equation}	\label{eq:decompose integral representation}
\chi\left(\frac{\log\Nrm(\idealc)}{\log R}\right)=\int_I\chihat(\xi)\Nrm(\idealc)^{-(1+\xi\kyo)/\log R}\rd\xi+\int_{\RR\setminus I}\chihat(\xi)\Nrm(\idealc)^{-(1+\xi\kyo)/\log R}\rd\xi.
\end{equation}
The two terms on the right-hand side are estimated in the following manner:
the first estimate~\eqref{Eq:ineq_of_chihat2pre} in Corollary~\ref{cor:ineq_of_chihat} implies that
\[
\left|\int_I\chihat(\xi)\Nrm(\idealc)^{-(1+\xi\kyo)/\log R}\rd\xi\right|\leq\Nrm(\idealc)^{-1/\log R}\int_I|\chihat(\xi)|\rd\xi=O_{A,\chi}\left(\Nrm(\idealc)^{-1/\log R}\right).
\]
Similarly, the second estimate~\eqref{Eq:ineq_of_chihat2} in Corollary~\ref{cor:ineq_of_chihat} implies that
\begin{equation}\label{eq:chichi-error2}
\left|\int_{\RR\setminus I}\chihat(\xi)\Nrm(\idealc)^{-(1+\xi\kyo)/\log R}\rd\xi\right|=O_{A,\chi}\left(\Nrm(\idealc)^{-1/\log R}(\log R)^{-A}\right).
\end{equation}
Note that  $O_{A,\chi}\left(\Nrm(\idealc)^{-1/\log R}(\log R)^{-A}\right)=O_{A,\chi}\left(\Nrm(\idealc)^{-1/\log R}\right)$ by $R\geq e$.
We decompose each of the $2m$ factors of
\[
\prod_{j\in[m]}\chi\left(\frac{\log\Nrm(\ideala_j)}{\log R}\right)\chi\left(\frac{\log\Nrm(\idealb_j)}{\log R}\right)
\]
as \eqref{eq:decompose integral representation},
and expand this product.
As a result, this product is expressed as the sum of the main term whose domain of integration is $I^{2m}$ and other $2^{2m}-1$ error terms.
Each error term
is the product of $2m$ factors.
We apply \eqref{eq:chichi-error2} to one of the
factors whose domain of integration is $\RR \setminus I$
to conclude that it is expressed as $O_{A,\chi}\left(\Nrm(\idealc)^{-1/\log R}(\log R)^{-A}\right)$.
The other factors can be estimated as $O_{A,\chi}\left(\Nrm(\idealc)^{-1/\log R}\right)$.
Hence the desired conclusion follows.
\end{proof}

In the rest of this subsection, we mainly show that the error caused by application of Lemma~\ref{lem:prod_chi} is small enough.
A rough estimate of the error suffices here; however we prepare a precise lemma for the estimation of the main term that comes later.

For each prime number $p\in\PP$, we define
\begin{equation}\label{eq:def-of-E_p}
\begin{split}
E_p&=E_p\left((\xi_j,\eta_j)_{j\in[m]}; R\right)\\
&\coloneqq\sum_{(\pideala_j,\pidealb_j)_{j\in[m]}\in\left(\Ideals_K\ppart\right)^{2m}}\Biggl(\prod_{j\in[m]}\left\{\mu(\pideala_j)\mu(\pidealb_j)\Nrm(\pideala_j)^{-z_j}\Nrm(\pidealb_j)^{-w_j}\right\}\Biggr)\E\left((\pideala_j,\pidealb_j)_{j \in [m]}\right).
\end{split}
\end{equation}
Since each summand of this sum contains $\mu(\alpha_j)\mu(\beta_j)$ as a factor,
it suffices to consider the sum over the tuples $(\pideala_j,\pidealb_j)_{j \in [m]}$ consisting of square free $p$-ideals.
The number of such tuples is at most $4^{nm}$ by Lemma~\ref{lemma=squarefree}.
The absolute value of the summand labeled by $(\pideala_j,\pidealb_j)_{j\in[m]} \neq (\OK,\dots,\OK)$ is at most $1/p$ by Lemma~\ref{Lem:value_of_E}.
Hence for every prime number $p$ greater than $w \ (\geq 4^{mn})$, $E_p$ is a non-zero finite value.
We will need the following more precise result.
\begin{lemma}\label{Lem:Into_Euler_product}
For every prime number $p$ greater than $w$, we have
\begin{equation}\label{eq:E_p>w}
E_p=1-\Sigma_p\left((\xi_j,\eta_j)_{j\in[m]};R\right)+O(4^{mn}/p^2),
\end{equation}
where
\[
\Sigma_p\left((\xi_j,\eta_j)_{j\in[m]};R\right)\coloneqq\sum_{j\in[m]}\sum_{\pidealp\in|\Spec(\OK)|\ppart}\left(\Nrm(\pidealp)^{-1-z_j}+\Nrm(\pidealp)^{-1-w_j}-\Nrm(\pidealp)^{-1-z_j-w_j}\right).
\]
The infinite series $E$ defined by \eqref{eq:def-of-E} and the infinite product $\prod_{p\in\PP_{>w}}E_p$ converge absolutely and uniformly on $(\xi_j,\eta_j)_{j\in[m]}\in I^{2m}$.
Furthermore, they coincide, that is,
\begin{equation}\label{Eq:sum-equal-product}
E=\prod_{p\in\PP_{>w}}E_p.
\end{equation}
\end{lemma}
\begin{proof}
Let $p$ be a prime number greater than $w$.
Consider the case of  Lemma~\ref{Lem:value_of_E}~\eqref{en:localfactor3}, namely,
there exists $j_0\in[m]$ such that $\pidealp\coloneqq\pideala_{j_0}\cap\pidealb_{j_0}\in|\Spec(\OK)|\ppart$ and $\pideala_j=\pidealb_j=\OK$ for all $j\in[m]\setminus\{j_0\}$. 
Then Lemma~\ref{Lem:value_of_E} asserts that $\E((\pideala_j,\pidealb_j)_{j \in [m]})=\Nrm(\pidealp)^{-1}$.
Moreover we have
\[
(\pideala_{j_0},\pidealb_{j_0}) \in \{ (\pidealp,\OK), (\OK,\pidealp), (\pidealp,\pidealp) \}.
\]
Hence \eqref{eq:E_p>w} follows from Lemma~\ref{Lem:value_of_E}.

Let $\aE$ be the sum of the absolute values of the summands of $E$, and for each prime number $p$, let $\aE_p$ be that of $E_p$.
Applying Lemma~\ref{Lem:value_of_E} to the summands of $\aE_p$, we have 
\[
	\aE_p = 
	\begin{cases}
		1 + O_{m,n}\left(p^{-1-\frac{1}{\log R}}\right) &\text{ if } p > w,\\
		1	&\text{ otherwise},
	\end{cases}
\]
and
\begin{equation}\label{Eq:abs_series}
\prod_{p\in\PP}\aE_p=\prod_{p\in\PP_{>w}}\left(1+O_{m,n}\left(p^{-1-\frac{1}{\log R}}\right)\right).
\end{equation}
Since
\[
\sum_{p\in\PP_{>w}}O_{m,n}\left(p^{-1-\frac{1}{\log R}}\right)=O_{m,n}\left(\zeta\left(1+\frac{1}{\log R}\right)\right)<\infty,
\]
we see that product~\eqref{Eq:abs_series} converges uniformly on $(\xi_j,\eta_j)_{j\in[m]}\in I^{2m}$ (with $R$ fixed).
This implies that $\prod_{p\in\PP_{>w}}E_p$ converges absolutely and uniformly.

By the multiplicativity of the M{\"o}bius function, norm (Lemma~\ref{lemma=completemultiplicativity}) and $\E$ (Lemma~\ref{lem:p-typical-E}), we see that for each $(\ideala_j,\idealb_j)_{j\in[m]}\in\Ideals_K^{2m}$
the following value decomposes into its $p$-parts:
\begin{align*}
&\Biggl(\prod_{j\in[m]}\left\{\mu(\ideala_j)\mu(\idealb_j)\Nrm(\ideala_j)^{-z_j}\Nrm(\idealb_j)^{-w_j}\right\}\Biggr)\E\left((\ideala_j,\idealb_j)_{j \in [m]}\right)\\
&=\prod_{p\in\PP}\left[\Biggl(\prod_{j\in[m]}\left\{\mu(\ideala_j\ppart)\mu(\idealb_j\ppart)\Nrm(\ideala_j\ppart)^{-z_j}\Nrm(\idealb_j\ppart)^{-w_j}\right\}\Biggr)\E\left((\ideala_j\ppart,\idealb_j\ppart)_{j\in[m]}\right)\right].
\end{align*}
Hence we have equalities of infinite series,
the latter ones of which have been shown to converge:
\[\aE=\prod_{p\in\PP}\aE_p=\prod_{p\in\PP_{>w}}\aE_p.\]
Therefore $E$ converges absolutely and uniformly on $(\ideala_j,\idealb_j)_{j\in[m]}\in\Ideals_K^{2m}$, and we have $E = \prod_{p \in \PP_{>w}} E_p$.
\end{proof}
The following lemma provides a further estimate of \eqref{Eq:abs_series}.
\begin{lemma}\label{lem:by_binomial_theorem}
Let $C$ be a positive real number.
Then for each positive real number $w$, we have
\[
\prod_{p\in\PP_{> w}}\left(1+ C\cdot p^{-1-\frac{1}{\log R}}\right)\leq(\log R+ O(1))^{C}.
\]
\end{lemma}
\begin{proof}
It suffices to prove this lemma in the case of $w=1$.
By the inequality
$1-Cx \leq (1+x)^{-C}$
for every $x > -1$, we have
\[
1+C\cdot p^{-1-\frac{1}{\log R}}\leq\left(1-p^{-1-\frac{1}{\log R}}\right)^{-C}.
\]
Hence
\[
\prod_{p\in\PP}\left(1+ C\cdot p^{-1-\frac{1}{\log R}}\right)\leq\prod_{p\in\PP}\left(1-p^{-1-\frac{1}{\log R}}\right)^{-C}=\zeta\left(1+\frac{1}{\log R}\right)^C.
\]
The Riemann zeta function $\zeta$ satisfies $|\zeta(s)-\frac{1}{s-1}|\leq 1$ for every $s>1$.
It follows that $\zeta\left(1+\frac{1}{\log R}\right)=\log R+O(1)$ and the desired result follows.
\end{proof}
We will prove Proposition~\ref{proposition=third} by combination of the above lemmas.
\begin{proof}[Proof of Proposition~$\ref{proposition=third}$]
We apply Lemma~\ref{lem:prod_chi} to $\prod_{j\in[m]}\chi\left(\frac{\log\Nrm(\ideala_j)}{\log R}\right)\chi\left(\frac{\log\Nrm(\idealb_j)}{\log R}\right)$ in \eqref{Eq:formula_to_show_3},
and then \eqref{Eq:formula_to_show_3} is decomposed into the main term represented by integrals and the error term
$
O_{A,\chi,m}\left((\log R)^{2m-A} \cdot
\tilde{E} \right).$
Since $E$ converges absolutely and uniformly by Lemma~\ref{Lem:Into_Euler_product},
we can interchange integration and summation in the main term.
Then it turns out that the main term is equal to \eqref{Eq:formula_to_show_4'}.
Next we estimate the error term.
By applying Lemma~\ref{lem:by_binomial_theorem} to $\aE = \prod_{p \in \PP_{> w}} \aE_p$ with~\eqref{Eq:abs_series},
we can rewrite the error term as
\begin{equation}\label{eq:2m-A}
O_{A,\chi,m}\left((\log R)^{2m-A}\left(\log R+O(1)\right)^{O_{m,n}(1)}\right).
\end{equation}
By replacing $A$ with some sufficiently large $A$ depending on $m$ and $n$, we may rewrite this error as $O_{A,\chi,m,n}\left((\log R)^{-A}\right)$.
\end{proof}
\subsection{Calculation of the main term}
%
In this subsection, we continue the calculation of the main term 
singled out in Proposition~\ref{proposition=third}.
Recall from~\eqref{eq:def-of-E_p} the definition of $E_p=E_p((\xi_j,\eta _j)_{j\in [m]};R)$.
\begin{proposition}\label{prop:formula_5}
Assume Setting~$\ref{setting=error}$.
Let $A$ be a positive real number.
Then \eqref{Eq:formula_to_show_3} equals
\begin{equation}\label{Eq:formula_to_show_5}
(\log R )^{2m}\int_{I^{2m}}\rd\uxi\rd\ueta\left[\Biggl(\prod_{j\in[m]}\chihat(\xi_j)\chihat(\eta_j)\Biggr)\cdot\prod_{p\in\PP_{>w}}E_p\right]
\end{equation}
with an additive error term $O_{A,\chi,m,n}((\log R)^{-A})$.
\end{proposition}
In this subsection, we assume the following setting, which determines the choice of the three parameters $w_0$, $R_0$ and $F_0$ in Theorem~\ref{Th:Goldston_Yildirim} up to $c_1$, $c_K$ and $c_2$, respectively.
\begin{setting}\label{setting=main_term}
Assume the setting in Theorem~\ref{Th:Goldston_Yildirim},
and fix $\chi$ as in Definition~\ref{def:chi}.
Set the positive real numbers $w_0, R_0$ and $F_0$ in Theorem~\ref{Th:Goldston_Yildirim} in the following manner:
\begin{enumerate}[(1)]
\item\label{en:w_0last}
$w_0=w_0((\psi_j)_{j\in[m]})\coloneqq\max\{c_14^{mn},w'_0,w''_0\}$,
where $c_1$ is a sufficiently large absolute constant.
Both $w'_0$ and $w''_0$ are given in Setting~\ref{Setting:w_0}.
\item\label{en:R_0last}
$R_0=R_0(m,K)\coloneqq e^{c_Km^2}$,
where $c_K$ is a sufficiently large constant depending only on $K$.
\item\label{en:F_0last}
$F_0=F_0(m,n)\coloneqq c_2(mn)^{-1}$,
where $c_2$ is a sufficiently small absolute constant. 
\end{enumerate}
Assume that $w$ and $R$ are positive real numbers such that $w \geq w_0$, $R\geq R_0$ and $\log w \leq F_0\cdot\sqrt{\log R}$.
\end{setting}
Here $c_1$ will be determined in the proofs of Lemmas~\ref{lemma=Ep_equals_Ep'} and \ref{lem:prod_Ep_Ep'},
$c_K$ will be determined in Lemma~\ref{lem:Ep'(all_prime)},
and $c_2$ will be determined by Lemma~\ref{lem:Ep'(p_leq_w)}.
Since their actual values are not important for our purpose, 
we will not specify them.

We continue to use the symbols as in Subsection~\ref{subsec:domain change}.
Let $I=I(R)\coloneqq[-\sqrt{\log R},\ +\sqrt{\log R}]_{\RR}$, and $\xi_j$ and $\eta_j$ be variables of integration.
We use shorthand symbols
$\rd\uxi = \rd\xi_1\cdots\rd\xi_m$ and $\rd\ueta = \rd\eta_1 \cdots \rd\eta_m$.
Let $z_j=\frac{1+\xi_j\kyo}{\log R}$ and $w_j=\frac{1+\eta_j\kyo}{\log R}$. 
The function $\chihat $ is the inverse Fourier transform of the function $x\mapsto e^x \chi (x)$.
For each $(\xi_j,\eta_j)_{j\in[m]}\in I^{2m}$ and prime number $p\in\PP$, define $E_p'=E_p'\left((\xi_j,\eta_j)_{j\in[m]}; R\right)$ by
\[
E_p'\coloneqq\prod_{j\in[m]}\prod_{\pidealp\in|\Spec(\OK)|\ppart}\frac{\left(1-\Nrm(\pidealp)^{-1-z_j}\right)\left(1-\Nrm(\pidealp)^{-1-w_j}\right)}{\left(1-\Nrm(\pidealp)^{-1-z_j-w_j}\right)}.
\]
We first show estimates for $E_p$ and $E_p'$ and their relation.
The following lemma provides elementary estimates.
\begin{lemma}\label{Lem:fool-lemma}
The following estimates hold.
\begin{enumerate}[$(1)$]
\item\label{item:really-basic}
For every complex number $\varepsilon$ with $|\varepsilon| \leq 1/2$, we have
\begin{align*}
&\log(1+\varepsilon)=O(|\varepsilon|),\\
&\frac{1}{1-\varepsilon}=1+\varepsilon+O(|\varepsilon |^2)=1+O(|\varepsilon|).
\end{align*}
\item\label{item:really-basic2}
For every complex number $\varepsilon$ with $|\varepsilon| \leq 1$, we have
\[
e^\varepsilon=1+O(|\varepsilon|).
\]
\item\label{item:sum-of-many}
For every positive integer $k$ and all complex numbers $\varepsilon_1,\ldots,\varepsilon_k$ with $|\varepsilon_1|,\ldots,|\varepsilon_k| \leq 1/k$, we have
\[
\prod_{i\in[k]}(1+\varepsilon_i)=1+\sum_{i\in [k]}\varepsilon_i+O\left(k^2 \cdot \max_{i\in[k]}|\varepsilon_i|^2\right)=1+O\left(k \cdot \max_{i\in[k]}|\varepsilon_i|\right).
\]
\end{enumerate}
\end{lemma}
\begin{proof}
\eqref{item:really-basic} and \eqref{item:really-basic2} follow from the Taylor expansions.
Next \eqref{item:sum-of-many} follows from
\begin{align*}
\left| \prod_{i\in[k]}(1+\varepsilon_i) -\left( 1+\sum_{i\in [k]}\varepsilon_i \right) \right|
&\leq \sum_{j \in [2,k]} \binom{k}{j} \max_{i \in [k]} |\varepsilon_i|^j
\leq k^2 \cdot \max_{i \in [k]} |\varepsilon_i|^2 \sum_{j \in [2,k]} \binom{k}{j} \frac{1}{k^j}\\
&\leq e \cdot k^2 \cdot \max_{i \in [k]} |\varepsilon_i|^2.
\end{align*}
This completes the proof.
\end{proof}
\begin{lemma}\label{lemma=Ep_equals_Ep'}
For every tuple $(\xi _j,\eta _j)_{j\in [m]} \in I^{2m}$ and prime number $p$ greater than $w$, we have
 \[
\frac{E_p}{E_p'}=1+O(4^{mn}/p^2).
\]
\end{lemma}
\begin{proof}
By Lemma~\ref{Lem:fool-lemma}~\eqref{item:really-basic}, we \obtainthat\
\begin{align*}
\frac{1}{E_p'}&=\prod_{j\in[m]}\prod_{\pidealp\in|\Spec(\OK)|\ppart}\frac{\left(1-\Nrm(\pidealp)^{-1-z_j-w_j}\right)}{ 		\left(1- \Nrm(\pidealp)^{-1-z_j}\right)\left(1-\Nrm(\pidealp)^{-1-w_j}\right)}\\
&=\prod_{j\in[m]}\prod_{\pidealp\in|\Spec(\OK)|\ppart}\left(1+\Nrm(\pidealp)^{-1-z_j}+O(1/p^2)\right)\\
&\qquad\qquad\qquad\cdot\left(1+\Nrm(\pidealp)^{-1-w_j}+O(1/p^2)\right)\left(1-\Nrm(\pidealp)^{-1-z_j-w_j}\right)	.
\end{align*}
Since the number of prime $p$-ideals is at most $n=[K:\QQ]$ by Lemma~\ref{lemma=squarefree}, 
the product above consists of at most $3mn$ factors.
In addition, the difference between $1$ and each factor is at most $2/p$.
Hence, in the case of $p>w \ (\geq 6mn)$, we may apply Lemma~\ref{Lem:fool-lemma}~\eqref{item:sum-of-many} to the product, and then \obtainthat\ 
\begin{equation*}  
1/E_p'=1+\Sigma_p\left((\xi_j,\eta_j)_{j\in[m]};R\right)+O(m^2n^2/p^2).
\end{equation*}
Recall that $\Sigma_p=\Sigma_p\left((\xi_j,\eta_j)_{k\in[m]};R\right)$ is given by
\[
\Sigma_p=\sum_{j\in[m]}\sum_{\pidealp\in|\Spec(\OK)|\ppart}\left(\Nrm(\pidealp)^{-1-z_j}+\Nrm(\pidealp)^{-1-w_j}-\Nrm(\pidealp)^{-1-z_j-w_j}\right).
\]
Note that $|\Sigma_p|\leq 3mn/p$ holds.
Recall from \eqref{eq:E_p>w} in Lemma~\ref{Lem:Into_Euler_product} that
\begin{equation*}
E_p=1- \Sigma_p\left((\xi_j,\eta_j)_{j\in[m]};R\right)+O(4^{mn}/p^2).
\end{equation*}
We set 
\begin{equation}	\label{epsiron1,2}
\varepsilon_1= 1/E'_p  -1 = \Sigma_p+O(m^2n^2/p^2), \quad \varepsilon_2= E_p -1 = -\Sigma_p+O(4^{mn}/p^2).
\end{equation}
If $c_1$ is large enough to match the implied constants in the two big-$O$ terms in~\eqref{epsiron1,2},
then we have $|\varepsilon_1|, |\varepsilon _2|\leq(3mn+1)/p\leq 1/2$ for every $p>c_14^{mn}$ ($\geq 6mn+1$).
Hence Lemma~\ref{Lem:fool-lemma}~\eqref{item:sum-of-many} can be applied to \eqref{epsiron1,2}, and then the desired result follows.
\end{proof}
\begin{lemma}\label{lem:prod_Ep_Ep'}
We have
\[
\prod_{p\in\PP_{>w}}E_p=\left(1+O\left(\frac{4^{mn}}{w\log w}\right)\right)\cdot\prod_{p\in\PP_{>w}}E_p'.	
\]
\end{lemma}
\begin{proof}
By Lemma~\ref{lemma=Ep_equals_Ep'}, we have
\[
\prod_{p\in\PP_{>w}}E_p=\prod_{p\in\PP_{>w}}\left(1+O(4^{mn}/p^2)\right)\cdot\prod_{p\in\PP_{>w}}E_p'.
\]
If $c_1$ is large enough to match the implied constant in the $O(4^{mn}/p^2)$ above, then for every $p>w$ ($>c_12^{mn}$), the $O(4^{mn}/p^2)$ can be bounded by $1/2$ from above .
Hence, by Lemma~\ref{Lem:fool-lemma}~\eqref{item:really-basic} and Lemma~\ref{lemma=Mertenssecond}, we \havethat\
\begin{align*}
\log \left( \prod_{p\in\PP_{>w}}E_p \right) - \log \left(\prod_{p\in\PP_{>w}}E_p'\right)
&=\log\left(\prod_{p\in\PP_{>w}}\left(1+O(4^{mn}/p^2)\right)\right)\\
&=\sum_{p\in\PP_{>w}}\log\left(1+O(4^{mn}/p^2)\right)\\
&=\sum_{p\in\PP_{>w}}O(4^{mn}/p^2)=O\left(\frac{4^{mn}}{w\log w}\right).
\end{align*}
If the constant $c_1$ is chosen large enough to match the implied constant in this  last $O(4^{mn}/(w\log w) )$, then the inequality $w\log w\geq w\geq c_14^{mn}$ implies that
the value of this $O$-term is at most $1$.
Hence by Lemma~\ref{Lem:fool-lemma}~\eqref{item:really-basic2}, we derive the desired estimate
\[
\frac{ \prod_{p\in\PP_{>w}}E_p }{\prod_{p\in\PP_{>w}}E_p'} 
= e^{O\left(\frac{4^{mn}}{w\log w}\right)}
=1+O\left(\frac{4^{mn}}{w\log w}\right)
.\]
The above requirements finalize our choice of $c_1$.
\end{proof}
By this lemma, the calculation of $\prod_{p\in\PP_{>w}}E_p$ is reduced to that of $\prod_{p\in\PP_{>w}}E_p'$.
We do this by calculating
the numerator and the denominator of
\[
\prod_{p\in\PP_{>w}}E_p'=\frac{\prod_{p\in\PP}E_p'}{\prod_{p\in\PP_{\leq w}}E_p'}.
\]
\begin{lemma}\label{lem:Ep'(all_prime)}
We have
\[
\prod_{p\in\PP}E_p'=\prod_{j\in[m]}\frac{\zeta^{}_K(1+z_j+w_j)}{\zeta^{}_K(1+z_j)\zeta^{}_K(1+w_j)}=\left(1+O_{K}\left(\frac{m}{\sqrt{\log R}}\right)\right)\cdot\prod_{j\in[m]}\frac{1}{\kappa}\cdot\frac{z_j w_j}{z_j + w_j}.
\]
\end{lemma}
\begin{proof}
Recall the definition of $E_p'$ from 
\eqref{eq:def-of-E_p}
and the Euler product of Dedekind zeta function from Proposition~\ref{proposition=Eulerproduct}.
We see that
\[
\prod_{p\in\PP}E_p'=\prod_{j\in[m]}\frac{\zeta^{}_K(1+z_j+w_j)}{\zeta^{}_K(1+z_j)\zeta^{}_K(1+w_j)}.
\]
Let $j$ be an arbitrary integer in $[m]$.
Since $\xi_j,\eta_j\in I=[-\sqrt{\log R},\sqrt{\log R}]_{\RR}$, we have
\[
|z_j|, |w_j|\leq\frac{\sqrt{1+\log R}}{\log R}\leq\sqrt{\frac{2}{\log R}}.
\]
Theorem~\ref{theorem=zeta_K} implies that, for all complex numbers $\varepsilon$ with $0 < |\varepsilon| \leq 1$,
\begin{equation}\label{eq:R_0-classnumber}
\zeta^{}_K(1+\varepsilon)=\frac{\kappa}{\varepsilon}\left(1+O_{K}(|\varepsilon|)\right).
\end{equation}
Since $R\geq R_0 \ (\geq c_K)$,
if $c_K$ is sufficiently large depending on $K$,
then the absolute value of $O_K(1/\sqrt{\log R})$ term below is at most $1/2$.
Lemma~\ref{Lem:fool-lemma}~\eqref{item:really-basic} and \eqref{eq:R_0-classnumber} imply that for $\varepsilon=z_j$ or $w_j$, 
\[
\frac{1}{\zeta^{}_K(1+\varepsilon)}=\frac{\varepsilon}{\kappa}\cdot\frac{1}{1+O_K\left(1/\sqrt{\log R}\right)}=\frac{\varepsilon}{\kappa}\left(1+O_{K}\left(1/\sqrt{\log R}\right)\right).
\]
Similarly, it follows from \eqref{eq:R_0-classnumber} that
\[
\zeta^{}_K(1+z_j+w_j)=\frac{\kappa}{z_j+w_k}\left(1+O_{K}\left(1/\sqrt{\log R}\right)\right).
\]
Hence applying Lemma~\ref{Lem:fool-lemma}~\eqref{item:sum-of-many} (by assuming $c_K$ is even larger if necessary) to the product of these three functions, we \obtainthat\
\[
\frac{\zeta^{}_K(1+z_j+w_j)}{\zeta^{}_K(1+z_j)\zeta^{}_K(1+w_j)}=\frac{1}{\kappa}\cdot\frac{z_jw_j}{z_j+w_j}\cdot\left(1+O_{K}\left(1/\sqrt{\log R}\right)\right).
\]
These estimates for all $j \in [m]$ together with Lemma~\ref{Lem:fool-lemma}~\eqref{item:sum-of-many}
(with $c_K$ assumed to be even larger if necessary) provide the desired estimate.
\end{proof}
\begin{lemma}\label{lem:Ep'(p_leq_w)}
We have
\[
\prod_{p\in\PP_{\leq w}}E_p'=\left(1+O\left(\frac{mn\log w}{\sqrt{\log R}}\right)\right)\cdot \left(\frac{\vph_K(W)}{\Nrm(W)}\right)^m.
\]
\end{lemma}
\begin{proof}
Let $p$ be a prime number at most $w$, and $\pidealp\in|\Spec(\OK)|\ppart$ a prime $p$-ideal.
Let $j\in[m]$, and suppose that $\varepsilon=z_j$ or $w_j$.
Note that $\varepsilon=O\left(1/\sqrt{\log R}\right)$, $\Nrm(\pidealp)\leq p^n$ and $\log w/\sqrt{\log R}\leq F_0  \ (\leq c_2n^{-1})$.
As a result we have 
$|\varepsilon \log \Nrm (\idealp )| = O(n\log w / \sqrt{\log R} )$
and it is at most $1$ if $c_2$ is small enough.
Therefore we have
\begin{align*}
	1-\Nrm(\pidealp)^{-1-\varepsilon}&=1-\Nrm(\pidealp)^{-1}\exp\left(-\varepsilon\log\Nrm(\pidealp)\right)\\
	&=1-\Nrm(\pidealp)^{-1}\left(1+O\left(\frac{\log\Nrm(\pidealp)}{\sqrt{\log R}}\right)\right)
	&&\text{(by Lemma~\ref{Lem:fool-lemma}~\eqref{item:really-basic2})} \\
	&=\left(1-\Nrm(\pidealp)^{-1}\right)\left(1+O\left(\frac{\log\Nrm(\pidealp)}{\Nrm(\pidealp)\sqrt{\log R}}\right)\right)\\
	&=\left(1-\Nrm(\pidealp)^{-1}\right)\left(1+O\left(\frac{\log p}{p\sqrt{\log R}}\right)\right) ,
\end{align*}
where 
the last estimate follows from the general fact 
$\frac{\log (p^d)}{p^d}\le \frac{\log p}{p}$ for integers $p,d\geq 2$. 
Similarly, we have
\[
\frac{1}{1-\Nrm(\pidealp)^{-1-z_j-w_j}}=\frac{1}{1-\Nrm(\pidealp)^{-1}}\left(1+O\left(\frac{\log p}{p\sqrt{\log R}}\right)\right).
\]
Under our assumptions $p\leq w $ and $\log w / \sqrt{\log R} \le F_0 = c_2/(mn)$,
we have a crude estimate
\[
	\frac{\log p}{p\sqrt{\log R}} \leq \frac{c_2}{mn}.
\]
Hence if $c_2$ is small enough, the absolute values of all $O(\frac{\log p}{p\sqrt{\log R}})$ above are at most $1/(3mn)$
(where $3mn$ is the maximum possible number of terms in the expression to come).
This together with Lemma~\ref{Lem:fool-lemma}~\eqref{item:sum-of-many} implies that
\begin{align*}
E_p'&=\prod_{j\in[m]}\prod_{\pidealp\in|\Spec(\OK)|\ppart}\frac{\left(1-\Nrm(\pidealp)^{-1-z_j}\right)\left(1-\Nrm(\pidealp)^{-1-w_j}\right)}{\left(1-\Nrm(\pidealp)^{-1-z_j-w_j}\right)}\\
&=\left(1+O\left(\frac{mn\log p}{p\sqrt{\log R}}\right)\right)\cdot\prod_{j\in[m]}\prod_{\pidealp\in|\Spec(\OK)|\ppart}\left(1-\Nrm(\pidealp)^{-1}\right).
\end{align*}
Recall that the set of prime divisors of $W$ equals $\PP_{\leq w}$. 
It follows from Proposition~\ref{prop=totient} that
\[
\prod_{p\in\PP_{\leq w}}\prod_{\pidealp\in|\Spec(\OK)|\ppart}\left(1-\Nrm(\pidealp)^{-1}\right)=\frac{\vph_K(W)}{\Nrm(W)}.
\]
Therefore, we have
\[
\prod_{p\in\PP_{\leq w}}E_p'=\left(\prod_{p\in\PP_{\leq w}}\left(1+O\left(\frac{mn\log p}{p\sqrt{\log R}}\right)\right)\right)\cdot\left(\frac{\vph_K(W)}{\Nrm(W)}\right)^m.
\]
Below we estimate the product over $\PP_{\leq w}$ on this right-hand side %
under the assumption $\log w/\sqrt{\log R}\leq F_0 \ (=c_2(mn)^{-1})$. By Lemma~\ref{Lem:fool-lemma}~\eqref{item:really-basic} and Proposition~\ref{proposition=Mertens}, we have
\begin{align*}
\log \left( \prod_{p\in\PP_{\leq w}}E_p' \right) - \log \left(\left(\frac{\vph_K(W)}{\Nrm(W)}\right)^m \right)
&=
\log\left(\prod_{p\in\PP_{\leq w}}\left(1+O\left(\frac{mn\log p}{p\sqrt{\log R}}\right)\right)\right)\\
&=\sum_{p\in\PP_{\leq w}}\log\left(1+O\left(\frac{mn\log p}{p\sqrt{\log R}}\right)\right)\\
&=\sum_{p\in\PP_{\leq w}}O\left(\frac{mn\log p}{p\sqrt{\log R}}\right)=O\left(\frac{mn\log w}{\sqrt{\log R}}\right).
\end{align*}
This together with Lemma~\ref{Lem:fool-lemma}~\eqref{item:really-basic2} implies that
\[
\frac{\prod_{p\in\PP_{\leq w}}E_p'}{\left(\vph_K(W)/\Nrm(W)\right)^m}
=e^{O\left(\frac{mn\log w}{\sqrt{\log R}}\right)}
=1+O\left(\frac{mn\log w}{\sqrt{\log R}}\right),
\]
and this completes the proof.
\end{proof}
By combining Lemmas~\ref{lem:prod_Ep_Ep'}, \ref{lem:Ep'(all_prime)} and \ref{lem:Ep'(p_leq_w)}, and using Lemma~\ref{Lem:fool-lemma}, %
we \obtainthat\
\begin{align} \label{eq:Ep;p>w}
	\prod_{p\in\PP_{>w}}E_p=\left(1+O\left(\frac{4^{mn}}{w\log w}\right)+O_K\left(\frac{m\log w}{\sqrt{\log R}}\right)\right)\cdot\left(\frac{W^n}{\vph_K(W)\cdot\kappa}\right)^m\prod_{j\in[m]}\frac{z_jw_j}{z_j+w_j}.
\end{align}
This concludes an estimate for every fixed $(\xi_j,\eta_j)_{j\in[m]}\in I^{2m}$. 
Now we integrate over $(\xi_j,\eta_j)_{j\in[m]}\in I^{2m}$;
we will use the following estimates.
\begin{proposition}\label{prop:formula_7}
\begin{enumerate}[$(1)$]
	\item 
		For every positive real number $A$, we have
		\begin{align*}
			&\int_{I^{2m}}\rd\uxi\rd\ueta\left[\prod_{j\in[m]}\chihat(\xi_j)\chihat(\eta_j)\frac{z_jw_j}{z_j+w_j}\right]\\
			&=\int_{\RR^{2m}}\rd\uxi\rd\ueta\left[\prod_{j\in[m]}\chihat(\xi_j)\chihat(\eta_j)\frac{z_jw_j}{z_j+w_j}\right]+O_{A,\chi,m}\left((\log R)^{-m-A}\right).
		\end{align*} 
		\label{prop:formula_7:1}
	\item We have
		\begin{align*}
			\int_{I^{2m}}\rd\uxi\rd\ueta\left|\prod_{j\in[m]}\chihat(\xi_j)\chihat(\eta_j)\frac{z_jw_j}{z_j+w_j}\right|
			= O_{\chi,m}((\log R)^{-m}).
		\end{align*}
		\label{prop:formula_7:2}
\end{enumerate}
\end{proposition}
\begin{proof}
For~\eqref{prop:formula_7:1}, we decompose the integral as $\int_{I^{2m}}=\int_{\RR^{2m}}-\int_{\RR^{2m}\setminus I^{2m}}$, and estimate $\int_{\RR^{2m}\setminus I^{2m}}$.
For every $(\xi_j,\eta_j)_{j\in[m]}\in\RR^{2m}$, we have
\[
\frac{z_jw_j}{z_j+w_j}=O\left((\log R)^{-1}(1+|\xi_j|)(1+|\eta_j|)\right).
\]
By Lemma~\ref{lem:Fourier-easy}, for all $B\geq 0$, it follows that
\[
	\prod_{j\in[m]}\chihat(\xi_j)\chihat(\eta_j)\frac{z_jw_j}{z_j+w_j}
	=O_{B,\chi,m}\left((\log R)^{-m}\prod_{j\in[m]}(1+|\xi_j|)^{-B}(1+|\eta_j|)^{-B}\right).
\]
We let $B=2A+1$ and note that the following:
\[
\int_{\RR}(1+|\xi|)^{-B}\rd\xi=O_A(1)\qquad \text{ and } \qquad
\int_{\RR\setminus I}(1+|\xi|)^{-B}\rd\xi=O_A\left((\log R)^{-A}\right).
\]
In an argument similar to that of the proof of  Lemma~\ref{lem:prod_chi}, we see that
\[
\int_{\RR\setminus I^{2m}}\rd\uxi\rd\ueta\left[\prod_{j\in[m]}\chihat(\xi_j)\chihat(\eta_j)\frac{z_jw_j}{z_j+w_j}\right]=O_{A,\chi,m}\left((\log R)^{-m-A}\right).
\]
Hence~\eqref{prop:formula_7:1} follows.

For~\eqref{prop:formula_7:2}, we obtain
\begin{align*}
	\int_{I^{2m}}\rd\uxi\rd\ueta\left|\prod_{j\in[m]}\chihat(\xi_j)\chihat(\eta_j)\frac{z_jw_j}{z_j+w_j}\right|
	\leq
	(\log R)^{-m}\left(\int_{\RR^2}|\widehat{\chi}(\xi)\widehat{\chi}(\eta)|(1+|\xi|)(1+|\eta|)\rd\xi\rd\eta\right)^m.
\end{align*}
Here the integral in the right-hand side is a constant depending on $\chi$.
\end{proof}

Our estimate of~\eqref{Eq:formula_to_show_5} is stated as follows.

\begin{proposition}\label{prop:formula_6}
Let $A > 0$. 
Then the main term~\eqref{Eq:formula_to_show_5} is equal to
\begin{equation}\label{Eq:formula_to_show_6}
	(\log R)^{2m}\left(\frac{W^n}{\vph_K(W)\cdot\kappa}\right)^m
	\int_{\RR^{2m}}\rd\uxi\rd\ueta\left[\prod_{j\in[m]}\chihat(\xi_j)\chihat(\eta_j)\frac{z_jw_j}{z_j+w_j}\right]
\end{equation}
with an additive error 
\begin{align*}
	\left(
		O_{A,\chi,m}((\log R)^{-A}) + 
		O_{\chi,m}\left(\frac{4^{mn}}{w\log w}\right)+O_{\chi,m,K}\left(\frac{\log w}{\sqrt{\log R}}\right)
	\right)
	\left(\frac{W^n \log R}{\vph_K(W)\cdot\kappa}\right)^m.
\end{align*}
\end{proposition}
\begin{proof}
	This follows from~\eqref{eq:Ep;p>w} and Proposition~\ref{prop:formula_7}.
\end{proof}

We write down a proof of the following lemma for the convenience of the reader.
\begin{lemma}[{See also \cite[(38)]{Conlon-Fox-Zhao14} or \cite[p.170]{Tao06Gaussian}}]\label{Lem:fourier-chi}
We have
\begin{equation}\label{eq:easy_double_int}
\int_{\RR^2}\rd\xi\rd\eta\left[\chihat(\xi) \chihat(\eta)\frac{(1+\xi\kyo)(1+\eta\kyo)}{2+(\xi+\eta)\kyo}\right]=c_{\chi}.
\end{equation}
\end{lemma}
\begin{proof}
Since
\[
\frac{1}{2+(\xi+\eta)\kyo}=\int_0^{\infty}e^{-x(1+\xi\kyo)}e^{-x(1+\eta\kyo)}\rd x,
\]
the left-hand side of \eqref{eq:easy_double_int} equals
\[
\int_0^{\infty}\left(\int_{\RR}\chihat(\xi)(1+\xi\kyo)e^{-x(1+\xi\kyo)}\rd\xi\right)^2\rd x.
\]

Then the integral over $\xi \in \RR$ is equal to $-\chi'(x)$ by the integral representation~\eqref{eq:int_rep_of_chi} of $\chi$.
Since $c_{\chi}$ is defined in \eqref{Eq:def-c-chi} as $c_{\chi}=\int_{0 }^{\infty } \chi'(x)^2\rd x$, it ends the proof.
\end{proof}
Now we are ready to complete the proof of Theorem~\ref{Th:Goldston_Yildirim}.
\begin{proof}[Proof of Theorem~$\ref{Th:Goldston_Yildirim}$]
First, note that by Lemma~\ref{Lem:fourier-chi},
the main term is calculated as follows
\begin{align}	\label{saigo main term}
(\log R)^{2m}\left(\frac{W^n}{\vph_K(W)\cdot\kappa}\right)^m\int_{\RR^{2m}}\rd\uxi\rd\ueta\left[\prod_{j\in[m]}\chihat(\xi_j)\chihat(\eta_j)\frac{z_jw_j}{z_j+w_j}\right]=\left(\frac{W^nc_{\chi}\log R}{\vph_K(W)\cdot\kappa}\right)^m.
\end{align}
Combining Propositions~\ref{prop:average-to-show-2.5}, \ref{prop:replace_B_by_D}, \ref{prop:formula_5}, \ref{prop:formula_6} and~\eqref{saigo main term},
we \havethat\ for all $A>0$,
\begin{align*}
&\EE\Biggl(\prod_{j\in[m]}(\Lambda_{R,\chi}\circ\theta_j)^2 \ \Bigg| \ \calB\Biggr) \notag \\
&=
	\left(
		1+
		O_{A,\chi,m}((\log R)^{-A}) + 
		O_{\chi,m}\left(\frac{4^{mn}}{w\log w}\right)+
		O_{\chi,m,K}\left(\frac{\log w}{\sqrt{\log R}}\right)
	\right)
	\cdot
	\left(
		\frac{W^nc_{\chi}\log R}{\vph_K(W)\cdot\kappa}
	\right)^m\\
&\quad+O_{A,\chi,m,n}\left((\log R)^{-A}\right)+O_{m,t,K}\left(\frac{(\log R)^{2m}}{R}\right).
\end{align*}
We write the first, second and third term of the right-hand side of the equality above, respectively, as $\EuScript{E}_1$, $\EuScript{E}_2$ and $\EuScript{E}_3$.
By setting $A=1/2$,
we obtain
\begin{align*}
\EuScript{E}_1 \cdot \left(\frac{W^nc_{\chi}\log R}{\vph_K(W)\cdot\kappa}\right)^{-m}
&=
		1+
		O_{\chi,m}\left( \frac{1}{\sqrt{\log R}} \right) + 
		O_{\chi,m}\left(\frac{4^{mn}}{w\log w}\right)+
		O_{\chi,m,K}\left(\frac{\log w}{\sqrt{\log R}}\right)
	\\
&=1+O_{\chi,m,n}\left(\frac{1}{w\log w}\right)+O_{\chi,m,K}\left(\frac{\log w}{\sqrt{\log R}}\right).
\end{align*}
Since $\vph_K(W)\leq W^n$, we in addition \havethat\
\[
\EuScript{E}_2 + \EuScript{E}_3
=
O_{\chi,m,n}\left(\frac{1}{\sqrt{\log R}}\right)+O_{m,t,K}\left(\frac{(\log R)^{2m}}{R}\right)=O_{\chi,m,t,K}\left(\frac{\log w}{\sqrt{\log R}}\right)\cdot\left(\frac{W^nc_{\chi}\log R}{\vph_K(W)\cdot\kappa}\right)^m.
  \]
Thus the proof is completed.
\end{proof}
We remark that, by replacing the definition of $ I=[-\sqrt{\log R}, \sqrt{\log R}]_\RR $ with $ [-(\log R) ^\varepsilon, (\log R) ^\varepsilon]_\RR $, $1/\sqrt{\log R}$ in the second big-$O$ term in \eqref{Eq:formula_to_show} can be improved to $1 / (\log R) ^ {1- \varepsilon} $.

In the remaining part of this subsection, we summarize differences from the previous work on Goldston--Y\i ld\i r\i m type estimates in this research field.  
%
%
%
%
\begin{remark}
Assumption~\eqref{Eq:no-inclusion} in Theorem~\ref{Th:Goldston_Yildirim} is stronger than the corresponding assumption in~\cite[Proposition~9.1]{Tao06Gaussian}.
However, this condition is always fulfilled in our applications. 
Furthermore, one of the great advantages of Theorem~\ref{Th:Goldston_Yildirim} is that it applies to \emph{all} number fields $K$, including those where $K/\QQ$ is \emph{not} Galois.

The assumption that $\psi _j$ has finite cokernel can be probably dropped in view of the fact that for asymptotically almost all prime elements $\pi $ the residue field $\OK / \pi \OK$ is a prime field,
see Lemma~\ref{lemma=density_degree1}.
In implementing this, we might modify the definition of 
$\Lambda _{R,\chi }$ and the normalizing coefficient in~\eqref{Eq:formula_to_showforideals}
slightly,
see~\cite[(49) on page~147 and Lemma~10.5]{Tao06Gaussian}.
\end{remark}
\begin{remark}
The condition `$\#I_i \geq R^{4m+1}$' in Theorem~\ref{Th:Goldston_Yildirim} corresponds to `$\#I_i \geq R^{5m}$' in~\cite[Proposition 9.1]{Tao06Gaussian}, and to `$\#I_i \geq R^{10m}$' in~\cite[Proposition 8.3]{Conlon-Fox-Zhao14}. 
Here note that the convention in~\cite{Tao06Gaussian} is slightly different from ours; 
Tao bounded ideal norms from above by $R^2$, not by $R$. 
In the present paper, we do not optimize the order of this bound on $\#I_i$; 
the proof of Theorem~\ref{Th:Goldston_Yildirim} remains to work, 
provided that $\#I_i\geq R^{4m}\cdot (\log R)^{2m}\cdot \frac{\sqrt{\log R}}{\log w}$.
\end{remark}

\begin{remark}
	We have an upper bound of $w_0((\psi _j)_{j\in [m]})$ in terms of the sizes of the coefficients of $\psi _j$'s.
	Fix a $\mathbb Z$-basis $\omom$ of $\OK$
	and suppose that the coefficients of the matrix $M_j$ representing the $\mathbb Z$-linear maps 
	$\psi _j \colon \mathbb Z^t \to \OK \cong \mathbb Z^n$
	are bounded by a positive number $L>0$ from above.
	By the theory of Smith normal forms, we know 
	\[
		\# \coker (\psi _j) = \gcd \{ n\times n \text{ minors of }M_j \} 
		\le n! L^n .
	\]
	This implies that $w_0' = O_n(L^n)$ in Setting~\ref{Setting:w_0}.

	Also, by standard algorithms we can find $x_{jk}\in \ker (\psi _j)\setminus \ker (\psi _k) $
	such that all of its components are estimated as $O_{t,n} (L^{O_{t,n}(1)})$.
	It follows that 
	all of the components of $\psi _k(x_{jk})$ are estimated as $tL\cdot O_{t,n} (L^{O_{t,n}(1)}) = O_{t,n} (L^{O_{t,n}(1)})$.
	Hence we have $N(\psi _k(x_{jk}))=O_{\omom ,t,n}(L^{n\cdot O_{t,n}(1)})$,
	and conclude $w_0'' = O_{\omom ,t,n}(L^{ O_{t,n}(1)})$.

	In Setting~\ref{setting=main_term}, we have set $w_0((\psi _j)_{j\in [m]}) \coloneqq\max\{c_14^{mn},w'_0,w''_0\}$.
	Therefore we have an estimate 
	\[
		w_0((\psi _j)_{j\in [m]}) = O_{\omom , t,n}(L^{O_{t,n}(1)}) + c_14^{mn}.
	\]
\end{remark}
\subsection{Goldston--Y\i ld\i r\i m type asymptotic formula for ideals}\label{subsection=GYforideals}
In this subsection, we present a generalization of Theorem~\ref{Th:Goldston_Yildirim} to the setting where the target $\OK $ is replaced by an ideal. 
The result, Theorem~\ref{theorem=GYfordieals}, will be employed in Section~\ref{section=quadraticform} in order to establish Theorem~\ref{mtheorem=quadraticform}. The reader who is interested in the proofs of results before Theorem~\ref{mtheorem=quadraticform}, such as Theorem~\ref{mtheorem=primeconstellationsfinite} and Theorem~\ref{mtheorem=TaoZieglergeneral}, may skip this subsection.

Let $\ideala$ be a non-zero ideal of $\OK $. Then $\ideala^{-1}=\{x\in K:x\ideala \subseteq \OK\}$ is a non-zero fractional ideal; recall the discussion before Theorem~\ref{theorem=primeideals_frac}. Let $\chi$ be a function as in Theorem~\ref{Th:Goldston_Yildirim} and let $R\geq 1$. From the $(R,\chi)$-von Mangoldt function $\Lambda_{R,\chi}\colon\Ideals_K\cup\{(0)\}\to\RR$, we construct a new function $\Lambda_{R,\chi}^{\ideala}\colon \ideala \to \RR$ in the following manner: 
\begin{equation}\label{eq:Lambda^a}
\Lambda_{R,\chi}^{\ideala}(\alpha)\coloneqq \Lambda_{R,\chi}(\alpha\ideala^{-1}) \quad \textrm{for all }\alpha\in \ideala.
\end{equation}
Here, note that $\alpha\ideala^{-1}\in \Ideals_K\cup\{(0)\}$ holds true. Also, note that unless $\ideala$ is principal or $\alpha=0$, the ideal $\alpha\ideala^{-1}$ is \emph{not} principal.

\begin{theorem}[Goldston--Y\i ld\i r\i m type asymptotic formula for ideals]\label{theorem=GYfordieals}
Let $K$ be a number field of degree $n$. Let $\ideala\in \Ideals_K$.
Let $m$ and $t$ be positive integers, and $\psi_1,\dots,\psi_m\colon\ZZ ^t\to\ideala$ be $\ZZ $-module homomorphisms.
Let $w$ be a positive real number, and $W$ a positive integer of which the set of prime divisors is $\PP_{\leq w}$.
Let $b_1,\dots ,b_m$ be elements in $\ideala$ such that 
\begin{equation}\label{eq:coprime_b}
b_i\OK+W\ideala=\ideala \quad \textrm{for all }i\in [m].
\end{equation}
Define affine transformations $\theta _1,\dots ,\theta _m\colon\ZZ^t \to \ideala$ as
\[
\theta_j(x)\coloneqq\Aff _{W,b_j}(\psi_j(x))=W\psi _j(x)+b_j.
\]
Let $R$ be a positive real number,
and $I_1, \ldots, I_t\subseteq \mathbb Z$ intervals of lengths at least $R^{4m+1}$.
Set $\calB\coloneqq I_1\times\dots\times I_t \subseteq\ZZ^t$. Fix a $C^{\infty}$-function $\chi\colon\mathbb{R}\to[0,1]_{\RR}$ which satisfies $\chi(0)=1$ and $\mathrm{supp}(\chi)\subseteq [-1,1]_{\RR}$.
Assume \eqref{Eq:no-inclusion}.
Then there exist positive real numbers $R_0=R_0(m,K)$, $F_0=F_0(m,n)$ and $w_0=w_0((\psi_j)_{j\in [m]})$ such that  if $R \geq R_0$, $ w \geq w_0$ and $\log w \leq F_0\cdot\sqrt{\log R}$, then
\begin{multline}\label{Eq:formula_to_showforideals}
\EE(\Lambda_{R,\chi }^{\ideala}(\theta_1(x))^2\cdots\Lambda_{R,\chi}^{\ideala}(\theta_m(x))^2\mid x\in\calB) \\
=\left(1+O_{\chi,m,n}\left(\frac{1}{w\log w}\right)+O_{\chi ,m,t,K}\left(\frac{\log w}{\sqrt{\log R}}\right)\right)\cdot \left(\frac{W^nc_\chi \log R}{\vph_K(W)\cdot\kappa}\right)^m
\end{multline}
holds true.
In particular, the error terms $O_{\chi,m,n}\left(\frac{1}{w\log w}\right)$ and $O_{\chi ,m,t,K}\left(\frac{\log w}{\sqrt{\log R}}\right)$ are bounded uniformly on $W$ and $(b_j)_{j\in[m]}$.
\end{theorem}

We remark that condition~\eqref{eq:coprime_b} is the counterpart of the coprime condition imposed on $b_1,\ldots ,b_m$ in Theorem~\ref{Th:Goldston_Yildirim}; see Section~\ref{section=maintheoremfull}, more specifically, Lemma~\ref{lemma=coprime}, for more details.

In the rest of this subsection, we prove Theorem~\ref{theorem=GYfordieals} under the setting of Theorem~\ref{theorem=GYfordieals}.
\begin{proposition}\label{prop:average-to-show-2.5forideals}
Expectation~\eqref{Eq:formula_to_showforideals} is equal to
\begin{equation}\label{Eq:average-to-show-2.5forideals}
 (\log R)^{2m}\sum_{(\ideala_j,\idealb_j )_{j \in[m]}\in\Ideals_K^{2m}}\Pi_{R,\chi}\left((\ideala_j,\idealb_j)_{j\in[m]}\right)\cdot\EE\Bigg(\prod_{j \in [m]}(\ichi_{\ideala \cdot (\ideala_j\cap\idealb_j)}\circ\theta_j) \ \Bigg| \ \calB\Biggr),
\end{equation}
where 
 \[
 \Pi_{R,\chi}\left((\ideala_j,\idealb_j)_{j\in[m]}\right) = \prod_{j \in[m]}\mu(\ideala_j)\mu(\idealb_j) \chi\left(\frac{\log\Nrm(\ideala_j)}{\log R}\right)\chi\left(\frac{\log\Nrm(\idealb_j)}{\log R}\right).
 \]
\end{proposition}
\begin{proof}
Let $x\in\calB$. 
We \havethat\
\begin{align*}
&(\log R)^{-2m}\prod_{j\in[m]}\Lambda^\ideala_{R,\chi}(\theta_j(x))^2\\
&=\sum_{\substack{(\ideala_j,\idealb_j)_{j\in[m]}\in\Ideals_K^{2m}\\ \theta_j(x)\in \ideala \cdot (\ideala_j\cap\idealb_j) \ (\forall j\in[m])}}\prod_{j\in[m]}\mu(\ideala_j)\mu(\idealb_j)\chi\left(\frac{\log\Nrm(\ideala_j)}{\log R}\right)\chi\left(\frac{\log\Nrm(\idealb_j)}{\log R}\right)\\
&=\sum_{(\ideala_j,\idealb_j)_{j\in[m]}\in\Ideals_K^{2m}}\left(\prod_{j\in[m]}\ichi_{\ideala \cdot ( \ideala_j\cap\idealb_j)}(\theta_j(x))\right)\cdot\Pi_{R,\chi}\left((\ideala_j,\idealb_j)_{j\in[m]}\right).
\end{align*}
Only the characteristic functions $\prod_{j\in[m]}\ichi_{\ideala \cdot (\ideala_j\cap\idealb_j)}(\theta_j(x))$ depend on $x \in \calB$, and hence the desired result holds.
\end{proof}
The following lemma is verified in the same way as the proof of Lemma~\ref{lem:E(D)}. Indeed, observe that 
for every $x\in(\ZZ/D\ZZ)^t$, the value $\ichi_{\ideala\cdot (\ideala_j\cap\idealb_j)}(\theta_j(x))\in\{0,1\}$ is well-defined.
\begin{lemma}\label{lem:E(D)forideals}
Let $(\ideala_j,\idealb_j)_{j\in[m]}\in\Ideals_K^{2m}$.
Let $D=D\left((\ideala_j,\idealb_j)_{j\in[m]}\right)$ be the positive integer defined in~\eqref{lem:E(D):1}.
If $\Nrm(\ideala_j),\Nrm(\idealb_j)\leq R$ holds for every $j\in[m]$, then the following hold true:
\[
\EE\Biggl(\prod_{j\in[m]}(\ichi_{\ideala \cdot (\ideala_j\cap\idealb_j)}\circ\theta_j) \ \Bigg| \ \calB\Biggr)\\
=\EE\Biggl(\prod_{j\in[m]}(\ichi_{\ideala \cdot (\ideala_j\cap\idealb_j)}\circ\theta_j) \ \Bigg| \ (\ZZ/D\ZZ)^t\Biggr)+O_t(R^{-2m-1}).	
\]
\end{lemma}

As in~\eqref{Eq:E}, we write for short the following expectation
\begin{equation}\label{Eq:Eforideals}
\Ea \left((\ideala_j,\idealb_j)_{j \in [m]}\right)=\Ea\left((\ideala_j,\idealb_j)_{j\in[m]} ; (\theta_j)_{j \in [m]}\right)\coloneqq\EE\Biggl(\prod _{j\in[m]}(\ichi_{\ideala \cdot (\ideala_j\cap\idealb_j)}\circ\theta_j) \ \Bigg| \ (\ZZ/D\ZZ)^t\Biggr),
\end{equation}
which depends on $(\theta_j)_{j \in [m]}$, $(\ideala_j,\idealb_j)_{j\in [m]}$ and $\ideala$.
Here a positive integer $D=D\left((\ideala_j,\idealb_j)_{j\in [m]}\right)$ is taken as in~Lemma~\ref{lem:E(D)forideals}.
Then we derive the following proposition in the same way as the proof of Proposition~\ref{prop:replace_B_by_D}.
\begin{proposition}\label{prop:replace_B_by_Dforideals}
Expectation \eqref{Eq:average-to-show-2.5forideals} equals
\begin{equation}\label{Eq:formula_to_show_3forideals}
(\log R)^{2m}\sum_{(\ideala_j,\idealb_j)_{j\in[m]}\in(\Ideals_K)^{2m}}\Pi_{R,\chi}\left((\ideala_j,\idealb_j)_{j\in[m]}\right)\cdot\Ea\left((\ideala_j,\idealb_j)_{j \in [m]}\right)
\end{equation}
with an additive error term $O_{m,t,K}\left(\frac{(\log R)^{2m}}{R}\right)$.
\end{proposition}
Recall that \eqref{Eq:formula_to_show} and \eqref{Eq:formula_to_showforideals} equal \eqref{Eq:formula_to_show_3} and \eqref{Eq:formula_to_show_3forideals}, respectively.
In addition, the properties of $\E$ required in the proof of Theorem~\ref{Th:Goldston_Yildirim} are the multiplicativity of $\E$ as in Lemma~\ref{lem:p-typical-E} and the estimates as in Lemma~\ref{Lem:value_of_E} for $\E$.
Hence, in order to show Theorem~\ref{theorem=GYfordieals}, it suffices to prove that $\Ea$ has the same properties of $\E$.
Namely, we show Lemmas~\ref{lem:p-typical-Eforideals} and \ref{Lem:value_of_Eforideals} below.
\begin{lemma}\label{lem:p-typical-Eforideals}
Let $(\ideala_j,\idealb_j)_{j\in [m]}$ be a tuple of arbitrary non-zero ideals of $\OK$.
Then Lemma~$\ref{lem:p-typical-E}$ with $\E$ 
replaced by $\Ea$ holds.
Namely, we \havethat\
\[
\Ea\left((\ideala_j,\idealb_j)_{j\in [m]}\right)=\prod_{p \in\PP}\Ea\left((\ideala^{(p)}_j,\idealb^{(p)}_j)_{j\in[m]}\right).
\]
\end{lemma}
\begin{proof}
We write $\idealc _j\coloneqq\ideala_j\cap\idealb_j$ for short,
and then we have $D\ZZ=\ZZ\cap(\bigcap_{j\in[m]}\idealc_j)$.
Note that $\idealc_j\ppart=\ideala_j\ppart\cap\idealb_j\ppart $ and $D\ppart\ZZ=\ZZ\cap(\bigcap_{j\in[m]}\idealc_j\ppart )$.
We consider the $\ZZ$-module homomorphisms and affine transformations
$
\overline{\psi_j}, \overline{\theta_j}\colon(\ZZ/D\ZZ)^t\to \ideala/\ideala \idealc_j
$
induced by $\psi_j$ and $\theta_j$, respectively.
Let
\[
\overline{\psi}, \overline{\theta}\colon(\ZZ/D\ZZ)^{t}\to\prod_{j\in[m]}\ideala/\ideala \idealc_j
\]
be the two maps defined by $\overline{\psi}(x) = (\overline{\psi_1}(x),\ldots,\overline{\psi_m}(x))$ and $\overline{\theta}(x) = (\overline{\theta_1}(x),\ldots,\overline{\theta_m}(x))$.
Then we see that
\begin{equation}	\label{lem:p-typical-Eforideals:1}
\Ea\left((\ideala_j,\idealb_j)_{j \in[m]}\right)
=\EE\Biggl(\prod_{j\in[m]}(\ichi_{\ideala \idealc_j}\circ\theta_j) \ \Bigg| \ (\ZZ/D\ZZ)^t\Biggr)
=\EE\left(\ichi_{\{0\}}\circ\overline{\theta}\mid (\ZZ/D\ZZ)^t\right).
\end{equation}
By Lemma~\ref{lem:Chinese}, the $\ZZ $-module homomorphism $\overline{\psi}$ equals the product of its restrictions $\overline{\psi}\ppart$ to $(\ZZ/D\ppart \ZZ)^t$, that means
\[
\overline{\psi}=\prod_{p\in\PP}\overline{\psi}\ppart\colon\prod_{p\in\PP}(\ZZ/D\ppart\ZZ)^t\to\prod_{p\in\PP}\left(\prod_{j\in[m]}\ideala/\ideala \idealc_j\ppart\right).
\]
Hence the affine transformation $\overline{\theta}$ is the product of the restrictions 
\[
\overline{\theta}\ppart\colon(\ZZ/D\ppart\ZZ)^t\to\prod_{j\in[m]}\ideala/\ideala \idealc_j\ppart  .
\]
Therefore, by \eqref{lem:p-typical-Eforideals:1}, we \obtainthat\
\begin{align*}\label{Eq:Chinese-equality-to-show}
\Ea\left((\ideala_j,\idealb_j)_{j \in[m]}\right)
&= \EE\left(\ichi_{\{0\}}\circ\overline{\theta}\mid (\ZZ/D\ZZ)^t\right)
=\prod_{p\in\PP}\EE \left(\ichi_{\{0\}}\circ\overline{\theta}\ppart\relmiddle|(\ZZ/D\ppart\ZZ)^t\right)\\
&=\prod_{p \in\PP}\Ea\left((\ideala^{(p)}_j,\idealb^{(p)}_j)_{j\in[m]}\right).
\end{align*}
This is the desired conclusion.
\end{proof}
In what follows, we assume the following setting.
\begin{setting}\label{Setting:p-idealsforideals}
Let $(\pideala_j,\pidealb_j)_{j\in [m]}\in(\Ideals _K^{(p)})^{2m}$ be a tuple of $p$-ideals for some prime number $p$.
Write $\pidealc_j\coloneqq\pideala_j\cap\pidealb _j$.
Let $D$ be the positive integer such that $D\ZZ=\ZZ \cap\left(\bigcap_{j\in[m]}\pidealc_j\right)$.
Let 
$
\overline{\psi_j}, \overline{\theta_j}\colon(\ZZ/D\ZZ)^t\to \ideala/\ideala \idealc_j
$
be the $\ZZ$-module homomorphisms induced by $\psi_j$ and $\theta_j$, respectively.
Let
$
\overline{\psi}, \overline{\theta}\colon(\ZZ/D\ZZ)^{t}\to\prod_{j\in[m]}\ideala/\ideala \idealc_j
$
be the two maps defined by $\overline{\psi}(x) = (\overline{\psi_1}(x),\ldots,\overline{\psi_m}(x))$ and $\overline{\theta}(x) = (\overline{\theta_1}(x),\ldots,\overline{\theta_m}(x))$.
\end{setting}
Note the equality $\Nrm(\pidealc) = \#( \ideala/\ideala \pidealc)$ for all ideal $\gamma \in \Ideals_K$.
The arguments in the proof of Lemma~\ref{Lem:smaller_or_bigger} hence shows the following lemma.
\begin{lemma}\label{Lem:smaller_or_biggerforideals}
Under Setting~$\ref{Setting:p-idealsforideals}$, Lemma~$\ref{Lem:smaller_or_bigger}$ with $\E$ replaced by $\Ea$ holds.
\end{lemma}
Although the next lemma is also verified by following the proof of Lemma~\ref{Lem:value_of_E}, we write down a proof for the convenience of the reader.
\begin{lemma}\label{Lem:value_of_Eforideals}
Under Settings~$\ref{Setting:w_0}$ and $\ref{Setting:p-idealsforideals}$, Lemma~$\ref{Lem:value_of_E}$ with $\E$ replaced by $\Ea$ holds.
\end{lemma}
\begin{proof}
First we prove \eqref{en:localfactor1}.
If $\pidealc_j=\OK$ for every $j \in [m]$, 
then $\prod_{j\in[m]}(\ichi_{\ideala \pidealc_j}\circ\theta_j)$ is identical with the constant function $1$, and hence $\Ea\left((\pideala_j,\pidealb_j)_{j\in[m]}\right)=1$.

Secondly, we prove \eqref{en:localfactor2}.
It suffices to show the non-membership of $\theta_{j_0}(x)=W\psi_{j_0}(x)+b_{j_0}$ in $\ideala\pidealc_{j_0}$
for all $x\in(\ZZ/D\ZZ)^t$.
Let $\idealp$ be an arbitrary prime ideal $\idealp\supseteq\pidealc_{j_0}$.
Since $\pidealc_{j_0}$ is a $p$-ideal, we have $\idealp\cap\ZZ=p\ZZ$.
From $p\leq w$, $p$ divides $W$, and hence $W\in\idealp$ follows.
In addition, the assumption $b_{j_0} \OK+ W \ideala =\ideala$ implies that  $b_{j_0}\not\in\ideala\idealp$.
Hence we see that for all $x \in \ZZ^t$, $\theta_{j_0}(x)=W\psi_{j_0}(x)+b_{j_0}\not\in\ideala\idealp$.
This together with $\pidealc_{j_0}\subseteq\idealp$ implies that $\theta_{j_0}(x)\not\in\ideala\pidealc_{j_0}$,
as desired.

Thirdly, we prove \eqref{en:localfactor3}.
Set $C_{j_0}\coloneqq\#\coker(\psi_{j_0})$.
Note that for every $x \in \ideala$, the element $C_{j_0} \cdot x$ is contained in the image of $\psi_{j_0}$.
By Setting~\ref{Setting:w_0}~\eqref{en:w_1}, $p$ and $C_{j_0}$ are coprime.
Since the order of $\ideala/\ideala \pidealc_{j_0}$ is a power of $p$, this implies that the multiplication by $C_{j_0}$ on $\ideala / \ideala\pidealc_{j_0}$ is an automorphism.
Hence we see that $\overline{\psi_{j_0}}\colon(\ZZ/D\ZZ)^t\to\ideala/\ideala\pidealc_{j_0}$ is surjective.
Since $W$ and $p$ are coprime, the map $\overline{\theta_{j_0}}=\Aff_{W,b_{j_0}}\circ\overline{\psi_{j_0}}\colon(\ZZ/D\ZZ)^t\to\ideala/\ideala\pidealc_{j_0}$ is also surjective.
Then Lemma~\ref{Lem:smaller_or_bigger} yields that $\Ea\left((\pideala_j,\pidealb_j)_{j\in[m]}\right)=(\#\Im(\overline{\psi}))^{-1}$;
recall that $\pidealc_j=\OK$ for all $j \in [m] \setminus \{j_0\}$.
We have
\[
\#\Im(\overline{\psi})=\#\Im(\overline{\psi_{j_0}})=\#(\ideala/\ideala \pidealc_{j_0})=\Nrm(\pidealc_{j_0}).
\]
This is the desired result.

Finally we prove \eqref{en:localfactor4}.
By Lemma~\ref{Lem:smaller_or_bigger}, it suffices to show that $\#\Im(\overline{\psi})\geq p^2$.
Without loss of generality, we may assume that $\pidealc_1,\pidealc_2 \subsetneq \OK$. 
Recall that two elements $x_{12}$ and $x_{21}$ are chosen in Setting~\ref{Setting:w_0}~\eqref{en:w_2}.
Since $p>w\geq w''_0$, both $\psi_2(x_{12})\not\in\ideala\pidealc_2$ and $\psi_1(x_{21})\not\in\ideala \pidealc_1$ hold.
We focus on the mapping $(\overline{\psi_1},\overline{\psi_2})$, which is defined as
\begin{equation*}\begin{array}{cccc}
(\overline{\psi_1},\overline{\psi_2})\colon
&(\ZZ/D\ZZ)^t 
&\to
&\ideala/\ideala \pidealc_1\times\ideala/\ideala \pidealc_2
\\[10pt]
& x
& \mapsto
& (\overline{\psi_1}(x),\overline{\psi_2}(x)).
\end{array}
\end{equation*}
This maps $x_{12}$ and $x_{21}$ to non-zero elements $(0,\overline{\psi_2}(x_{12}))$ and $(\overline{\psi_1}(x_{21}),0)$, respectively.
The order of the linear span of these two images is at least $p^2$.
Hence
\[
p^2\leq\#\Im(\overline{\psi_1},\overline{\psi_2})\leq\#\Im(\overline{\psi}),
\]
and \eqref{en:localfactor4} follows.
\end{proof}
This provides the desired estimate for $\Ea$, and the proof of Theorem~\ref{theorem=GYfordieals} is completed.

%% file: chapter7.tex
\section{Construction of pseudorandom measures and estimates of weighted densities}\label{section=positiveweighteddensity}
In the present section, we will prove our first goal Theorem~\ref{theorem=primeconstellationsfinite}, as mentioned in Subsection~\ref{subsection=domain}. The main argument in this section consists of two parts: switching our `worlds,' which treat $\OK$, among the three worlds appearing in Steps~1--4 in Subsection~\ref{subsection=ideasofproof}, and estimations of weighted densities corresponding to these switches. Let us recall the three worlds above.
\begin{itemize}
\item\emph{The $N$-world $(\OK, \|\cdot\|_{\infty,\omom}, N)$}: it is an auxiliary world to apply the relative multidimensional Szemer\'{e}di theorem (Theorem~\ref{thm:RMST}). We use the $\lmugen$-length scale and the parameter $N$.
Starting from a set $A$ in the $M$-world, we choose appropriate $W\in\ZZ$ and $b\in\OK$, and connect the $M$-world (range) and the $N$-world (domain) via the affine transformation
\[
\Aff_{W,b}\colon\OK\to\OK;\quad \beta\mapsto W\beta+b.
\]
Then we construct a set $B$ from $A$ in this $N$-world; we will apply Theorem~\ref{thm:RMST} to this set $B$.
\item\emph{The $M$-world $(\OK, \|\cdot\|_{\infty,\omom}, M)$}: this is the world where the set $A$ in the setting of Theorem~\ref{theorem=primeconstellationsfinite} lives. We use the $\lmugen$-length scale and the parameter $M\in \RR_{\geq1}$. We choose  the parameter $N$ above appropriately depending on $M$. 
\item\emph{The $L$-world $(\OK, \Nrm, L)$}: the underlying space is the same as that of the $M$-world. The differences between these two worlds are the scales we take: in the $L$-world, we consider the (ideal) norm $\Nrm$. We use the parameter $L\in \mathbb{R}_{\geq 1}$. 
We appeal to Theorem~\ref{theorem=Chebotarev}~\eqref{Chebotarev} for counting elements in a given set $A$ with respect to a certain scale; since Theorem~\ref{theorem=Chebotarev}~\eqref{Chebotarev} is stated in terms of ideals, this $L$-world is suited for this counting procedure.
\end{itemize}

We assume the following setting throughout the current section.
\begin{setting}\label{setting=section7-1}
Let $K$ be a number field of degree $n$ and $\omom$ an integral basis of $K$.
Let $\DD$ be an NL-compatible $\OKt$-fundamental domain; recall Definitions~\ref{definition=fundamentaldomain} and \ref{definition=normrespecting}.
Let $C=C(\omom,\DD)>0$ and $C'=C'(\omom)>0$, respectively, be constants which satisfy  \eqref{NLC}.
Let $S$ be a finite subset of $\OK$. We moreover assume that $S$ is a standard shape; recall Definition~\ref{definition=standardshape}. As mentioned after Theorem~\ref{mtheorem=primeconstellationsfinite}, we may assume this without loss of generality. Set $r\coloneqq\#S-1$.
\end{setting}
In this section and Sections~\ref{section=maintheoremfull} and \ref{section=slidetrick}, we use the following convention on cosets without mentioning it: each element of $\OK/W\OK$ may be seen as a subset of $\OK$.
In particular, if $b\in \OK$ is sent to $\overline{b}\in\OK/W\OK$ by the natural projection, then we may express it as `$b\in \overline{b}$.' Thus we frequently regard $\overline{b}\in\OK/W\OK$ as a subset of $\OK$. In our arguments in these sections, we often take $\overline{b}\in\OK/W\OK$ first, and then choose $b$ in $\overline{b}$.

\subsection{Outline of the proof of Theorem~\ref{mtheorem=primeconstellationsfinite}}\label{subsection=setting_section7}

Let $\delta>0$. Assume that for a sufficiently large real number $M$, $A\subseteq\PP_K \cap\DD\cap\OK(\omom,M)$ satisfies
\[
\#A\geq\delta\cdot\#(\PP_K\cap\DD\cap\OK(\omom,M)).
\]
The goal is to show that  $A$ contains an $S$-constellation.
Here we outline the argument to achieve this goal. 
\begin{itemize}
\item Choose a parameter $w$, which arises from the $W$-trick employed in the proof of the Goldston--Y\i ld\i r\i m type asymptotic formula (Theorem~\ref{Th:Goldston_Yildirim}). 
\item For the parameter $w$ and the main parameter $M$, determine three auxiliary parameters $W=W_w$, $R=R_{M;r}$ and $N=N_{w,M}$ appropriately in order to apply Theorem~\ref{thm:RMST}.
\item Under the condition that $w$ and $M$ are sufficiently large depending on a given $\rho>0$, construct a pseudorandom measure $\tilde{\lambda}\colon\OK\to\RR_{\geq 0}$, meaning that $\tilde{\lambda}$ satisfies the $(\rho,N,S)$-linear forms condition in the sense of Definition~\ref{definition=S-linearform}.
\end{itemize}

The exact argument will be presented in Subsection~\ref{subsection=pseudorandommeasure}. The proof of pseudorandomness of $\tilde{\lambda}$ is based on Theorem~\ref{Th:Goldston_Yildirim}. Here, one of the keys to the proof is that the error term of Theorem~\ref{Th:Goldston_Yildirim} decomposes into two parts: one is of the form $o_{w\to \infty}(1)$, depending only on $w$ out of $w$ and $M$; the other is of the form $o_{M\to \infty}(1)$, as long as $w$ is fixed.  To switch from the $M$-world to the $N$-world, we need to take $b\in \OK$ as well as $W\in \NN$.
Although in Subsection~\ref{subsection=pseudorandommeasure} we argue with an arbitrarily taken $b$ which is prime to $W$, we will eventually choose $b$ according to the given set $A$.

In Subsection~\ref{subsection=normlengthcomparison}, 
we prove Proposition~\ref{proposition=densityequivalence}, 
which describes the switch between the $L$-world and the $M$-world in the setting of the infinitary version of  Theorem~\ref{theorem=primeconstellationsfinite} (Corollary~\ref{corollary=primeconstellationsupperdense}). Next, we proceed to the following two steps:
\begin{itemize}
\item trim $A$ by removing an `exceptional set' to obtain $A'$, choose an appropriate $b\in\OK$ according to $A'$, 
\item and set $B\subseteq\OK(\omom, N)$ as $B\coloneqq(\Aff_{W,b})^{-1}(A')\cap\OK(\omom,N)$, and establish the following estimate of the weighted density of $B$,
\[
\EE\left(\ichi_{B}\cdot\tilde{\lambda}\relmiddle|\OK(\omom,N)\right)\geq\tide,
\]
provided that $M$ is sufficiently large. Here, $\tide$ is a strictly positive constant depending on $\delta$, $\omom$, $\DD$ and $S$.
\end{itemize}
\input{Fig2_02_06}

In Subsection~\ref{subsection=positiverelativedensity}, these two steps will be implemented by switching among the three worlds.
Figure~\ref{fig:4} illustrates the switchings; see Proposition~\ref{proposition=positiverelativeweighteddensity} for the precise statement.
To define the measure $\tilde{\lambda}$, we also need to choose $\chi$ as in Setting~\ref{setting=section7-2}. 
Since the choice of $\chi $ can be fixed throughout the paper,  
in what follows we sometimes omit to write the dependence on $\chi$ explicitly.

Finally, in Subsection~\ref{subsection=proofofmaintheorem}, we choose the parameter $w$ in the $W$-trick appropriately. 
We determine this $w$ depending only on $\omom$, $\DD$, $\delta$ and $S$, but independently of the main parameter $M$. Then $W=W_w$ is also fixed. We will deduce from previous arguments that Theorem~\ref{thm:RMST} applies to our case, provided that $M$ is sufficiently large depending on $\omom$, $\DD$, $\delta$ and $S$; here we also confirm the remaining condition \eqref{en:smallness} (the smallness condition). 
This will allow us to verify that $B$ contains an $S$-constellation. 
As we explained in Remark~\ref{remark=affine}, we then conclude that $A$ contains an $S$-constellation. 
For the estimate of $\scrN_{S}(A)$, we appeal to Theorem~\ref{theorem=weighted-counting}. 
This will complete the proof of Theorem~\ref{theorem=primeconstellationsfinite}.
\subsection{Construction of the pseudorandom measure}\label{subsection=pseudorandommeasure}
To construct the pseudorandom measure mentioned in Subsection~\ref{subsection=setting_section7}, we first  choose the parameter $w$, arising from the $W$-trick in the Goldston--Y\i ld\i r\i m type asymptotic formula. Then, assume that the main parameter  $M\in \RR_{\geq 1}$ is sufficiently large depending on this $w$. We choose parameters $W$, $R$ and $N$ according to these two parameters $w$ and $M$, and a function $\chi$ in the following manner.
\begin{setting}[The choice of the parameters]\label{setting=section7-2}
Assume that two parameters $w\geq 2$ and $M\in \RR_{\geq 1}$ satisfy that $M\geq e^{2w}$. Under this assumption, set $W=W_{w}\in\NN$ and $R=R_{M;r}\in\RR$ as follows:
\begin{equation}\label{eq:RwW}
W\coloneqq\prod_{p\in\PP_{\leq w}}p,\quad \textrm{and}\quad 
R\coloneqq M^{\frac{1}{17(r+1)2^r}}.
\end{equation}
Also, set $N=N_{w,M}\in\NN$ as
\begin{equation}\label{eq:def_of_N}
N\coloneqq\left\lceil\frac{M}{W}\right\rceil.
\end{equation}
Fix a $C^{\infty}$-function $\chi\colon\mathbb{R}\to\mathbb{R}_{\geq 0}$ which satisfies the conditions $\chi(0)=1$, $\mathrm{supp}(\chi)\subseteq [-1,1]_{\RR}$, and $\chi(x)\leq 1$ for every $x\in \mathbb{R}$. Let $c_{\chi}\coloneqq\int_{0}^{\infty} \chi '(x)^2\rd x$.
\end{setting}
We note that 
\begin{equation}\label{eq=W_N_f}
W\leq M^{\log 2}
\end{equation}
holds true.
Indeed, this follows from the definition of $W$, $M\geq e^{2w}$, and \eqref{eq:Chebyshev-ineq} in Proposition~\ref{prop=elementaryChebyshev}.

Here is the construction of the pseudorandom measure $\tilde{\lambda}$, which plays a key role in the proof of the main theorems in the present paper.
\begin{definition}[The pseudorandom measure $\tilde{\lambda}$]\label{def=our-measure}
We use Settings~$\ref{setting=section7-1}$ and $\ref{setting=section7-2}$. Let $\kappa=\kappa_K>0$ be the constant as in Theorem~\ref{theorem=zeta_K}. Let $\vph_K$ be the totient function of $K$ (Definition~\ref{def=totient}) and $\Lambda_{R,\chi}$ be the $(R,\chi)$-von Mangoldt function (Definition~\ref{def:chi}).
Then, define a function $\lambda=\lambda_{M; \chi,r,K}\colon\OK\to\RR_{\geq 0}$ by
\begin{equation}\label{eq=lambda_0}
\lambda(\alpha)\coloneqq\frac{\kappa\cdot\Lambda_{R,\chi}(\alpha)^2}{c_{\chi}\log R}.
\end{equation}
For a fixed $b\in\OK$, define a function $\tilde{\lambda}=\tilde{\lambda}_{w,M;\chi,r,K,b}\colon\OK\to\RR_{\geq 0}$ by
\begin{equation}\label{eq=lambda}
\tilde{\lambda}(\beta)\coloneqq\frac{\vph_K(W)}{W^n}(\lambda\circ \Aff_{W,b})(\beta).
\end{equation}
\end{definition}
We will deduce from the Goldston--Y{\i}ld{\i}r{\i}m type asymptotic formula (Theorem~\ref{Th:Goldston_Yildirim}) that the function $\tilde{\lambda}$ indeed satisfies the linear forms condition. 
For this purpose, %
we rewrite the assertion of the asymptotic formula in the following manner.

\begin{theorem}\label{Th:Goldston_Yildirim_rho}
Let $K,n,m,t$, $(\psi_j)_{j\in [m]}$, $\chi$ be as in Theorem~$\ref{Th:Goldston_Yildirim}$.  Then there exist positive numbers $w_0=w_0((\psi_j)_{j\in [m]})$, $R_0=R_0(m,K)$ and  $F_0=F_0(m,n)$ such that the following hold true. Let $w\geq w_0$ and $W\in \NN$ such that the set of prime factors of $W$ coincides with $\PP_{\leq w}=\{ p\in \PP : p\le w\}$. Let $b_1,\dots ,b_m\in\OK$ be coprime with $W$. Define affine transformations $\theta _1,\dots ,\theta _m\colon\ZZ^t \to \OK$ by $\theta_j(x)\coloneqq\Aff _{W,b_j}(\psi_j(x))=W\psi _j(x)+b_j$. Let
$R\geq R_0$, $I_1,\ldots ,I_t\subseteq  \ZZ$ be intervals  with lengths at least $R^{4m+1}$, and $\calB\coloneqq I_1\times\dots\times I_t \subseteq\ZZ^t$. Then, 
under the assumption that $\log w \leq F_0\cdot\sqrt{\log R}$, there exist a positive constant $c_{\mathrm{GY};m,n}^{(1)}$, depending only on $\chi$, $m$ and $n$, and a positive constant $c^{(2)}_{\mathrm{GY};\chi,m,t,K}$, depending only on $\chi,m,t$ and $K$, such that for
\begin{equation}\label{Eq:rho_bigO}
\rho_{\mathrm{GY};\chi,m,n}^{(1)}(w)\coloneqq\frac{c_{\mathrm{GY};\chi,m,n}^{(1)}}{w\log w},\quad \rho_{\mathrm{GY};\chi,m,t,K}^{(2)}(w,R)\coloneqq c^{(2)}_{\mathrm{GY};\chi,m,t,K}\cdot \frac{\log w}{\sqrt{\log R}},
\end{equation} 
we \havethat\ 
\begin{multline}\label{Eq:GY_rho}
\left|\left(\frac{\vph_K(W)\cdot\kappa}{W^nc_\chi \log R}\right)^{m}\cdot \EE(\Lambda_{R,\chi }(\theta_1(x))^2\cdots\Lambda_{R,\chi}(\theta_m(x))^2\mid x\in\calB)-1\right|\\
\leq \rho_{\mathrm{GY};\chi,m,n}^{(1)}(w)+\rho_{\mathrm{GY};\chi,m,t,K}^{(2)}(w,R).
\end{multline}
\end{theorem}
\begin{theorem}[Pseudorandomness of $\tilde{\lambda}$]\label{theorem=pseudorandommeasure}
Assume Settings~$\ref{setting=section7-1}$ and $\ref{setting=section7-2}$. Fix
 $\rho > 0$. Then there exist real numbers $w_{\mathrm{PR}}(\rho,\chi,S)$, depending only on $\rho$, $\chi$ and $S$, and $M_{\mathrm{PR}}(w,\rho,\chi,S)$, depending only on $w$, $\rho$, $\chi$ and $S$, such that the following holds: if $w\geq w_{\mathrm{PR}}(\rho,\chi,S)$, $M\geq M_{\mathrm{PR}}(w,\rho,\chi,S)$, and $b\in \OK$ is coprime with $W$, then $\tilde{\lambda}=\tilde{\lambda}_{w,M;\chi,r,K,b}$, constructed in Definition~$\ref{def=our-measure}$, is a $(\rho,N,S)$-pseudorandom measure.
\end{theorem}
\begin{proof}
We will check that the conditions of Theorem~\ref{Th:Goldston_Yildirim_rho} are fulfilled. Let $t=2r+2$. Take an arbitrary non-empty subset $\mathcal{J}$ of $\bigsqcup_{j\in [r+1]}\{0,1\}^{e_j}$, and let $m\coloneqq \#\mathcal{J}$. Here for each $j\in [r+1]$, $e_j=[r+1]\setminus\{j\}$.
Write $S$ as $S=\{s_1,\dots,s_r\}\sqcup\{0\}$. Then for each $\omega\in \mathcal{J}$, define a homomorphism $\psi_S^{(\omega)}\colon\ZZ^t\to\OK$ of $\ZZ$-modules by \eqref{Eq:the-linear-maps-1} and \eqref{Eq:the-linear-maps-2} in Definition~\ref{definition=S-linearform}. (Thus we consider $m$ homomorphisms in total for a fixed $\mathcal{J}$.) Since $S$ is assumed to be a standard shape, these $\psi_S^{(\omega)}$ are all surjective. Hence, by Lemma~\ref{lem:indep}, condition~\eqref{Eq:no-inclusion} holds for $(\psi_S^{(\omega)})_{\omega\in\mathcal{J}}$.

Take $R_0(m,K)$, $F_0(m,n)$ and $w_0((\psi_S^{(\omega)})_{\omega \in \mathcal{J}})$ as in Theorem~\ref{Th:Goldston_Yildirim_rho}. Set the maximums and minimum over all non-empty $\mathcal{J}$ as
\begin{equation}\label{eq:R_1,F_1,w_1}
\begin{split}
&R_1(r,K)\coloneqq \max_{m\in[(r+1)2^r]}R_0(m,K),\\
&F_1(r,n)\coloneqq \min_{m\in[(r+1)2^r]}F_0(m,n),\\
&w_1(S)\coloneqq \max_{\varnothing\neq\mathcal{J}\subseteq \bigsqcup_{j\in [r+1]}\{0,1\}^{e_j}}w_0((\psi_S^{(\omega)})_{\omega\in\mathcal{J}}).
\end{split}
\end{equation}
Under the assumption of $w\geq w_1(S)$, set
\[
\rho_{\mathrm{PR};\chi,r,n}^{(1)}(w)\coloneqq \max_{m\in[(r+1)2^r]}\rho_{\mathrm{GY};\chi,m,n}^{(1)}(w).
\]
Under the assumptions of $R\geq R_1(r,K)$ and $\log w \leq F_1(r,n)\cdot\sqrt{\log R}$, set
\[
\rho_{\mathrm{PR};\chi,r,K}^{(2)}(w,M)\coloneqq \max_{m\in[(r+1)2^r]}\rho_{\mathrm{GY};\chi,m,t,K}^{(2)}(w,R).
\]
Here recall that $R$ is determined from $M$ and $r$ by \eqref{eq:RwW}.

By \eqref{Eq:rho_bigO}, if $w\geq w_1(S)$ is sufficiently large depending on $\chi, r,n$ and $\rho$, then we have
\begin{equation}\label{eq:rho1}
\rho_{\mathrm{PR};\chi,r,n}^{(1)}(w)\leq \frac{1}{2}\rho.
\end{equation}
Take $w_{\mathrm{PR}}(\rho,\chi,S)$ with $w_{\mathrm{PR}}(\rho,\chi,S)\geq w_1(S)$ in such a way that the condition $w\geq w_{\mathrm{PR}}(\rho,\chi,S)$ implies \eqref{eq:rho1}. 
Let $w\geq w_{\mathrm{PR}}(\rho,\chi,S)$. 
Then, choose $M_{\mathrm{PR}}(w,\rho,\chi,S)$ with $M_{\mathrm{PR}}(w,\rho,\chi,S)\geq e^{2w}$ in such a way that if $M\geq M_{\mathrm{PR}}(w,\rho,\chi,S)$, then the following three conditions
\begin{equation}\label{eq:rho2}
\rho_{\mathrm{PR};\chi,r,K}^{(2)}(w,M)\leq \frac{1}{2}\rho,
\end{equation}
$R\geq R_1$, and $\log w \leq F_1\cdot \sqrt{\log R}$ all hold true. 

Take an arbitrary subset $\cal{B}$ of $\ZZ^{r+1}$ which is the product of intervals of lengths at least $N$. By \eqref{eq=W_N_f}, the choice of $R$ in \eqref{eq:RwW} and that of $N$ in \eqref{eq:def_of_N}, we have $N\geq M^{\frac{5}{17}}\geq R^{4m+1}$. Here, recall that $m\leq (r+1)2^r$. 
Hence, $\calB\times\calB\subseteq\ZZ^t$ is the product of intervals of lengths at least $R^{4m+1}$. 
Therefore, for every $b\in\OK$ coprime with $W$, we may appeal to Theorem~\ref{Th:Goldston_Yildirim_rho}. By \eqref{Eq:GY_rho}, \eqref{eq:rho1} and \eqref{eq:rho2}, 
we conclude that $\tilde{\lambda}$ satisfies the $(\rho,N,S)$-linear forms condition as in Definition~\ref{definition=S-linearform}, 
provided that $w\geq w_{\mathrm{PR}}(\rho,\chi,S)$ and that $M\geq M_{\mathrm{PR}}(w,\rho,\chi,S)$.
\end{proof}
\subsection{Comparison between \counting s with the norm scale and with the $\lmugen$-length scale}
\label{subsection=normlengthcomparison}
 In Subsection~\ref{subsection=positiverelativedensity}, we will estimate the weighted density of the set $B$ appearing in Subsection~\ref{subsection=setting_section7}. Before that, in this subsection, for a subset of $\PP_K\cap\DD$, we make a comparison of the relative upper asymptotic density measured by norm and that measured by $\lmugen$-length; see Proposition~\ref{proposition=densityequivalence}. This relates to Corollary~\ref{corollary=primeconstellationsupperdense}, 
 which is the infinitary version of Theorem~\ref{theorem=primeconstellationsfinite}. The key to this comparison is the following lemma for an  NL-compatible $\OKt$-fundamental domain $\DD$.
\begin{lemma}\label{lemma=NLC-incl}
Let $\DD$, $C$ and $C'$ be as in Setting~$\ref{setting=section7-1}$. Then the following hold true.
\begin{enumerate}[$(1)$]
\item\label{en:NLC-incl1}
For $L,M\in\RR_{\geq0}$ with $L\leq CM^n$,
\[
\DD\cap\OK(L)\subseteq\DD\cap\OK(\omom,M).
\]
\item\label{en:NLC-incl2}
For $L,M\in\RR_{\geq0}$ with $L\geq C'M^n$,
\[
\DD\cap\OK(\omom,M)\subseteq\DD\cap\OK(L).
\]
\end{enumerate}
\end{lemma}
\begin{proof}
By the assumptoin, we have the following inequality (which is \eqref{NLC})
\[
C\|\alpha\|_{\infty,\omom}^n\leq \Nrm (\alpha)\leq C'\|\alpha\|_{\infty,\omom}^n
\]
for all $\alpha\in \DD$.
Both items easily follow from this.
\end{proof}
We will derive the following proposition from Lemma~\ref{lemma=NLC-incl} and Theorem~\ref{theorem=Chebotarev}~\eqref{Chebotarev}. Recall the definition of the relative upper asymptotic densities from Definition~\ref{definition=relativedensity}.
\begin{proposition}\label{proposition=densityequivalence}
Assume Setting~$\ref{setting=section7-1}$.
Then for every $A\subseteq\PP_K\cap\DD$, the density $\overline{d}_{\PP_K\cap \DD}(A)$ is strictly positive  if and only if $\overline{d}_{\PP_K\cap\DD,\omom}(A)$ is strictly positive.
\end{proposition}
\begin{proof}
Since $\DD$ is an $\OKt$-fundamental domain, the following holds true: for a finite subset $\mathscr{A}\subseteq\Ideals_K$ consisting of principal ideals, we have 
\begin{equation}\label{eq:princ-ideal-counting}
\#\{\alpha\in\DD : (\alpha)\in\mathscr{A}\}=\#\mathscr{A}.
\end{equation}

Apply \eqref{eq:princ-ideal-counting} to $\mathscr{A}=\{\pidealp\in|\Spec(\OK)|^{\mathrm{PI}} : \Nrm(\pidealp)\leq L\}$. Then by %
the \naturaldensityversionofthe Chebotarev density theorem (Theorem~\ref{theorem=Chebotarev}~\eqref{Chebotarev}), for a sufficiently large $L$, we \obtainthat\ 
\begin{equation}\label{eq:weak-Chebotarev}
\frac{1}{2h}\cdot\frac{L}{\log L}\leq\#(\PP_K\cap\DD\cap\OK(L))\leq \frac{2}{h}\cdot\frac{L}{\log L}.
\end{equation}
Here $h=h_K$ denotes the class number of $K$.

First, we will prove $\overline{d}_{\PP_K\cap\DD,\omom}(A)>0$ if $\overline{d}_{\PP_K\cap\DD}(A)>0$. Let $\delta\coloneqq\overline{d}_{\PP_K\cap\DD}(A)>0$. Then there exists a strictly increasing real sequence $(L_k)_{k\in\NN}$ with $\lim\limits_{k\to\infty}L_k=\infty$ such that for all $k\in \NN$,
\[
\#(A\cap\OK(L_k))\geq\frac{\delta}{2}\cdot\#(\PP_K\cap\DD\cap\OK(L_k))
\]
holds. By \eqref{eq:weak-Chebotarev}, if $k$ is sufficiently large, then we have
\[
\#(A\cap\OK(L_k))\geq\delta\cdot\frac{1}{4h}\cdot\frac{L_k}{\log L_k}.
\]
Now set $M_k\coloneqq(L_k/C)^{1/n}$. Then $\lim\limits_{k\to\infty}M_k=\infty$ holds. By Lemma~\ref{lemma=NLC-incl}~\eqref{en:NLC-incl1}, for a sufficiently large $k$, 
\[
\#(A\cap\OK(\omom,M_k))\geq\delta\cdot\frac{1}{4h}\cdot\frac{CM_k^n}{\log(CM_k^n)}
\]
holds. In addition, by Lemma~\ref{lemma=NLC-incl}~\eqref{en:NLC-incl2} and by \eqref{eq:weak-Chebotarev}, we have
\[
\#(\PP_K\cap\DD\cap\OK(\omom,M_k))\leq\frac{2}{h}\cdot\frac{C'M_k^n}{\log(C'M_k^n)}.
\]
By combining them, we \havethat\  for a sufficiently large $k$,
\[
\frac{\#(A\cap\OK(\omom,M_k))}{\#(\PP_K\cap\DD\cap\OK(\omom,M_k))}\geq\frac{C}{8C'}\delta;
\]
here recall that $C'\geq C$. This implies that
\[
\overline{d}_{\PP_K\cap \DD,\omom}(A)\geq\frac{C}{8C'}\cdot\overline{d}_{\PP_K\cap \DD}(A) .
\]
Since the right-hand side is assumed to be positive, we conclude that $\overline{d}_{\PP_K\cap \DD,\omom}(A)>0$.

Finally, we will prove $\overline{d}_{\PP_K\cap\DD}(A)>0$ if $\overline{d}_{\PP_K\cap\DD,\omom}(A)>0$. 
We reset $\delta $ as $\delta\coloneqq\overline{d}_{\PP_K\cap\DD,\omom}(A)>0$.
By the definition of $\overline{d}_{\PP_K\cap\DD,\omom}(A)$,
there exists a strictly increasing real sequence $(M_k)_{k\in\NN}$ with $\lim\limits_{k\to\infty}M_k=\infty$ such that for every $k\in \NN$, 
\[
\#(A\cap\OK(\omom,M_k))\geq\frac{\delta}{2}\cdot\#(\PP_K\cap\DD\cap\OK(\omom,M_k))
\]
holds. Set $L_k\coloneqq C'M_k^n$. Then $\lim\limits_{k\to\infty}L_k=\infty$. By Lemma~\ref{lemma=NLC-incl} and \eqref{eq:weak-Chebotarev}, if $k$ is sufficiently large, then we \havethat\
\begin{align*}
\#(A\cap\OK(L_k))&\geq \#(A\cap\OK(\omom,M_k))\\
&\geq\frac{\delta}{2}\cdot\#(\PP_K\cap\DD\cap\OK(\omom,M_k))\\
&\geq\frac{\delta}{2}\cdot\#(\PP_K\cap\DD\cap\OK(CM_k^n))
\geq\delta\cdot\frac{1}{4h}\cdot\frac{CM_k^n}{\log(CM_k^n)}.
\end{align*}
Again by \eqref{eq:weak-Chebotarev}, for a sufficiently large $k$, we obtain
\[
\frac{\#(A\cap\OK(L_k))}{\#(\PP_K\cap\DD\cap\OK(L_k))}\geq \delta\cdot\frac{1}{4h}\cdot\frac{CM_k^n}{\log(CM_k^n)}\cdot\frac{h}{2}\cdot\frac{\log L_k}{L_k}\geq\frac{C}{8C'}\delta.
\]
This implies that
\[
\overline{d}_{\PP_K\cap \DD}(A)\geq\frac{C}{8C'}\cdot\overline{d}_{\PP_K\cap \DD,\omom}(A)
\]
and hence that $\overline{d}_{\PP_K\cap \DD}(A)>0$.
\end{proof}
We remark that in the above proof of Proposition~\ref{proposition=densityequivalence}, the multiplicative constant $\frac{C}{8C'}$, appearing twice, may be improved to $\frac{C}{C'}$. Indeed, for every $\varepsilon>0$, take $L$ and $k$ both sufficiently large according to $\varepsilon$, and improve the factors $\frac{1}{2h}$, $\frac{2}{h}$ and $\frac{\delta}{2}$ in the proof to $\frac{1-\varepsilon}{h}$, $\frac{1+\varepsilon}{h}$ and $(1-\varepsilon)\delta$, respectively. Finally, let $\varepsilon\to 0$.

The following counting result will be employed in Subsection~\ref{subsection=positiverelativedensity}.
\begin{proposition}\label{proposition=counting_below_DD}
Under Setting~$\ref{setting=section7-1}$, we \havethat\
\[
\liminf_{M\to \infty}\frac{\#(\PP_K\cap \DD \cap \OK(\omom, M))}{M^n(\log M)^{-1}}\geq \frac{C}{(n+1)h},
\]
where $h$ is the class number of $K$.
\end{proposition}
\begin{proof}
Use the estimate from below in \eqref{eq:weak-Chebotarev} (with a finer constant $(1-\varepsilon)/h$ for each $\varepsilon>0$ and for a sufficiently large $L$ depending on $\varepsilon$) with $L=CM^n$. Then, the desired result follows from Lemma~\ref{lemma=NLC-incl}~\eqref{en:NLC-incl1} and $CM^n\leq M^{n+1}$, which is valid for $M\geq C$.
\end{proof}
\subsection{Estimates of weighted densities}\label{subsection=positiverelativedensity}
Let $A$ be a subset of $\OK$ in which we hope to find an $S$-constellation.
In Subsection~\ref{subsection=pseudorandommeasure}, we have constructed a pseudorandom measure $\tilde{\lambda}$ corresponding to suitable parameters and for a (yet unspecified) $b\in \OK$ prime to $W$. 
In this subsection, we specify a suitable $b$ and make an estimation of the weighted density of $B$, whose construction is outlined in Subsection~\ref{subsection=setting_section7}, with respect to the weight $\tilde{\lambda}$. 
As mentioned in Subsection~\ref{subsection=setting_section7}, to construct $B$ from $A$, we will trim  $A$ by removing an exceptional set $T$, which behaves badly in our proof. The following lemma collects the properties of the exceptional set $T$ employed  in this section.
\begin{lemma}[Exceptional set]\label{lemma=jogai}
Assume Settings~$\ref{setting=section7-1}$ and $\ref{setting=section7-2}$. Let 
$A\subseteq\PP_K\cap\DD\cap\OK(\omom,M)$. Then 
\begin{equation}\label{eq:jogai_T}
T\coloneqq A\cap \OK(R)
\end{equation}
satisfies the following, provided that $M$ is sufficiently large depending on $r$ and $K$:
\begin{enumerate}[$(1)$]
\item\label{en:noudo_T} $\#T\leq M^{\frac{1}{16}}$,
 \item\label{en:log_compatible} for every $\alpha\in A\setminus T$,
\[
\lambda(\alpha)=\frac{\kappa}{17(r+1)2^rc_{\chi}}\cdot\log M,
\]
 \item\label{en:coprime_A'} every $\alpha\in A\setminus T$ is prime to $W$.
\end{enumerate}
\end{lemma}
\begin{proof}
First, we will prove \eqref{en:noudo_T}. By Proposition~\ref{proposition=idealdensity}, if $M$ is sufficiently large depending on $r$, then the number of ideals whose ideal norms are at most $R$ does not exceed $2\kappa R$. By \eqref{eq:princ-ideal-counting}, $\#T\leq 2\kappa R$ holds.
From the choice of $R$, if $M$ is sufficiently large depending on $K$, then $R\leq \frac{1}{2\kappa}M^{\frac{1}{16}}$ holds. 

Secondly, we will prove \eqref{en:log_compatible}. Here note that for every $\alpha\in A\setminus T$, we have $\Nrm(\alpha)>R$ and $\alpha$ is a prime element.  
Since $\chi(0)=1$ and $\mathrm{supp}(\chi)\subseteq [-1,1]_{\RR}$, we have
\[
\sum_{\idealb\mid \alpha \OK}\mu(\idealb)\chi\left(\frac{\log\Nrm(\idealb)}{\log R}\right)=1,
\]
which ensures \eqref{en:log_compatible}.

Finally, we will prove \eqref{en:coprime_A'}. Since $w\leq\frac{1}{2}\log M$, for a sufficiently large $M$ depending on $r$ and $K$, we have $w^n\leq R$. 
Hence, every $\alpha\in A\setminus T$ satisfies that $\Nrm(\alpha)>w^n$. 
Let $p$ be the characteristic of the residue field $\OK /\alpha \OK $.
Since it is a field extension of $\mathbb F _p$ of degree at most $n$, 
we have $p^n\ge \Nrm (\alpha )$ and hence $p > w$.
This implies that $p$ is not a factor of $W$.
Since $p\in \alpha \OK $, it follows that $\alpha $ and $W$ are coprime in $\OK $.
\end{proof}
The next proposition is the goal of the present subsection.
\begin{proposition}[Estimate of the weighted density with respect to the weight $\tilde{\lambda}$]\label{proposition=positiverelativeweighteddensity}
Assume Settings~$\ref{setting=section7-1}$ and $\ref{setting=section7-2}$.
Then there exist {a} positive real number $M_{\mathrm{DS}}=M_{\mathrm{DS}}(\omom,\DD,\delta,r)$, depending only on $\omom$, $\DD$, $\delta$ and $r$, and a positive real number $u=u_{\mathrm{DS}}(\omom,\DD,\chi,r)>0$, depending only on $\omom$, $\DD$, $\chi$ and $r$, such that the following holds true:
let $\delta>0$ and $M\geq M_{\mathrm{DS}}$. Assume that a set
$A\subseteq\PP_K\cap\DD\cap\OK(\omom,M)$ satisfies
\begin{equation}\label{eq:A-hyp}
\# A\geq\delta\cdot\#(\PP_K\cap\DD\cap\OK(\omom,M)).
\end{equation}
Let $A'\coloneqq A\setminus T$, where $T$ is the exceptional set defined in \eqref{eq:jogai_T}. Then there exists $b\in\OK$ such that $b$ is prime to $W$ and that for
\begin{equation}\label{def_of_B}
B\coloneqq(\Aff_{W,b})^{-1}(A')\cap\OK(\omom,N),
\end{equation}
the following estimate 
\[
\EE\left(\ichi_{B}\cdot\tilde{\lambda} \relmiddle| \OK(\omom,N)\right)\geq\tide
\]
of the weighted density of $B$ with respect to the weight $\tilde{\lambda}=\tilde{\lambda}_{w,M;\chi,r,K,b}$, which is defined in \eqref{eq=lambda}, holds true. Here,  $\tide\coloneqq u\cdot\delta$.
\end{proposition}
An important point here is that not all $b\in \OK$ prime to $W$ satisfy the estimate of the weighted density in Proposition~\ref{proposition=positiverelativeweighteddensity}, whereas Theorem~\ref{theorem=pseudorandommeasure} holds true for all $b\in \OK$ prime to $W$. We will go back to this point in Section~\ref{section=slidetrick} in more detail. In the proof of Proposition~\ref{proposition=positiverelativeweighteddensity} below, note that $C=C(\omom,\DD)$ only depends on $\omom$ and $\DD$.
\begin{proof}
By Proposition~\ref{proposition=counting_below_DD}, for a sufficiently large $M$ depending on $\omom, \DD$ and $\delta$, we \havethat\
\[
\#A\geq\frac{C}{2(n+1)h}\delta\cdot\frac{M^n}{\log M}.
\]
Since $M^{\frac{1}{16}}=o_{M\to\infty}\bigl(\frac{M^n}{\log M}\bigr)$, Lemma~\ref{lemma=jogai} \eqref{en:noudo_T} implies that for a sufficiently large $M$ depending on $\omom$, $\DD$ and $\delta$, 
\begin{equation}\label{eq:M-world-counting2}
\#A'\geq\frac{C}{3(n+1)h}\delta\cdot\frac{M^n}{\log M}
\end{equation}
holds true.

Next, we will choose an appropriate $b\in \OK$.
By Lemma~\ref{lemma=jogai}~\eqref{en:coprime_A'}, if $M$ is sufficiently large depending on $r$ and $K$, then the image of $A'$ by the natural projection $\OK\twoheadrightarrow \OK/W\OK$ is a subset of $(\OK/W\OK)^{\times}$. Hence $A'$ is partitioned as
\begin{equation}\label{eq:partition_of_A}
A'=\bigsqcup_{\overline{c}\in(\OK/W\OK)^{\times}}(A'\cap \overline{c});
\end{equation}
here, we regard $\overline c$ as a subset $\overline{c}\subseteq \OK$ as mentioned at the beginning of the current section.
Apply the pigeonhole principle to \eqref{eq:partition_of_A}. Then \eqref{eq:M-world-counting2} implies that there exists $\overline{b}\in(\OK/W\OK)^{\times}$ such that
\begin{equation}\label{eq:M-world-counting3}
\#(A'\cap \overline{b})\geq\frac{1}{\vph_K(W)}\cdot\frac{C}{3(n+1)h}\delta\cdot\frac{M^n}{\log M}
\end{equation}
holds.
Fix such a coset $\overline{b}\in \OK/W\OK$, and choose $b\in \overline{b}$ which satisfies
\begin{equation}\label{eq=b}
\|b\|_{\infty,\omom}<W.
\end{equation}
We set $\tilde{\lambda}=\tilde{\lambda}_{w,M;\chi,r,K,b}$ corresponding to this $b$.

In what follows, we will make estimates of weight densities in the $N$-world. Define $B$ by \eqref{def_of_B}, corresponding to the element $b$ above. Since 
$\Aff_{W,b}\colon\OK\to\OK$ is injective, 
\[
\EE(\ichi_B\cdot\tilde{\lambda}\mid\OK(\omom,N))=\frac{\vph_K(W)}{W^n}\cdot\EE\bigl(\ichi_{\Aff_{W,b}(B)}\cdot\lambda \ \big| \ \Aff_{W,b}(\OK(\omom,N))\bigr)
\]
holds true.
Note that $\Aff_{W,b}(B)=A'\cap\Aff_{W,b}(\OK(\omom,N))$. Hence, by Lemma~\ref{lemma=jogai}~\eqref{en:log_compatible}, we \havethat\
\begin{equation}\label{eq:Step4-counting1}
\EE(\ichi_B\cdot\tilde{\lambda}\mid\OK(\omom,N))\geq\frac{\vph_K(W)}{W^n}\cdot\frac{\kappa}{17(r+1)2^r c_{\chi}}\cdot\frac{\#\bigl(A'\cap\Aff_{W,b}(\OK(\omom,N))\bigr)}{(2N+1)^n}\cdot \log M.
\end{equation}
Here observe by the choice of $N$ in \eqref{eq:def_of_N}, \eqref{eq=b} and by the triangle inequality for $\|\cdot\|_{\infty,\omom}$,
\[
\Aff_{W,b}(\OK(\omom,N))\supseteq\OK(\omom, M)\cap (W\OK +b)
\]
holds true; indeed,
our choice \eqref{eq:def_of_N} of $N$ is made in such a way that the inclusion above is satisfied.

Therefore, by \eqref{eq:M-world-counting3} and \eqref{eq:Step4-counting1}, we conclude that
\[
\EE(\ichi_B\cdot\tilde{\lambda}\mid\OK(\omom,N))\geq\frac{\kappa}{17(r+1)2^rc_{\chi}}\cdot\frac{C}{3(n+1)h}\delta\cdot\frac{M^n}{W^n(2N+1)^n}.
\]
Finally, since $N\leq \frac{2M}{W}$ and $(2N+1)^n\leq 3^n N^n$, we \obtainthat\ 
\begin{equation}\label{eq=tide_PES}
\EE(\ichi_B\cdot\tilde{\lambda}\mid\OK(\omom,N))\geq\frac{\kappa}{51h(n+1)6^{n}}\cdot\frac{C}{(r+1)2^{r}c_{\chi}}\cdot \delta.
\end{equation}
Hence we can take $u=\frac{\kappa C}{51h(n+1)6^{n}(r+1)2^{r}c_{\chi}}$.
\end{proof}
\subsection{Proof of Theorem~\ref{theorem=primeconstellationsfinite}}
\label{subsection=proofofmaintheorem}
In this subsection, we present the proof of Theorem~\ref{theorem=primeconstellationsfinite}, which is the first goal of the present paper. Here, we restate Theorem~\ref{theorem=primeconstellationsfinite}.
\begin{theorem}[Theorem~\ref{theorem=primeconstellationsfinite}, restated]\label{theorem=primeconstellationsfiniteagain}
Assume Setting~$\ref{setting=section7-1}$.
Let $\delta>0$.
Then, there exists a positive integer $M_{\mathrm{PES}}=M_{\mathrm{PES}}(\omom,\DD,\delta,S)$ depending %
on $\omom$, $\DD$, $\delta$ and $S$ such that the following holds true: assume that $M\geq M_{\mathrm{PES}}$ and  that a subset $A$ of $\PP_K\cap\DD\cap\OK(\omom,M)$ satisfies
\[
\#A\geq\delta\cdot\#(\PP_K\cap \DD \cap\OK(\omom,M)).
\]
Then there exists an $S$-constellation in $A$. Moreover, there exists a positive real number $\gamma=\gamma_{\mathrm{PES}}(\omom,\DD,\delta,S)$ depending only on $\omom$, $\DD$, $\delta$ and $S$ such that in the setting above,
\[
\scrN_S(A)\geq \gamma\cdot \frac{M^{n+1}}{(\log M)^{\# S}}
\]
holds true.
\end{theorem}
\begin{proof}
We use Setting~\ref{setting=section7-2}. First, we will prove the former assrtion. For $A$ in the assumption, let $A'$ be as in Proposition~\ref{proposition=positiverelativeweighteddensity} and let $\tide$ as in \eqref{eq=tide_PES}. Here, note that $\tide$ depends only on $\omom,\DD,\delta,\chi$ and $r$. Take $\gamma'=\gamma_{\mathrm{RMS}}(\omom,\tide,S)$ and  $\rho=\rho_{\mathrm{RMS}}(\omom,\tide,S)$, which are determined from the relative multidimensional Szemer\'{e}di theorem (Theorem~\ref{thm:RMST}); note that we here substitute $\tide$ for $\delta$. For this $\rho$, set $w=w_{\mathrm{PES}}(\omom,\DD,\delta,\chi,S)$ as $w\coloneqq w_{\mathrm{PR}}(\rho,\chi,S)$; set $W$ from $w$ by \eqref{eq:RwW}. Take $b, B, \tilde{\lambda}=\tilde{\lambda}_{w,M;\chi,r,K,b}$ as in the statement of Proposition~\ref{proposition=positiverelativeweighteddensity}. Then, we set the lower bound of the main parameter $M$ in the following manner:

\begin{itemize}
\item (Pseudorandomness) Define $M_{\mathrm{PES}}^{(1)}=M_{\mathrm{PES}}^{(1)}(\omom,\DD,\delta,\chi,S)$ by $M_{\mathrm{PES}}^{(1)}\coloneqq M_{\mathrm{PR}}(w,\rho,\chi,S)$.
Then by Theorem~\ref{theorem=pseudorandommeasure}, if $M\geq M_{\mathrm{PES}}^{(1)}$, then $\tilde{\lambda}$ is a $(\rho,N,S)$-pseudorandom measure. 
\item (Weighted density) Define $M_{\mathrm{PES}}^{(2)}\coloneqq M_{\mathrm{DS}}(\omom,\DD,\delta,r)$. By Proposition~\ref{proposition=positiverelativeweighteddensity}, if $M\geq M_{\mathrm{PES}}^{(2)}$, then the weighted density condition
\[
\EE(\ichi_B\cdot\tilde{\lambda}\mid\OK(\omom,N))\geq\tide
\]
is fulfilled.
\item (Smallness) By Lemma~\ref{lemma=jogai}~\eqref{en:log_compatible}, we \havethat\

\begin{align*}
\EE(\ichi_B\cdot\tilde{\lambda}^{r+1}\mid\OK(\omom,N))&\leq\EE\bigl(\ichi_{A'}\cdot\lambda^{r+1} \ \big| \ \Aff_{W,b}(\OK(\omom,N))\bigr)\\
&\leq\left(\frac{\kappa}{17(r+1)2^rc_{\chi}}\right)^{r+1} (\log M)^{r+1}.
\end{align*}
It then follows that
\begin{align*}
\frac{1}{N}\cdot\EE(\ichi_{B}\cdot\tilde{\lambda}^{r+1}\mid\OK(\omom ,N))&\leq W\cdot\left(\frac{ \kappa}{17(r+1)2^rc_{\chi}}\right)^{r+1} \frac{(\log M)^{r+1}}{M}\\
&=o_{M\to\infty; \omom,\DD,\delta,\chi,S}(1).
\end{align*}
Hence, there exists $M^{(3)}_{\mathrm{PES}}=M^{(3)}_{\mathrm{PES}}(\omom,\DD,\delta,\chi,S)$ such that the following holds: if $M\geq M^{(3)}_{\mathrm{PES}}$, then the smallness condition

\[
\EE(\ichi_{B}\cdot\tilde{\lambda}^{r+1}\mid\OK(\omom ,N))\leq \gamma'\cdot N
\]
is fulfilled.
\end{itemize}
Finally, set $M'_{\mathrm{PES}}=M'_{\mathrm{PES}}(\omom,\DD,\delta,\chi,S)$ as $M'_{\mathrm{PES}}\coloneqq\max\{M^{(1)}_{\mathrm{PES}},M^{(2)}_{\mathrm{PES}},M^{(3)}_{\mathrm{PES}}\}$.
Then, for every $M$ with $M\geq M'_{\mathrm{PES}}$, Theorem~\ref{thm:RMST} applies to the set $B$ above. This implies that there exists an $S$-constellation in $B$. By sending it by $\Aff_{W,b}$, we obtain an $S$-constellation in $A'$; recall Remark~\ref{remark=affine}. Since $A'\subseteq A$, in particular, there exists an $S$-constellation in $A$, as desired.

On dependence on parameters, observe that $\chi$ is taken in order to construct $\tilde{\lambda}$; it does not appear in the setting itself of Theorem~\ref{theorem=primeconstellationsfiniteagain}. Hence, we may set 

\[
M_{\mathrm{PES}}\coloneqq\min\limits_{\chi}\left\lceil M'_{\mathrm{PES}}(\omom,\DD,\delta,\chi,S)\right\rceil.
\]
Then $M_{\mathrm{PES}}$ only depends on $\omom$, $\DD$, $\delta$ and $S$. This ends the proof of the former assertion.

Next, we prove the latter assertion on $\scrN_S(A)$. Fix $\chi$ with $M_{\mathrm{PES}}=\left\lceil M'_{\mathrm{PES}}(\omom,\DD,\delta,\chi,S)\right\rceil$. For this $\chi$, take $w$ and $W$ as in the proof of the former assertion. In what follows, let $M\geq M_{\mathrm{PES}}$. By Theorem~\ref{theorem=weighted-counting}, we have for $\gamma'=\gamma_{\mathrm{RMS}}(\omom,\tide,S)>0$ that 
\[
\frac{\mathscr{N}_S(B)\times\left(\frac{ \kappa}{17(r+1)2^rc_{\chi}}\log M\right)^{r+1}}{N(2N+1)^n}\geq\gamma'.
\]
Indeed, since $\frac{\vph_K(W)}{W^n} \leq 1$, for every $S$-constellation $\eS$, the value of $\prod_{s'\in \eS}(\ichi_{B}\cdot\tilde{\lambda})(s')$ does not exceed $\left(\frac{\kappa}{17(r+1)2^rc_{\chi}}\log M\right)^{r+1}$ by Lemma~\ref{lemma=jogai}~\eqref{en:log_compatible}. Also note that $\mathscr{N}_S(B)\leq \mathscr{N}_S(A')\leq \mathscr{N}_S(A)$ from the argument in the last part of the proof of the former assertion. Together with \eqref{eq:def_of_N}, we \havethat\ 
\[
\mathscr{N}_S(A)\geq\left(\frac{\gamma'\cdot(17(r+1)2^rc_{\chi})^{r+1}2^n}{\kappa^{r+1}\cdot W^{n+1}}\right)\cdot\frac{M^{n+1}}{(\log M)^{r+1}}.
\]
Recall that our $W$ depends only on the data $\omom,\DD,\delta,\chi$ and $S$.
Set $\gamma=\gamma_{\mathrm{PES}}(\omom,\DD,\delta,S)$ as
\[
\gamma\coloneqq \frac{(17(r+1)2^rc_{\chi})^{r+1}2^n}{\kappa^{r+1}\cdot W^{n+1}}\cdot\gamma';
\]
here we do not indicate the dependence on $\chi$ since $\chi$ is already fixed.
Therefore, for $M\geq M_{\mathrm{PES}}$, we obtain the desired estimate of $\mathscr{N}_S(A)$. This completes the proof of Theorem~\ref{theorem=primeconstellationsfiniteagain}.
\end{proof}
\begin{proof}[Proof of Corollary~$\ref{corollary=primeconstellationsupperdense}$]
Note that Proposition~\ref{proposition=densityequivalence} implies that the two conditions in Corollary~\ref{corollary=primeconstellationsupperdense} are equivalent to each other.
Take an arbitrary finite subset $S$ of $\OK$.
Assume that $\overline{d}_{\PP_K\cap \DD,\omom}(A)>\delta>0$. 
Then there exists a sequence $M_1<M_2<M_3<\cdots$ of positive real numbers with $\lim\limits_{n\to \infty}M_n=\infty$ such that for all $n\in \mathbb{N}$, the set $A\cap\OK(\omom,M_n)$ witnesses the relative density at least $\delta$. 
There exists $m$ such that $M_m\geq M_{\mathrm{PES}}(\omom,\DD,\delta,S)$. 
Apply Theorem~\ref{theorem=primeconstellationsfiniteagain} to $A\cap\OK(\omom,M_m)$ with the parameter $M_m$ with such $m$, we can find an $S$-constellation in $A\cap\OK(\omom,M_m)(\subseteq A)$.
Since $S$ is arbitrarily taken, this completes the proof; note that $m$ itself does depend on $S$ but $A\supseteq A\cap\OK(\omom,M_n)$ for all $n\in \mathbb{N}$.
\end{proof}
\begin{proof}[Proof of Corollary~$\ref{corollary=primeconstellationssimple}$]
Proposition~\ref{proposition=normrespectingfundamentaldomain} ensures the existence of an NL-compatible $\OKt$-fundamental domain $\DD$. Apply Corollary~\ref{corollary=primeconstellationsupperdense} with $A=\PP_K\cap \DD$; since $\DD$ is an $\OKt$-fundamental domain, $\DD$ itself admits no associate pairs.
\end{proof}
\begin{remark}[On the choice of $W$]
	On the estimate of $\scrN_{S}(A)$ in Theorem~\ref{theorem=primeconstellationsfinite}, the important observation is that we can take $w$, and hence also $W$, independent of $M$ as long as $M$ is sufficiently large. This observation is based on \cite[footnote~20]{Green-Tao08}.
\end{remark}

%% file: Fig2_02_06.tex
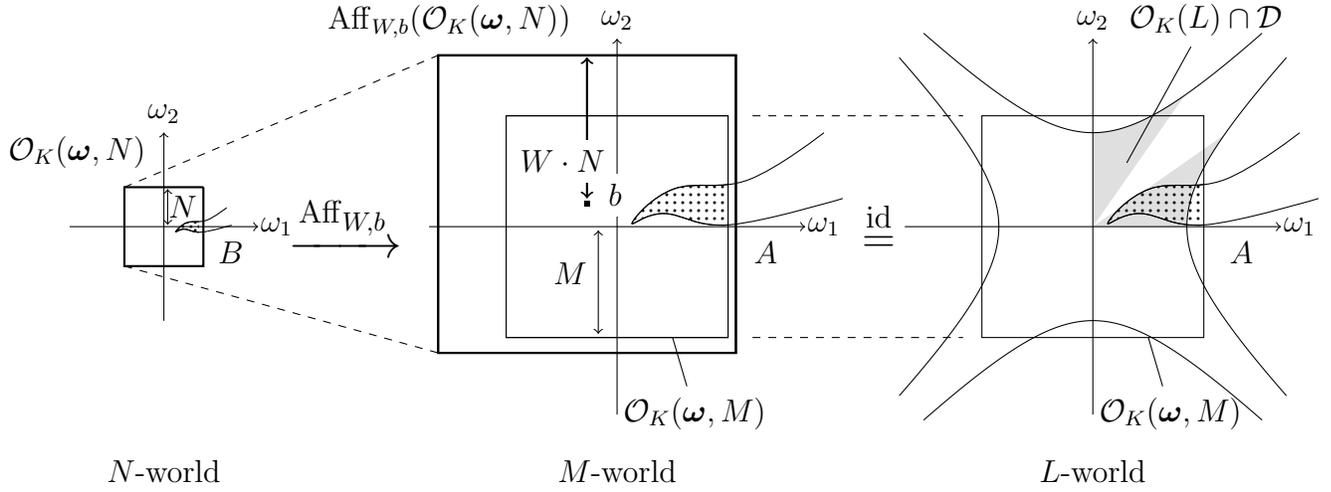
\begin{figure}[htbp]
    \def\myW{3.77}
    \def\myN{2.1}
    \def\smallaxes{5.01}
    \def\mydistance{24.1}
    \def\myadvance{25.3} 
    \newcommand{\shift}{(0,\mydistance)}

        \begin{tikzpicture}[scale=0.25,xshift = -1000]
        
            \draw[thick,fill=lightgray!0] (\myN-\mydistance,\myN)--(\myN-\mydistance,-\myN)--(-\myN-\mydistance,-\myN)--(-\myN-\mydistance,\myN)--(\myN-\mydistance,\myN); 
            
            \draw[->] (-\smallaxes-\mydistance,0)--(\smallaxes-\mydistance,0); 
            \node at (1+\smallaxes-\mydistance,0) {$\omega_1$}; 
            \draw[->] (0-\mydistance,-\smallaxes)--(0-\mydistance,\smallaxes); 
            \node at (0-\mydistance,1+\smallaxes) {$\omega_2$};
            \node at (0-\mydistance,-13) {$N$-world};

            \node at ($\myN*(-2.2,1.9)+(0-\mydistance,0)$) {$\mathcal O_K(\mathbf{\omom},N)$};
            \draw[<->] (0.2-\mydistance,0.1)--(0.2-\mydistance,\myN -0.1); 
            \node at (1-\mydistance,\myN/2) {$N$};

            \coordinate (bN) at (-1.6,1.2); 
            \def\myM{5.9} 
            \draw[thin,fill=lightgray!0] (\myM,\myM)--(\myM,-\myM)--(-\myM,-\myM)--(-\myM,\myM)--cycle; 

            \newcommand{\drawhyperbolas}[2]{ 
            \draw[samples=100, domain=-9: 9, variable=\y #2] plot( {sqrt((#1 +((\y)^2) ) ) +\myadvance},\y ) 
                  [samples=100, domain=9:-9, variable=\x #2] plot( \x+\myadvance, {sqrt((#1 +(\x)^2  ) ) } ) 
                  [samples=100, domain=9:-9, variable=\y #2] plot( {-sqrt((#1 +((\y)^2) ) ) +\myadvance},\y ) 
                 [samples=100, domain=-9: 9, variable=\x #2] plot( \x+\myadvance, -{sqrt((#1 +(\x)^2 ) ) } ); 
            } 
            \newcommand{\fillhyperbolas}[2]{
            \fill[pattern=dots, pattern color=lightgray!50, samples=30, domain=-10:10, variable=\y #2] plot( {sqrt((#1 +((\y)^2) ) ) +\myadvance},\y ) 
            [samples=30, domain=10:-10, variable=\x #2] -- plot( \x+\myadvance, {sqrt((#1 +(\x)^2  ) ) } ) 
            [samples=30, domain=10:-10, variable=\y #2] -- plot( {-sqrt((#1 +((\y)^2) ) ) +\myadvance},\y ) 
            [samples=30, domain=-10:10, variable=\x #2] -- plot( \x+\myadvance, -{sqrt((#1 +(\x)^2 ) ) } ); 
            }
            \newcommand{\fillOurDomain}[3]{
                                            
            \fill[#3]
            (0+\myadvance,0) \foreach \y in {0,0.1,...,#2} {--({sqrt(#1+\y*\y)+\myadvance},\y)} -- cycle;
            
            \fill[#3]
            (0+\myadvance,0) \foreach \x in {0,0.1,...,#2} {--(\x+\myadvance,{sqrt(#1+\x*\x)})} -- cycle;
            \node (labelHere) at (0.4*#2,{0.7*sqrt(#1+0.16*#2*#2)}) {}; 
            }
            
            \node (OL) at (6+\myadvance,11) {$\mathcal O_K(L)\cap \mathcal D$}; 
            \def\widthOfDomain{25}
            \fillOurDomain{\widthOfDomain}{5}{color=lightgray!50}
            \draw (OL) -- ($0.9*(labelHere)+(\myadvance ,0)$);

            \node[fill=white] (label-OM) at (0.7*\myM,-\myN*\myW-2) {$\mathcal O_K(\omom ,M)$}; 
            \draw[thin] (label-OM)--(0.5*\myM,-\myM); 
            \draw[<->] (-1,-0.2)--(-1,0.2-\myM); 
            \node[fill=white] at (-2.5,-2.5) {$M$}; 
            
            \node[fill=white] (label-OM-shift) at ($(label-OM)+(\myadvance,0)$) {$\mathcal O_K(\omom ,M)$}; 
            \draw[thin] (label-OM-shift)--(0.5*\myM+\myadvance,-\myM);

            \draw[->] (-10,0)--(10,0); 
            \node at (11,0) {$\omega_1$}; 
            \draw[->] (0,-10)--(0,10); 
            \node at (0,11) {$\omega_2$}; 
            \node at (0,-13) {$M$-world}; 

            \draw[->] (-10+\myadvance,0)--(10+\myadvance,0); 
            \node at (11+\myadvance,0) {$\omega_1$}; 
            \draw[->] (0+\myadvance,-10)--(0+\myadvance,10); 
            \node at (0+\myadvance,11) {$\omega_2$}; 
            \node at (0+\myadvance,-13) {$L$-world}; 

            \draw[thick] ($(bN)+\myW*\myN*(1,1)$)--($(bN)+\myW*\myN*(1,-1)$)--($(bN)+\myW*\myN*(-1,-1)$)--($(bN)+\myW*\myN*(-1,1)$)--cycle; 
            \node[fill=black,inner sep=1pt] at (bN) {};
            \node [fill=white] at ($(bN)+(1.5,+0.5)$) {$b$}; 
            \draw[thick,<->] ($(bN)+(0,0.3)$)--($(bN)+(0,\myN*\myW -0.1)$); 
            \node[fill=white] at (-2.9,+3.5) {$W\cdot N$}; 
            \node[fill=white] at (-\myN*\myW-1,\myN*\myW+3) {$\mathrm{Aff}_{W,b} (\mathcal O_K(\omom,N))$};

            \newcommand{\drawDottedHyperbolas}[2]{ 
            \draw[samples=#2, domain=-10:10, variable=\y ,dotted] plot( {sqrt((#1 +((\y)^2) ) ) },\y ) 
            [samples=#2, domain=10:-10, variable=\x ,dotted] plot( \x, {sqrt((#1 +(\x)^2  ) ) } ) 
            [samples=#2, domain=10:-10, variable=\y ,dotted] plot( -{sqrt((#1 +((\y)^2) ) ) },\y ) 
            [samples=#2, domain=-10:10, variable=\x ,dotted] plot( \x, -{sqrt((#1 +(\x)^2 ) ) } ); 
            }

            \newcommand{\drawMiniHyperbolas}[3]{
            \draw[samples=#2, domain=-#3:#3, variable=\y ,dotted] plot( {sqrt((#1 +((\y)^2) ) ) },\y ) 
            [samples=#2,  domain=#3:-#3, variable=\x ,dotted] plot( \x, {sqrt((#1 +(\x)^2  ) ) } ) 
            [samples=#2,  domain=#3:-#3, variable=\y ,dotted] plot( -{sqrt((#1 +((\y)^2) ) ) },\y ) 
            [samples=#2, domain=-#3:#3, variable=\x ,dotted] plot( \x, -{sqrt((#1 +(\x)^2 ) ) } ); 
            }

    \draw[thin] (\myM,\myM)--(\myM,-\myM)--(-\myM,-\myM)--(-\myM,\myM)--cycle; 
    \draw[thin] (\myM+\myadvance,\myM)--(\myM+\myadvance,-\myM)--(-\myM+\myadvance,-\myM)--(-\myM+\myadvance,\myM)--cycle; 
    
    \draw[thin,dashed] (\myM+1.3,\myM)--(-\myM+\myadvance-1,\myM);
    \draw[thin,dashed] (\myM+1.3,-\myM)--(-\myM+\myadvance-1,-\myM);

    \drawhyperbolas{\widthOfDomain}{}
    
    \draw[thick] ($(bN)+\myW*\myN*(1,1)$)--($(bN)+\myW*\myN*(1,-1)$)--($(bN)+\myW*\myN*(-1,-1)$)--($(bN)+\myW*\myN*(-1,1)$)--cycle; 

    \node at (0.55*\myadvance,0) {\Large $\overset{\mathrm{id}}{=}$};
    \node at (-0.6*\mydistance,0) {\Large$\xrightarrow{\mathrm{Aff}_{W,b}}$};
    \draw[thin, dashed] (-\myN-\mydistance,\myN)--($(-\myN*\myW,\myN*\myW)+(bN)$); 
    \draw[thin, dashed] (-\myN-\mydistance,-\myN)--($(-\myN*\myW,-\myN*\myW)+(bN)$);

    \newcommand{\drawA}[5]{ 
            \def\coordA{
            ($#2*(12,1.5)-#3*(bN)+(#1,0)$) 
            ($#2*(6,0.1) -#3*(bN)+(#1,0)$) 
            ($#2*(2.5,0.7) -#3*(bN)+(#1,0)$) 
            ($#2*(0.8,0.2)   -#3*(bN)+(#1,0)$) 
            ($#2*(3,2)   -#3*(bN)+(#1,0)$) 
            ($#2*(7,2.5)   -#3*(bN)+(#1,0)$) 
            ($#2*(11,5)  -#3*(bN)+(#1,0)$)}

    \begin{scope}
            \clip ($(-#2*\myM+#1+#3,-#2*\myM+#3)$) rectangle ($(+#2*\myM+#1+#3,+#2*\myM+#3)$);
            \filldraw [#5,smooth,tension=0.7] plot coordinates \coordA ; 
    \end{scope}
    \draw [smooth,tension=0.7] plot coordinates \coordA ; 
    \node[fill=white] at 
        ($#2*(\myM ,0)+(2,-1.4)+(#1,0)$) {#4}; 
    }
    \drawA{0}{1}{0}{$A$}{pattern=dots}
    \drawA{\myadvance}{1}{0}{$A$}{pattern=dots}
    \drawA{-\mydistance}{0.265}{0.265}{$B$}{pattern=crosshatch dots}

\end{tikzpicture}
    \caption{Switching among the three worlds \label{fig:4}}
\end{figure}

%% file: chapter8.tex
\section{Szemer\'{e}di-type theorems in prime elements of number fields}\label{section=maintheoremfull}
In the proof of Theorem~\ref{theorem=primeconstellationsfinite} in Section~\ref{section=positiveweighteddensity},
the first key was the existence of a suitable pseudorandom measure. 
Once this %
was ensured, the main part of the rest of the proof was %
counting arguments of elements. 
We will \emph{axiomatize} such processes and establish Theorem~\ref{mtheorem=primeconstellationsfinite}.

We will have two types of \emph{axiomatized constellation theorems}: %
\begin{enumerate}[(i)]
\item\label{en:douhan_kamo_seiza} Axiomatized constellation theorems of type~$1$: Theorem~\ref{theorem=package}, Theorem~\ref{theorem=package_infinite}, Corollary~\ref{corollary=package_infinite}.
\item\label{en:hidouhan_seiza} Axiomatized constellation theorems of type~$2$: Theorem~\ref{theorem=package_DD}, Theorem~\ref{theorem=package_infinite_DD}, Corollary~\ref{corollary=package_infinite_DD}.
\end{enumerate}
Here, by constellation theorems \emph{of type~$1$}, we mean those that guarantee the mere existence of constellations; constellation theorems \emph{of type~$2$} those that guarantee the existence of constellations \emph{without associate pairs}. 
As the cost of this stronger conclusion, constellation theorems of type~2
require an additional hypothesis. %
More specifically, whereas we assume a certain condition on \counting s from below of elements in both types, in %
type~2 we furthermore impose a certain condition on \counting s \emph{from above}. %

Theorem~\ref{mtheorem=primeconstellationsfinite} in Introduction will be dedued from the axiomatic Theorem~\ref{theorem=package_DD} applied to the number field situation.
For this, the required estimates of the number of prime elements,
Propositions~\ref{proposition=PK_kouri} and \ref{proposition=prime_counting_above},
will be verified using the materials in Sections~\ref{section=NLC} in addition to the \Cheb\ density theorem. %

We collect the setting in this section.
\begin{setting}\label{setting=package}
Let $K$ be a number field, and $n$ the degree.
Let $\vph_K$ be the totient function of $K$ (Definition~\ref{def=totient}).
Let $\ideala\subseteq \OK$ be a non-zero ideal of $\OK$.
Then $\ideala$ is a free module of rank $n$ as a $\ZZ$-module; we fix a $\ZZ$-basis $\bv$ of $\ideala$.
\end{setting}
In Section~\ref{section=relativeSzemeredi}, the $(\rho,N,S)$-pseudorandom condition on measures (Definition~\ref{definition=S-linearform}) is stated for a non-negative integer parameter $N$. Hereafter, we relax the range of the parameter $N$ and consider $N$ to be a non-negative real parameter. More precisely, for a non-negative real number $N$, we define the $(\rho,N,S)$-pseudorandom condition as the $(\rho,\lceil N\rceil,S)$-pseudorandom condition in the original sense. This relaxation enables us to avoid inessential issues caused by the integrality of the parameters. In the present paper, when we express dependences of constants, we omit writing that on $K$ if the constant depends on $\bv$. This is because $K$ equals the $\QQ$-span of $\bv$, and hence $\bv$ remembers $K$.
\subsection{Preliminaries to the axiomatized constellation theorems}\label{subsection=package}
In this subsection and Subsection~\ref{subsection=packagetheorem}, we will axiomatize the arguments in Section~\ref{section=positiveweighteddensity}.
At the same time, we extend the setting to a more general case; more precisely, we consider a general ideal $\ideala$ rather than the whole $\OK$. 
The case $\ideala = \OK $ suffices 
for the proofs of Theorem~\ref{mtheorem=primeconstellationsfinite} and Theorem~\ref{mtheorem=TaoZieglergeneral}; 
however, to establish Theorem~\ref{mtheorem=quadraticform} in the full generality, we will %
need results %
for a general non-zero ideal $\ideala$. %
This generalization does not require additional work.

Recall from Section \ref{section=relativeSzemeredi} the symbol
$\ideala(\bv,M)=\{ \sum_{i \in [n]} a_i v_i : a_i \in [-M,M] \text{ for all } i \in [n] \} $.
In our axiomatic formulation, %
the sufficient condition for 
the existence of $S$-constellations in $A\subseteq \ideala(\bv,M)$ with %
$\#A \geq \delta \cdot \frac{M^n}{\log M}$
(for a large enough $M$) is formulated as the \compati. 
The counterpart in the infinitary version is described as the  family $\logpseu$ of subsets of $\ideala$. 
These two notions will be introduced in Definitions~\ref{definition=logpseudorandom} and \ref{definition=logpseudorandom_infinite} below.
\begin{definition}[\compatiW]\label{definition=logpseudorandom}
Assume Setting~$\ref{setting=package}$. 
Let $S\subseteq \ideala$ be a standard shape. 
Let $\rho>0$, $M\in \RR_{\geq 1}$, $D_1, D_2 >0$, and $\varepsilon \in (0,1)_{\RR}$.
Let $W\in \NN$ be a natural number with $W\leq M^{\varepsilon}$. 
Then  $A\subseteq \ideala$ is said to satisfy the \emph{\compatiW\ with parameters $(D_1,D_2,\varepsilon)$} if $A\subseteq \ideala(\bv,M)$ and if there exists $\lambda\colon \ideala\to \RR_{\geq 0}$ such that the following hold.
\begin{enumerate}[(1)]
  \item \label{en:pseudorandom}
For every $b\in \ideala$ with $b\OK +W\ideala=\ideala$, the function $\beta \mapsto \frac{\vph_K(W)}{W^n}(\lambda\circ \Aff_{W,b})(\beta)$ on $\ideala$
is a $(\rho,\frac{M}{W},S)$-pseudorandom measure.
  \item \label{en:measurebelow} 
There exists $T\subseteq A$ with $\#T\leq M^{\varepsilon n}$ 
such that, for every $\alpha\in A\setminus T$,
\begin{equation*}
D_1 \cdot \log M \leq \lambda(\alpha)\leq D_2\cdot  \log M
\end{equation*}
holds.
  \item \label{en:coprime} 
  For the subset  $T$ in \eqref{en:measurebelow}
  and for every $\alpha\in A\setminus T$, the equality $\alpha \OK +W\ideala =\ideala$ holds.\end{enumerate}
\end{definition}
\begin{definition}[\compati]\label{definition=logpseudorandomW}
Assume Setting~$\ref{setting=package}$.
Let $S\subseteq \ideala$ be a standard shape.
Let $\rho>0$, $M\in \RR_{\geq 1}$, and take $D_1, D_2>0$ and $\varepsilon \in (0,1)_{\RR}$. Then $A\subseteq \ideala$ is said to satisfy the \emph{\compati\ with parameters $(D_1,D_2,\varepsilon)$} if there exists a natural number $W\in \NN$ with $W\leq M^{\varepsilon}$ such that $A$ satisfies the \compatiW\ with parameters $(D_1,D_2,\varepsilon)$.
\end{definition}

In the three conditions in Definition~\ref{definition=logpseudorandom}, \eqref{en:pseudorandom} corresponds to Theorem~\ref{theorem=pseudorandommeasure} in the proof of Theorem~\ref{theorem=primeconstellationsfinite} in Section~\ref{section=positiveweighteddensity}; \eqref{en:measurebelow} and \eqref{en:coprime} correspond to Lemma~\ref{lemma=jogai}.

To formulate the infinitary version of the \compati, dependence between $M$ and $\rho$ is of importance. However, on the data $(D_1, D_2,\varepsilon)$, the only requirement is that they exist without depending on $M$; we do not care the exact values of them. For this reason, in Definition~\ref{definition=logpseudorandom}, we divide the seven data into two classes, $(\rho,M,\bv,S)$ and $(D_1,D_2,\varepsilon)$, and use the terminology of `the \compati\ with parameters $(D_1,D_2,\varepsilon)$.' Note also that $\bv$ and $S$ are given data.
\begin{definition}[The family $\logpseu$]\label{definition=logpseudorandom_infinite}
Under Setting~$\ref{setting=package}$,
we define a family $\logpseu$ of subsets of 
$\ideala$ as follows.
For $A\subseteq\ideala$, we declare that
$A\in \logpseu$ if for every standard shape
$S\subseteq \ideala$, there exist 
$D_1,D_2>0$ and $\varepsilon\in (0,1)_{\RR}$ such that the following holds true:
for every $\rho>0$, there exists $M(\rho)=M(\rho,S)\in \RR$
such that, for every $M\geq M(\rho)$,
$A\cap \ideala(\bv,M)$
satisfies the \compati\ with parameters
$(D_1,D_2,\varepsilon)$.
\end{definition}
In the symbol `$\logpseu$' above, `S', `$\Psi$' and `$\log$', respectively, stand for \emph{subset}, \emph{pseudorandom}, and \emph{having the sparsity of order $1/\log$}; here, we will show in  Lemma~\ref{lemma=logpseudorandom_subset}~\eqref{en:pseudosubset} that the family $\logpseu$ is closed under taking subsets. 
The condition of $A\in \logpseu$ is formulated for a fixed $\ZZ$-basis $\bv$; however, it may be easily seen that this condition does not depend on the choice of $\bv$. 

When $\ideala = \OK $, the condition `$b\OK+W\ideala=\ideala$'
in conditions~\eqref{en:pseudorandom} and \eqref{en:coprime} in Definition~\ref{definition=logpseudorandom}
is equivalent to saying that $b$ is prime to $W$.
Also note the following fact.
\begin{lemma}\label{lemma=coprime}
Let $\ideala\in \Ideals_K$ and  let $W\in \NN$.
Then the image of the set
$\{b \in \ideala : b\OK+W\ideala=\ideala\}$
under the natural projection
$\ideala\twoheadrightarrow \ideala/W\ideala$
has cardinality $\vph_K(W)$.
\end{lemma}
\begin{proof}
Apply Lemma~\ref{lemma=a/Wa} with
$\idealb=W\OK$; this yields an isomorphism 
$\ideala/W\ideala\simeq \OK/W\OK$ as $\OK$-modules.
An element $\gamma\in \OK/W\OK$ belongs to $(\OK/W\OK)^{\times}$ if and only if it generates $\OK/W\OK$ as an $\OK$-module. For $b\in \ideala$, the equality $b\OK+W\ideala=\ideala$ holds if and only if $\overline{b}$ generates $\ideala/W\ideala$ as an $\OK$-module, where $\overline{b}$ is the image of $b$ under $\ideala\twoheadrightarrow \ideala/W\ideala$. Combination of these observations ends the proof.
\end{proof}
\begin{lemma}[Heredity to subsets]\label{lemma=logpseudorandom_subset}
The following statements hold true.
\begin{enumerate}[$(1)$]
\item \label{en:compatisubset} 
If $A\subseteq \ideala(\bv,M)$ satisfies the \compatiW\
with parameters $(D_1,D_2,\varepsilon)$, and
$A_1\subseteq A$, then $A_1$ satisfies the \compatiW\
with parameters $(D_1,D_2,\varepsilon)$.
\item \label{en:pseudosubset}
If $A\in \logpseu$ and 
$A_1\subseteq A$, then $A_1\in \logpseu$.
\end{enumerate}
\end{lemma}
\begin{proof}
First we will prove \eqref{en:compatisubset}. 
We use the same measure $\lambda $.
Let $\lambda $ and $T$ be the measure and exceptional set (a set satisfying~\eqref{en:measurebelow} and~\eqref{en:coprime} of Definition~\ref{definition=logpseudorandom_infinite}) for $A$.
Then $T_1\coloneqq T\cap A_1$ works as %
an exceptional set for $A_1$. 
Hence \eqref{en:compatisubset} holds. Item~\eqref{en:pseudosubset} immediately follows from \eqref{en:compatisubset}.
\end{proof}
The following theorem provides a motivating example of a member of the family $\logpseuOK$.
\begin{theorem}\label{theorem=logpseudorandom}
The set $\PP_K$ of prime elements
of a number field $K$ satisfies that
$\PP_K\in  \logpseuOK$.
Furthermore, for every integer $r\geq [K:\QQ]$, there exist $D_1,D_2>0$ and $\varepsilon\in (0,1)_{\RR}$ such that the following holds true. 
Let $S\subseteq \OK$ be a standard shape with $\#S=r+1$, and $\omom$ an integral basis of $K$. Let $\rho>0$. Then there exist a natural number $W=W_{\PP_K,\mathrm{S}\Psi_{\log}}(\rho,S)$ and a positive real number $M_{\PP_K,\mathrm{S}\Psi_{\log}}(\rho,\omom,S)$ such that for every  $M\geq M_{\PP_K,\mathrm{S}\Psi_{\log}}(\rho,\omom,S)$, the set $\PP_K\cap \OK(\omom,M)$ satisfies the $(\rho,W,M,\omom,S)$-condition with parameters $(D_1,D_2,\varepsilon)$.
\end{theorem}
Theorem~\ref{theorem=pseudorandommeasure} works for the proof of the pseudorandomness required in Theorem~\ref{theorem=logpseudorandom}. Nevertheless, prior to the proof of Theorem~\ref{theorem=logpseudorandom}, we prove the following theorem, which treats a more general setting. This is because similar arguments to Theorem~\ref{theorem=package_PR} will be needed in Sections~\ref{section=slidetrick} and \ref{section=quadraticform}. 

\begin{theorem}\label{theorem=package_PR}
Let $K$ be a number field of degree $n$, $S\subseteq \OK$ a finite subset such that $0\in S$ and that $S$ generates $\OK$ as a $\ZZ$-module. Let $\chi$ and $c_{\chi}$ be as in Setting~$\ref{setting=section7-2}$. 
Let $\rho>0$, $\uvarsigma>0$ and $a\in (0,1]_{\RR}$. 
Then there exist positive real numbers $w_{\mathrm{PRSI}}(\rho,\chi,S)$, depending only on $\rho$, $\chi$ and $S$, and $M_{\mathrm{PRSI}}(w,\rho,\uvarsigma,\chi,S,a)$, depending only on $w,\rho,\uvarsigma,\chi,S$ and $a$, such that the following holds true. Assume that $w\geq w_{\mathrm{PRSI}}(\rho, \chi, S)$ and $M\geq M_{\mathrm{PRSI}}(w,\rho,\uvarsigma,\chi,S,a)$. Define $W=W_w$ by \eqref{eq:RwW}, and set $R=R_{M;r,a}$ as
\begin{equation}\label{eq:RMa}
R\coloneqq M^{\frac{a}{17(r+1)2^r}}.
\end{equation}
Define a function $\lambda=\lambda_{M;\chi,r,a,K}\colon \OK\to\RR_{\geq 0}$ by
\[
\lambda(\alpha)\coloneqq\frac{\kappa\cdot\Lambda_{R,\chi}(\alpha)^2}{c_{\chi}\log R}.
\]
Here, $\kappa=\kappa_K>0$ is the constant appearing in Theorem~$\ref{theorem=zeta_K}$, $\vph_K$ is the totient function of $K$ $($Definition~$\ref{def=totient}$$)$, and $\Lambda_{R,\chi}$ is the $(R,\chi)$-von Mangoldt function $($Definition~$\ref{def:chi}$$)$. For $b\in\OK$ coprime with $W$, define $\tilde{\lambda}=\tilde{\lambda}_{w,M;\chi,r,a,K,b}\colon\OK\to\RR_{\geq 0}$ by
\[
\tilde{\lambda}(\beta)\coloneqq\frac{\vph_K(W)}{W^n}(\lambda\circ \Aff_{W,b})(\beta).
\]
Then, $w\leq \frac{a}{2}\log M$ holds, and $\tilde{\lambda}$ is a $(\rho,\frac{\uvarsigma M^a}{W},S)$-pseudorandom measure.
\end{theorem}

In the assertion of Theorem~\ref{theorem=package_PR}, by $w\leq \frac{a}{2}\log M$ and \eqref{eq:Chebyshev-ineq}, we in particular \havethat\ 
\begin{equation}\label{eq:WaM}
W\leq M^{(\log 2)a}.
\end{equation}

\begin{proof}
We prove Theorem~\ref{theorem=package_PR} by generalizing the proof of Theorem~\ref{theorem=pseudorandommeasure}. Let $t=2r+2$.
Take an arbitrary non-empty subset $\mathcal{J}$ of $\bigsqcup_{j\in [r+1]}\{0,1\}^{e_j}$, and let $m\coloneqq \# \mathcal{J}$. Here, $e_j=[r+1]\setminus\{j\}$.
Write $S=\{s_1,\dots,s_r\}\sqcup\{0\}$. Then for each $\omega\in\mathcal{J}$, the homomorphism $\psi_S^{(\omega)}\colon\ZZ^t\to\OK$ of $\ZZ$-modules is defined by Definition~\ref{definition=S-linearform}~\eqref{Eq:the-linear-maps-1} and \eqref{Eq:the-linear-maps-2}. By the assumptions of $S$, these maps are all surjective.

Take $R_0(m,K)$, $F_0(m,n)$ and $w_0((\psi_S^{(\omega)})_{\omega\in\mathcal{J}})$ as in Theorem~\ref{Th:Goldston_Yildirim_rho}. Set $R_1(r,K)$, $F_1(r,n)$ and $w_1(S)$ in the same manner as \eqref{eq:R_1,F_1,w_1}. 
Recall $\rho_{\mathrm{GY};\chi, m,n}^{(1)}(w)$ and  $\rho_{\mathrm{GY};\chi,m,t,K}^{(2)}(w,R)$ from Theorem~\ref{Th:Goldston_Yildirim_rho}. For $w\geq w_1(S)$, set
\[
\rho_{\mathrm{PRSI};\chi,r,n}^{(1)}(w)\coloneqq \max_{m\in[(r+1)2^r]}\rho_{\mathrm{GY};\chi,m,n}^{(1)}(w),
\]
and under the conditions that $R\geq R_1(r,K)$ and $\log w \leq F_1(r,n)\cdot\sqrt{\log R}$, set
\[
\rho_{\mathrm{PRSI};\chi,r,K}^{(2)}(w,M,a)\coloneqq \max_{m\in[(r+1)2^r]}\rho_{\mathrm{GY};\chi,m,t,K}^{(2)}(w,R).
\]
Here recall that $R$ is defined from $M,r$ and $a$ by \eqref{eq:RMa}.

For $w\geq w_1(S)$, it follows from \eqref{Eq:rho_bigO} that if $w$ is sufficiently large depending on $\chi, r,n$ and $\rho$, then $\rho_{\mathrm{PRSI};\chi,r,n}^{(1)}(w)\leq \rho/2$  holds. Set $w_{\mathrm{PRSI}}(\rho,\chi,S)$ with $w_{\mathrm{PRSI}}(\rho,\chi,S)\geq w_1(S)$ in such a way that $w\geq w_{\mathrm{PRSI}}(\rho,\chi,S)$ implies $\rho_{\mathrm{PRSI};\chi,r,n}^{(1)}(w)\leq\rho/2$. Let $w\geq w_{\mathrm{PRSI}}(\rho,\chi,S)$. Then, choose $M_{\mathrm{PRSI}}(w,\rho,\uvarsigma,\chi,S,a)$ with $M_{\mathrm{PRSI}}(w,\rho,\uvarsigma,\chi,S,a)\geq e^{2w/a}$ such that if $M\geq M_{\mathrm{PRSI}}(w,\rho,\uvarsigma,\chi,S,a)$, then the following four inequalities
\begin{align*}
&R\geq R_1,\\
&\log w \leq F_1\cdot \sqrt{\log R},\\
&M^{\left(\frac{12}{17}-\log 2\right)a}\geq \frac{1}{u},\\
&\rho_{\mathrm{PRSI};\chi,r,K}^{(2)}(w,M,a)\leq \frac{\rho}{2}
\end{align*}
are all satisfied. Then since $M\geq e^{2w/a}$, we in particular have $w\leq \frac{a}{2}\log M$.

Take an arbitrary subset $\cal{B}$ of $\ZZ^{r+1}$ which may be written as the product of intervals of lengths at least $\frac{\uvarsigma M^a}{W}$. 
Since $m\leq (r+1)2^r$, we have $R^{4m+1}\leq R^{5(r+1)2^r}$. It follows from $M^{a((12/17)-\log 2)}\geq u^{-1}$  that $R^{4m+1}\leq \frac{\uvarsigma M^a}{W}$. Hence, $\calB\times\calB\subseteq\ZZ^t$ is the product of intervals of lengths at least $R^{4m+1}$. 
Therefore, we may appeal to Theorem~\ref{Th:Goldston_Yildirim_rho} and conclude the following: if $w\geq w_{\mathrm{PRSI}}(\rho,\chi,S)$ and if $M\geq M_{\mathrm{PRSI}}(w,\rho,\uvarsigma,\chi,S,a)$, then $\tilde{\lambda}$ satisfies the $(\rho,\frac{\uvarsigma M^a}{W},S)$-linear forms condition.
\end{proof}
\begin{proof}[Proof of Theorem~$\ref{theorem=logpseudorandom}$]
The proof goes along the same line as the arguments in Section~\ref{section=positiveweighteddensity}, except one on the bound of $\#T$. For the reader's convenience, we go into details of the proof.
Let $W=W_{\PP_K,\mathrm{S}\Psi_{\log}}(\rho,\chi,S)$ be the positive  integer determined from $w=w_{\mathrm{PRSI}}(\rho,\chi,S)$ by \eqref{eq:RwW}. 
Let $M$ be a parameter with $M\geq M_{\mathrm{PRSI}}(w_{\mathrm{PRSI}}(\rho,\chi,S),\rho,1,\chi,S,1)$, and take the function $\lambda=\lambda_{M;\chi,r,1,K}\colon \OK\to\RR_{\geq 0}$ defined  in Theorem~\ref{theorem=package_PR}. In particular, we set $R=M^{\frac{1}{17(r+1)2^r}}$. Now apply Theorem~\ref{theorem=package_PR} for $(a,\uvarsigma)=(1,1)$. Then we conclude that for every $b\in \OK$ coprime with $W$, the function $\frac{\vph_K(W)}{W^n}(\lambda\circ \Aff_{W,b})$ is a $(\rho,M/W,S)$-pseudorandom measure. In addition, by \eqref{eq:WaM}, we have $W\leq M^{\log 2}\leq M^{\frac{3}{4}}$.

We define the exceptional set $T\subseteq \PP_K\cap \OK(\omom,M)$ by $T\coloneqq \PP_K\cap \OK(\omom,M)\cap \OK(R)$.
Now, recall the proof of Lemma~\ref{lemma=jogai}.
The arguments for \eqref{en:log_compatible} and \eqref{en:coprime_A'} still work in the current setting. More precisely, the following holds: if $M$ is sufficiently large depending on $r$ and $K$, then \eqref{en:coprime} in the \compatiW\ is fulfilled and $\lambda(\alpha)=\frac{\kappa}{17(r+1)2^r \cdot c_{\chi}}\cdot \log M$ holds for every $\alpha\in A\setminus T$.

Therefore, what remains is the estimate of $\#T$; this part is more involved than the proof of Lemma~\ref{lemma=jogai}~\eqref{en:noudo_T}.
Proposition~\ref{proposition=idealdensity} implies that if $M$ is sufficiently large depending on $r$, then the number of ideals of $\OK$ whose ideal norms are at most $R$ does not exceed $2\kappa R$. In particular, the same estimate holds for the number of principal ideals with the same condition. For each such a principal ideal $\ideala$, consider $\alpha\in \OK(\omom,M)$ which is sent to $\ideala$ by the map $\alpha\mapsto \alpha\OK$. For every $\ideala$ above, the number of such $\alpha$ does not exceed of $\Xi'\cdot (\log M)^{n-1}$; here  $\Xi'=\Xi'(\omom)>0$ is the constant appearing in Lemma~\ref{lemma=OKt_orbit}~\eqref{en:roughcounting}. Indeed, observe that $r_1+r_2\leq n$, where $r_1$ and $r_2$ are as in Setting~\ref{setting=section4}. Hence, we conclude that
\begin{equation}\label{eq:boundT}
\#T\leq \Xi' \cdot (\log M)^{n-1} \cdot (2\kappa R).
\end{equation}
For sufficiently large $M$ depending on $\omom$, we have
$R\leq \frac{1}{2\kappa \Xi'}\cdot \frac{M^{\frac{1}{16}}}{(\log M)^{n-1}}$; in this case, we \obtainthat\ 
\[
\#T\leq M^{\frac{1}{16}}\leq M^{\frac{3}{4}n}.
\]

Therefore, by setting $M_{\PP_K,\mathrm{S}\Psi_{\log}}(\rho,\omom,\chi,S)$ as the minimum of integers for which the arguments above work, we conclude the following: if $M\geq M_{\PP_K,\mathrm{S}\Psi_{\log}}(\rho,\omom,\chi,S)$, then 
$\PP_K\cap \OK(\omom,M)$ satisfies the $(\rho,W,M,\omom,S)$-condition with parameters
\[
(D_1,D_2,\varepsilon)=\left(\frac{\kappa}{17(r+1)2^r \cdot c_{\chi}},\frac{\kappa}{17(r+1)2^r \cdot c_{\chi}},\frac{3}{4}\right).
\]
This ends the proof of the latter assertion of Theorem~\ref{theorem=logpseudorandom}. 
Here, by fixing $\chi$, we omit to write dependences of 
$W_{\PP_K,\mathrm{S}\Psi_{\log}}(\rho,\omom,\chi,S)$ and 
$M_{\PP_K,\mathrm{S}\Psi_{\log}}(\rho,\omom,\chi,S)$ on $\chi$. Thus, we write 
$W_{\PP_K,\mathrm{S}\Psi_{\log}}(\rho,\omom,S)$ and 
$M_{\PP_K,\mathrm{S}\Psi_{\log}}(\rho,\omom,S)$ for short.
Then the former assertion immediately follows.
\end{proof}

As in the proof above, in the case where $A$ is not inside an $\OKt$-fundamental domain, then the map $A\ni \alpha \mapsto \alpha \OK\in \Ideals_K$ is not injective in general. Then, we need to take into account the contribution of the action of the group of units in order to transfer ideal \counting s to element \counting s. To treat this, Lemma~\ref{lemma=OKt_orbit} is a key tool, as we have seen in the deduction of \eqref{eq:boundT} in the above proof.

\subsection{Axiomatized constellation theorems of type~1}
\label{subsection=packagetheorem}

In this subsection, we establish axiomatized constellation theorems of type~1 with the aid of the \compati\ and $\logpseu$ introduced in Subsection~\ref{subsection=package}. Here exhibit the two statements: one is the finitary version and the other is the infinitary version. 
Recall the definition of $\scrN_S(A)$ from Definition~\ref{definition=standardshape}~\eqref{en:kosuu}.
\begin{theorem}\label{theorem=package}
Assume Setting~$\ref{setting=package}$. Let $S\subseteq \ideala$ be a standard shape.
Let $\delta>0$, $D_1,D_2>0$ and $\varepsilon\in (0,1)_{\RR}$. 
Then there exist $\rho=\rho_{\rmI}(D_1,\bv,\delta,S)>0$ and
$M_{\rmI}=M_{\rmI}(D_1,D_2,\varepsilon,\bv,\delta,S)\in \NN$ such that the following
holds true.
Assume that $M\in \NN$ with 
$M\geq M_{\rmI}$ and $A\subseteq\ideala$ with
\begin{equation}\label{88a}
A\subseteq  \ideala(\bv,M)
\end{equation}
satisfies the following two conditions:
\begin{enumerate}[$(i)$]
 \item\label{en:counting} the cardinality $\#A$ satisfies
\begin{equation}\label{eq:delta_log}
\#A\geq \delta \cdot \frac{M^n}{\log M},
\end{equation}
 \item\label{en:logpseudorandom} $A$
 satisfies the \compati\ with parameters $(D_1,D_2,\varepsilon)$.
\end{enumerate}
Then $A$ contains an $S$-constellation.
Furthermore, there exists $\gamma=\gamma_{\rmI}(D_1,D_2,\bv,\delta,S)>0$ such that
\begin{equation}\label{eq:scrNN}
\scrN_S(A)\geq \gamma W^{-(n+1)}\cdot \frac{M^{n+1}}{(\log M)^{\#S}}
\end{equation}
holds.
Here, $W$ is an integer appearing in the 
\compati, which comes from condition~\eqref{en:logpseudorandom}.
\end{theorem}
\begin{theorem}\label{theorem=package_infinite}
Assume Setting~$\ref{setting=package}$.
Assume that a subset
$A\subseteq \ideala$ satisfies the following two conditions:
\begin{enumerate}[$(i)$]
\item\label{en:counting_infinite} the inequality
\begin{equation}
\limsup_{M\to \infty}\frac{\#(A\cap \ideala(\bv,M))}{M^n(\log M)^{-1}}>0\label{eq:upperlogdensity}
\end{equation}
holds,
 \item\label{en:logpseudo_infinite}$A\in \logpseu$.
\end{enumerate}
Then for every finite subset $S\subseteq\ideala$,
there exists an $S$-constellation in $A$.
\end{theorem}

In Theorem~\ref{theorem=package}, the first condition is a condition on \counting s, and the second condition is a condition which inherits to subsets; recall  Lemma~\ref{lemma=logpseudorandom_subset}. The same holds true for the two conditions in Theorem~\ref{theorem=package_infinite}.
The existence of a constellation no way inherits to subsets. Nevertheless, in the proof of constellation theorems, we can decompose our criteria into the following two parts: conditions on mere \counting s, and conditions related to pesudorandomness, which inherit to subsets. In this manner, we may  have a clear description of the proofs. For instance, Theorem~\ref{theorem=package_infinite} yields the following relative Szemer\'{e}di-type theorem. To state it, we extend the definition of the relative asymptotic density in Definition~\ref{definition=relativedensity}~\eqref{en:lenghtdensity} in the following manner. 

\begin{definition}\label{definition=reldens_c}
Let $\calZ$ be a free $\ZZ$-module of finite rank and $\bv$ a $\ZZ$-basis. For a non-empty set $X\subseteq \calZ$ and  a subset $A\subseteq X$, define the \emph{relative upper asymptotic  density of $A$ in $X$ measured by $\|\cdot\|_{\infty,\bv}$} as
\[
\overline{d}_{X,\bv}(A)\coloneqq\limsup_{M\to\infty}\frac{\#(A\cap \calZ(\bv,M))}{\#(X \cap \calZ(\bv,M))}.
\]
The \emph{relative lower asymptotic  density of $A$ in $X$ measured by $\|\cdot\|_{\infty,\bv}$} is also defined as
\[
\underline{d}_{X,\bv}(A)\coloneqq\liminf_{M\to\infty}\frac{\#(A\cap \calZ(\bv,M))}{\#(X \cap \calZ(\bv,M))}.
\]
\end{definition}

\begin{corollary}\label{corollary=package_infinite}
Assume Setting~$\ref{setting=package}$.
Assume that a subset
$A\subseteq \ideala$ satisfies the following two conditions:
\begin{enumerate}[$(i)$]
\item\label{en:counting_infinite_X}the inequality
\begin{equation}
\liminf_{M\to \infty}\frac{\#(A\cap \ideala(\bv,M))}{M^n(\log M)^{-1}}>0\label{eq:lowerlogdensity}
\end{equation}
holds,
\item\label{en:logpseudo_infinite_X}$A\in \logpseu$.
\end{enumerate}
Then for every subset $A'\subseteq A$ with
$\overline{d}_{A,\bv}(A')>0$, the following holds:
for every finite subset $S$, 
$A'$ contains an
$S$-constellation.
\end{corollary}
\begin{proof}[Proof of ``Theorem~$\ref{theorem=package_infinite}$ $\Longrightarrow$ Corollary~$\ref{corollary=package_infinite}$'']
Suppose $A'\subseteq A$ and
$\overline{d}_{A,\bv}(A')>0$.
By Lemma~\ref{lemma=logpseudorandom_subset}~\eqref{en:logpseudo_infinite_X}, we have $A'\in\logpseu$.
The set $A'$ satisfies
\eqref{eq:upperlogdensity} with $A$ replaced by $A'$.
Therefore, we can apply Theorem~\ref{theorem=package_infinite}
with replacing $A$ by $A'$.
\end{proof}

Theorem~\ref{theorem=package_infinite} is deduced from Theorem~\ref{theorem=package} in the following manner.
\begin{proof}[Proof of ``Theorem~$\ref{theorem=package}$ $\Longrightarrow$ Theorem~$\ref{theorem=package_infinite}$'']
Let $A\subseteq\ideala$ be a subset fulfilling the two conditions in Theorem~$\ref{theorem=package_infinite}$. Take an arbitrary standard shape $S\subseteq\ideala$.
Since $A\in \logpseu$, we can take parameters $(D_1,D_2,\varepsilon)$ associated with $S$. Set $\delta>0$ as the left-hand side of \eqref{eq:upperlogdensity}; if the left-hand side is $+\infty$, then set $\delta=2$. By Theorem~\ref{theorem=package}, there exist $\rho=\rho_{\mathrm{I}}(D_1,\bv,\delta/2,S)>0$ and $M_{\mathrm{I}}=M_{\mathrm{I}}(D_1,D_2,\varepsilon,\bv,\delta/2,S)\in \NN$.
Then since $A\in \logpseu$, there exist $M(\rho,\bv,S)$ as in Definition~\ref{definition=logpseudorandom_infinite} and a positive integer $M$ at least $M_{\mathrm{I}}$ such that the following hold true: we have
\[
\#(A\cap \ideala(\bv,M))\geq \frac{\delta}{2}\cdot \frac{M^n}{\log M},
\]
and $A\cap \ideala(\bv,M)$ satisfies the \compati\ with parameters $(D_1,D_2,\varepsilon)$. Therefore, Theorem~\ref{theorem=package} applies to $A\cap \ideala(\bv,M)$, and we can find an $S$-constellation in $A$. As $S$ is arbitrarily taken, this ends our proof.
\end{proof}
It remains to prove Theorem~\ref{theorem=package}. The following proof is motivated by the proofs of 
Proposition~\ref{proposition=positiverelativeweighteddensity} and Theorem~\ref{theorem=primeconstellationsfiniteagain}.
\begin{proof}[Proof of Theorem~$\ref{theorem=package}$]
Let $S$, $\delta$, $D_1$, $D_2$ and $\varepsilon$ be as in the statement.
Let $M$ be a real parameter; we will take it sufficiently large. Take an arbitrary $\rho>0$.

Let $A\subseteq \ideala(\bv,M)$ be a set which satisfies \eqref{eq:delta_log} and which fulfills the \compatiW\ with parameters $(D_1,D_2,\varepsilon)$; here $W \leq M^{\varepsilon}$ holds. Take $\lambda\colon \ideala\to \RR_{\geq 0}$ and $T\subseteq A$ associated with $A$. Trim $A$ as $A'\coloneqq A\setminus T$. First, we count $A'$ from below.
For a sufficiently large $M$ depending on $\varepsilon$ and $\delta$, we have
\begin{align*}
\#T\leq M^{\varepsilon n}\leq \frac{\delta}{2}\cdot \frac{M^n}{\log M}.
\end{align*}
Hence by \eqref{eq:delta_log} the inequality
\begin{equation}\label{eq:countingA'}
\#A'\geq \frac{\delta}{2}\cdot\frac{M^n}{\log M}
\end{equation}
holds.

In what follows, we will verify that Theorem~\ref{thm:RMST} applies to $A'$, provided that $M$ is sufficiently large. By the \compatiW~\eqref{en:coprime} and by Lemma~\ref{lemma=coprime}, the cardinality of the image of $A'$ under the natural projection $\ideala \twoheadrightarrow \ideala/W\ideala$ is at most $\vph_K(W)$. The pigeonhole principle tells us that there exists 
$\overline{b}\in \ideala/W\ideala$ such that
\[
\#(A'\cap \overline{b})\geq \frac{1}{\vph_K(W)}\cdot \#A'
\]
holds.
Here we regard $\overline{b}$ as a subset of $\ideala$.
Combine this with \eqref{eq:countingA'}, and \obtainthat\ 
\begin{equation}\label{eq:countingA'b}
\#(A'\cap \overline{b})\geq \frac{1}{2\vph_K(W)}\delta \cdot \frac{M^n}{\log M}.
\end{equation}

Set a positive integer $N$ from $M$ and $W$ by $N\coloneqq\left\lceil\frac{M}{W}\right\rceil$.
Since $W\leq M^{\varepsilon}$, we have
\begin{equation}\label{eq:NtoMW}
M^{1-\varepsilon}\leq  N\leq  \frac{2M}{W}.
\end{equation}
Choose a representative $b\in \overline{b}$ in such a way that $\|b\|_{\infty,\bv}<W$. 
Since $A'\cap\overline{b}\neq\varnothing$, we have $b\OK+W\ideala=\ideala$ by the \compatiW~\eqref{en:coprime}. By the definition of $N$ and by the triangle inequality, we have
\[
\Aff_{W,b}(\ideala (\bv,N))\supseteq\ideala(\bv,M)\cap(W\ideala+b)\supseteq A' \cap \overline{b}.
\]
This implies that $B\coloneqq\Aff_{W,b}^{-1}(A' \cap \overline{b})$  is a subset of $\ideala(\bv,N)$. For this $B$, by \eqref{eq:countingA'b}, we \havethat\ 
\begin{equation}\label{eq:relativeooi}
\#B \geq \frac{1}{2\vph_K(W)}\delta \cdot \frac{M^n}{\log M}.
\end{equation}
Define $\tilde{\lambda}\colon\ideala\to\RR_{\geq0}$ by $\tilde{\lambda}\coloneqq \frac{\vph_K(W)}{W^n}( \lambda\circ \Aff_{W,b})$.
By the \compatiW~\eqref{en:pseudorandom}, this $\tilde{\lambda}$ is a $(\rho,N,S)$-pseudorandom measure.
Furthermore, by the estimate from below in the \compatiW~\eqref{en:measurebelow} and by \eqref{eq:relativeooi} and \eqref{eq:NtoMW}, we have the following estimate of the weighted density:
\begin{equation}\label{eq:type1_weighted_density}
\begin{split}
\EE(\ichi_B\cdot \tilde{\lambda}\mid\ideala(\bv,N) ) &\geq\left(\frac{1}{2\vph_K(W)}\delta \cdot \frac{M^n}{\log M}\right)\cdot \left(\frac{\vph_K(W)}{W^n} D_1 \cdot \log M\right) \cdot (2N+1)^{-n}\\
&\geq\frac{D_1\delta}{2} \cdot \left(\frac{M}{3WN}\right)^n\geq\frac{D_1}{2\cdot 6^n}\cdot  \delta;
\end{split}
\end{equation}
it ensures the weighted density condition in the relative multidimensional Szemer\'{e}di theorem.
On the smallness condition, we employ 
\eqref{eq:NtoMW} and the estimate from above in the \compatiW~\eqref{en:measurebelow}. Then, we \havethat\ 
\begin{equation}\label{eq:small}
\frac{1}{N}\cdot \EE(\ichi_B\cdot \tilde{\lambda}^{r+1}   \mid \ideala(\bv,N) )\leq D_2^{r+1}\cdot  \frac{(\log M)^{r+1}}{M^{1-\varepsilon}},
\end{equation}
where $r\coloneqq\#S-1$.

Finally, we specify $\rho>0$ and $M_{\rmI}$ so as to activate Theorem~\ref{thm:RMST}. 
We set $\rho=\rho_{\rmI}(D_1,\bv,\delta,S)\coloneqq\rho_{\mathrm{RMS}}(\bv,\frac{D_1}{2\cdot 6^n}\delta,S)$.
Also, set $\gamma'_{\rmI}=\gamma'_{\rmI}(D_1,\bv,\delta,S)>0$ as $\gamma'_{\rmI}\coloneqq \gamma_{\mathrm{RMS}}(\bv,\frac{D_1}{2\cdot 6^n} \delta,S)$. The arguments in the current proof up to this point work for a sufficiently large $M$ depending on $\varepsilon$ and $\delta$; if necessary, we replace $M$ with a larger number depending on $D_1$, $D_2$, $\varepsilon$, $\bv$, $\delta$ and $S$ in such a way that
\[
D_2^{r+1}\cdot  (\log M)^{r+1}\leq \gamma'_{\rmI} \cdot M^{1-\varepsilon}
\]
holds true. Set $M_{\rmI}$ as the smallest positive integer which satisfies the inequality above. 

Then, we may appeal to the relative multidimensional Szemer\'{e}di Theorem (Theorem~\ref{thm:RMST}); indeed, by \eqref{eq:type1_weighted_density} and \eqref{eq:small}, if $M\geq M_{\rmI}$, then it applies to $B\subseteq \ideala(\bv,N)$. Therefore, there exists an $S$-constellation in $B$. Since $\Aff_{W,b}(B)\subseteq A'\subseteq A$, we can find an $S$-constellation in $A$.

Finally, we make an estimate of $\scrN_{S}(A)$. 
By Theorem~\ref{theorem=weighted-counting} and the estimate from above in the \compatiW~\eqref{en:measurebelow}, we \havethat\ 
\[
\frac{1}{N(2N+1)^n} \cdot \scrN_{S}(B) \cdot (D_2\log M)^{r+1}\geq \gamma'_{\rmI}.
\]
Hence, by setting $\gamma_{\rmI}=\gamma_{\rmI}(D_1,D_2,S,\delta,\bv)$ as $\gamma_\rmI \coloneqq 2^n D_2^{-(r+1)}\gamma'_{\rmI}$, we \obtainthat\ 

\[
\scrN_{S}(B)\geq \gamma_\rmI \cdot N^{n+1} (\log M)^{-(r+1)}\geq\gamma_{\rmI}   W^{-(n+1)}\cdot \frac{M^{n+1}}{(\log M)^{r+1}}.
\]
Since $\scrN_{S}(B)\leq \scrN_{S}(A')\leq \scrN_{S}(A)$, this ends our proof.
\end{proof}
\begin{remark}
Condition \eqref{en:measurebelow} in the definition of the {\compatiW } involved a bound from above $\lambda (\alpha ) \le D_2\cdot \log M$.
If we only need to prove the existence part of Theorem~\ref{theorem=package},
we may relax it to $o_{M\to\infty}(M^{\frac{1-\varepsilon}{r+1}})$, where $\varepsilon $ is the constant appearing in the assumption $W\leq M^\varepsilon $.
However, this makes the estimate of $\scrN_{S}(A)$ in Theorem~\ref{theorem=package} worse.
Examples in this paper all obey estimates of the form $\lambda (\alpha ) \le D_2\cdot \log M$.
%
\end{remark}
In the last part of this subsection, we will prove Theorem~\ref{theorem=primeconstellationsfiniteagain_full} below as an application of Theorem~\ref{theorem=package}. For the proof, we employ the following estimate from below.
\begin{proposition}\label{proposition=PK_kouri}
Let $\omom$ be an integral basis of $K$. Then there exist $C_{\PP_K,\rmI}(\omom)>0$ and a positive integer $M_{\PP_K,\rmI}(\omom)$ such that for every $M\geq M_{\PP_K,\rmI}(\omom)$, the inequality
\begin{equation}\label{eq:counting_inPPK}
\#(\PP_K\cap \OK(\omom,M))\geq C_{\PP_K,\rmI}(\omom) \cdot \frac{M^n}{\log M}
\end{equation}
holds. 
In particular,  for $\ideala=\OK$, the set $A=\PP_K$ fulfills the two conditions $\eqref{en:counting_infinite_X}$, $\eqref{en:logpseudo_infinite_X}$ of Corollary~$\ref{corollary=package_infinite}$.
\end{proposition}
\begin{proof}
First we will show \eqref{eq:counting_inPPK}. 
Take an NL-compatible $\OKt$-fundamental domain $\DD$; such a $\DD$ exists by Proposition~\ref{proposition=normrespectingfundamentaldomain}. 
Proposition~\ref{proposition=counting_below_DD} provides a constant $\tilde{C}=\tilde{C}(\omom,\DD)>0$ depending on  $\omom$ and $\DD$ such that for sufficiently large $M$, the inequality
\[
\#(\PP_K\cap \DD\cap \OK(\omom,M))\geq \tilde{C}\cdot \frac{M^n}{\log M}
\]
holds.
Then we obtain \eqref{eq:counting_inPPK} by observing that $\PP_K\cap \DD\subseteq \PP_K$.
The rest of the statement of Proposition~\ref{proposition=PK_kouri} now immediately follows from \eqref{eq:counting_inPPK} and Theorem~\ref{theorem=logpseudorandom}.
\end{proof}
Theorem~\ref{theorem=primeconstellationsfiniteagain_full} is a type~1 version of Theorem~\ref{mtheorem=primeconstellationsfinite}, that means, we do not go into the point whether an $S$-constellation admits an associate pair.
We will present the full proof of Theorem~\ref{mtheorem=primeconstellationsfinite} in the latter part of the present section; for this, we will develop an axiomatic framework for constellation theorems of type~2.
Note that, by Proposition~\ref{proposition=PK_kouri}, Corollary~\ref{corollary=package_infinite} implies the infinitary version of Theorem~\ref{theorem=primeconstellationsfiniteagain_full}.
For the reader's convenience, we sketch the proof of Theorem~\ref{theorem=primeconstellationsfiniteagain_full} itself. 
\begin{theorem}\label{theorem=primeconstellationsfiniteagain_full}
Let $K$ be a number field of degree $n$, and  $\omom$ an integral basis of $K$. Let  $S\subseteq\OK$ be a standard shape. Let $\delta>0$.
Then there exists a positive integer $M_{\mathrm{PES},\rmI}=M_{\mathrm{PES},\rmI}(\omom,\delta,S)$ depending only on $\omom$, $\delta$ and $S$ such that the following holds true:
if $M\geq M_{\mathrm{PES},\rmI}$ and if a subset $A$ of $\PP_K\cap\OK(\omom,M)$ satisfies
    \begin{align*}  %
        \#A\geq\delta\cdot\#(\PP_K \cap\OO_K(\omom,M)),
    \end{align*}
    then there exists an $S$-constellation in $A$.
    Furthermore, there exist a constant $\gamma=\gamma_{\mathrm{PES},\rmI}(\omom,\delta,S)>0$,  depending only on $\omom$, $\delta$ and $S$, such that in the setting above,
    \[
        \scrN_{S}(A)\geq\gamma\cdot \frac{M^{n+1}}{(\log M)^{\#S}}
    \]
    holds.
\end{theorem}
\begin{proof}
Recall that we have fixed $\chi$ to obtain 
$M_{\PP_K,\mathrm{S}\Psi_{\log}}(\rho,\omom, S)=M_{\PP_K,\mathrm{S}\Psi_{\log}}(\rho, \omom, \chi,S)$ 
in the proof of Theorem~\ref{theorem=logpseudorandom}.
Set $D\coloneqq\kappa\cdot(17(r+1)2^r\cdot c_{\chi})^{-1}$. Set $\rho\coloneqq\rho_{\rmI}(D,\omom,\delta\cdot C_{\PP_K,\rmI}(\omom),S)$. 
Define
\[
M_{\PP_K,\rmI}\coloneqq\max\{M_{\PP_K,\rmI}(\omom),M_{\PP_K,\mathrm{S}\Psi_{\log}}(\rho,\omom,S),M_{\rmI}(D,D,3/4,\omom,\delta\cdot C_{\PP_K,\rmI}(\omom),S)\}.
\]
Then, for $M\geq M_{\PP_K,\rmI}$ and for $A\subseteq \PP_K\cap\OK(\omom,M)$ with $\#A\geq\delta\cdot\#(\PP_K\cap\OK(\omom,M))$, Proposition~\ref{proposition=PK_kouri} implies that
\[
\#A\geq\delta\cdot C_{\PP_K,\rmI}(\omom)\cdot \frac{M^n}{\log M}.
\]
Hence by Theorem~\ref{theorem=logpseudorandom} and Lemma~\ref{lemma=logpseudorandom_subset}~\eqref{en:compatisubset}, $A$ satisfies the $(\rho,W,M,\omom,S)$-condition with parameters $(D,D,3/4)$, 
where $W=W_{\PP_K,\mathrm{S}\Psi_{\log}}(\rho,S)$.
Therefore, Theorem~\ref{theorem=package} applies to $A$, and we can find an $S$-constellation in $A$. 
Observing that $W$ is independent of $M$, we obtain the desired estimate of $\scrN_{S}(A)$.
\end{proof}

\begin{remark}
	In the seminal work \cite{Bloom-Sisask}, Bloom and Sisask have proved the existence of an absolute constant $c>0$ such that every $A\subseteq \NN$ with
	$$\limsup_{M \to \infty} \frac{\#(A \cap [M])}{M(\log M)^{-(1+c)}} >0$$ contains infinitely many $3$-APs. 
	Hence, in the case where $K=\QQ$ and $S=\{ -1,0,1\}$ is fixed, 
	condition~\eqref{en:logpseudo_infinite} in Theorem~\ref{theorem=package_infinite} is redundant.
\end{remark}

\subsection{Counting prime elements from above}\label{subsection=counting_above}In Subsection~\ref{subsection=Omega}, we will upgrade Theorem~\ref{theorem=primeconstellationsfiniteagain_full} to its type~2 version, Theorem~\ref{mtheorem=primeconstellationsfinite}. The key to this upgrading is counting prime elements from above. 
In this subsection, as a prototype of this counting, we will make an estimate of $\#(\PP_K\cap \OK(\omom,M))$ from above; see also Remark~\ref{remark=Mitsui}.

\begin{proposition}\label{proposition=prime_counting_above}
Let $\omom$ be an integral basis of $K$. 
Then there exist a constant $C_{\PP_K,\rmII}(\omom)>0$ and a positive integer $M_{\PP_K,\rmII}(\omom)$ such that for every $M\geq M_{\PP_K,\rmII}(\omom)$, the inequality
\[
\#(\PP_K\cap \OK(\omom,M))\leq C_{\PP_K,\rmII}(\omom) \cdot \frac{M^n}{\log M}
\]
holds true.
\end{proposition}
Recall the definition of the logarithmic integral $\Li$: %
\begin{equation}\label{eq:def-of-log-int}
  \Li(L)\coloneqq\int_2^L\frac{\rd t}{\log t}=(1+o_{L\to\infty}(1))\cdot\frac{L}{\log L}.
  \end{equation}
The following lemma will be employed in the proof of Proposition~\ref{proposition=prime_counting_above}. 
\begin{lemma}\label{lemma=integral}
For every $k\in \ZZ_{\geq 0}$, there exists a constant $C_{\Li}(k)>0$ such that  for every $L,\eta\in\RR_{\geq 2}$, the inequalities
\begin{equation}\label{eq:int_iterated_2}
\int_{2}^{L}\frac{1}{\log t}\left\{\log\left(\frac{\eta L}{t}\right)\right\}^k \rd t \leq C_{\Li}(k)\cdot(\log \eta)^k\cdot \frac{L}{\log L}
\end{equation}
and
\begin{equation}\label{eq:int_iterated_1}
\int_{2}^{L}\frac{1}{\log t}\left\{\log\left(\frac{L}{t}\right)\right\}^k \rd t \leq C_{\Li}(k)\cdot \frac{L}{\log L}
\end{equation}
hold true.
\end{lemma}
\begin{proof}
First, we will prove \eqref{eq:int_iterated_2} by induction on $k$. For 
$k=0$, this follows from \eqref{eq:def-of-log-int}.
Now we proceed to the induction step. We will reduce the assertion for $k\geq 1$ to that for $k-1$. By integration by parts, we \havethat\ 
\begin{align*}
\int_{2}^{L}\frac{1}{\log t}\left\{\log\left(\frac{\eta L}{t}\right)\right\}^k \rd t=\Li(L)\cdot(\log\eta)^k+k\cdot \int_{2}^{L}\frac{\Li(t)}{t}\left\{\log\left(\frac{\eta L}{t}\right)\right\}^{k-1}\rd t.
\end{align*}
By noting that $(\log \eta)^{k}\geq (\log 2)(\log \eta)^{k-1}$, we can make the desired reduction; recall also  \eqref{eq:def-of-log-int}. Therefore, \eqref{eq:int_iterated_2} holds.

Next we will prove  \eqref{eq:int_iterated_1}. Apply  \eqref{eq:int_iterated_2} with $\eta=e>2$, and obtain
\[
\int_{2}^{L}\frac{1}{\log t}\left\{\log\left(\frac{eL}{t}\right)\right\}^k \rd t \leq C_{\Li}(k)\cdot \frac{L}{\log L}.
\]
Since $\int_{2}^{L}\frac{1}{\log t}\left\{\log\left(\frac{eL}{t}\right)\right\}^k \rd t\geq \int_{2}^{L}\frac{1}{\log t}\left\{\log\left(\frac{L}{t}\right)\right\}^k \rd t$, we conclude \eqref{eq:int_iterated_1}.
\end{proof}

\begin{proof}[Proof of Proposition~$\ref{proposition=prime_counting_above}$]
Let $r_1$ and $r_2$ be the numbers, respectively, of real embeddings and of imaginary embeddings. Let $k\coloneqq r_1+r_2-1$.
Take the constant $\Xi=\Xi(\omom)>0$ as in Lemma~\ref{lemma=OKt_orbit}~\eqref{en:precisecounting} in such a way that $C'\leq \Xi$, where $C'=C'(\omom)$ is as in \eqref{NLC}. In particular, for every $\pi\in \PP_K\cap \OK(\omom,M)$, we have $\Nrm(\pi)\in [2, \Xi M^n]_{\RR}\cap \ZZ$.
For each $t\in \ZZ_{\geq 2}$, set
\[
p_{K}(t)\coloneqq \#\{\idealp\in |\Spec(\OK)|^{\PI}\colon \Nrm(\idealp)=t\},
\]
and for every $L\in \RR_{\geq 2}$, define
\[
P_{K}(L)\coloneqq \sum_{t\in [2,L]_{\RR}\cap\ZZ}p_K(t).
\]
By  Theorem~\ref{theorem=Chebotarev}~\eqref{Landau}, there exists $C_{\mathrm{Lan}}=C_{\mathrm{Lan}}(K)>0$ such that for every $L\in\RR_{\geq 2}$,
\begin{equation}\label{eq:Landau}
P_K(L)\leq C_{\mathrm{Lan}}\cdot \frac{L}{\log L}
\end{equation}
holds true; this rough estimate suffices for the present proof.
If $k=0$, in other words, if $\#(\OKt)<\infty$, then \eqref{eq:Landau} already provides the desired estimate. In what follows, we treat the case where $k\geq 1$.

Focus on the map $\PP_K\cap \OK(\omom,M)\ni \pi \mapsto \pi\OK\in |\Spec(\OK)|^{\PI}$. By considering the multiplicities of this map, we derive the following inequality from Lemma~\ref{lemma=OKt_orbit} \eqref{en:precisecounting}:
\[
\#(\PP_K\cap \OK(\omom,M))\leq \Xi \cdot\sum_{t\in [2,\Xi M^n]_{\RR}\cap \ZZ} p_K(t) \left\{\log\left(\frac{\Xi M^n}{t}\right)\right\}^{k}.
\]
By Abel's summation formula (see \cite[Theorem~421]{Hardy-Wright}), this implies that
\[
\#(\PP_K\cap \OK(\omom,M))\leq k\Xi\cdot \int_{2}^{\Xi M^n}\frac{P_K(t)}{t}\left\{\log\left(\frac{\Xi M^n}{t}\right)\right\}^{k-1}\rd t.
\]
Recall \eqref{eq:Landau}, and apply \eqref{eq:int_iterated_1} with $k$ replaced by $k-1$ and with $L=\Xi M^n$. Then, we obtain the desired estimate for the case of $k\geq 1$. It completes our proof.
\end{proof}
\begin{remark}\label{remark=Mitsui}
Proposition~\ref{proposition=PK_kouri} and
Proposition~\ref{proposition=prime_counting_above} assert that 
\[
0<\liminf_{M\to \infty}\frac{\#(\PP_K\cap \OK(\omom,M))}{M^n(\log M)^{-1}}\quad \textrm{and}\quad  \limsup_{M\to \infty}\frac{\#(\PP_K\cap \OK(\omom,M))}{M^n(\log M)^{-1}}<\infty
\]
for a fixed $\omom$.
Mitsui's generalized prime number theorem \cite[Corollary on p.35]{Mitsui56} implies that the limit infimum and limit supremum coincide and the common limit is equal to $ 2^n / \kappa $ (where $2^n$ reflects the count $\# \OK (\omom ,M)= (2M+1)^n$).
See \cite[Theorem~2, 3]{Kuperberg-Rodgers-RodittyGershon20} for the details of this deduction.%
However, 
we chose not to use Mitsui's theorem
partly because 
its proof is considerably more involved than 
the total effort needed to establish the inequalities above.

\end{remark}
\subsection{Reduction to the case of a fixed fundamental domain}\label{subsection=Omega}
In this subsection, we present our axiomatic framework for constellation theorems \emph{of type~$2$}, i.e., constellation theorems that ensure the existence of constellations \emph{without associate pairs}. 
Recall that two non-zero elements in an ideal $\ideala $ are said to be \emph{associate} if they are in the same orbit of the multiplication action $\OKt\curvearrowright \ideala\setminus \{0\}$. 

The technical key here is Theorem~\ref{theorem=fundamental_Omega}; the 
main results
are
Theorem~\ref{theorem=package_DD} (finitary version) and Theorem~\ref{theorem=package_infinite_DD} (infinitary version). 

With the aid of them, Theorem~\ref{mtheorem=primeconstellationsfinite} and Theorem~\ref{theorem=primeconstellationsdensesemiprecise} will be established.
\begin{lemma}\label{lemma=Omega}
Assume Setting~$\ref{setting=package}$.
Let $\Omega>0$ and $M\in \RR_{\geq 1}$.
Assume that $A\subseteq \ideala(\bv,M)\setminus \{0\}$ satisfies
\begin{equation}\label{eq:half}
\#(A \setminus \OK(\Omega M^n))\geq \frac{1}{2}\cdot \#A.
\end{equation}
Then there exist a constant $c_{\Omega,\bv}\in (0,1]_{\RR}$, depending only on $\Omega$, $\bv$, and a subset $A_0\subseteq A\setminus\OK(\Omega M^n)$ such that the following hold true:
\begin{enumerate}[$(1)$]
\item\label{en:c_Omega} $\#A_0\geq c_{\Omega,\bv}\cdot \#A$,
\item\label{en:hidouhan} the subset $A_0$ admits no associate pairs.
\end{enumerate}
\end{lemma}
\begin{proof}
Let $k\coloneqq r_1+r_2-1$ be the rank of $\OKbar$. Take $\Xi=\Xi(\bv)>0$ as in Corollary~\ref{corollary=OKt_orbit_ideal}~\eqref{en:precisecounting_ideal}.
Then, for every $\alpha \in \ideala(\bv,M)\setminus (\{0\}\cup\OK(\Omega M^n))$, there exist at most $\Xi\cdot \left\{\log\left(\frac{\Xi}{\Omega}\right)\right\}^k$ elements in $\ideala(\bv,M)\setminus \{0\}$ which are associate to $\alpha$. Set 
\[
c_{\Omega,\bv}\coloneqq \frac{1}{2\Xi}\cdot \left\{\log\left(\frac{\Xi}{\Omega}\right)\right\}^{-k}.
\]
Consider the quotient set of $A\setminus \OK(\Omega M^n)$ by the equivalence relation of being associate. Take  a complete system of representatives for this, and write $A_0$ for it. Then, under \eqref{eq:half}, we have \eqref{en:c_Omega} with the constant $c_{\Omega,\bv}$ as defined above. By construction of $A_0$, \eqref{en:hidouhan} also holds.
\end{proof}
We remark that Lemma~\ref{lemma=Omega}~\eqref{en:hidouhan} may be rephrased as follows: there exists a fundamental domain $\DD$ for the action $\OKt\curvearrowright \ideala\setminus \{0\}$ such that $A_0\subseteq A\cap \DD$ holds. Also \eqref{en:c_Omega} asserts that the counting in $A$ is comparable to that in $A_0$, as long as $\Omega$ is fixed. Hence, in order to reduce the general case of $A$ to that of fixing a fundamental domain, it suffices to find $\Omega>0$ for $A$ in a certain controlled way.

The following proposition provides a criterion on $A$ for which such a controlled constant $\Omega>0$ exists.

\begin{proposition}\label{proposition=get_Omega}
Assume Setting~$\ref{setting=package}$. Let $\delta>0$ and $\Delta>0$.
Then, there exist a constant $\Omega=\Omega_{\mathrm{red}}(\bv,\delta,\Delta)>0$ and a positive integer $M_{\mathrm{red}}=M_{\mathrm{red}}(\bv,\delta,\Delta)$ such that for all $M\geq M_{\mathrm{red}}$, the following holds true:
if $A\subseteq \ideala(\bv,M)\setminus \{0\}$ satisfies that
\begin{equation}\label{eq:ideal_Delta}
\textrm{for all $L\in \RR_{\geq 2}$},\quad \# \{\alpha \OK \in \Ideals_K\colon \alpha \in A\cap \OK(L)\}\leq \Delta \cdot \frac{L}{\log L},
\end{equation}
then we \havethat\ 
\begin{equation}\label{eq:delta_Omega}
\#(A\cap \OK(\Omega M^n))\leq \delta\cdot \frac{M^n}{\log M}.
\end{equation}
\end{proposition}
Note that in the \counting\ in \eqref{eq:ideal_Delta}, we count \emph{ideals} instead of the elements themselves.
\begin{proof}
Let $k\coloneqq r_1+r_2-1$ be the rank of $\OKbar$.
For $k=0$, we can set
\[
\Omega\coloneqq\frac{\delta n}{2\#\mu(K)\cdot \Delta}.
\]
Indeed, \eqref{eq:ideal_Delta} directly ensures \eqref{eq:delta_Omega}; note that for a sufficiently large $M$ depending on $\delta$, $\Delta$ and $K$, we have $\Omega M^n\geq M^{n/2}$.

Hence, in what follows, we focus on the case of $k\geq 1$; in particular, we have $n\geq 2$ in this case. 
We will generalize the argument of the proof of Proposition~\ref{proposition=prime_counting_above} in the following manner. 
For each $t\in \NN$, define $h_A(t)\coloneqq \#\{\alpha\OK \in \Ideals_K \colon \alpha \in A,\ \Nrm(\alpha)=t\}$, and for each $L\in \RR_{\geq 1}$, set $H_A(L)\coloneqq \sum\limits_{t\in [1,L]_{\RR}\cap \ZZ} h_A(t)$.
Then, assumption \eqref{eq:ideal_Delta} is equivalent to saying that
\begin{equation}\label{eq:Delta_L}
\textrm{for all $L\in \RR_{\geq 2}$},\quad H_A(L)\leq \Delta \cdot \frac{L}{\log L}.
\end{equation}

Take a constant $\Xi=\Xi(\bv)$ as in Corollary~\ref{corollary=OKt_orbit_ideal}~\eqref{en:precisecounting_ideal} with $\Xi\geq 2$. 
Take a parameter $\theta \in (0,1)_{\RR}$; we will indicate its value later.
Then, Corollary~\ref{corollary=OKt_orbit_ideal}~\eqref{en:precisecounting_ideal} implies that
\begin{align}\label{eq:inequality-h_A}
\#(A\cap \OK(\theta  M^n))&\leq \Xi\cdot \sum_{t\in [1,\theta  M^n ]_{\RR}\cap \ZZ} h_A(t) \left\{\log \left(\frac{\Xi M^n}{t}\right)\right\}^k
\notag\\
&\leq \Xi\cdot \left(\left\{\log(\Xi M^n)\right\}^k+\sum_{t\in [2,\theta  M^n]_{\RR}\cap \ZZ} h_A(t) \left\{\log \left(\frac{\Xi M^n}{t}\right)\right\}^k\right).
\end{align}
Indeed, note that $h_A(1)\leq 1$.
Since $\{ \log (\Xi M^n )\} ^k$, as a function of $M$, has a smaller order than $M^n/\log M$, we \havethat\  for a sufficiently large $M$ depending on $\bv$,
\begin{equation}\label{eq:the-initial-term}
\Xi\cdot\{\log(\Xi M^n)\}^k\leq\frac{1}{2}\delta\cdot\frac{M^n}{\log M}.
\end{equation}
In what follows, we will make estimates of $\sum\limits_{t\in [2,\theta  M^n]_{\RR}\cap \ZZ} h_A(t) \left\{\log \left(\frac{\Xi M^n}{t}\right)\right\}^k$ appearing in \eqref{eq:inequality-h_A}. By Abel's summation formula, we \havethat\ 
\begin{align*}
&\sum_{t\in[2,\theta M^n]_{\RR}\cap\ZZ}h_A(t) \left\{\log\left(\frac{\Xi M^n}{t}\right)\right\}^k\\
&\leq H_A(\theta M^n)\left\{\log\left(\frac{\Xi}{\theta}\right)\right\}^k+k\cdot\int_2^{\theta M^n}\frac{H_A(t)}{t}\left\{\log\left(\frac{\Xi M^n}{t}\right)\right\}^{k-1}\rd t.
\end{align*}
By assumption \eqref{eq:Delta_L}, the value $\sum\limits_{t\in [2,\theta  M^n]_{\RR}\cap \ZZ} h_A(t) \left\{\log \left(\frac{\Xi M^n}{t}\right)\right\}^k$ does not exceed the following:
\[
\Delta\cdot\frac{\theta M^n}{\log (\theta M^n)}\left\{\log\left(\frac{\Xi}{\theta}\right)\right\}^k+k\Delta\cdot\int _2^{\theta M^n}\frac{1}{\log t}\left\{\log\left(\frac{\Xi M^n}t\right)\right\}^{k-1}\rd t.
\]
The second term is estimated as follows: apply \eqref{eq:int_iterated_2} with $k$ replaced by $k-1$ and with $\eta=\Xi/\theta$ and $L=\theta M^n$. Then, with the constant $C=kC_{\Li}(k-1)$, we \havethat\ 
\[
k\Delta\cdot\int _2^{\theta M^n}\frac{1}{\log t}\left\{\log\left(\frac{\Xi M^n}t\right)\right\}^{k-1}\rd t\leq C\Delta\cdot\frac{\theta M^n}{\log (\theta M^n)}\left\{\log\left(\frac{\Xi}{\theta}\right)\right\}^{k-1}.
\]
Hence, if $M\geq 1/\theta$, then we \obtainthat\ 
\begin{align*}
&\sum_{t\in[2,\theta M^n]_{\RR}\cap\ZZ}h_A(t) \left\{\log\left(\frac{\Xi M^n}{t}\right)\right\}^k\\
&\leq\frac{\theta\Delta}{n+\frac{\log\theta}{\log M}}\left(\left\{\log\left(\frac{\Xi}{\theta}\right)\right\}^k+C\left\{\log\left(\frac{\Xi}{\theta}\right)\right\}^{k-1}\right)\cdot\frac{M^n}{\log M}\\
&\leq\frac{\theta\Delta}{n-1}\left(\left\{\log\left(\frac{\Xi}{\theta}\right)\right\}^k+C\left\{\log\left(\frac{\Xi}{\theta}\right)\right\}^{k-1}\right)\cdot\frac{M^n}{\log M}.
\end{align*}
There exists a real number $\theta\in(0,1]_{\RR}$, depending on $\delta$, $\Delta$ and $\bv$, such that
\[
\frac{\theta\Delta}{n-1}\left(\left\{\log\left(\frac{\Xi}{\theta}\right)\right\}^k+C\left\{\log\left(\frac{\Xi}{\theta}\right)\right\}^{k-1}\right)\leq\frac{\delta}{2\Xi};
\]
we take such a small $\theta$, and write $\Omega$ for it.
Take  $M_{\mathrm{red}}\in \NN$ such that $M_{\mathrm{red}}\geq1/\Omega $ and that for all $M\geq M_{\mathrm{red}}$, \eqref{eq:the-initial-term} holds. Then, by \eqref{eq:inequality-h_A} and the arguments after this, we conclude that for every $M\geq M_{\mathrm{red}}$,
\[
\frac{\#(A\cap \OK(\Omega M^n))}{M^n(\log M)^{-1}}\leq\frac{1}{2}\delta+\frac{1}{2}\delta=\delta;
\]
this is the desired estimate. It completes our proof.
\end{proof}
Proposition~\ref{proposition=get_Omega} together with Lemma~\ref{lemma=Omega}  derives the following theorem.

\begin{theorem}[Reduction to the case of fixing a fundamental domain]\label{theorem=fundamental_Omega}
Assume Setting~$\ref{setting=package}$.
Let $\delta>0$ and $\Delta>0$.
Then there exist $\delta'=\delta'_{\mathrm{red}}(\bv,\delta,\Delta)>0$, $\Omega'=\Omega'_{\mathrm{red}}(\bv,\delta,\Delta)>0$ and $M'_{\mathrm{red}}=M'_{\mathrm{red}}(\bv,\delta,\Delta)\in\NN$ such that for every $M\geq M'_{\mathrm{red}}$, the following holds true.
Let $A\subseteq \ideala(\bv,M)\setminus \{0\}$ be a set that satisfies \eqref{eq:delta_log} and \eqref{eq:ideal_Delta}.
Then, there exist a fundamental domain $\DD$ for $\OKt\curvearrowright(\ideala\setminus\{0\})$ and a subset $A_0\subseteq\DD$ such that
\begin{equation}\label{eq:AcapDD}
\#A_0\geq \delta'\cdot \frac{M^n}{\log M}
\end{equation}
and
\begin{equation}\label{eq:Omega_norm}
A_0\subseteq A \setminus\OK(\Omega'M^n)
\end{equation}
hold true.
\end{theorem}
\begin{proof}
Take $\Omega'=\Omega'_{\mathrm{red}}(\bv,\delta,\Delta)\coloneqq \Omega_{\mathrm{red}}(\bv,\delta/2,\Delta)$ and $M'_{\mathrm{red}}(\bv,\delta,\Delta)\coloneqq M_{\mathrm{red}}(\bv,\delta/2,\Delta)$.
Let $A\subseteq \ideala(\bv,M)\setminus \{0\}$ be the set in the assertion of the theorem. 
Then, by assumption~\eqref{eq:ideal_Delta}, Proposition~\ref{proposition=get_Omega} implies that
\[
\#(A\cap \OK(\Omega'M^n))\leq \frac{\delta}{2}\cdot \frac{M^n}{\log M}.
\]
By combining this with assumption~\eqref{eq:delta_log}, we \havethat\ 
\[
\#(A\setminus \OK(\Omega'M^n))\geq \frac{1}{2}\cdot \# A.
\]
Therefore, we can apply Lemma~\ref{lemma=Omega}. This provides a subset $A_0\subseteq A\setminus\OK(\Omega'M^n)$ with $\#A_0\geq c_{\Omega}\cdot \#A$ such that $A_0$ admits no associate pairs.
This $A_0$ satisfies \eqref{eq:Omega_norm}.
Finally, to ensure \eqref{eq:AcapDD}, take an arbitrary fundamental domain $\DD$ for the $\OKt$-action in such a way that $A_0\subseteq \DD$ holds. Then, set $\delta'\coloneqq c_{\Omega',\bv}\cdot\delta$.
\end{proof}
Theorem~\ref{theorem=fundamental_Omega} enables us to upgrade axiomatized constellation theorems of type~1 to those of type~2 in the following manner. 
Recall here that we have three axiomatized constellation theorems of type~1: 
Theorem~\ref{theorem=package}, Theorem~\ref{theorem=package_infinite} and Corollary~\ref{corollary=package_infinite}.
\begin{theorem}\label{theorem=package_DD}
Assume Setting~$\ref{setting=package}$.
Let $S\subseteq\ideala$ be a standard shape.
Let $\delta>0$, $\Delta>0$, $D_1,D_2>0$, and  $\varepsilon\in (0,1)_{\RR}$.
Then, there exist a positive real number $\rho=\rho_{\rmII}(D_1,\bv,\delta,\Delta,S)>0$ and a positive integer $M_{\rmII}=M_{\rmII}(D_1,D_2,\varepsilon,\bv,\delta,\Delta,S)\in\NN$ such that the following holds true.
Assume that $M\geq M_{\rmII}$ and a set $A\subseteq  \ideala(\bv,M)\setminus\{0\}$ fulfill the following three conditions:
\begin{enumerate}[$(i)$]
 \item\label{en:counting_DD}the inequality
\[
\#A\geq \delta \cdot \frac{M^n}{\log M}
\]
holds, 
 \item\label{en:counting_ideal_DD}for every $L\in \RR_{\geq 2}$, 
\[
\#\{\alpha\OK\in \Ideals_K\colon \alpha \in A\cap \OK(L)\}\leq \Delta\cdot \frac{L}{\log L}
\]
holds,
 \item\label{en:logpseudorandom_DD} $A$ satisfies the \compati\ with parameters $(D_1,D_2,\varepsilon)$.
\end{enumerate}
Then, there exists an $S$-constellation without associate pairs in $A$. 
Furthermore, there exists $\gamma=\gamma_{\rmII}(D_1,D_2,\bv,\delta,\Delta,S)>0$ such that 
\[
\scrN_S^{\sharp}(A)\geq \gamma W^{-(n+1)}\cdot \frac{M^{n+1}}{(\log M)^{\# S}}
\]
holds true.
Here, $W$ is an integer appearing in the \compati, which comes from condition~\eqref{en:logpseudorandom_DD}.
\end{theorem}
\begin{theorem}\label{theorem=package_infinite_DD}
Assume Setting~$\ref{setting=package}$. Assume that $A\subseteq \ideala\setminus\{0\}$ fulfills the following three conditions:
\begin{enumerate}[$(i)$]
 \item\label{en:counting_infinite_DD} the inequality
\[
\limsup_{M\to \infty}\frac{\#(A\cap \ideala(\bv,M))}{M^n(\log M)^{-1}}>0
\]
holds, 
\item\label{en:counting_ideal_infinite_DD}there exists $\Delta>0$ such that for every $L\in \RR_{\geq 2}$,
\[
\#\{\alpha\OK\in \Ideals_K\colon \alpha \in A\cap \OK(L)\}\leq \Delta\cdot \frac{L}{\log L}
\]
holds,
 \item\label{en:logpseudo_infinite_DD}$A\in \logpseu$.
\end{enumerate}
Then, for every finite subset $S\subseteq \ideala$, there exists an $S$-constellation with no associate pairs in $A$.
\end{theorem}
\begin{corollary}\label{corollary=package_infinite_DD}
Assume Setting~$\ref{setting=package}$.
Assume that $A\subseteq \ideala\setminus\{0\}$ fulfills the following three conditions:
\begin{enumerate}[$(i)$]
 \item\label{en:counting_infinite_DD_above_cor} the inequality
\[
\liminf_{M\to \infty}\frac{\#(A\cap \ideala(\bv,M))}{M^n(\log M)^{-1}}>0
\]
holds, 
\item\label{en:counting_ideal_infinite_DD_cor} condition~\eqref{en:counting_ideal_infinite_DD} in Theorem~$\ref{theorem=package_infinite_DD}$ is satisfied,
 \item\label{en:logpseudo_infinite_DD_cor}$A\in \logpseu$.
\end{enumerate}
Then, for every $A'\subseteq A$ with $\overline{d}_{A,\bv}(A')>0$, the following holds true: for every finite subset $S\subseteq \ideala$, there exists an $S$-constellation with no associate pairs in $A'$.
\end{corollary}
\begin{remark}\label{remark=Delta_detekuru}
In the statement of Theorem~\ref{theorem=package_infinite_DD}, \eqref{en:counting_ideal_infinite_DD} implies that the limit suprimum in \eqref{en:counting_infinite_DD} is finite.
To verify this, run an argument similar to that of the proof of Proposition~\ref{proposition=prime_counting_above}; examine also the proof of Proposition~\ref{proposition=get_Omega}.
\end{remark}
\begin{proof}[Proofs of Theorem~$\ref{theorem=package_DD}$, Theorem~$\ref{theorem=package_infinite_DD}$ and Corollary~$\ref{corollary=package_infinite_DD}$]
First, we will prove Theorem~\ref{theorem=package_DD}.
Set
\begin{align*}
\rho_{\rmII}(D_1,\bv,\delta,\Delta,S)&\coloneqq\rho_{\rmI}(D_1,\bv,\delta',S),\\
M_{\rmII}(D_1,D_2,\varepsilon,\bv,\delta,\Delta,S)&\coloneqq M_{\rmI}(D_1,D_2,\varepsilon,\bv,\delta',S),\\
\gamma_{\rmII}(D_1,D_2,\bv,\delta,\Delta,S)&\coloneqq\gamma_{\rmI}(D_1,D_2,\bv,\delta',S);
\end{align*}
here $\delta'=\delta'_{\mathrm{red}}(\bv,\delta,\Delta)$ is the one as in Theorem~\ref{theorem=fundamental_Omega}.
Consider a set $A$ that fulfills the three conditions of Theorem~\ref{theorem=package_DD}. Take an integer $W$ appearing in an \compati, which comes from condition~\eqref{en:logpseudorandom_DD}.
By assumptions~\eqref{en:counting_DD} and \eqref{en:counting_ideal_DD}, Theorem~\ref{theorem=fundamental_Omega} applies. Hence, there exist a fundamental domain $\DD$ for the $\OKt$-action and a set $A_0\subseteq\DD$ such that \eqref{eq:AcapDD} holds.
Then, by Lemma~\ref{lemma=logpseudorandom_subset}~\eqref{en:compatisubset}, $A_0$ satisfies the \compatiW\  with parameters $(D_1,D_2,\varepsilon)$.
Therefore, we can apply Theorem~\ref{theorem=package} to this $A_0$ and $\delta'$, thus proving Theorem~\ref{theorem=package_DD}. Here, recall that since $A_0\subseteq\DD$, we have $\scrN_S^{\sharp}(A_0)=\scrN_S(A_0)$.

In a similar manner to the deduction of Theorem~\ref{theorem=package_infinite} from Theorem~\ref{theorem=package}, we can deduce Theorem~\ref{theorem=package_infinite_DD} from Theorem~\ref{theorem=package_DD}.
We can also derive Corollary~\ref{corollary=package_infinite_DD} from Theorem~\ref{theorem=package_infinite_DD} in a way similar to the deduction of Corollary~\ref{corollary=package_infinite} from Theorem~\ref{theorem=package_infinite}. Here, observe that assumption~\eqref{en:counting_ideal_infinite_DD_cor}  in Corollary~\ref{corollary=package_infinite_DD} on $A$ inherits to subsets.
\end{proof}
\begin{proposition}\label{proposition=Masani_Landau}
For a number field  $K$, $\PP_K\subseteq \OKnz$ satisfies the three conditions in Corollary~$\ref{corollary=package_infinite_DD}$.
\end{proposition}
\begin{proof}
We have already proved in Proposition~\ref{proposition=PK_kouri} that $\PP_K$ satisfies assumptions~\eqref{en:counting_infinite_DD_above_cor} and \eqref{en:logpseudo_infinite_DD_cor}.
By Theorem~\ref{theorem=Chebotarev}~\eqref{Landau}, $\PP_K$ satisfies assumption~\eqref{en:counting_ideal_infinite_DD_cor} as well.
\end{proof}
Now we are ready to complete the proofs of Theorem~\ref{mtheorem=primeconstellationsfinite} and Theorem~\ref{theorem=primeconstellationsdensesemiprecise}.
\begin{proof}[Proofs of Theorem~$\ref{mtheorem=primeconstellationsfinite}$ and Theorem~$\ref{theorem=primeconstellationsdensesemiprecise}$]
Theorem~$\ref{theorem=primeconstellationsdensesemiprecise}$ immediately follows from Proposition~\ref{proposition=Masani_Landau} and Corollary~\ref{corollary=package_infinite_DD}.
Recall that we have shown Theorem~\ref{theorem=primeconstellationsfiniteagain_full} from Proposition~\ref{proposition=PK_kouri}, Theorem~\ref{theorem=logpseudorandom} and Theorem~\ref{theorem=package}. 
In a manner similar to this, we can deduce Theorem~\ref{mtheorem=primeconstellationsfinite} from Proposition~\ref{proposition=Masani_Landau} by replacing Theorem~\ref{theorem=package} with Theorem~\ref{theorem=package_DD}. Here, for the estimate of $\scrN_S^{\sharp}(A)$, recall that we can take $W$ independently of $M$ by Theorem~\ref{theorem=logpseudorandom}.
\end{proof}
In the last part of this subsection, we will prove the following theorem; it may be regarded as the infinitary version of Theorem~\ref{theorem=fundamental_Omega}.
\begin{theorem}[Reduction to the case of fixing a fundamental domain, infinitary version]\label{theorem=fundamental_Omega_infinite}
Assume Setting~$\ref{setting=package}$. Assume that $A\subseteq \ideala\setminus \{0\}$ satisfies condition~\eqref{en:counting_ideal_infinite_DD} of Theorem~$\ref{theorem=package_infinite_DD}$.
Then the following hold true.
\begin{enumerate}[$(1)$]
\item\label{en:limsup} Assume that $A$ satisfies condition~\eqref{en:counting_infinite_DD} of Theorem~$\ref{theorem=package_infinite_DD}$. Then there exists an NL-compatible fundamental domain $\DD=\DD(A,\bv)$ for the action $\OKt\curvearrowright(\ideala\setminus\{0\})$ such that 
\begin{equation}\label{eq:limsup_X}
\limsup_{M\to \infty}\frac{\#(A\cap \DD\cap \ideala(\bv,M))}{M^n(\log M)^{-1}}>0
\end{equation}
holds true.
\item\label{en:liminf}
Assume that $A$ satisfies condition~\eqref{en:counting_infinite_DD_above_cor} in Corollary~$\ref{corollary=package_infinite_DD}$. Then there exists an NL-compatible fundamental domain $\DD'=\DD'(A,\bv)$ for the action $\OKt\curvearrowright(\ideala\setminus\{0\})$ such that 
\begin{equation}\label{eq:liminf_X}
\liminf_{M\to \infty}\frac{\#(A\cap \DD'\cap \ideala(\bv,M))}{M^n(\log M)^{-1}}>0
\end{equation}
holds true.
\end{enumerate}
\end{theorem}
Note that the notion of the NL-compatibility is defined for subsets of $\OK$; in particular, it is defined for a subset of $\ideala$. Also recall that for a fixed integral basis $\omom$ of $K$, the restriction of $\|\cdot\|_{\infty,\omom}$ on $\ideala$ is bi-Lipschitz equivalent to $\|\cdot\|_{\infty,\bv}$.

\begin{proof}
There exists a constant $C'>0$ depending only on $\bv$ such that for every $\alpha\in \ideala\setminus \{0\}$, $\Nrm(\alpha)\leq C'\|\alpha\|_{\infty,\bv}^n$ holds. Indeed, this can be verified in a similar manner to the proof of Lemma~\ref{lemma=NLCreversed}. We fix such $C'>0$ in the present proof.

First, we will prove \eqref{en:limsup}.
Since condition~\eqref{en:counting_infinite_DD} in Theorem~$\ref{theorem=package_infinite_DD}$ is fulfilled, there exist a strictly increasing positive real sequence $(M_l)_{l\in \NN}$ with $\lim\limits_{l\to \infty}M_l=\infty$ and $\delta>0$ such that for every $l\in \NN$, the inequality
\[
\frac{\#(A\cap  \ideala(\bv,M_l))}{M_l^n(\log M_l)^{-1}}\geq\delta
\]
holds.
Take constants $\delta'=\delta'_{\mathrm{red}}(\bv,\delta,\Delta)>0$, $\Omega'=\Omega'_{\mathrm{red}}(\bv,\delta,\Delta)>0$ and $M'_{\mathrm{red}}=M'_{\mathrm{red}}(\bv,\delta,\Delta)\in \RR_{>0}$ as in Theorem~\ref{theorem=fundamental_Omega} associated with $\delta$ and with $\Delta$ appearing in condition~\eqref{en:counting_ideal_infinite_DD} of Theorem~$\ref{theorem=package_infinite_DD}$. By passing to a subsequence of $(M_l)_{l\in \NN}$ if necessary, we may assume that $M_1\geq M_{\mathrm{red}}$ and that for every $l\in \NN$, the inequality
\begin{equation}\label{eq:C'Omega}
M_{l+1}\geq \left(\frac{C'}{\Omega'}\right)^{\frac{1}{n}}M_l
\end{equation}
holds.
Note that \eqref{eq:C'Omega} implies that 
\begin{equation}\label{eq:C'Omega_seq}
\Omega'M_1^n<C'M_1^n\leq\Omega'M_2^n\leq C'M_2^n\leq\Omega'M_3^n\leq C'M_3^n\leq\Omega'M_4^n \leq\cdots. 
\end{equation}
Let $l\in \NN$. Apply Theorem~\ref{theorem=fundamental_Omega} to $A\cap\ideala(\bv,M_l)$; then we can find $A_0^{(l)}\subseteq A\cap \ideala(\bv,M_l)$ which admits no associate pairs such that
\begin{equation}\label{eq:A_0^l}
\#A_0^{(l)}\geq \delta'\cdot \frac{M_l^n}{\log M_l}
\end{equation}
and that for every $\alpha\in A_0^{(l)}$, the inequality
\begin{equation}\label{eq:A_0^lnorm}
\Omega'M_l^n<\Nrm(\alpha)\leq C' M_l^n
\end{equation}
holds.
Thus, we obtain a family of sets $(A_0^{(l)})_{l\in \NN}$. Set
\begin{equation}\label{eq:A_0_infinite}
A_0\coloneqq \bigsqcup_{l\in \NN}A_0^{(l)}.
\end{equation}
Here, we can show that the union in the right-hand side of \eqref{eq:A_0_infinite} is indeed a disjoint union in the following manner: by \eqref{eq:A_0^lnorm} and \eqref{eq:C'Omega_seq}, for distinct $l_1,l_2\in \NN$, the intersection of $\Nrm(A_0^{(l_1)})$ and $\Nrm(A_0^{(l_2)})$ is empty. This argument, furthermore, implies that $A_0$ admits no associate pairs.
By \eqref{eq:A_0^lnorm}, $A_0$ is NL-compatible; recall the remark after the statement of Theorem~\ref{theorem=fundamental_Omega_infinite}.
Take an NL-compatible $\OKt$-fundamental domain $\DD_0\subseteq \OKnz$ with the aid of Proposition~\ref{proposition=normrespectingfundamentaldomain}. Set $\DD_1\coloneqq (\DD_0\cap\ideala)\setminus(\OKt\cdot A_0)$ and $\DD\coloneqq A_0\sqcup \DD_1$.
Then since $A_0$ does not admit an associate pair, by construction of $\DD$, this $\DD$ is a fundamental domain for the action $\OKt\curvearrowright \ideala\setminus \{0\}$. Moreover, $\DD$ is NL-compatible: indeed, it is the union of  two NL-compatible sets $A_0$ and $\DD_1$.
By \eqref{eq:A_0^l} and  \eqref{eq:A_0_infinite}, we have \eqref{eq:limsup_X}. Therefore, we have proved \eqref{en:limsup}.

Secondly, we will show \eqref{en:liminf}.
By condition~\eqref{en:counting_infinite_DD_above_cor} in Corollary~$\ref{corollary=package_infinite_DD}$, there exists $\delta>0$ such that for every sufficiently large $M$, the inequality
\begin{equation}\label{eq:liminf_delta}
\frac{\#(A\cap \ideala(\bv,M))}{M^n(\log M)^{-1}}\geq \delta
\end{equation}
holds.
Take constants $\delta'=\delta'_{\mathrm{red}}(\bv,\delta,\Delta)>0$, $\Omega'=\Omega'_{\mathrm{red}}(\bv,\delta,\Delta)>0$ and $M'_0=M'_{\mathrm{red}}(\bv,\delta,\Delta)\in \RR_{>0}$ as in Theorem~\ref{theorem=fundamental_Omega} associated with this $\delta$ and $\Delta$ appearing in condition~\eqref{en:counting_ideal_infinite_DD} of Theorem~$\ref{theorem=package_infinite_DD}$. By replacing $M'_0$ with a bigger number if necessary, we may assume that for every $M\geq M'_0$, \eqref{eq:liminf_delta} holds.
Now, define a sequence $(M_l)_{l\in \NN}$ inductively as follows: set  $M_1\coloneqq M'_0$ and for each $l\in \NN$, set
\begin{equation}\label{eq:M_l_liminf}
M_{l+1}\coloneqq\left(\frac{C'}{\Omega'}\right)^{\frac{1}{n}}M_l.
\end{equation}
Note that \eqref{eq:C'Omega_seq} holds by construction.

From this sequence $(M_l)_{l\in \NN}$, construct $(A_0^{(l)})_{l\in \NN}$ and an NL-compatible fundamental domain $\DD'$ for $\OKt\curvearrowright \ideala\setminus \{0\}$ in the same manner as in the proof of \eqref{en:limsup}. What remains to verify is that $\DD'$ fulfills \eqref{eq:liminf_X}.
Take an arbitrary $M\geq M'_0$.
Then by \eqref{eq:M_l_liminf}, there exists $l\in \NN$ such that
\begin{equation}\label{eq:MtoM_l}
M_l\leq M< \left(\frac{C'}{\Omega'}\right)^{\frac{1}{n}} M_{l}
\end{equation}
holds. For this $l\in \NN$, we note that $A\cap \DD'\cap \ideala(\bv,M)\supseteq A_0^{(l)}$. Hence by  \eqref{eq:liminf_delta}, the definition of $\delta'$ and \eqref{eq:MtoM_l}, we conclude that
\[
\frac{\#(A\cap \DD'\cap \ideala(\bv,M))}{M^n(\log M)^{-1}}\geq \left(\frac{M_l}{M}\right)^n\cdot \frac{\log M}{\log M_l} \cdot \delta' \geq \frac{\Omega'}{C'}\cdot \delta',
\]
thus proving \eqref{eq:liminf_X}. It completes our proof.
\end{proof}
Here we state the reduction theorem again, which was mentioned as Theorem~\ref{theorem=no_DD_to_with_DD} in Section~\ref{section=organization}.
\begin{corollary}[{Theorem~$\ref{theorem=no_DD_to_with_DD}$, restated}]\label{corollary=no_DD_to_with_DD_re}
Let $K$ be a number field and  $\omom$ an integral basis.
Assume that $A\subseteq \PP_K$ satisfies $\overline{d}_{\PP_K,\omom}(A)>0$. Then there exists  an NL-compatible $\OKt$-fundamental domain $\DD=\DD(A,\omom)$ such that
\[
\overline{d}_{\PP_K\cap \DD,\omom}(A\cap \DD)>0
\]
holds.
\end{corollary}
\begin{proof}
Recall that we have estimates of the number of prime elements, both from above and below, from Proposition~\ref{proposition=counting_below_DD} and Proposition~\ref{proposition=prime_counting_above}. Then, apply Theorem~\ref{theorem=fundamental_Omega_infinite}~\eqref{en:limsup} to $A$.
\end{proof}
\begin{remark}\label{remark=toruno_toranaino}
  By Propositions~\ref{proposition=counting_below_DD}, \ref{proposition=prime_counting_above} and
  Corollary~\ref{corollary=no_DD_to_with_DD_re},
  constellation theorems in a given NL-compatible fundamental domain
  and ones without mention of NL-compatible domains are in fact equivalent.
%

\end{remark}

%% file: chapter9.tex
\section{Szemer\'{e}di-type theorems for short intervals}
\label{section=slidetrick}

The main goal of this section is to prove
the finitary version of the
Szemer\'{e}di-type theorem for short intervals in prime elements of number fields
(Theorem~\ref{mtheorem=TaoZieglergeneral}). 
For the rational field $\QQ$, 
a stronger form of the Green--Tao theorem for short intervals
is proved; see Theorem~\ref{theorem=primenumbertheoremshort} and Theorem~\ref{theorem=BanachGreenTao0.525}.

In Subsection~\ref{subsection=BanachGreenTao},
we first prove
a stronger form of the Green--Tao theorem for short intervals,
and later in Subsection~\ref{subsection=package_shortinterval_proof},
Theorem~\ref{mtheorem=TaoZieglergeneral} will be established.
A notable difference between the proof of 
Theorem~\ref{mtheorem=TaoZieglergeneral} and that of
Theorem~\ref{theorem=primenumbertheoremshort}, which arises
when $[K:\QQ]\geq2$, is overcome by sophisticated use of
the pigeonhole principle. This technique is commonly used in
combinatorics; we call this the \emph{slide trick}, and present
it in Subsection~\ref{subsection=slidetrick}. In Subsection~\ref{subsection=close_norm}, we prove Theorem~\ref{theorem=package_infinite_a_close} as an application of our constellation theorems for short intervals; Theorem~\ref{theorem=package_infinite_a_close} plays a key role in Section~\ref{section=quadraticform}.

\subsection{Statements of the theorems for short intervals}
\label{subsection=statement}
\begin{setting}\label{setting:short_theorems}
Let $K$ be a number field of degree $n$.
Let $\omom$ be an integral basis of $K$.
\end{setting}
Recall from Definition~\ref{definition=shortinterval} that
the $\lmugen$-interval $\OO_K(\omom,x,M)$ for $x\in\OK$
and $M\in\RR_{\geq 0}$ is defined as
\[
\OO_K(\omom,x,M)\coloneqq\{\alpha\in\OK : \|\alpha-x\|_{\infty,\omom}\leq M\}.
\]
We say this $\OO_K(\omom,x,M)$ is a `short interval' if
$M$ is sufficiently small compared to $\|x\|_{\infty,\omom}$.
Our concern in this section is to prove the existence of a
constellation in short intervals in the sense above. 
Informally speaking, this amounts to proving 
the following statements \eqref{en:weak} and \eqref{en:strong}. 
Note that, it is possible to formulate Szemer\'{e}di-type theorems
after suitable modifications. We state, however, 
as constellation theorems for the set $\PP_K$ for brevity.
Let $f\colon \RR_{>0}\to \RR_{\geq 0}$ be a monotonically non-decreasing
function which diverges to infinity slower than the identity
function $t\mapsto t$.
\begin{enumerate}[(I)]
\item\label{en:weak}
For every finite set $S\subseteq\OK$, there exists a 
sequence $(y_l)_{l\in\NN}$ in $\OK$ with
$\|y_l\|_{\infty,\omom}\to \infty$ such that, for each $l\in\NN$,
there exists an $S$-constellation in 
$\PP_K\cap \OO_K(\omom,y_l,f(\|y_l\|_{\infty,\omom}))$.
\item\label{en:strong}
For every finite set $S\subseteq\OK$, there exists $M\in\RR_{>0}$ such that,
for each $x\in \OO_K$ with $\|x\|_{\infty,\omom}\geq M$, 
there exists an $S$-constellation in 
$\PP_K\cap \OO_K(\omom,x,f(\|x\|_{\infty,\omom}))$.
\end{enumerate}
We refer to \eqref{en:strong} the \emph{strong form} for short intervals,
because the assertion \eqref{en:strong} implies \eqref{en:weak}.

For the case $K=\QQ$, we prove the strong form \eqref{en:strong}, as follows.

\begin{theorem}[The Green--Tao theorem for short intervals: strong version]\label{theorem=BanachGreenTao}
Let $a\in (0,1)_{\RR}$, and assume that there exists
$\delta>0$ such that, for a sufficiently large $M>0$,
\begin{equation}\label{eq:primecounting}
\#(\PP\cap [M,M+M^a]_{\RR})\geq \delta \cdot \frac{M^a}{\log M}
\end{equation}
holds. Then we have the following, where $k$ is an integer 
with $k\geq3$.
\begin{enumerate}[$(1)$]
\item\label{en:BGT1}
There exists $M_{\mathrm{GTSI}}=M_{\mathrm{GTSI}}(a,\delta,k)\in\NN$ 
depending only on $a,\delta$ and $k$ such that, for every 
$M\in\RR$ with $M\geq M_{\mathrm{GTSI}}$,
$\PP\cap [M,M+M^{a}]_{\RR}$ contains an arithmetic progression of length $k$.
\item\label{en:BGT2}
Moreover, there exists $\gamma=\gamma_{\mathrm{GTSI}}(a,\delta,k)>0$, depending only on $a,\delta$ and $k$, such that for every 
$M\in\RR$ with $M\geq M_{\mathrm{GTSI}}$,
\[
\mathscr{N}_{k}(\PP\cap [M,M+M^{a}]_{\RR})\geq\gamma\cdot \frac{M^{2a}}{(\log M)^{k}}
\]
holds. Here, for a finite subset $X\subseteq \ZZ$, $\mathscr{N}_k(X)$
denotes the number of arithmetic progressions $($as sets$)$ of length $k$
in $X$.
\end{enumerate}
\end{theorem}

We prove Theorem~\ref{theorem=BanachGreenTao} in  Subsection~\ref{subsection=BanachGreenTao}.
As for the possible values of $a$ in
\eqref{eq:primecounting}, that is, 
`prime number theorem in short intervals,' a number of results are known; see \cite{Baker-Harman-Pintz2001} and references therein.
Among those, we mention the celebrated result of Baker--Harman--Pintz.

\begin{theorem}[{\cite[p.562]{Baker-Harman-Pintz2001}}]\label{theorem=primenumbertheoremshort}
For a sufficiently large real number $M$,
we have
\[
\#(\PP \cap [M,M+M^{0.525}]_{\RR})\geq \frac{9}{100}\cdot \frac{M^{0.525}}{\log M}
.\]
\end{theorem}

This theorem implies the following unconditional result:

\begin{theorem}\label{theorem=BanachGreenTao0.525}
Assertions \eqref{en:BGT1} and \eqref{en:BGT2} of
Theorem~$\ref{theorem=BanachGreenTao}$ hold for
$a=0.525$ and $\delta=0.09$.
\end{theorem}

The strong form for short version, that is,
the assertion of Theorem~\ref{theorem=BanachGreenTao} 
makes sense only for real numbers $a$
satisfying \eqref{eq:primecounting}. This requires
deep results in the theory of distribution of primes.
Under the Riemann hypothesis, we can ensure
\eqref{eq:primecounting} for the range $a>\frac12$.

For a general number field $K$, we are content with the
assertion \eqref{en:weak}, which is weaker than the strong form 
\eqref{en:strong}. While no results are needed from 
the theory of distribution of primes, we need an additional
argument beyond those done in Section~\ref{section=maintheoremfull},
namely, the slide trick.

\begin{theorem}[Theorem~\ref{mtheorem=TaoZieglergeneral}, restated]\label{theorem=shortinterval}
Let $K$ be a number field and $\omom$ an integral basis of $K$.
Let $\delta$ be a positive real number and $S$ a finite subset of $\OK$.
Take a real number $a$ with $0<a<1$.
Then the following hold.
\begin{enumerate}[$(1)$]
\item\label{en:tan_seiza} There exist a positive integer 
$M_{\mathrm{PESSI}}=M_{\mathrm{PESSI}}(\omom,\delta,S,a)
$, depending on $\omom,\delta,S$ and $a$, and a positive real number $\etaUpsilon = \etaUpsilon_{\mathrm{PESSI}}(\omom,\delta)>0$, depending only on $\omom$ and $\delta$, such that the following holds: if $M\geq M_{\mathrm{PESSI}}$ and a subset $A$ of $\PP_K\cap\OO_K(\omom,M)$ satisfies
\begin{equation}\label{eq:condition_for_A_thmB_again}
\#A\geq\delta\cdot\#(\PP_K \cap\OO_K(\omom,M)),
\end{equation}
then there exists $x\in A$ with

\begin{equation}\label{eq=condition_for_x_in_short_interval}
\etaUpsilon M\leq \|x\|_{\infty,\omom} \leq M
\end{equation}
such that $A\cap\OK(\omom,x,\|x\|_{\infty,\omom}^a)$
contains an $S$-constellation without associate pairs.
\item\label{en:tan_seiza_kosuu} If $S$ is a standard shape, then there exists a constant $\gamma
=\gamma_{\mathrm{PESSI}}(\omom,\delta,S,a)>0$,  
depending on $\omom,\delta,S$ and $a$, such that the following holds: if $M\geq M_{\mathrm{PESSI}}$ 
and a subset $A$ of $\PP_K\cap\OO_K(\omom,M)$ satisfies \eqref{eq:condition_for_A_thmB_again}, then there exists $x\in A$ with \eqref{eq=condition_for_x_in_short_interval} which satisfies
\[
\mathscr{N}_S^{\sharp}(A \cap \OO_K(\omom,x,\|x\|_{\infty,\omom}^a))\geq \gamma \cdot  \frac{M^{a(n+1)}}{ (\log M)^{\#S}}.
\]
\end{enumerate}
\end{theorem}

The following corollary is an infinitary version of 
Theorem~\ref{mtheorem=TaoZieglergeneral} (=
Theorem~\ref{theorem=shortinterval}).

\begin{corollary}\label{corollary=shortinterval}
Let $K$ be a number field, and let 
$\omom$ be an integral basis of $K$. If a subset $A\subseteq \PP_K$ satisfies
\[
\overline{d}_{\PP_K,\omom}(A)>0,
\]
then there exists a sequence $(y_l)_{l\in \NN}$ in $A$ satisfying the
following:
for every $a\in (0,1)_{\RR}$ and every finite set $S\subseteq A$,
there exists a finite subset $L\subseteq\NN$ such that for all $l\in\NN\setminus L$, the set
$A\cap \OK(\omom,y_l,\|y_l\|_{\infty,\omom}^a)$
contains an $S$-constellation consisting of pairwise non-associate
elements.
\end{corollary}

The proofs of Theorem~\ref{theorem=shortinterval} and
of Corollary~\ref{corollary=shortinterval} will be presented in
Subsection~\ref{subsection=package_shortinterval_proof}.

\begin{remark}\label{remark=TaoZiegler}
The statement of the form \eqref{en:weak} for the case $K=\QQ$ in terms
of upper density can be found in \cite[Remark~2.4]{Tao-Ziegler08},
as follows. Let $k$ be an integer with $k\geq3$.
If $A\subseteq \PP$ satisfies $\overline{d}_{\PP}(A)>0$,
then for every $a>0$, there is a sequence of real numbers $(M_l)_{l\in\NN}$
tending to $\infty$ such that,
$A\cap[M_l,M_l+M_l^a]_{\RR}$ contains an arithmetic progression
of length $k$.
\end{remark}

\subsection{Strong form for the case $K=\QQ$}\label{subsection=BanachGreenTao}

In this subsection, we prove Theorem~\ref{theorem=BanachGreenTao},
which is the strong form for a short interval version for the case $K=\QQ$,
along the lines of the axiomatic framework given in 
Section~\ref{section=maintheoremfull}.

\begin{proposition}\label{proposition=SPsilog_QQ}
Let $a\in(0,1)_{\RR}$.
Let $k$ be an integer at least $3$, and set $S_k\coloneqq\{0,1,\ldots ,k-1\}$.
Then, for every $\rho>0$, there exist positive integers $W=W_{\PP,\mathrm{S}\Psi_{\log}^{\mathrm{SI}}}(\rho,k)$ and $M_{\PP,\mathrm{S}\Psi_{\log}^{\mathrm{SI}}}(\rho,k,a)$ such that for every $M\in \RR$ with $M\geq M_{\PP,\mathrm{S}\Psi_{\log}^{\mathrm{SI}}}(\rho,k,a)$, there exists a function $\lambda\colon \ZZ\to \RR_{\geq 0}$ which fulfills the following three conditions.
\begin{enumerate}[$(1)$]
  \item\label{en:GTshort_Psi} For every $b\in \ZZ$ coprime with $W$, $\tilde{\lambda}_b\coloneqq \frac{\vph(W)}{W}(\lambda\circ\Aff_{W,b})$ is a $(\rho,\frac{M^a}{W},S_k)$-pseudorandom measure. Here, $\vph=\vph_{\QQ}$ is the Euler totient function.
\item\label{en:GTshort_log} For every $q\in \PP\cap [M,M+M^a]_{\RR}$, 
\[
\lambda(q)=\frac{a}{17k\cdot 2^{k-1}c_{\chi}}\cdot \log M
\]
holds true.
\item\label{en:GTshort_coprime} For every $q\in \PP\cap [M,M+M^a]_{\RR}$, $q$ is coprime with $W$.
\end{enumerate}
\end{proposition}

Notice that conclusions 
\eqref{en:GTshort_Psi}, \eqref{en:GTshort_log} and \eqref{en:GTshort_coprime}
are slightly different from conditions
\eqref{en:pseudorandom}, \eqref{en:measurebelow} and \eqref{en:coprime}
of Definition~\ref{definition=logpseudorandom}.
First, the width of the interval for the pseudorandomness 
is $M^a/W$ instead of $M/W$; secondly, there is no exceptional set $T$
in \eqref{en:GTshort_log} or \eqref{en:GTshort_coprime}.
The first difference comes from the fact that the width of
the interval $[M,M+M^a]_{\RR}$ is $M^a$; the absence of $T$
means that we may take $T=\varnothing$.

\begin{proof}
Set $\chi$ in a similar manner to Setting~\ref{setting=section7-2}, and fix it. Apply Theorem~\ref{theorem=package_PR} for $K=\QQ$, $S=S_k$, and $\uvarsigma=1$; recall the definitions of $w_{\mathrm{PRSI}}$ and  $M_{\mathrm{PRSI}}$. Let 
$W=W_{\PP,\mathrm{S}\Psi_{\log}^{\mathrm{SI}}}(\rho,\chi,k)$ be the positive integer determined by \eqref{eq:RwW} with $w=w_{\mathrm{PRSI}}(\rho,\chi,S_k)$.
Set $M_{\PP,\mathrm{S}\Psi_{\log}^{\mathrm{SI}}}(\rho,\chi,k,a)\coloneqq M_{\mathrm{PRSI}}(w_{\mathrm{PRSI}}(\rho,\chi,S_k),\rho,1,\chi,S_k,a)$.
For $M\geq M_{\PP,\mathrm{S}\Psi_{\log}^{\mathrm{SI}}}(\rho,\chi,k,a)$, take $\lambda\colon\ZZ\to \RR_{\geq 0}$ as in Theorem~\ref{theorem=package_PR} in the current setting; note here that $\kappa=\kappa_{\QQ}=1$. Then, Theorem~\ref{theorem=package_PR} implies that $\tilde{\lambda}_b$ in \eqref{en:GTshort_Psi} is a $(\rho,\frac{M^a}{W},S_k)$-pseudorandom measure.
For \eqref{en:GTshort_log}, by the construction of $\lambda$ and by $R<M$, we have for all $q\in \PP\cap [M,M+M^a]_{\RR}$,
\[
\lambda(q)=\frac{1}{c_{\chi}}\log R=\frac{a}{17k\cdot 2^{k-1}c_{\chi}}\cdot \log M.
\]
Since $w_{\mathrm{PRSI}}(\rho,\chi,S_k)\leq \frac{a}{2}\log M< M$, we also obtain \eqref{en:GTshort_coprime}. 
Since we have fixed a function $\chi$, we omit to write dependences of 
$M_{\PP,\mathrm{S}\Psi_{\log}^{\mathrm{SI}}}(\rho,\chi,k,a)$
and
$W_{\PP,\mathrm{S}\Psi_{\log}^{\mathrm{SI}}}(\rho,\chi,k)$
on $\chi$.
Thus we write 
$M_{\PP,\mathrm{S}\Psi_{\log}^{\mathrm{SI}}}(\rho,k,a)$
and
$W_{\PP,\mathrm{S}\Psi_{\log}^{\mathrm{SI}}}(\rho,k)$
for short.
This completes the proof.
\end{proof}
\begin{proof}[Proof of Theorem~$\ref{theorem=BanachGreenTao}$]
Let $S_k\coloneqq \{0,1,\dots,k-1\}\subseteq\ZZ$. 
Despite that this set $S_k$ is not a standard shape,
it meets all the requirements in 
Definition~\ref{definition=standardshape}
except `$S=-S$.'
We will prove the theorem along the
same lines of Theorem~\ref{theorem=package}, using
Proposition~\ref{proposition=SPsilog_QQ}.
Let $a\in(0,1)_{\RR}$, and define 
$D\coloneqq a\cdot(17k\cdot 2^{k-1}c_{\chi})^{-1}$. 
We will take $\rho>0$, depending only on $a,\delta$ and $k$, in what follows; for this $\rho$, let $W\coloneqq W_{\PP,\mathrm{S}\Psi_{\log}^{\mathrm{SI}}}(\rho,k)$. Let $M$ be sufficiently large; we will specify later.
Let 
\[X_{M}\coloneqq \PP \cap [M,M+M^{a}]_{\RR}.\]
By assumption, there exists $\delta>0$ such that 
\eqref{eq:primecounting} holds.
If $M$ is taken to be sufficiently large depending on $a$ and $\delta$,
then
\[
\#X_M\geq \delta \cdot \frac{M^{a}}{\log M}
\]
holds. By the pigeonhole principle, Proposition~\ref{proposition=SPsilog_QQ}~\eqref{en:GTshort_coprime} implies that there exists $\overline{b}\in(\ZZ/W\ZZ)^{\times}$ with
\begin{equation}
\#(X_M\cap \overline{b})\geq \frac{1}{\vph(W)}\delta \cdot \frac{M^{a}}{\log M}.\label{eq:countingPPb}
\end{equation}
Set $N\coloneqq\left\lceil\frac{M^{a}}{W}\right\rceil$. Then,
\begin{equation}\label{eq:NtoMW_a}
\frac{M^a}{W}\leq N\leq\frac{2M^{a}}{W}
\end{equation}
holds true.
Let $b\in\overline{b}$ be the largest element of $\overline{b}$ with $b\leq M$. Then, by definition and by the triangle inequality, we have $\Aff_{W,b}([-N,N]) \supseteq X_M \cap \overline{b}$.
Hence for $B\coloneqq \Aff_{W,b}^{-1}(X_M \cap \overline{b})$, we obtain $B\subseteq [-N,N]$. By \eqref{eq:countingPPb},
\begin{equation}\label{eq:relativeooi_a}
\#B \geq \frac{1}{\vph(W)}\delta \cdot \frac{M^{a}}{\log M}
\end{equation}
holds.
Next, let $\lambda\colon\ZZ\to\RR_{\geq 0}$ be the function as in Proposition~\ref{proposition=SPsilog_QQ}, and set $\tilde{\lambda}\colon\ZZ\to \RR_{\geq0}$ as $\tilde{\lambda}\coloneqq \frac{\vph(W)}{W}(\lambda\circ\Aff_{W,b})$. Then, by Proposition~\ref{proposition=SPsilog_QQ}~\eqref{en:GTshort_Psi}, if $M\geq M_{\PP,S\Psi^{\mathrm{SI}}_{\log}}(\rho,k,a)$, then $\tilde{\lambda}$ is a $(\rho,N,S_k)$-pseudorandom measure.
By Proposition~\ref{proposition=SPsilog_QQ} \eqref{en:GTshort_log}, \eqref{eq:relativeooi_a} and \eqref{eq:NtoMW_a}, we have
\[
\EE(\ichi_B\cdot\tilde{\lambda}\mid[-N,N])\geq D\cdot\delta\cdot\frac{M^{a}}{W\cdot (2N+1)}\geq\frac{D}{6}\cdot\delta
\]
and
\[
\frac{1}{N}\cdot \EE(\ichi_B\cdot \tilde{\lambda}^{k}  \mid [-N,N] )\leq D^{k}W\cdot  \frac{(\log M)^{k}}{M^{a}}.
\]

Now we define 
$\rho\coloneqq \rho_{\mathrm{RMS}}(\omom,\frac{D}{6}\delta,S_k)$ and 
$\gamma'\coloneqq \gamma_{\mathrm{RMS}}(\omom,\frac{D}{6} \delta,S_k)$,
where $\omom=(\omega)$ with $\omega=1$.
Choose $M_{\mathrm{GTSI}}$ in such a way that $M\geq M_{\mathrm{GTSI}}$
satisfies all the previous arguments, and that 
$D^{k}W\cdot  (\log M)^{k}\leq\gamma'\cdot M^{a}$
holds.
Then for $M\geq M_{\mathrm{GTSI}}$, 
we can apply the relative 
Szemer\'{e}di theorem ($n=1$ in Theorem~\ref{thm:RMST})
to $B\subseteq [-N,N]$. It follows that $B$ contains an $S_k$-constellation.
Applying the transformation $\Aff_{W,b}$, we obtain an $S_k$-constellation in $X_M=\PP\cap [M,M+M^{a}]_{\RR}$.

Finally, we prove \eqref{en:BGT2}.
Theorem~\ref{theorem=weighted-counting2}
and Proposition~\ref{proposition=SPsilog_QQ}~\eqref{en:GTshort_log}
imply

\[
\frac{1}{4N(2N+1)} \cdot (\scrN_{S_k}(B)+\scrN_{-S_k}(B)) \cdot (D_2\log M)^{k}\geq\frac{\gamma'}{2}.
\]
Since $\scrN_{S_k}(B)+\scrN_{-S_k}(B)=2\scrN_k(B)$, we have

\[
\scrN_k(X_M)\geq\scrN_k(B)\geq \frac{2D^{-k}\gamma'}{W^2}\cdot \frac{M^{2a}}{(\log M)^{k}}.
\]
Therefore, by setting $\gamma=\gamma_{\mathrm{GTSI}}(a,\delta,k)\coloneqq\frac{2D^{-k}}{W^2}\cdot \gamma'$, we obtain the desired estimate of $\scrN_{k}(X_M)$.\end{proof}

\subsection{Slide trick}\label{subsection=slidetrick}

In this subsection, we present a technique necessary in
proving constellation theorems of the form \eqref{en:weak}.
More precisely, we first describe the difficulty which arises
when $[K:\QQ]\geq2$, and then present the technique called
a slide trick to overcome this difficulty. 

Note that the \naturaldensityversionofthe
Chebotarev density theorem (Theorem~\ref{theorem=Chebotarev}~\eqref{Chebotarev})
provides an estimate of the number of prime elements in the domain
of the form $\OO_K(\omom, [M_1,M_1'])$. Here, we define,
for a free $\ZZ$-module $\calZ$, its $\ZZ$-basis $\bv$,
and $M_1,M_1'\in \RR_{\geq 0}$ with $M_1\leq M_1'$,
\begin{equation}\label{eq=M1M1'}
\calZ(\bv, [M_1,M_1'])\coloneqq\{\alpha\in \calZ : \|\alpha\|_{\infty,\bv}\in [M_1,M_1']_{\RR}\}.
\end{equation}
We suppose that $M_1'$ is reasonably close to $M_1$. If $n=1$, that is,
if $K=\QQ$, then with respect to its basis $\omom=(\omega)$, where
$\omega=1$, we have
\[
\OO_{\QQ}(\omom, [M_1,M_1'])=([-M_1',-M_1]_{\RR}\cap \ZZ) \cup ([M_1,M_1']_{\RR}\cap \ZZ).
\]
Since $\PP_{\QQ}=\PP\sqcup (-\PP)$, this amounts to considering
the single interval $[M_1,M_1']_{\RR}\cap \ZZ$. If $n\geq2$, then
there is a significant difference between
$\OO_{K}(\omom, [M_1,M_1'])$ and a short interval
$\OO_K(\omom,x,M)$. The former is an $n$-dimensional cube with
a smaller $n$-dimensional cube removed, while the latter is just a small
$n$-dimensional cube. In order to apply the relative multidimensional
Szemer\'{e}di theorem (Theorem~\ref{thm:RMST}), 
it is necessary to transfer the estimate in the former to the latter.
This is an extra step which arises when $n\geq2$.
The transfer is possible if we can guarantee the density of prime elements
is large in the latter, provided that the density of prime elements is large
in the former. This assertion can be proved using the pigeonhole principle,
which we formulate explicitly as the slide trick. 
In the setting of Szemer\'{e}di-type theorems of finitary version, the slide trick
can be thought of taking a better representative $b\in\overline{b}$ of
$\overline{b}\in(\OK/W\OK)^{\times}$ in the proof; retaking a representative
is nothing but `sliding' it by an element of $W\OK$, hence the name.
We remark that this argument already appeared in 
the proof of Theorem~\ref{theorem=BanachGreenTao}
when a representative $b\in \overline{b}$ was chosen.
The reason why the slide trick is useful in our application is that 
the assertion of the 
Goldston--Y\i ld\i r\i m type asymptotic formula
(Theorem~\ref{Th:Goldston_Yildirim}) is strong enough that
$b\in\OK$ can be arbitrary as long as $b\OK +W\OK=\OK$.

We now describe the slide trick. This can be formulated in the
following general setting. Let $n\in\NN$. Suppose that the additive
group $\ZZ^n$ acts on a non-empty set $\LL$ and that this action
is simply transitive. This means that, for all $l,l'\in\LL$,
there exists a unique $z\in\ZZ^n$ such that $l'=z\cdot l$.
In this case, for a non-empty subset $P$ of $\LL$, the set of the
form $z\cdot P$, where $z\in\ZZ^n$, is called a \emph{translate}
of $P$. The following is the fundamental lemma for the slide trick.

\begin{lemma}\label{lemma=slidetrick}
Let $n\in\NN$, and let $\ZZ^n\curvearrowright\LL$ be a simply
transitive action.
Let $P,X\subseteq\LL$ be finite non-empty sets, and 
let $\mathcal{Q}$ be the family of all translates $Q$ of $P$ satisfying
$Q\cap X\neq\varnothing$.
Then there exists $Q_X\in\mathcal{Q}$ such that
\[\frac{\#(Q_X\cap X)}{\#Q_X}\geq \frac{\#X}{\#\mathcal{Q}}\]
holds.
\end{lemma}
\begin{proof}
Since the action is simply transitive, the number of $Q\in\mathcal{Q}$
containing a given $x\in X$ is exactly $\#P$. This implies
\[
\sum_{Q\in\mathcal{Q}}\#(Q\cap X)=\#\{(x,Q)\in X\times\mathcal{Q} \mid x\in Q\}=\sum_{x\in X}\#\{Q\in \mathcal{Q} \mid x\in Q\}=\#X\cdot\#P,
\]
and hence
\[
\EE\left(\frac{\#(Q\cap X)}{\#Q} \ \middle| \ Q\in \mathcal{Q}\right)=\frac{1}{\#\mathcal{Q}}\sum_{Q\in\mathcal{Q}} \frac{\#(Q\cap X)}{\#Q}=
\frac{1}{(\#\mathcal{Q})\cdot(\#P)}\sum_{Q\in\mathcal{Q}} \#(Q\cap X)=
\frac{\#X}{\#\mathcal{Q}}.
\]
The result then follows by the pigeonhole principle.
\end{proof}
\begin{remark}\label{remark=slidetrick}
The above lemma can be generalized as follows. Let 
$f\colon \LL\to \RR_{\geq 0}$ be a function which is $0$ on
$\LL\setminus X$. Then there exists a translate $Q_{f}$ of $P$
such that
\[
\EE\left(f\mid Q_f\right)\geq\frac{\#X}{\#\mathcal{Q}}\cdot\EE(f\mid X)
\]
holds. 
Lemma~\ref{lemma=slidetrick} is exactly the case
where $f=\ichi_X$.
\end{remark}

If the subset $P\subseteq\LL$ tiles $\LL$, that is, if there exists
a subset $Z\subseteq\ZZ^n$ such that $\LL=\bigsqcup_{z\in Z}z\cdot P$,
then an analogous statement as Lemma~\ref{lemma=slidetrick}
can be proved more directly. 
Lemma~\ref{lemma=slidetrick} is used in the proof of
Theorem~\ref{mtheorem=TaoZieglergeneral} with $P$ being an
$n$-dimensional cube, and this falls in the above situation.
We have, however, stated Lemma~\ref{lemma=slidetrick} as
a more general statement.

\subsection{Axiomatized constellation theorems for short intervals}\label{subsection=package_shortinterval}

In this subsection and the next, we use
Setting~\ref{setting=package}.
In a manner similar to the argument in Section~\ref{section=maintheoremfull},
we present axiomatized constellation theorems of both finitary and infinitary versions of the form \eqref{en:weak} for short intervals.
We present our axiomatized constellation theorems only for type~2, that is,
ones for constellations without associate pairs.
\begin{definition}[\compatiaW]\label{definition=logpseudorandom_a}
Assume Setting~\ref{setting=package}. 
Let $a\in (0,1)_{\RR}$, and let $S\subseteq\ideala$ be a standard
shape.
Let $\rho>0$, $\uvarsigma>0$, $D_1,D_2>0$, $\varepsilon\in (0,1)_{\RR}$
and let $M\in\RR_{\geq 1}$. Let $W\in \NN$ be a narutal number with $W\leq M^{\varepsilon a}$.
A subset $A\subseteq \ideala$ is said to satisfy the \emph{\compatiaW\
with parameters $(D_1,D_2,\varepsilon)$}, if
$A\subseteq \ideala(\bv,M)$, and the following holds: there exists 
$\lambda\colon \ideala\to \RR_{\geq 0}$ such that the following conditions
are fulfilled.
\begin{enumerate}[$(1)$]
  \item \label{en:pseudorandom_a} 
  For every $b\in \ideala$ satisfying $b\OK +W\ideala=\ideala$,
  the mapping $\beta\mapsto\frac{\vph_K(W)}{W^n}(\lambda\circ\Aff_{W,b})(\beta)$ on $\ideala$ is a $(\rho,\frac{\uvarsigma M^a}{W},S)$-pseudorandom measure.
  \item \label{en:measurebelow_a} 
 There exists $T\subseteq A$ with 
$ \#T\leq M^{\varepsilon an}$
such that,  for every $\alpha\in A\setminus T$,
\[
D_1 \cdot \log M \leq \lambda(\alpha)\leq D_2\cdot  \log M
\]
holds.
  \item \label{en:coprime_a}
  For the subset  $T$ in \eqref{en:measurebelow_a}
  and $\alpha\in A\setminus T$, 
$\alpha \OK +W\ideala =\ideala$
holds.
\end{enumerate}
\end{definition}
\begin{definition}[\compatia]\label{definition=logpseudorandom_aW}
Assume Setting~\ref{setting=package}.
Let $a\in (0,1)_{\RR}$, and $S\subseteq\ideala$ a standard shape.
Let $\rho>0$, $\uvarsigma>0$, $M\in\RR_{\geq 1}$, $D_1,D_2>0$, and $\varepsilon\in(0,1)_{\RR}$.
A subset $A\subseteq\ideala$ is said to satisfy the \emph{\compatia\ with parameters $(D_1,D_2,\varepsilon)$} if there exists $W \in\NN$ with $W\leq M^{\varepsilon a}$ such that $A$ satisfies the \compatiaW\ with parameters $(D_1,D_2,\varepsilon)$.
\end{definition}

The difference of Definition~\ref{definition=logpseudorandom_a}
from Definition~\ref{definition=logpseudorandom} is that
the appearance of $a$ in the exponent for the upper bound on $W$,
and the lower bound on the width of the interval for the
pseudorandom condition in \eqref{en:pseudorandom_a}
is changed from $\frac{M}{W}$ to $\frac{\uvarsigma M^a}{W}$.
The reason for the change 
is similar to the situation in Proposition~\ref{proposition=SPsilog_QQ};
the width of the interval is $M^a$ instead of $M$.
The factor $\uvarsigma$ is needed for the proof of Theorem~\ref{theorem=package_a}. The role of the
exponent $a$ in the upper bound
on the size of an exceptional set $T$ in 
\eqref{en:measurebelow_a} and \eqref{en:coprime_a} is a minor
issue. It is placed so as to force the containment of 
the set $\logpseua$ in $\logpseu$; the former will be defined in Definition~\ref{definition=logpseudorandom_infinite_a}.
The removal of an exceptional set $T$ is necessary even 
in the weak form of the short interval version. Indeed,
it cannot be avoided unless $\OKt$ is finite; see the proof
of Theorem~\ref{theorem=logpseudorandom_a}.
This is in contrast to Proposition~\ref{proposition=SPsilog_QQ},
where no exceptional set was needed.

\begin{definition}[The family $\logpseua$]\label{definition=logpseudorandom_infinite_a}
We define a family $\logpseua$ of subsets of 
$\ideala$ as follows: we declare
$A\in \logpseua$ if and only if,
for every $a\in (0,1)_{\RR}$ and for every standard shape
$S\subseteq \ideala$, there exist $D_1,D_2>0$ and
$\varepsilon\in (0,1)_{\RR}$ such that the following holds:
for every $\rho>0$ and for every $\uvarsigma>0$, there exists $M(\rho,\uvarsigma)=M(\rho,\uvarsigma,\bv,S,a)\in \RR_{\geq0}$
such that, for all $M\geq M(\rho,\uvarsigma)$, 
$A\cap \ideala(\bv,M)$ satisfies the \compatia\ with parameters
$(D_1,D_2,\varepsilon)$.
\end{definition}

We chose to attach symbols `SI' in $\logpseua$ to signify `short interval.'

\begin{lemma}\label{lemma=stronglog}
The following statements hold true.
\begin{enumerate}[$(1)$]
  \item\label{en:logsubset}
  If $A\subseteq\ideala$ satisfies the \compatiaW\ with parameters 
  $(D_1,D_2,\varepsilon)$, then so do its subsets.
  \item\label{en:strlogsubset}
If  $A\in \logpseua$, $A_1\subseteq A$, then $A_1\in \logpseua$.
  \item\label{en:stronglog} $\logpseua \subseteq \logpseu$.
\end{enumerate}
\end{lemma}
\begin{proof}
Items \eqref{en:logsubset} and \eqref{en:strlogsubset} 
can be proved in a manner similar to that of Lemma~\ref{lemma=logpseudorandom_subset}.

We prove \eqref{en:stronglog}. Let $A\in\logpseua$. 
Let $S\subseteq\ideala$ be a standard shape, and fix
$a\in(0,1)_{\RR}$ arbitrarily. Then 
there exist $D_1,D_2>0$ and $\varepsilon\in (0,1)_{\RR}$ 
as described in Definition~\ref{definition=logpseudorandom_infinite_a}.
Take $\rho>0$ and let $\uvarsigma=1$. Let $M>0$ be a real number satisfying
$M^{1/a}\geq M(\rho,1,\bv,S,a)$. Since $A\in\logpseua$,
the set $A\cap \ideala(\bv,M^{1/a})$ satisfies the 
$(\rho,1, M^{1/a},\bv,S,a)$-condition with parameters
$(D_1,D_2,\varepsilon)$. It follows from \eqref{en:logsubset} that
$A\cap \ideala(\bv,M)$ satisfies the 
$(\rho,1, M^{1/a},\bv,S,a)$-condition with parameters
$(D_1,D_2,\varepsilon)$.
By Definition~\ref{definition=logpseudorandom_a},
this means that there exist 
$W\in \NN$ with $W\leq M^\varepsilon$,
$\lambda\colon\ideala\to\RR_{\geq0}$
and an exceptional set $T\subseteq \ideala(\bv ,M)$ with $\#T\leq M^{\varepsilon n}$ such that conditions~\eqref{en:pseudorandom_a}--\eqref{en:coprime_a} of 
Definition~\ref{definition=logpseudorandom_a} hold.
Then we see that
$A\cap\ideala(\bv,M)$ satisfies the \compati\ with parameters
$(D_1/a,D_2/a,\varepsilon)$. Therefore,
$A\in\logpseu$.
\end{proof}
\begin{theorem}\label{theorem=logpseudorandom_a}
For a number field $K$, we have $\PP_K\in \logpseuaOK$. Furthermore, for every  $a\in(0,1)_{\RR}$ and every integer $r$ at least $[K:\QQ]$, there exist $D_1,D_2>0$ and $\varepsilon\in (0,1)_{\RR}$ such that the following holds: 
let $S\subseteq \OK$ be a standard shape with $\#S=r+1$, and $\omom$ be an integral basis of $K$. Let $\rho>0$ and $\uvarsigma>0$.
Then, there exist integers $W=W_{\PP_K,\mathrm{S}\Psi_{\log}^{\mathrm{SI}}}(\rho,S)$ and $M_{\PP_K,\mathrm{S}\Psi_{\log}^{\mathrm{SI}}}(\rho,\uvarsigma,\omom,S,a)$ such that if $M\geq M_{\PP_K,\mathrm{S}\Psi_{\log}^{\mathrm{SI}}}(\rho,\uvarsigma,\omom,S,a)$, then $\PP_K\cap \OK(\omom,M)$ satisfies the $(\rho,\uvarsigma,W,M,\omom,S,a)$-condition with parameters $(D_1,D_2,\varepsilon)$.
\end{theorem}
\begin{proof}
The proof proceeds along the same lines of that of
Proposition~\ref{proposition=SPsilog_QQ} except the treatment of
an exceptional set $T$.
Let $a\in(0,1)_{\RR}$ be arbitrary. Recall $w_{\mathrm{PRSI}}$ and $M_{\mathrm{PRSI}}$ from Theorem~\ref{theorem=package_PR}. 
Define the positive integer $W=W_{\PP_K,\mathrm{S}\Psi_{\log}^{\mathrm{SI}}}(\rho,\chi,S)$ by \eqref{eq:RwW} with $w= w_{\mathrm{PRSI}}(\rho,\chi,S)$. 
Let $M$ be a parameter with $M\geq M_{\mathrm{PRSI}}(w_{\mathrm{PRSI}}(\rho,\chi,S),\rho,u,\chi,S,a)$, and take the function $\lambda=\lambda_{M;\chi,r,a,K}\colon \OK\to\RR_{\geq 0}$ as in Theorem~\ref{theorem=package_PR}. Then, by Theorem~\ref{theorem=package_PR}, the condition of Definition~\ref{definition=logpseudorandom_a} \eqref{en:pseudorandom_a} is satisfied for $\ideala=\OK$. Moreover, by \eqref{eq:WaM}, we have $W\leq M^{(\log 2) a}\leq M^{\frac{3}{4}a}$.

Define an exceptional set
$T\subseteq\PP_K\cap\OK(\omom,M)$
by $T\coloneqq \PP_K\cap\OK(\omom,M)\cap\OK(R)$.
As in the proof of 
Lemma~\ref{lemma=jogai}, if $M$ is sufficiently large depending on 
$r,K$ and $a$, then every $\alpha\in \PP_K\setminus T$ is prime to
$W$, and $\lambda(\alpha)=\frac{\kappa a}{17(r+1)2^r \cdot c_{\chi}}\cdot \log M$
holds. Moreover, if $M$ is sufficiently large depending on 
$\omom$ and $a$, then by \eqref{eq:boundT}, we obtain $\#T\leq M^{\frac{a}{16}}$.
Therefore, for every $\rho>0$ and for every $\uvarsigma>0$, if $M$ is sufficiently large depending on $\rho$,$\uvarsigma$,$\omom$,$S$ and $a$, the $\PP_K\cap \OK(\omom,M)$ satisfies the $(\rho,\uvarsigma,W,M,\omom,S,a)$-condition with parameters
\[
(D_1,D_2,\varepsilon)=\left(\frac{\kappa a}{17(r+1)2^r \cdot c_{\chi}},\frac{\kappa a}{17(r+1)2^r \cdot c_{\chi}},\frac{3}{4}\right).
\]
Define $M_{\PP_K,\mathrm{S}\Psi_{\log}^{\mathrm{SI}}}(\rho,\uvarsigma,\omom,\chi,S,a)$ as the smallest integer for which all of the arguments above work. 
This ends the proof of the latter assertion. 
Here, since we have fixed a function $\chi$, we omit to write dependences of 
$W_{\PP_K,\mathrm{S}\Psi_{\log}^{\mathrm{SI}}}(\rho,\chi,S)$
and
$M_{\PP_K,\mathrm{S}\Psi_{\log}^{\mathrm{SI}}}(\rho,\uvarsigma,\omom,\chi,S,a)$
on $\chi$.
Thus we write 
$W_{\PP_K,\mathrm{S}\Psi_{\log}^{\mathrm{SI}}}(\rho,S)$
and
$M_{\PP_K,\mathrm{S}\Psi_{\log}^{\mathrm{SI}}}(\rho,\uvarsigma,\omom,S,a)$
for short.
In particular, we conclude that $\PP_K\in \logpseuaOK$.
This completes the proof.
\end{proof}

The following two theorems are the short interval versions of
Theorem~\ref{theorem=package_DD} and
Theorem~\ref{theorem=package_infinite_DD}, respectively.
Recall Definition~\ref{definition=shortinterval}.
\begin{theorem}\label{theorem=package_a}
We use Setting~$\ref{setting=package}$. 
Let $a\in (0,1)_{\RR}$, and let $S\subseteq\ideala$ be a standard shape.
Let $\delta, \Delta,D_1,D_2>0$ and $\varepsilon\in (0,1)_{\RR}$.
Then there exist 
$\rho=\rho_{\mathrm{SI}}(D_1,\bv,\delta,\Delta,S)>0$,
$\uvarsigma=\uvarsigma_{\mathrm{SI}}(\bv,\delta,\Delta,a)>0$, 
$\etaUpsilon=\etaUpsilon_{\mathrm{SI}}(\bv,\delta,\Delta)>0$
and 
$M_{\mathrm{SI}}=M_{\mathrm{SI}}(D_1,D_2,\varepsilon,\bv,\delta,\Delta,S,a)\in\NN$
such that the following holds.
Assume an integer $M\geq M_{\mathrm{SI}}$ and a set
$A\subseteq\ideala(\bv,M)\setminus\{0\}$ satisfy the following three conditions:
\begin{enumerate}[$(i)$]
\item\label{en:counting_a}
the inequality
\[
\#A\geq \delta \cdot \frac{M^n}{\log M}
\]
holds, 
\item\label{en:counting_above}
for every $L\in \RR_{\geq 2}$, 
\[
\#\{\alpha\OK\in \Ideals_K\colon \alpha \in A\cap \OK(L)\}\leq \Delta\cdot \frac{L}{\log L}
\] 
holds,
\item\label{en:logpseudorandom_a} the set $A$ satisfies the \compatia\ with parameters
$(D_1,D_2,\varepsilon)$.
\end{enumerate}
Then there exists $x\in A$ satisfying

\begin{equation}\label{eq:log^2_bv}
\etaUpsilon M\leq \|x\|_{\infty,\bv}\leq M
\end{equation}
such that
$A\cap \ideala(\bv,x,\|x\|_{\infty,\bv}^a)$
contains an $S$-constellation without associate pairs.
Moreover, there exists $\gamma=\gamma_{\mathrm{SI}}(D_1,D_2,\bv,\delta,\Delta,S,a)>0$
such that the above $x\in A$ can be taken in such a way that
\[
\scrN_S^{\sharp}(A\cap \ideala(\bv,x,\|x\|_{\infty,\bv}^a))\geq \gamma W^{-(n+1)}\cdot \frac{M^{a(n+1)}}{(\log M)^{\#S}}
\]
holds.
Here, $W$ is an integer appearing in the \compatia, which comes from condition~\eqref{en:logpseudorandom_a}.
\end{theorem}
\begin{theorem}\label{theorem=package_infinite_a}
We use Setting~$\ref{setting=package}$.
Assume that a subset $A\subseteq \ideala\setminus\{0\}$
satisfies the following three conditions:
\begin{enumerate}[$(i)$]
\item\label{en:counting_infinite_a}
the inequality
\[
\limsup_{M\to \infty}\frac{\#(A\cap \ideala(\bv,M))}{M^n(\log M)^{-1}}>0
\]
holds, 
\item\label{en:counting_above_a}
there exists $\Delta>0$ such that for every $L\in \RR_{\geq 2}$,
\begin{equation}\label{eq:ideal_Delta_S9}
\#\{\alpha\OK\in \Ideals_K\colon \alpha \in A\cap \OK(L)\}\leq \Delta\cdot \frac{L}{\log L}
\end{equation}
holds,
\item\label{en:logpseudo_infinite_a}
$A \in \logpseua$.
\end{enumerate}
Then, there exists a sequence 
$(y_l)_{l\in \NN}$ in $A$ satisfying the following: for every $a\in (0,1)_{\RR}$ and every finite set
$S\subseteq \ideala$, 
there exists a finite subset $L\subseteq\NN$ such that for all $l\in\NN\setminus L$,
$A\cap \ideala(\bv,y_l,\|y_l\|_{\infty,\bv}^a)$
contains an $S$-constellation consisting of pairwise non-associate
elements.
\end{theorem}
We prove Theorem~\ref{theorem=package_a} in 
Subsection~\ref{subsection=package_shortinterval_proof},
and then deduce Theorem~\ref{theorem=package_infinite_a}
from Theorem~\ref{theorem=package_a}.

We obtain the following corollary from
Theorem~\ref{theorem=package_infinite_a}
and Lemma~\ref{lemma=stronglog}.
\begin{corollary}\label{corollary=package_infinite_a}
We use Setting~$\ref{setting=package}$.
Assume that a subset $A\subseteq \ideala\setminus\{0\}$ satisfies the following three conditions:

\begin{enumerate}[$(i)$]
\item\label{en:counting_infinite_X_a} 
the inequality
\[
\liminf_{M\to \infty}\frac{\#(A\cap \ideala(\bv,M))}{M^n(\log M)^{-1}}>0
\]
holds, 
\item\label{en:counting_above_X_a}
condition~\eqref{en:counting_above_a} of Theorem~$\ref{theorem=package_infinite_a}$ is satisfied,
\item\label{en:logpseudo_infinite_X_a}
$A\in \logpseua$.
\end{enumerate}
Then for every $A'\subseteq A$ satisfying
$\overline{d}_{A,\bv}(A')>0$, 
there exists a sequence 
$(y_l)_{l\in \NN}$ in $A'$ satisfying the following: for every $a\in (0,1)_{\RR}$ and every finite set
$S\subseteq \ideala$, 
there exists a finite subset $L\subseteq\NN$ such that for all $l\in\NN\setminus L$,
$A'\cap \ideala(\bv,y_l,\|y_l\|_{\infty,\bv}^a)$
contains an $S$-constellation consisting of pairwise non-associate
elements.
\end{corollary}
\begin{proof}[Proof of
Theorem~$\ref{theorem=package_infinite_a}$ $\Longrightarrow$ Corollary~$\ref{corollary=package_infinite_a}$]
This is analogous to the deduction of 
Corollary~\ref{corollary=package_infinite_DD}
from Theorem~\ref{theorem=package_infinite_DD}.
\end{proof}

We note that the sequence $(y_l)_{l\in \NN}$ can be found independent
of the choice of $a$ and $S$ in 
Theorem~\ref{theorem=package_infinite_a} and Corollary~\ref{corollary=package_infinite_a}.

\subsection{Proofs of Theorems~\ref{theorem=package_a} and \ref{theorem=package_infinite_a}}\label{subsection=package_shortinterval_proof}

In this subsection, we prove
Theorems~\ref{theorem=package_a} and~\ref{theorem=package_infinite_a},
and using these theorems, we prove
Theorem~\ref{theorem=shortinterval}
and Corollary~\ref{corollary=shortinterval}.

Let $n\in \NN$. 
Define the $\lmugen$-length 
$\|\cdot\|^{}_{\infty}$
on $\ZZ^n$ with respect to the standard basis.
For $M_1,M_2\in\RR$ with
$0\leq M_1\leq M_2$, let
$\ZZ^n([M_1,M_2])\coloneqq\{x\in \ZZ^n: 
\|x\|_{\infty}\in [M_1,M_2]_{\RR}\}$.

\begin{lemma}\label{lemma=countingmelonpan}
Let $M_0,M$ and $a\in(0,1)_{\RR}$ be real numbers
with $0<M_0\leq M$. Then,
for every $A\subseteq \ZZ^n([M_0,M])$, 
there exists $M_1\in [M_0,M]_{\RR}$ such that
\[
\#(A\cap \ZZ^n([M_1,M_1+M_1^a]))\geq \frac{1}{2^{n}}\cdot \frac{(M_1+M_1^a)^n-M_1^n}{M^n} \cdot \#A
\]
\end{lemma}
\begin{proof}
We define a finite sequence
$m_1,m_2,\ldots$ of real numbers as follows. We first set
$m_0\coloneqq M_0$. Assuming $m_i$ has been defined,
we stop constructing the sequence if $m_i\geq M$, and otherwise
define $m_{i+1}\coloneqq m_{i}+m_{i}^a$.
This process terminates because we continue to have
$m_{i+1}\geq m_i+M_0^a$. Let $m_l$ be the last term of this
sequence. Then there exists $l_A\in[0,l-1]$ such that
\[
\#(A\cap \ZZ^n([m_{l_A},m_{l_A+1}])\geq \frac{1}{2^n }\cdot  \frac{m_{l_A+1}^n-m_{l_A}^n}{M^n} \cdot \#A
\]
holds. Indeed, otherwise, taking the summation from $i=0$ to $l-1$
gives
\[
\#A< \frac{1}{2^n}\cdot \frac{m_l^n-m_0^n}{M^n}\cdot\#A\leq\frac{1}{2^n}\cdot \frac{(M+M^a)^n-M_{0}^n}{M^n}\cdot \#A \leq \#A,
\]
which is a contradiction.
The desired inequality holds by setting $M_1\coloneqq m_{l_A}$.
\end{proof}

\begin{lemma}\label{lemma=girigiri_chop}\normalfont
Assume Setting~$\ref{setting=package}$. Let $\sigma_1,\ldots ,\sigma_{n}$ be the embeddings of $K$ into $\CC$. Let $\Omega>0$. Then, there exist $D(\bv,\Omega),D'(\bv,\Omega)>0$, depending on $\bv$ and $\Omega$, such that the following holds true: for $M\in \RR_{\geq 1}$, if $\alpha\in \ideala(\bv,M)$ satisfies $\Nrm(\alpha)\geq \Omega M^n$, then
\[
\min_{i\in [n]}|\sigma_i(\alpha)|\geq D(\bv,\Omega)\cdot M \quad \textrm{and} \quad \|\alpha\|_{\infty,\bv}\geq D'(\bv,\Omega)\cdot M
\]
holds.
\end{lemma}

\begin{proof}
For $\bv=(v_1,\ldots ,v_n)$, set $C'(\bv)\coloneqq \max_{i\in [n]}\sum_{j\in[n]}|\sigma_i(v_j)|$. Then for all $k\in [n]$, we have $|\sigma_k(\alpha)|\leq C'(\bv)M$.
Since $\Nrm(\alpha)=|\sigma_i(\alpha)|\cdot \prod_{k \in [n]\setminus \{i\}}|\sigma_k(\alpha)|$ by Lemma~\ref{lemma=idealnorm}, we have
\[
\min_{i\in [n]}|\sigma_i(\alpha)|\geq \frac{\Omega}{(C'(\bv))^{n-1}}\cdot  M.
\]
Similarly, since for every $k\in[n]$, $|\sigma_k(\alpha)|\leq C'(\bv)\|\alpha\|_{\infty,\bv}$, we obtain 
\[
\|\alpha\|_{\infty,\bv}\geq \frac{\sqrt[n]{\Omega}}{C'(\bv)}\cdot M.
\]
Therefore, $D(\bv,\Omega)\coloneqq \frac{\Omega}{(C'(\bv))^{n-1}}$ and $D'(\bv,\Omega)\coloneqq \frac{\sqrt[n]{\Omega}}{C'(\bv)}$ work.
\end{proof}

\begin{proof}[Proof of Theorem~$\ref{theorem=package_a}$]
Let $a,\delta,\Delta,D_1,D_2$ and $\varepsilon$ be as in the
statement of the theorem.
We let $M$ be a sufficiently large real number, to be determined
exactly later. We also take $\rho>0$ and $\uvarsigma>0$ arbitrarily at this point.
Let $A\subseteq\ideala(\bv,M)\setminus\{0\}$ be a subset satisfying
conditions \eqref{en:pseudorandom_a}, \eqref{en:measurebelow_a}
and \eqref{en:coprime_a} of the theorem. Thus we can take an integer $W\leq M^{\varepsilon a}$, 
a function $\lambda\colon \ideala\to \RR_{\geq 0}$
and an exceptional set $T\subseteq A$ as in these conditions.
For real numbers $M_1$ and $M_2$ with $0\leq M_1\leq M_2$, define $\ideala(\bv,[M_1,M_2])$ by setting $\calZ=\ideala$ in \eqref{eq=M1M1'}.

Take $M'_{\mathrm{red}}(\bv,\delta,\Delta)\in\NN$, $\delta'=\delta'_{\mathrm{red}}(\bv,\delta,\Delta)>0$ and $\Omega'=\Omega'_{\mathrm{red}}(\bv,\delta,\Delta)>0$ as in Theorem~\ref{theorem=fundamental_Omega}.
Now assume that $M\geq M'_{\mathrm{red}}(\bv,\delta,\Delta)$.
Then we can apply Theorem~\ref{theorem=fundamental_Omega} to the set $A$. 
This implies that there exist a fundamental domain $\DD$ for the action 
$\OKt\curvearrowright \ideala\setminus \{0\}$ and a subset
$A_0\subseteq A\cap \DD$ such that
\begin{equation}\label{eq:A0_delta'}
\#A_0\geq \delta'\cdot \frac{M^n}{\log M}
\end{equation}
holds. In addition \eqref{eq:Omega_norm} holds, that is,
for $\alpha\in A_0$, we have $\Nrm(\alpha)>\Omega' M^n$.
Let $D'(\bv,\Omega')>0$ be the constant in Lemma~\ref{lemma=girigiri_chop}. Set
\begin{equation}\label{eq:Mflat}
M_{\flat}\coloneqq D'(\bv,\Omega')M.
\end{equation}
Then, we have
\begin{equation}\label{eq:suppori_hairu}
A_0\subseteq \ideala(\bv,[M_{\flat},M]).
\end{equation}

Let $A_0'\coloneqq A_0\setminus T$.
If $M$ is sufficiently large depending on 
$\varepsilon,\bv,\delta,\Delta$ and $a$, it follows from
the upper bound $\#T\leq M^{\varepsilon an}$ and \eqref{eq:A0_delta'} that
\begin{equation}\label{eq:A'below}
\#A_0'\geq  \frac{1}{2}\delta' \cdot \frac{M^n}{\log M}
\end{equation}
holds. Under the isometry 
$(\ideala,\|\cdot\|_{\infty,\bv})\simeq (\ZZ^n,\|\cdot\|_{\infty})$,
we may regard $A_0'$ as a subset of $\ZZ^n([M_{\flat},M])$ by \eqref{eq:suppori_hairu}. We can thus apply 
Lemma~\ref{lemma=countingmelonpan} by setting
$M_0=M_{\flat}$, $A=A_0'$. This, together with
\eqref{eq:A'below}, implies that there exists
$M_{\natural}\in[M_{\flat},M]_{\RR}$ such that
\begin{equation}\label{eq:A''below}
\#A_0''\geq \frac{1}{2^{n+1}}\delta' \cdot \frac{(M_{\natural}+M_{\natural}^a)^n-M_{\natural}^n}{\log M}
\end{equation}
holds, where
$A_0''\coloneqq A_0'\cap \ideala(\bv,[M_{\natural},M_{\natural}+M_{\natural}^a])$.

Next we apply Lemma~\ref{lemma=coprime} to $A_0''$, noting that
condition \eqref{en:coprime_a} holds by the 
\compatiaW. By \eqref{eq:A''below}, we see that there exists
$\overline{b}\in\ideala/W\ideala$ such that
\begin{equation}\label{eq:countingmelonpan}
\#(A_0''\cap \overline{b})\geq \frac{1}{2^{n+1}\vph_K(W)}\delta' \cdot \frac{(M_{\natural}+M_{\natural}^a)^n-M_{\natural}^n}{\log M}
\end{equation}
holds. Note that the set $A_0''\cap \overline{b}$ is 
not located inside an $n$-dimensional cube with a small diameter for which the relative multidimensional
Szemer\'edi theorem applies. 

We will determine the center $x\in A$ of a short interval by
suitably choosing a representative $b\in\overline{b}$ using the
slide trick, in such a way that we can apply the relative multidimensional
Szemer\'edi theorem.
Let $N\coloneqq 2\left\lceil\frac{M_{\natural}^a}{8W}\right\rceil$.
Since $W\leq M^{\varepsilon a}$ and $M_{\flat}\leq M_{\natural}$, for a sufficiently large $M$ depending on $\varepsilon$, $\bv,\delta,\Delta$ and $a$, we have
\begin{equation}\label{eq:NtoMtoMnatural}
\frac{M_{\natural}^a}{4W}\leq N\leq\frac{M_{\natural}^a}{3W}.
\end{equation}
Observe that the additive group 
$\ideala\simeq \ZZ^n$ acts simply transitively on $\overline{b}$ by
$\beta\cdot x\coloneqq x+W\beta$. Fix an arbitrary $b_0\in \overline{b}$.
We aim to apply Lemma~\ref{lemma=slidetrick} by setting
$P=W\ideala(\bv,N/2)+b_0$ and $X=A_0''\cap \overline{b}$.
The size of the set $\calQ$ in Lemma~\ref{lemma=slidetrick} for this setting
can be estimated by considering the location of the `upper left' corner of
$W\ideala(\bv,N/2)+b$, a translate of $P$, as follows.
\begin{align*}
\#\calQ&\leq \left(\frac{2M_{\natural}+2M_{\natural}^a+WN+1}{W}+1\right)^n-\left(\frac{2M_{\natural}-WN}{W}-1\right)^n \\
&\leq \frac{(2M_{\natural}+3M_{\natural}^a)^n-(2M_{\natural}-M_{\natural}^a)^n}{W^n}.
\end{align*}
We use the following general inequalities: for $t,t'\geq0$ with
$t'/t$ small enough depending on $n$,
\[
 t^n+nt^{n-1}t' \leq (t+t')^n\leq t^n+2nt^{n-1}t'
\]
holds.
In addition, Bernoulli's inequality implies $(t-t')^n\geq t^n-nt^{n-1}t'$.
Since $a\in (0,1)_{\RR}$, for a sufficiently large $M$ depending on $a$ and $n$,
we have
\[
(2M_{\natural}+3M_{\natural}^a)^n-(2M_{\natural}-M_{\natural}^a)^n\leq 2^{n+2}((M_{\natural}+M_{\natural}^a)^n-M_{\natural}^n).
\]
Therefore,
\begin{equation}\label{eq:calQ}
\#\calQ\leq \frac{2^{n+2}}{W^n}\cdot ((M_{\natural}+M_{\natural}^a)^n-M_{\natural}^n).
\end{equation}
It follows from 
Lemma~\ref{lemma=slidetrick}, \eqref{eq:countingmelonpan} and \eqref{eq:calQ}
that, there exists  a representative $b\in \overline{b}$ such that the following holds:
\begin{align*}
&\#(A_0''\cap \Aff_{W,b}(\ideala(\bv,N/2)))\\
&\geq\frac{\#X}{\#\calQ}\cdot \#P\\
&\geq\left(\frac{1}{2^{n+1}\vph_K(W)}\delta' \cdot \frac{(M_{\natural}+M_{\natural}^a)^n-M_{\natural}^n}{\log M}\right) \cdot \left(\frac{2^{n+2}}{W^n}\cdot ((M_{\natural}+M_{\natural}^a)^n-M_{\natural}^n)\right)^{-1} \cdot (N+1)^n.
\end{align*}
In other words,
\begin{equation}\label{eq:densityshou}
\#(A_0''\cap \Aff_{W,b}(\ideala(\bv,N/2)))\geq \frac{W^n}{\vph_K(W)} \cdot \frac{1}{2^{2n+3}}\delta' \cdot \frac{(N+1)^n}{\log M}.
\end{equation}
We fix such a representative $b$. Since
$A_0''\cap \Aff_{W,b}(\ideala(\bv,N/2))\ne \varnothing$, we can choose
$x\in A_0''\cap \Aff_{W,b}(\ideala(\bv,N/2))$, which will also be fixed for the rest of the proof.
Note, in particular, that $x\in A$.
By \eqref{eq:suppori_hairu}, we have
$\|b\|_{\infty,\bv}- WN/2\leq \|x\|_{\infty,\bv}\leq M$
and $\|b\|_{\infty,\bv}\geq M_{\flat}-WN/2$.
Thus, if $M$ is sufficiently large depending on 
$a$, we see that \eqref{eq:log^2_bv} holds.

We now define
$\tilde{\lambda}\coloneqq\frac{\vph_K(W)}{W^n}(\lambda\circ\Aff_{W,x})$
and $B\coloneqq\Aff_{W,x}^{-1}(A_0'')\cap\ideala(\bv,N)$.
By \eqref{eq:NtoMtoMnatural} and the assumption on $\lambda$,
the function $\tilde{\lambda}$  is a $(\rho,\frac{\uvarsigma M^a}{W},S)$-pseudorandom measure.
By \eqref{eq:NtoMtoMnatural}, we have
$WN\leq \frac{M_{\natural}^a}{3}$. Since
$x\in A''_0$, we have
$\|x\|_{\infty,\bv}\geq M_{\natural}$. Thus
$3WN\leq \|x\|_{\infty,\bv}^a$ holds.
In particular,
\begin{equation}
A_0'' \cap \Aff_{W,b}(\ideala(\bv,N/2))\subseteq A_0''\cap \Aff_{W,x}(\ideala(\bv,N)) \subseteq A_0''\cap \ideala(\bv,x,\|x\|_{\infty,\bv}^a)\label{eq:dobleinclusions}
\end{equation}
holds. It follows from the first containment in 
\eqref{eq:dobleinclusions} and \eqref{eq:densityshou} that
\begin{equation}
\#B\geq \frac{W^n}{\vph_K(W)} \cdot \frac{1}{2^{2n+3}}\delta' \cdot \frac{(N+1)^n}{\log M}\label{eq:densityshouB}
\end{equation}
holds. We now treat the 
weighted density and smallness conditions.
By \eqref{eq:densityshouB} and the lower bound in \eqref{en:measurebelow_a}
in the \compatiaW, we obtain
\begin{equation}\label{eq:deltaD_1}
\EE(\ichi_B \cdot \tilde{\lambda}\mid \ideala(\bv,N))\geq \frac{D_1}{2^{3n+3}}\cdot \delta'.
\end{equation}
By \eqref{eq:NtoMtoMnatural}, 
$W\leq M^{\varepsilon a}$ and $M_{\flat}\leq M_{\natural}$, if
$M$ is sufficiently large depending on $\varepsilon$, $\bv,\delta,\Delta$ and $a$, then 
$N\geq M^{\frac{1-\varepsilon}{2}a}$ holds. This, together with the upper bound in
\eqref{en:measurebelow_a} in the \compatiaW, we obtain
\begin{equation}
\frac{1}{N}\cdot \EE(\ichi_B\cdot \tilde{\lambda}^{r+1}\mid \ideala(\bv,N))\leq D_2^{r+1}\cdot  \frac{(\log M)^{r+1}}{M^{\frac{1-\varepsilon}{2}a}},
\label{eq:gammahyouka}
\end{equation}
where $r\coloneqq\#S-1$.

In what follows, we will specify $\rho$,$\uvarsigma$,$\etaUpsilon$ and $M_{\mathrm{SI}}$ according to the arguments above.
First, we define
$\rho=\rho_{\mathrm{SI}}(D_1,\bv,\delta,\Delta,S) \coloneqq\rho_{\mathrm{RMS}}(\bv,\frac{D_1}{2^{3n+3}}\delta',S)$. Secondly, we define $\uvarsigma = \uvarsigma_{\mathrm{SI}}(\bv,\delta,\Delta,a)\coloneqq (D'(\bv,\Omega'))^a/4$ and  $\etaUpsilon = \etaUpsilon_{\mathrm{SI}}(\bv,\delta,\Delta)\coloneqq D'(\bv,\Omega')$.
Before defining $M_{\mathrm{SI}}$, we need some preparations. Let
$\gamma'$$=\gamma'_{\mathrm{SI}}(D_1,\bv,\delta,\Delta,S)>0$ be $\gamma' \coloneqq\gamma_{\mathrm{RMS}}(\bv,\frac{D_1}{2^{3n+3}}\delta',S)$.
Note that $M$ has been assumed to be sufficiently large up to this point. We further
make $M$ large enough in such a way that
$D_2^{r+1}\cdot  (\log M)^{r+1}\leq \gamma' \cdot M^{\frac{1-\varepsilon}{2}a}$
holds, and we set $M_{\mathrm{SI}}$ to be the smallest positive integer 
such that all these requirements hold for as long as $M\geq M_{\mathrm{SI}}$.
Note that $\|x\|_{\infty,\bv}\geq \etaUpsilon M$ and $N\geq \frac{\uvarsigma M^a}{W}$ by \eqref{eq:Mflat} and \eqref{eq:NtoMtoMnatural}. Hence by  \eqref{eq:deltaD_1} and \eqref{eq:gammahyouka},
Theorem~\ref{thm:RMST} can be applied to $B\subseteq \ideala(\bv,N)$ under
the hypothesis $M\geq M_{\mathrm{SI}}$.
It follows that there exists an $S$-constellation in $B$.
Applying $\Aff_{W,x}$, this leads to the existence of an $S$-constellation
in $A_0''\cap \Aff_{W,x}(\ideala(\bv,N))$.
Since $A_0''\subseteq \DD$, such an $S$-constellation contains no associate pairs.
By the second containment in \eqref{eq:dobleinclusions}, this implies the
existence of an  $S$-constellation in 
$A\cap \ideala(\bv,x,\|x\|_{\infty,\bv}^a)$
without associate pairs.

Finally, we estimate the number of $S$-constellations.
Since $\scrN_{S}(B)\leq \scrN_{S}^{\sharp}(A\cap \ideala(\bv,x,\|x\|_{\infty,\bv}^a))$,
it suffices to give a lower bound on $\scrN_{S}(B)$. By
Theorem~\ref{theorem=weighted-counting} and condition
\eqref{en:measurebelow_a} in the \compatiaW, we obtain
\[
\frac{1}{N(2N+1)^n} \cdot \scrN_{S}(B) \cdot (D_2\log M)^{r+1}\geq \gamma'.
\]
Thus, by setting
$\gamma=\gamma_{\mathrm{SI}}(D_1,D_2,\bv,\delta,\Delta,S,a)$ as $\gamma\coloneqq 2^{n} D_2^{-(r+1)}u^{n+1}\gamma' $, we deduce from \eqref{eq:NtoMtoMnatural} that 
\[
\scrN_{S}(B)\geq \gamma W^{-(n+1)}\cdot \frac{M^{a(n+1)}}{(\log M)^{r+1}}.
\]
This provides the desired estimate.
\end{proof}

\begin{proof}[Proof of Theorem~$\ref{theorem=package_infinite_a}$]
By \eqref{en:counting_infinite_a}, 
there exist $\delta>0$ and an increasing sequence 
$(M_l)_{l\in \NN}$ of real numbers with
$\lim\limits_{l\to \infty}M_l=\infty$ such that, 
for all $l\in \NN$,
\begin{equation}\label{eq:Ml}
\#(A\cap \ideala(\bv,M_l))\geq \delta \cdot \frac{M_l^n}{\log M_l}
\end{equation}
holds.
By \eqref{en:counting_above_a}, 
there exists $\Delta>0$ such that 
\eqref{eq:ideal_Delta_S9} holds.

Fix $a\in (0,1)_{\RR}$ and a standard shape
$S\subseteq \ideala$.
By \eqref{en:logpseudo_infinite_a}, we can choose 
$D_1,D_2>0$ and
$\varepsilon\in (0,1)_{\RR}$ such that 
the property described in 
Definition~\ref{definition=logpseudorandom_infinite_a} holds.
Then take $\rho>0$ and $M_{\mathrm{SI}}\in\NN$
such that the conclusion of Theorem~\ref{theorem=package_a}
holds. In order to apply this conclusion, 
choose $l\in\NN$ in such a way that 
$M_l\geq \max\{M_{\mathrm{SI}},
M_{A,\mathrm{S}\Psi^{\mathrm{SI}}_{\log}}(\rho,\bv,S,a)
\}$
and set
$A'\coloneqq A\cap\ideala(\bv,M_l)$. Then 
by \eqref{eq:Ml},
\[\#A'\geq \delta \cdot \frac{M_l^n}{\log M_l}\]
holds.
This implies that 
condition \eqref{en:counting_a} of Theorem~\ref{theorem=package_a}
is fulfilled with $(A,M)$ replaced by $(A',M_l)$.
Since \eqref{eq:ideal_Delta_S9} is valid even if $A$ is replaced by $A'$,
condition \eqref{en:counting_above} of Theorem~\ref{theorem=package_a}
is fulfilled with $A$ replaced by $A'$.
Finally, since 
$M_l\geq M_{A,\mathrm{S}\Psi^{\mathrm{SI}}_{\log}}(\rho,\bv,S,a)$,
$A'$ satisfies the $(\rho,M_l,\bv,S,a)$-condition with parameters
$(D_1,D_2,\varepsilon)$.
This implies that 
condition \eqref{en:logpseudorandom_a} of Theorem~\ref{theorem=package_a}
is fulfilled with $(A,M)$ replaced by $(A',M_l)$.
Thus, by Theorem~\ref{theorem=package_a}, there exists 
$x_{(a,S)}\in A'$ 
such that
$A'\cap \ideala(\bv,x_{(a,S)},\|x_{(a,S)}\|_{\infty,\bv}^a)$
contains an $S$-constellation without associate pairs.

In order to complete the proof of the theorem, we employ the
diagonal argument.
Fix a decreasing sequence $(a_l)_{l\in\NN}$ in $(0,1)_{\RR}$
with $\lim\limits_{l\to \infty}a_l=0$.
Since $\ideala$ is countable, we can take a sequence
$S_1\subseteq S_2\subseteq\cdots$ of standard shapes in $\ideala$ such that
$\bigcup_{l\in\NN}S_l=\ideala$. 
Applying the above argument to
$(a_l,S_l)$ for each $l\in\NN$, we find
a sequence $(y_l)_{l\in\NN}$
of elements in $A'$ 
such that
$A'\cap \ideala(\bv,y_l,\|y_l\|_{\infty,\bv}^{a_l})$
contains an $S_l$-constellation without associate pairs.
It remains to show that the sequence
$(y_l)_{l\in\NN}$ satisfies the desired property.
Indeed, let
$a\in (0,1)_{\RR}$ and a standard shape $S\subseteq \ideala$
be arbitrary. Then there 
the set $L\coloneqq\{l\in\NN: a<a_l\text{ or }S\not\subseteq S_l\}$
is finite. Then for all $l\in\NN\setminus L$, 
we have $a\geq a_l$ and $S\subseteq S_l$.
Since 
$A'\cap \ideala(\bv,y_l,\|y_l\|_{\infty,\bv}^{a_l})$
contains an $S_l$-constellation without associate pairs,
we see that
$A\cap \ideala(\bv,y_l,\|y_l\|_{\infty,\bv}^{a})$
contains an $S$-constellation without associate pairs.
\end{proof}

We have completed the necessary axiomatization.
The proofs of Theorem~\ref{theorem=shortinterval} 
and Corollary~\ref{corollary=shortinterval} are now within reach.

\begin{proposition}\label{proposition=PK_kouri_a}
For a number field $K$, the set
$\PP_K$ satisfies all the assumptions of Corollary~$\ref{corollary=package_infinite_a}$.
\end{proposition}
\begin{proof}
By Proposition~\ref{proposition=PK_kouri},
condition \eqref{en:counting_infinite_X_a} holds.
By Theorem~\ref{theorem=Chebotarev}~\eqref{Landau},
condition \eqref{en:counting_above_X_a} holds.
Finally, 
by Theorem~\ref{theorem=logpseudorandom_a}, 
condition \eqref{en:logpseudo_infinite_X_a} holds.
\end{proof}

\begin{proof}[Proof of Theorem~$\ref{theorem=shortinterval}$]
We may assume without loss of generality that
$S$ is a standard shape.
Recall that we have fixed $\chi$ to obtain
$W_{\PP_K,S\Psi_{\log}^{\mathrm{SI}}}(\rho,S)=W_{\PP_K,S\Psi_{\log}^{\mathrm{SI}}}(\rho,\chi,S)$
and
$M_{\PP_K,\mathrm{S}\Psi^{\mathrm{SI}}_{\log}}(\rho, \uvarsigma, \omom, S, a) = M_{\PP_K,\mathrm{S}\Psi^{\mathrm{SI}}_{\log}}(\rho, \uvarsigma, \omom, \chi, S, a)$ in the proof of Theorem~\ref{theorem=logpseudorandom_a}.
Recall $\kappa$ from Theorem~\ref{theorem=zeta_K}.
Let 
$r\coloneqq\#S-1$.
Set $D=D_{\PP_K}(K,r,a)\coloneqq\kappa a\cdot(17(r+1)2^r\cdot c_{\chi})^{-1}$. Set $\rho\coloneqq\rho_{\mathrm{SI}}(D,\omom,\delta\cdot C_{\PP_K,\rmI}(\omom),C_{\mathrm{Lan}},S)$, $\uvarsigma\coloneqq \uvarsigma_{\mathrm{SI}}(\omom,\delta\cdot C_{\PP_K,\rmI}(\omom),C_{\mathrm{Lan}},a)$, and $\etaUpsilon\coloneqq\etaUpsilon_{\mathrm{SI}}(\omom,\delta\cdot C_{\PP_K,\rmI}(\omom),C_{\mathrm{Lan}})$. Define $M_{\mathrm{PESSI}}(\omom,\delta,S,a)$ by
\begin{align*}
&M_{\mathrm{PESSI}}(\omom,\delta,S,a)\\
&\coloneqq\max\{M_{\PP_K,\rmI}(\omom),M_{\PP_K,\mathrm{S}\Psi^{\mathrm{SI}}_{\log}}(\rho,\uvarsigma,\omom,S,a),M_{\mathrm{SI}}(D,D,3/4,\omom,\delta\cdot C_{\PP_K,\rmI}(\omom),C_{\mathrm{Lan}},S,a)\};
\end{align*}
by Proposition~\ref{proposition=PK_kouri}, for $M\geq M_{\PP_K,\rmI}(\omom)$ and for $A$ with $\#A\geq\delta\cdot\#(\PP_K\cap\OK(\omom,M))$, 
\[
\#A\geq\delta\cdot C_{\PP_K,\rmI}(\omom)\cdot \frac{M^n}{\log M}
\]
holds true.
Here, $C_{\mathrm{Lan}}$ is as in \eqref{eq:Landau}.
Suppose that $M\geq M_{\mathrm{PESSI}}$.
It then follows from the proof of 
Theorem~\ref{theorem=logpseudorandom_a} and Lemma~\ref{lemma=stronglog}~\eqref{en:logsubset}
that $A$ satisfies 
the $(\rho,\uvarsigma, M,\omom,S,a)$-condition with parameters $(D,D,3/4)$. 
Thus, we can apply
Theorem~\ref{theorem=package_a} 
with $\Delta=C_{\mathrm{Lan}}$, $D_1=D_2=D$ and 
$\varepsilon=3/4$ to conclude that 
there exists $x\in A$ 
such that $A\cap\OK(\omom,x,\|x\|_{\infty,\omom}^a)$
contains an $S$-constellation without associate pairs.
This proves \eqref{en:tan_seiza}.

As for \eqref{en:tan_seiza_kosuu}, 
employ $W=W_{\PP_K,S\Psi_{\log}^{\mathrm{SI}}}(\rho,S)$ as in Theorem~\ref{theorem=logpseudorandom_a}. Let
\[
\gamma_{\mathrm{PESSI}}(\omom,\delta,S,a)\coloneqq \gamma_{\mathrm{SI}}(D,D,\omom,\delta\cdot C_{\PP_K,\rmI}(\omom),C_{\mathrm{Lan}},S,a)\cdot W^{-(n+1)}.
\]
Then, Theorem~\ref{theorem=package_a} provides the desired estimate of  $\scrN_S^{\sharp}(A\cap \OK(\bv,x,\|x\|_{\infty,\bv}^a))$.
\end{proof}

\begin{proof}[Proof of Corollary~$\ref{corollary=shortinterval}$]
Immediate from
Corollary~\ref{corollary=package_infinite_a} and 
Proposition~\ref{proposition=PK_kouri_a}.
\end{proof}

\subsection{Constellations consisting of elements whose norms are close}\label{subsection=close_norm}

In this subsection, we exhibit one application of our constellation theorems for short intervals. For simplicity, we only state the infinitary version.

\begin{theorem}\label{theorem=package_infinite_a_close}
We use Setting~$\ref{setting=package}$.
Assume that a subset $A\subseteq \ideala\setminus\{0\}$
satisfies the three conditions in Theorem~$\ref{theorem=package_infinite_a}$. 
Then, there exists a sequence 
$(\eT_l)_{l\in \NN}$ of pairwise disjoint finite subsets in $A$ with no associate pairs satisfying the following. For every $a\in (0,1)_{\RR}$, every finite set
$S\subseteq \ideala$ and every $\eta>0$, 
there exists a finite subset $L\subseteq\NN$ such that for all $l\in\NN\setminus L$, the following hold true.
\begin{enumerate}[$(1)$]
  \item\label{en:seiza_aa}
The set $\eT_l$ contains an $S$-constellation.
  \item\label{en:close_aa}
For every $\alpha_1,\alpha_2\in \eT_l$,
\[
\frac{\Nrm(\alpha_2)}{\Nrm(\alpha_1)}\leq 1+\eta\cdot (\min\{\Nrm(\alpha)\colon \alpha\in \eT_l\})^{\frac{a-1}{n}}
\]
holds.
\end{enumerate}
\end{theorem}

\begin{proof}
Let $\sigma_1,\dots,\sigma_n$
be the embeddings of $K$ into $\CC$. 
Recall that there exists a constant $D=D_{\bv}>0$, depending only on $\bv$, such that for every $\alpha\in \ideala$ and for each $i\in [n]$, we have $|\sigma_i(\alpha)|\leq D \|\alpha\|_{\infty,\bv}$; see Lemma~\ref{lemma=NLCreversed}.
Take an arbitrary $\delta>0$ which is strictly smaller than the limit supremum appearing in condition~\eqref{en:counting_infinite_a} of Theorem~\ref{theorem=package_infinite_a}.
Take $\Delta>0$ as in condition~\eqref{en:counting_above_a} of Theorem~\ref{theorem=package_infinite_a}.
From $\delta$ and $\Delta$, set $\Omega'=\Omega'_{\mathrm{red}}(\bv,\delta,\Delta)>0$ as in Theorem~\ref{theorem=fundamental_Omega}. Take the constant $D(\bv,\Omega')>0$ as in Lemma~\ref{lemma=girigiri_chop}  and $C'(\bv)>0$ as in its proof.

Consider a sequence $(a_l,S_l,\eta_l)_{l\in \NN}$, where $(a_l)_{l\in \NN}$ is a decreasing sequence in $(0,1)_{\RR}$ with $\lim\limits_{l\to \infty}a_l=0$ and $(\eta_l)_{l\in \NN}$ is a decreasing sequence of positive real numbers with $\lim\limits_{l\to \infty}\eta_l=0$. The sequence $(S_n)_{n\in \NN}$ comes from a filtration $S_1\subseteq S_2\subseteq \cdots$ of standard shapes of $\ideala$ with $\bigcup_{l\in\NN}S_l=\ideala$. 
 
Apply the proof of Theorem~\ref{theorem=package_a} to this setting. Then, we obtain a strictly increasing sequence $(M_l)_{l\in \NN}$ in $\RR_{\geq 1}$ with $\lim\limits_{l\to \infty}M_l=\infty$, a sequence $(y_l)_{l\in \NN}$ in $A$, and a sequence of sets $(A''_0(M_l))_{l\in \NN}$ such that for every $l\in \NN$, the following conditions are fulfilled:
\begin{enumerate}[(a)]
  \item\label{en:shokukan} $A''_0(M_l)\subseteq A\cap \ideala(\bv,M_l)\cap \ideala(\bv,y_l,\frac{\eta_l}{2}\|y_l\|_{\infty,\bv}^{a_l})$,
  \item\label{en:normchikai} for every $\alpha\in A''_0(M_l)$, $\Nrm(\alpha)\geq \Omega' M_l^n$,
  \item\label{en:seizaaruyo} $A''_0(M_l)$ admits no associate pairs, and it contains an $S_l$-constellation.
\end{enumerate}
Indeed, for \eqref{en:shokukan}, consider $(a_l/2)_{l\in \NN}$ instead of $(a_l)_{l\in \NN}$, and take sufficiently large $M_l$ to obtain the $\eta_l$-factor. Moreover, we can take in such a way that $(A''_0(M_l))_{l\in \NN}$ is pairwise disjoint. 

Now, for each $l\in\NN$, we set $\eT_l\coloneqq A''_0(M_l)$. What remains is to verify \eqref{en:close_aa}. Fix $l\in \NN$, and take $\alpha\in \eT_l$.  By Lemma~\ref{lemma=girigiri_chop}, we have for each $i\in [n]$
\begin{equation}\label{eq:sigma_dekai}
|\sigma_i(\alpha)|\geq D(\bv,\Omega')\cdot M_l.
\end{equation}
Now take $\alpha_1,\alpha_2\in \eT_l$. Then, since $\|\alpha_1-\alpha_2\|_{\infty,\bv}\leq \eta_lM_l^{a_l}$ by \eqref{en:shokukan}, we have for each $i\in [n]$, 
\begin{equation}\label{eq:sigma_sachiisai}
|\sigma_i(\alpha_1)-\sigma_i(\alpha_2)|\leq \eta_lC'(\bv) M_l^{a_l}.
\end{equation}
By combining Lemma~\ref{lemma=idealnorm}, \eqref{eq:sigma_dekai} and \eqref{eq:sigma_sachiisai}, we have
\[
\frac{\Nrm(\alpha_2)}{\Nrm(\alpha_1)}\leq \prod_{i\in [n]}\left(1+\frac{|\sigma_i(\alpha_1)-\sigma_i(\alpha_2)|}{|\sigma_i(\alpha_1)|}\right)\leq \left(1+\frac{\eta_lC'(\bv)}{D(\bv,\Omega')}\cdot M_l^{a_l-1} \right)^n.
\]
By \eqref{en:normchikai}, for every $a\in (0,1)_{\RR}$ and $\eta>0$, if $l$ is sufficiently large depending on $\bv,A,a$ and $\eta$, then we have for all $\alpha_1,\alpha_2\in \eT_l$,
\[
\frac{\Nrm(\alpha_2)}{\Nrm(\alpha_1)}\leq 1+\eta\cdot \left(\min_{\alpha\in \eT_l}\Nrm(\alpha)\right)^{\frac{a-1}{n}}.
\]
Hence we confirm \eqref{en:close_aa}, and the proof is completed.
\end{proof}

%% file: chapter10.tex
\section{Constellations in ideals and quadratic forms}\label{section=quadraticform}

The goal of the present section is to establish Theorem~\ref{mtheorem=quadraticform};
we in fact prove its refinement, Theorem~\ref{theorem=quadraticformcloseprimes}.
For a binary quadratic form $F\colon \ZZ^2\to \ZZ; F(x,y)\coloneqq ax^2+bxy+cy^2$ with integral coefficients, recall that $F$ is said to be \emph{primitive} if $\mathrm{gcd}(a,b,c)=1$, and that $D_F$ denotes the discriminant $b^2-4ac$. We say that $F$ is \emph{non-degenerate} if $D_F$ is not a perfect square. If $D_F$ is a perfect square, then $F$ decomposes into the product of two linear polynomials in the polynomial ring $\ZZ[x,y]$; such a form can take prime values only if one of the two factors is equal to $\pm 1$, and hence it is not of our interest. Recall also the definition of the relative upper asymptotic density
$\overline{d}_{X,\bv}(A)$
from Definition~\ref{definition=reldens_c}. 

\begin{theorem}[Refinement of Theorem~\ref{mtheorem=quadraticform}]\label{theorem=quadraticformcloseprimes}
Let $F(x,y)\coloneqq ax^2+bxy+cy^2 \in \ZZ[x,y]$ be a non-degenerate primitive quadratic form.
Assume that $a>0$.
Let $\boldsymbol{u}$ be the standard basis of $\ZZ^2$.
\begin{enumerate}[$(1)$]
\item\label{en:positiveprime_again} Let $A\subseteq F^{-1}(\PP)$ be a set with $\overline{d}_{F^{-1}(\PP),\boldsymbol{u}}(A)>0$. 
Then, for a finite set $S\subseteq \ZZ^2$, there exists an $S$-constellation in $A$.
\item\label{en:negativeprime_again}
Assume that $D_F>0$.
Let $A\subseteq F^{-1}(-\PP)$ be a set with $\overline{d}_{F^{-1}(-\PP),\boldsymbol{u}}(A)>0$. 
Then, for a finite set $S\subseteq \ZZ^2$, there exists an $S$-constellation in $A$.
\item\label{en:close_again} In both \eqref{en:positiveprime_again} and  \eqref{en:negativeprime_again}, the following furthermore holds true for $A$: there exists a sequence of pairwise disjoint finite subsets $(\eT_l)_{l\in \NN}$ of $A$ which fulfills the following two conditions.

\begin{enumerate}[$(a)$]
  \item\label{en:kotonaru_prime_again} For every $l\in \NN$, $F\mid_{\eT_l}\colon \eT_l\to \PP_{\QQ}$ is injective. %
  \item\label{en:chikai_prime_again} For every $\theta\in (0,1)_{\RR}$, for every finite subset $S \subseteq \ZZ^2$ and for every $\eta>0$, there exists a finite subset $L\subseteq \NN$ such that for every $l\in \NN\setminus L$, the following hold:
\begin{enumerate}[$(b1)$]
  \item\label{en:chikai1} $\eT_l$ contains an $S$-constellation,
  \item for every $p_1,p_2\in F(\eT_l)$,
\[
\frac{|p_2|}{|p_1|}\leq 1+\eta\cdot (\min\{|p|\colon p\in F(\eT_l)\})^{\frac{\theta-1}{2}}
\]
holds. \label{en:chikai2}
\end{enumerate}
\end{enumerate}
\end{enumerate}
\end{theorem}
\begin{remark}\label{remark=exponent} %
The exponent $\frac{\theta-1}{2}$ in condition~($b2$) in Theorem~\ref{theorem=quadraticformcloseprimes}~\eqref{en:close_again} can \emph{not} be improved to any constant less than $-\frac{1}{2}$ in general. Indeed, 
consider the case where $F(x,y)=x^2+y^2$ and $S=\{(0,0),(1,0),(0,1)\}$. Note that if $|x|\geq |y|$ and if $0<d<|x|$, then
\[
\left|\frac{F(x+d,y)}{F(x,y)}-1\right|\geq \frac{|d|}{2}\cdot (F(x,y))^{-\frac{1}{2}}.
\]
\end{remark}

We exhibit key theorems to the proof of Theorem~\ref{theorem=quadraticformcloseprimes} in Subsection~\ref{subsection=outline10}, where we describe the organization of the present section.

\subsection{Strategy for the proof of Theorem~\ref{theorem=quadraticformcloseprimes}}\label{subsection=outline10}

In this subsection, we describe the strategy for the proof of Theorem~\ref{theorem=quadraticformcloseprimes}. 
The first key to the proof is the following correspondence between primitive quadratic forms and invertible fractional ideals of orders. 
This is a classical result from the times of Gauss, Dirichlet and Dedekind. In the present paper, we exhibit a proof of Theorem~\ref{th:quadratic-forms-ideals} in the appendix for the convenience of the reader; see Theorem~\ref{th:quad-classical}. See also \cite[Subsections~3.2 and 3.3]{Bhargava04}.

\begin{theorem}\label{th:quadratic-forms-ideals}
Let $F(x,y)=ax^2+bxy+cy^2\in \ZZ[x,y]$ be a primitive and non-degenerate binary quadratic form. Then there exist an order $\Or$ in $K\coloneqq \QQ (\sqrt{D_F})$, 
an invertible fractional ideal $\idealc $ of $\Or$,
    a basis $(\gamma_1 ,\gamma_2)$ of $\idealc $ as a $\ZZ $-module,
    and a sign $\signF \in \{ \pm 1 \} $
    such that the following identity holds:
    \begin{equation}\label{eq:quadratic-forms-ideals}
        F(x,y) = \signF\cdot \frac{\Nelm (\gamma_1 x + \gamma_2 y )}{\Nrm (\idealc ) } \quad \text{ for all $(x,y)\in \ZZ ^2$}.
    \end{equation}
\end{theorem}

We will present the definitions of \emph{orders} and \emph{invertible fractional ideals} in Subsection~\ref{subsection=invideal}; for an invertible ideal $\idealc$ of an order $\Or$, the definition of the \emph{ideal norm} $\Nrm(\idealc)$ of $\idealc$ will be provided in Subsection~\ref{subsection=idealnorm}. Throughout this section, we use the symbols $\idealc$ and $\ideald$ for invertible fractional ideals of an order, and $\ideala$ and $\idealb$ for non-zero fractional ideals of $\OK $.

Inspired by Theorem~\ref{th:quadratic-forms-ideals}, we define the following subsets of fractional ideals of $\OK$ and $\Or$, which may be regarded as  counterparts of the set $\PP_K$ of prime elements. 
Indeed, if $\ideala =\OK $, then the set $\PP _{\OK }$ coincides with $\PP_K$.

\begin{definition}\label{def:of-set-P}
Let $K$ be a number field, and $\OK$ the ring of integers of $K$. Let $\Or$ be an order in $K$.
\begin{enumerate}[(1)]
\item\label{en:PPa}
    Let $\ideala$ be a non-zero fractional ideal of $\OK $.
        Define the set $\PP_{\ideala} \subseteq \ideala $ by 
        \[
            \PP_{\ideala }\coloneqq\{ \alpha \in \ideala : \text{the ideal $\alpha \ideala ^{-1}\subseteq\OK $ is a prime ideal}  \} ,
        \]
        where $\ideala ^{-1}$ is the inverse fractional ideal of $\ideala $ (see the discussion before Theorem~\ref{theorem=primeideals_frac}
        for this concept).
\item\label{en:PPc}
     Let $\idealc $ be an invertible fractional ideal of $\Or$ (Definition~\ref{def:invertible-fractional-ideal} below). Define the set $\PP_{\idealc} \subseteq \idealc $ by 
     \[ 
         \PP_{\idealc } \coloneqq \PP_{\idealc \OK} \cap \idealc ,
     \]
     where $\idealc \OK \subseteq K $ denotes the (non-zero) fractional ideal of $\OK $ generated by $\idealc $.
\end{enumerate}
\end{definition}

With these definitions, we will establish the following theorem. 
As it turns out, this will immediately imply Theorem~\ref{theorem=quadraticformcloseprimes}.

\begin{theorem}\label{th:constellations-in-ideals}
    Let $K$ be a number field of degree $n$ and $\Or$ an order in $K$.
    Let $\idealc$ be an invertible fractional ideal of $\Or $ and $\bw$ be a $\ZZ$-basis of $\idealc$.
    Let $\PP_{\idealc} \subseteq \idealc $ 
    be the set 
    defined in Definition~$\ref{def:of-set-P}$~\eqref{en:PPc}. 
    Assume that $A\subseteq \PP_{\idealc }$ satisfies that $\overline{d}_{\PP_{\idealc },\bw}(A)>0$. 
Then there exists a sequence of pairwise disjoint finite subsets $(\eT_l)_{l\in \NN}$ of $A$ such that the following holds: for every $\theta\in (0,1)_{\RR}$, for every finite subset $S \subseteq \idealc$, and for every $\eta>0$, there exists a finite subset $L\subseteq \NN$ such that for every $l\in \NN\setminus L$, the following hold.
\begin{enumerate}[$(1)$]
  \item\label{en:seiza_aruyo} $\eT_l$ contains an $S$-constellation,
  \item\label{en:saga_sukunaiyo} for every $\alpha_1,\alpha_2\in \eT_l$,
\[
\frac{|\Nelm(\alpha_2)|}{|\Nelm(\alpha_1)|}\leq 1+\eta\cdot (\min\{|\Nelm(\alpha)| : \alpha\in \eT_l\})^{\frac{\theta-1}{n}}
\]
holds.
\end{enumerate}
\end{theorem}

Theorem~\ref{th:constellations-in-ideals}, in fact, can derive a result for the norm forms associated with the triple $(\Or,\idealc,\bw)$; %
see Setting~\ref{setting=normform} for the definition of the associated norm form.
We will state the precise statement in Theorem~\ref{theorem=normform}.
Elsholtz and Frei \cite{Elsholtz-Frei19} have studied prime numbers represented by norm forms. 
See also the work of Maynard \cite{maynard_2020}.
The novel point of our work is that we study combinatorics for the \emph{set of tuples} $(x_1,x_2,\ldots ,x_n)$ 
at which the norm form represents primes. %

In Subsections~\ref{subsection=invideal} and \ref{subsection=idealnorm}, we recall basics %
of the theory of invertible fractional ideals
over orders.
Proposition~\ref{theorem=reduction_order} is the goal in these two subsections. 
In Subsection~\ref{subsection=SPsiSI}, we prove that $\PP_{\ideala}\in \logpseua$ for every  $\ideala\in \Ideals_K$ with the aid of Theorem~\ref{theorem=GYfordieals}. In Subsection~\ref{subsection=counting_Cheb}, we estimate the sizes of certain subsets of $\PP_{\idealc}$ by employing a version of the Chebotarev density theorem (Theorem~\ref{theorem=Chebotarev_narrow}). We prove Theorem~\ref{th:constellations-in-ideals} in Subsection~\ref{subsection=proof_orders}. Finally, in Subsection~\ref{subsec:proof-constellations-quadratic-forms}, we establish Theorem~\ref{theorem=quadraticformcloseprimes}, 
which in turn implies Theorem~\ref{mtheorem=quadraticform}.

\subsection{Preliminaries on invertible ideals of orders}\label{subsection=invideal}
Here we review generalities of invertible fractional ideals of orders.
The reader with an algebraic background may skip this and the next subsections.

For a number field $K$ of degree $n$, an \emph{order} $\Or$ in $K$ is a  subring of $\OK$ 
which
is isomorphic to $\ZZ^n$ as a $\ZZ$-module. In particular, $\OK$ itself is an example of an order in $K$; this is why it is also called the \emph{maximal order}. One example of a non-maximal order is $\ZZ[\sqrt{5}]$, where  $K=\QQ(\sqrt{5})$. For this $K$, the maximal order $\OK$ is $\ZZ[\frac{1+\sqrt{5}}{2}]$. Orders in number fields are always %
$1$-dimensional Noetherian integral domains; see for instance \cite[Chapter~I, Proposition 12.2]{Neukirch}.
Recall that an (always associative commutative unital) ring is said to be {\em Noetherian} if every non-empty set of its ideals has a maximal element with respect to the inclusion relation.
For a ring, the condition that it is a $1$-dimensional integral domain is equivalent to saying that it is not a field and its only prime ideals are the zero ideal and the non-zero maximal ideals. %

\begin{lemma}\label{lem:the-only-property-we-need}
    Let $\thering $ be a $1$-dimensional Noetherian integral domain.
    For every non-zero ideal $\ideald \subseteq \thering $,
    there exist finitely many maximal ideals
    $\idealp_1,\dots,\idealp_s$, not necessarily distinct,
    such that $\idealp _1\cdots \idealp _s \subseteq \ideald$ holds.
\end{lemma}
\begin{proof}
    Let $\Phi $ be the set of non-zero ideals which do not satisfy the claimed property.
    By way of  contradiction, suppose $\Phi $ is non-empty.
    Then by the Noetherian assumption, it has a maximal element. We write $\ideald $ for it. 
   Note that $\ideald $ is not equal to $\thering $
   or a maximal ideal because they trivially satisfy the claimed condition.
        Since $\thering $ is a $1$-dimensional integral domain,
    it follows that $\ideald $ is not a prime ideal.
        Hence there exist elements $a,b\in \thering \setminus \ideald $ with $ab\in \ideald $.
    The two ideals $a\Or+\ideald $ and $b\Or+\ideald $ are strictly larger than $\ideald $. By the maximality of $\ideald $, they do not belong to $\Phi $.
    By the definition of $\Phi $, there exist maximal ideals $\idealp _1,\dots ,\idealp _r , \idealp _{r+1},\dots, \idealp _s$ such that $\idealp _1\cdots \idealp _r \subseteq a\Or+\ideald $ and that $\idealp _{r+1}\cdots \idealp _s \subseteq b\Or+\ideald$. Take the product of the two inclusions above.
The right-hand side of the product is contained in $\ideald$ by the relation $ab\in \ideald $. Hence we obtain $\idealp _1\cdots \idealp _s \subseteq \ideald$. This contradicts $\ideald \not\in \Phi $, and we conclude that $\Phi $ is empty.
\end{proof}

Also recall the following general fact.
\begin{lemma}\label{lem:power-maximal-ideal}
    Let $\idealp $ and $\idealq $ be maximal ideals in a given ring.
    If there exists a positive integer $e>0$ such that $\idealp ^e \subseteq \idealq $ holds,
    then we have $\idealp =\idealq $.
\end{lemma}
\begin{proof}
    Since $\idealq $ is in particular a prime ideal,
    from the given inclusion we have $\idealp \subseteq \idealq $.
    Then by the fact that $\idealp $ is maximal,
    we obtain the equality $\idealp = \idealq $.
\end{proof}

Now we recall the definition of invertible fractional ideals of an order.
\begin{definition}\label{def:invertible-fractional-ideal}
   Let $K$ be a number field, and $\Or$ an order in $K$. A non-zero $\thering $-submodule $\idealc$ of  $K$ is called an {\em invertible fractional ideal} of $\thering $ if there exists a  non-zero $\thering $-submodule $\ideald$ of $K$ such that $\idealc \ideald =\thering$.
    Here, the left-hand side is defined to be the $\thering $-submodule of $K$ generated by the set $\{ cd : c\in \idealc ,\ d\in \ideald \}$.
    Such a $\ideald $ is uniquely determined by $\idealc$, and is called the {\em inverse fractional ideal} of $\idealc $. It is written as $\idealc ^{-1}$.
    If moreover $\idealc \subseteq \Or$ holds, then we say that $\idealc$ is an \emph{invertible ideal} of $\Or$.
\end{definition}

In view of Lemma~\ref{lem:monogenic-modulo-p^e} \eqref{item:finitely-generated} below, the use of the adjective `fractional' in the term `invertible fractional ideal' is consistent with Section~\ref{section=preliminarynumbertheory}, where fractional ideals of $\OK $ meant finitely generated $\OK $-submodules of $K$.
By \cite[Chapter~I, Proposition~3.8]{Neukirch}, every non-zero fractional ideal $\ideala $ of $\OK $ is invertible as was pointed out in the discussion before Theorem~\ref{theorem=primeideals_frac}. 

\begin{lemma}\label{lem:monogenic-modulo-p^e}
Let $K$ be a number field, $\Or$ an order in $K$. 
Let $\idealc $ be an invertible fractional ideal of $\thering $. 
    \begin{enumerate}[$(1)$]
        \item\label{item:finitely-generated} The $\Or$-module $\idealc$ is finitely generated. %
        \item\label{item:p^e}
        Let $\idealp $ be a maximal ideal. 
        Then there exists an element $c\in \idealc $ such that 
        for every ideal $\ideald $ containing a power $\idealp ^e$ of $\idealp$, where $e\in \NN$,
        the multiplication map
        $\thering \ni x \mapsto cx \in \idealc $ induces a bijection 
        \begin{equation}\label{eq:monogenic-modulo-p^e}
            \thering / \ideald  \xrightarrow \simeq \idealc / \idealc \ideald .
        \end{equation}
    \end{enumerate}
\end{lemma}
\begin{proof}
    By  $\idealc \cdot \idealc ^{-1}=\thering $,
there exist elements $c_1,\dots ,c_r \in \idealc $ and
    $d_1,\dots ,d_r\in \idealc ^{-1}$ such that we have
    \begin{equation}\label{eq:sum-ab-1}
        c_1d_1+\dots +c_rd_r =1 . 
    \end{equation}
    Multiply this by an arbitrary $c\in \idealc$ to get $c_1(d_1c) + \dots + c_r(d_rc) = c $.
    Since 
    we have $d_ic\in \Or $ for all $i\in [r]$,
    this shows that $\idealc $ is generated by the elements $c_1,\dots ,c_r$. This proves \eqref{item:finitely-generated}.

    To show \eqref{item:p^e}, in \eqref{eq:sum-ab-1} note that for all $i\in [r]$, $c_i d_i \in\thering $ holds. It follows that there exists $i\in [r]$ such that $c_id_i \in \thering \setminus \idealp $. Fix such $i$.
    We claim that $c\coloneqq c_i $ is an element with the desired property. To prove this claim, first observe  the following equality of 
    ideals
    of $\thering $: $c \idealc ^{-1} + \ideald  = \thering$.
    Indeed, suppose that it is not the case. Then, there must exist a maximal ideal $\idealq $ of $\thering $ containing the left-hand side.
    Since $c d_i \in c\idealc ^{-1}$ is not contained in $\idealp $, we \havethat\   $\idealq \neq \idealp $.
    However, by Lemma~\ref{lem:power-maximal-ideal}, 
    $\idealp $ is the only maximal ideal that can contain $\ideald $, a contradiction.
    Thus we obtain $c \idealc ^{-1} + \ideald  = \thering$.
    By multiplying this by $\idealc $, we have $c \thering +\idealc \ideald = \idealc$ as an equality of submodules of $\idealc $.
    This is equivalent to saying that 
    the map \eqref{eq:monogenic-modulo-p^e} is surjective.

    For the injectivity,
    apply \eqref{eq:Chinese-ideals} to $c \idealc ^{-1} + \ideald  = \thering$, and  \obtainthat\  $c \idealc ^{-1} \cap \ideald = c \idealc ^{-1}\cdot \ideald$. 
Since $\idealc$ is an invertible fractional ideal, 
we \havethat\  $c \thering \cap \idealc \ideald  = c \ideald $.
    This implies that if an element of the form $c d$ with $d\in \thering $ belongs to $\idealc\ideald $, then $d$ is necessarily in $\ideald $.
    It is equivalent to the injectivity of the map \eqref{eq:monogenic-modulo-p^e}.
\end{proof}

\begin{proposition}\label{prop:monogenic-modulo-b}
Let $K$ be a number field, and $\Or$ an order in $K$.
Let $\idealc $ be an invertible fractional ideal of $\thering $
    and $\ideald \subseteq \thering $ a non-zero ideal.
    Then the $\thering $-module $\idealc /\idealc \ideald $ is 
    isomorphic to $\thering / \ideald $.
    In particular, it is 
    generated by a single element.
\end{proposition}
\begin{proof}
Recall that $\Or$ is a $1$-dimensional Noetherian integral domain.    By Lemma \ref{lem:the-only-property-we-need}, there exist distinct maximal ideals
    $\idealp _1,\dots ,\idealp _s$ of $\thering $ and an exponent $e\ge 0$ such that $(\idealp _1\cdots \idealp _s )^e\subseteq \ideald$ holds. 
    For each $i\in [s]$, set $\ideald _i \coloneqq \ideald + \idealp _i^e $.
    By Lemma \ref{lem:power-maximal-ideal}, the ideals $\ideald _i$, $i\in [s]$, are mutually coprime.
    Then by \eqref{eq:Chinese-ideals} we have $\bigcap_{i\in[s]} \ideald _i = \prod_{i\in[s]} \ideald _i $. Since $(\idealp _1\cdots \idealp _s )^e\subseteq \ideald$, 
    we conclude 
    \[
        \bigcap_{i\in[s]} \ideald _i
        = \prod _{i\in [s]} (\ideald +\idealp _i^e)
        \subseteq \ideald .
    \]
    On the other hand, we trivially have $\ideald \subseteq \bigcap_{i\in [s]} \ideald _i$. Therefore, $\ideald = \bigcap _{i \in [s]} \ideald _i$ holds.
    Hence by the Chinese remainder theorem (formula~\eqref{eq:Chinese-modules}),
    we have an isomorphism of $\thering$-modules $\idealc/\idealc\ideald \simeq \prod_{i \in [s]} \idealc / \idealc \ideald_i.$
    By Lemma \ref{lem:monogenic-modulo-p^e}, for each $i\in [s]$, the $i$-th factor on the right-hand side is 
    isomorphic to $\thering / \ideald _i $.
    Finally, the product $\prod_{i\in [s]} \thering / \ideald _i$ is isomorphic to $\thering /\ideald $ 
    again by the Chinese remainder theorem (formula~\eqref{eq:Chinese-rings}).
\end{proof}

Proposition~\ref{prop:monogenic-modulo-b} in particular implies Lemma~\ref{lemma=a/Wa}, which has been employed in Section~\ref{section=maintheoremfull} in axiomatized constellation theorems. 

\begin{corollary}\label{cor:coprime}
Let $K$ be a number field, and $\Or$ an order in $K$.
Let $\idealc $ be an invertible fractional ideal of $\Or$, and $\cond \in \thering \setminus \{ 0 \} $.
    Then there exists an element $x\in K^{\times} $ such that 
$x \idealc + \cond \thering = \thering$ holds as an equality of $\Or$-submodules of $K$. That means, the submodule $x\idealc $ is contained in $\thering $ and coprime with the given element $\cond $.
\end{corollary}
\begin{proof}
    Apply Proposition \ref{prop:monogenic-modulo-b} to $\idealc ^{-1} $ 
    and $\ideald = \cond \thering $ to find an element $x\in \idealc ^{-1}\setminus \{ 0\} $ such that $x\thering + \cond \idealc ^{-1} = \idealc ^{-1} $.
    By multiplying this by $\idealc $, we obtain the desired equality.
\end{proof}

\subsection{Ideal norms of invertible fractional ideals}\label{subsection=idealnorm}
The goal of this subsection is to prove Proposition~\ref{theorem=reduction_order}, which enables us to reduce the setting of an invertible fractional ideal of an order to a non-zero (integral) ideal of $\OK $. %
First we recall the notion of the norm of an invertible fractional ideal in the context of number fields.

Throughout this subsection, we assume the following setting:
\begin{setting} \label{setting:ideal_norm}
Let $K$ be a number field of degree $n$,
and $\Or $ an 
order in $K$.
Let $\cond$ be a positive integer satisfying $\cond \OK \subseteq \Or $ (for instance, $\cond \coloneqq \# (\OK / \Or )$, which is finite because $\OK $ and $\Or $ are both isomorphic to $\ZZ ^n$ as $\ZZ $-modules).
\end{setting}

\begin{definition}\label{def:of-norm}
Let $\idealc$ be an invertible fractional ideal of $\Or$.
    \begin{enumerate}[(1)]
        \item\label{en:Nrm} 
        If $\idealc \subseteq \Or$,
        we define its \emph{ideal norm} by $\Nrm (\idealc ) \coloneqq \# (\Or / \idealc)$.        
        \item\label{en:Nrmfrac}
        In general, we define the {\em ideal norm} of $\idealc $ in the following manner: choose an element $d\in \Or \setminus \{ 0\}$ such that $d\idealc \subseteq \Or$ (which exists by Lemma \ref{lem:monogenic-modulo-p^e} \eqref{item:finitely-generated}), and set 
        \begin{equation}\label{eq:def-of-norm}
            \Nrm (\idealc ) \coloneqq \frac{\Nrm (d\idealc )}{\Nrm (d\Or )} .
        \end{equation}
    \end{enumerate}
\end{definition}
Note that in \eqref{en:Nrm}, the quotient group  $\Or /\idealc $ is finite because both $\Or $ and $\idealc $ are free abelian groups of rank $n$. 
Also, when $\Or =\OK$, \eqref{en:Nrm} is 
consistent with the definition of the ideal norm in Section~\ref{section=preliminarynumbertheory}.
In Proposition~\ref{lem:multiplicative-norm}, we will verify that the right-hand side of~\eqref{eq:def-of-norm} is independent of the choice of $d$;
in particular, \eqref{eq:def-of-norm} is consistent with~\eqref{en:Nrm}.

\begin{remark}\label{remark=absence}
    An invertible fractional ideal $\idealc$ of $\Or$ %
    can be an $\Or' $-module for several different orders $\Or'$. 
    The absence of the ring $\Or $ from the symbol $\Nrm (\idealc )$ is nonetheless justified by the fact %
    that
under  the assumption that $\idealc $ is an {\em invertible} fractional ideal of $\Or $, the ring $\Or $ can be recovered from $\idealc \subseteq K$ as $\Or = \{ x\in K \mid x\idealc \subseteq \idealc\}$.
\end{remark}

\begin{lemma}\label{lem:norm-principal-case}
    Let $\idealc \subseteq K$ be a subgroup isomorphic to $\ZZ ^n$ and $x\in \OK \setminus \{ 0\}$.
    Assume that $x\idealc \subseteq \idealc $ as subsets of $K$.
    Then we have $\# (\idealc / x\idealc ) = | \Nelm (x) | $.
\end{lemma}
\begin{proof}
    Choose an  arbitrary $\ZZ $-basis $\bw $ for $\idealc $ and let $X$ be the matrix representing the map $ \idealc \ni c \mapsto xc\in \idealc $ with respect to $\bw $.
    Then by the theory of finitely generated abelian groups, we have $\# (\idealc / x\idealc )=|\det (X) | $.
Recall from Remark~\ref{remark=idealnorm} that the norm $\Nelm (x)$ equals the determinant of the $\QQ $-linear endomorphism $K\ni y \mapsto xy \in  K$.
    Since $\bw$ may be seen as a $\QQ$-basis of $K$, the result follows.
\end{proof}

\begin{proposition}\label{lem:multiplicative-norm}
    Let $\idealc $ be an invertible fractional ideal of $\Or $. 
    \begin{enumerate}[$(1)$]
        \item\label{item:multiplicative-special-case} Assume that $\idealc \subseteq \Or $, and let $d\in \Or \setminus \{ 0\}$. Then we have
        \begin{equation}\label{eq:norm-is-multiplicative}
            \Nrm (d\idealc ) = |\Nelm (d) | \cdot \Nrm (\idealc ).
        \end{equation}
        \item\label{item:norm-well-defined} The ideal norm of $\idealc $ in \eqref{eq:def-of-norm} is independent of the choice of $d$.
        Moreover \eqref{eq:norm-is-multiplicative} holds for all invertible fractional ideals $\idealc $ and all $d\in K^{\times}$.
    \end{enumerate}
\end{proposition}
\begin{proof}
    First, we prove \eqref{item:multiplicative-special-case}. By the filtration $d\idealc \subseteq \idealc \subseteq \Or $, we have 
\[
        \# (\Or / d \idealc ) = \# (\Or / \idealc ) \cdot \# (\idealc /d \idealc ) .
\]
    By Lemma \ref{lem:norm-principal-case}, we have $\# (\idealc / d\idealc ) = |\Nelm (d) |$. Therefore, we obtain the desired formula \eqref{eq:norm-is-multiplicative}.

    For \eqref{item:norm-well-defined}, let $c,d\in \Or \setminus \{ 0 \}$ be two elements with $c\idealc \subseteq \Or$  and  $d\idealc \subseteq \Or $. By \eqref{item:multiplicative-special-case}, we \havethat\  $\Nrm (cd \idealc )=|\Nelm (c)| \cdot \Nrm (d\idealc)= |\Nelm (d)| \cdot \Nrm (c\idealc)$.
    From this, we \obtainthat\  $\Nrm (d\idealc )/|\Nelm (d)|=\Nrm (c\idealc )/ |\Nelm (c)|$. It ensures that Definition~\ref{def:of-norm}~\eqref{en:Nrmfrac} is well-defined.
    Equality \eqref{eq:norm-is-multiplicative} for the general case follows in a similar manner by repeated application of \eqref{item:multiplicative-special-case}.
\end{proof}

The following lemma describes a relationship between the ideal norm of an ideal of $\Or$ and that of an ideal of $\OK$.

\begin{lemma}\label{lem:preservation-special-case}
    Let $\idealc \subseteq \Or $ be an invertible ideal satisfying $\idealc + \cond \Or = \Or $. 
    Then we have $\Nrm (\idealc ) = \Nrm (\idealc \OK )$.
\end{lemma}
\begin{proof}
    Consider the following commutative diagram 
\[
    \xymatrix{
        0 \ar[r]& \idealc \ar@{_{(}->}[d]\ar[r] & \Or \ar@{_{(}->}[d]\ar[r] & \Or /\idealc \ar[d]^v\ar[r] & 0
        \\
        0 \ar[r]& \idealae \ar[r]& \OK \ar[r]& \OK /\idealae \ar[r]& 0
    }
\]
with exact rows.
        By the assumption $\idealc + \cond \Or = \Or$, the integer $\cond $ is invertible in the rings $\Or /\idealc $ and $\OK /\idealc \OK $.
        By the assumption $\cond \OK \subseteq \Or $ from Setting~\ref{setting:ideal_norm}, $\cond$ acts as zero on the quotient groups
        $\idealc \OK / \idealc $ and $\OK /\Or $.
        Hence, by the snake lemma, we conclude that on the kernel and cokernel of the vertical map $v$, the integer $\cond $ acts as zero and invertibly at the same time.
        Therefore, $\ker (v)$ and $\coker (v)$ are both zero; in other words, $v$ is bijective.
        In particular, we \obtainthat\ 
\[
\Nrm (\idealc )=\# (\Or /\idealc ) = \# (\OK /\idealc \OK ) = \Nrm (\idealc \OK ),
\]
as desired.
\end{proof}

\begin{proposition}\label{prop:preservation-norms}
    Let $\idealc $ be an arbitrary invertible fractional ideal of $\Or $.
    Then we have $\Nrm (\idealc ) = \Nrm (\idealc \OK)$.
\end{proposition}
\begin{proof}
    By Corollary \ref{cor:coprime}, there exists an element $x\in K^{\times}$ such that $x\idealc + \cond \Or = \Or $ holds.
    By Lemma \ref{lem:preservation-special-case}, we have $\Nrm (x\idealc ) = \Nrm (x\idealc \OK ).$ By Proposition \ref{lem:multiplicative-norm}, this then implies that $|\Nelm (x) | \cdot \Nrm (\idealc )=|\Nelm (x)| \cdot \Nrm (\idealc \OK ) $. Hence, $\Nrm (\idealc ) = \Nrm (\idealc \OK )$.
\end{proof}

The following reduction result (Proposition~\ref{theorem=reduction_order}) plays a key role in the proof of Theorem~\ref{th:constellations-in-ideals}, as well as in the 
deduction
of Theorem~\ref{theorem=quadraticformcloseprimes} from Theorem~\ref{th:constellations-in-ideals}. 
To prove the reduction result, we employ the following concept of sign. %

Recall that $r_1$ is the number of real embeddings of $K$.
\begin{definition}\label{def:sign-of-algebraic-numbers}
    Let $\xi \in K^\times$. 
    We define its {\em sign} $\sgn (\xi)$
    to be the tuple 
     \[ \sgn (\xi) \coloneqq \big(\sgn(\sigma _{i } (\xi))\big )_{i\in [r_1] } \in \{  \pm 1\} ^{r_1}\]
of signs $\pm 1$.
\end{definition}

\begin{lemma}\label{lem:existence-sign-modulo}
    For every sign $s \in \{ \pm 1\} ^{r_1}$ and for every class $\tau \in \mathcal O_K/\cond\mathcal O_K $, there exists an element of $\OK $ whose sign is $s$ and whose residue class is $\tau$.
\end{lemma}

\begin{proof}
	By the approximation theorem (\cite[Chapter~II, Theorem~3.4]{Neukirch}), there exists an element $\xi$ of $K^\times$ such that $\sgn(\xi)=s$.
	By multiplying by an appropriate positive integer if necessary, we may assume that $\xi$ is an element of $\OK$.
	Let $\alpha_0$ be a representative of $\tau$ and $t$ a positive integer. 
	Set $\alpha\coloneqq\alpha_0+ft\xi$.
	Then $\alpha$ is an element of $\OK$, and the residue class of $\alpha$ is $\tau$.
	Furthermore, since $\sigma_i (\alpha)=\sigma_i(\alpha_0)+ft\sigma_i(\xi)$ for every $i \in[r_1]$, 
	the signs of $\sigma_i(\alpha)$ and $\sigma_i(\xi)$ coincide if $t$ is sufficiently large.
\end{proof}

For the definitions of the sets $\PP_{\ideala}$ and $\PP_{\idealc}$ in the following statement, see Definition~\ref{def:of-set-P}.

\begin{proposition}[Reduction to ideals in maximal orders]\label{theorem=reduction_order}
Assume Setting~$\ref{setting:ideal_norm}$.
Let $\idealc$ be an invertible fractional ideal of $\Or$. Then there exists $\xi \in K^{\times}$ such that the following three hold true. Below, we write $\ideala\coloneqq \xi \idealc \OK$.
\begin{enumerate}[$(1)$]
 \item\label{en:ideal_desuyo} The set $\xi\idealc$ is an ideal of $\Or$, and it is coprime with $f$. 
 \item\label{en:coprime_f} 
    The ideal $\ideala$ of $\OK $ is coprime with $\cond $.
 \item\label{en:inj_incl} The multiplication map $\alpha\mapsto \xi \alpha$ in $K$ induces injective maps from $\idealc$ to $\ideala$ and from $\PP_{\idealc}$ to $\PP_{\ideala}$.
 \item\label{en:norm_huhen} For every $\alpha\in \idealc$, we have
\[
\frac{\Nelm(\alpha)}{\Nrm(\idealc)}=\frac{\Nelm(\xi \alpha)}{\Nrm(\ideala)}.
\]
\end{enumerate}
\end{proposition}
Note that \eqref{en:coprime_f} is implied by \eqref{en:ideal_desuyo}. We have explicitly stated \eqref{en:coprime_f} for later reference.

\begin{proof}
Apply Corollary~\ref{cor:coprime} to $\idealc$ and $f$. Then, we obtain $\xi \in K^{\times}$ such that $\xi \idealc+f\Or=\Or$. 
We claim that $\xi $ can be taken to satisfy $\Nelm (\xi )>0$.
To show this claim, suppose that the originally taken $\xi $ satisfies $\Nelm (\xi )<0$; in this case we necessarily have $r_1\geq 1$.
By Lemma~\ref{lem:existence-sign-modulo} with $f$ replaced by $f^2$, we can find an element $\alpha \in 1+f^2\OK $ ($\subseteq 1+f\Or $) such that $\Nelm (\alpha )<0$.
Write $\xi $ afresh for $\alpha \xi $.
This verifies our claim.

Our task is to verify \eqref{en:inj_incl} and \eqref{en:norm_huhen} for this $\xi $.
First, by definition, we have $\PP_{\idealc}\subseteq \PP_{\idealc\OK}$. 
Observe that the multiplication map $\alpha\mapsto \xi \alpha$ induces a bijection from $\PP_{\idealc\OK}$ to $\PP_{\ideala}$. Indeed, for every $\alpha \in \idealc$, we have the equality of fractional ideals $\alpha (\idealc\OK)^{-1}=(\xi \alpha)(\xi\idealc\OK)^{-1}$. Hence, \eqref{en:inj_incl} holds. 

Next, by  
Proposition~\ref{lem:multiplicative-norm}~\eqref
{item:norm-well-defined} 
and 
Proposition~\ref{prop:preservation-norms} 
we have (recalling the definition 
$\ideala =\xi \idealc \OK $)
\[
 \Nrm (\ideala )= |\Nelm (\xi )|\cdot \Nrm (\idealc \OK )
 = \Nelm (\xi ) \cdot \Nrm (\idealc ).
\]
Here we used the condition $\Nelm(\xi)>0$ as well.
Assertion \eqref{en:norm_huhen} now follows from the multiplicativity of the norm: $\Nelm (\xi \alpha )=\Nelm (\xi )\cdot \Nelm (\alpha )$.
\end{proof}

Proposition~\ref{theorem=reduction_order} explains partly why Theorem~\ref{theorem=GYfordieals} and our axiomatized constellation theorems in Sections~\ref{section=maintheoremfull} and \ref{section=slidetrick}
are formulated for non-zero ideals $\ideala$ of $\OK$ instead of invertible fractional ideals.

\subsection{`Prime elements' in ideals and subpseudorandom condition}\label{subsection=SPsiSI}

The goal of this subsection is the following.

\begin{theorem}\label{theorem=subpseudorandom_PPa}
Let $K$ be a number field of degree $n$ and $\ideala\in \Ideals_K$. Then $\PP_{\ideala}$ is a member of $\logpseua$. Moreover, for every $\theta\in(0,1)_{\RR}$ and for every integer $r$ at least $[K:\QQ]$, there exist $D_1,D_2>0$ and $\varepsilon\in (0,1)_{\RR}$ such that the following holds true. Let $S\subseteq \ideala$ be a standard shape with $\#S=r+1$, and $\bv$ a $\ZZ$-basis of $\ideala$. Then, for every $\rho >0$ and every $\uvarsigma>0$, there exist natural numbers $W=W_{\PP_{\ideala},\mathrm{S}\Psi_{\log}^{\mathrm{SI}}}(\rho,S)$ and $M_{\PP_{\ideala},\mathrm{S}\Psi_{\log}^{\mathrm{SI}}}(\rho,\uvarsigma,\bv,S,\theta)$ such that if $M\geq M_{\PP_{\ideala},\mathrm{S}\Psi_{\log}^{\mathrm{SI}}}(\rho,\uvarsigma,\bv,S,\theta)$, then $\PP_{\ideala}\cap\ideala(\bv,M)$ satisfies the $(\rho,\uvarsigma,W,M,\bv,S,\theta)$-condition with parameters $(D_1,D_2,\varepsilon)$.
\end{theorem}

The key to the proof is the Goldston--Y\i ld\i r\i m type asymptotic formula for ideals, stated in Theorem~\ref{theorem=GYfordieals}. Since we use the symbol $a$ for the coefficient of $F(x,y)=ax^2+bxy+cy^2$, we use the symbol $\theta$ for the `parameter $a\in (0,1)_{\RR}$' in the proof.

\begin{proof}
We will first prove the latter assertion.
Fix a $C^{\infty}$-function $\chi\colon \RR\to [0,1]_{\RR}$ with $\chi(0)=1$ and $\mathrm{supp}(\chi)\subseteq [-1,1]_{\RR}$.
Using Theorem~\ref{theorem=GYfordieals}, we can go along a similar line to the proof of Theorem~\ref{theorem=package_PR} so that we obtain the following: 
for every $\rho>0$ and for every $\uvarsigma>0$, 
there  exist $W=W_{\PP_{\ideala},S\Psi_{\log}^{\mathrm{SI}}}(\rho,\chi,S)$ and $M'_{\PP_{\ideala},S\Psi_{\log}^{\mathrm{SI}}} = M'_{\PP_{\ideala},S\Psi_{\log}^{\mathrm{SI}}}(\rho,\uvarsigma,\chi,S,\theta)\in \NN$ with $W\leq (M'_{\PP_{\ideala},S\Psi_{\log}^{\mathrm{SI}}})^{(\log 2)\theta}$ such that the following holds. 
For $M\geq M'_{\PP_{\ideala},\mathrm{S}\Psi_{\log}^{\mathrm{SI}}}$, let $R\coloneqq M^{\frac{\theta}{17(r+1)2^r}}$ and define $\lambda\colon \ideala\to \RR_{\geq 0}$ by
\[
\lambda(\alpha)\coloneqq \frac{\kappa \cdot \Lambda_{R,\chi}^{\ideala}(\alpha)^2}{c_{\chi}\log R}.
\]
Here $\kappa$ is the constant appearing in Theorem~\ref{theorem=zeta_K}, and $c_{\chi}\coloneqq \int_{0}^{\infty}{\chi'}^2(x)\rd x$; 
recall the definition of $\Lambda_{R,\chi}^{\ideala}$ from \eqref{eq:Lambda^a}. Then, for every $b\in \ideala$ with $b\OK+W\ideala=\ideala$, 
the function $\frac{\varphi_K(W)}{W^n}(\lambda\circ \Aff_{W,b})$ is a $(\rho,\frac{\uvarsigma M^{\theta}}{W},S)$-pseudorandom measure.

Let $\bv$ be a $\ZZ$-basis of $\ideala$.
Define the exceptional set
$T\subseteq\PP_{\ideala}\cap\ideala(\bv,M)$
by $T\coloneqq \PP_{\ideala}\cap\ideala(\bv,M)\cap\OK(\Nrm(\ideala)\cdot R)$.
We claim that if $M$ is sufficiently large depending on 
$\rho,S$ and $\theta$, then every $\alpha\in \PP_{\ideala}\setminus T$ satisfies that $\alpha \OK +W\ideala=\ideala$. Indeed, since $\alpha \in \PP_{\ideala}$, there exists a prime ideal $\idealp_{\alpha}$ such that $\alpha \OK=\idealp_{\alpha} \ideala$. If $M$ is sufficiently large depending on $\rho,S$ and $\theta$, then the inequality $\Nrm(\idealp_{\alpha})> R$ implies that $\idealp_{\alpha}$ is prime to $W$. Hence, $\alpha \OK +W\ideala=\ideala$ in this case. This argument in addition shows that for every $\alpha\in \PP_{\ideala}\setminus T$, we have 
$\lambda(\alpha)=\frac{\kappa \theta}{17(r+1)2^r \cdot c_{\chi}}\cdot \log M$.
Moreover, if $M$ is sufficiently large depending on 
$\bv$ and $\theta$, then by \eqref{eq:boundT}, we obtain $\#T\leq M^{\frac{\theta}{16}}$.

Define $M_{\PP_{\ideala},\mathrm{S}\Psi_{\log}^{\mathrm{SI}}}(\rho,\uvarsigma,\bv,\chi,S,\theta)$ as the minimal integer such that all of the arguments above hold true. Then, for
\[
(D_1,D_2,\varepsilon)=\left(\frac{\kappa \theta}{17(r+1)2^r \cdot c_{\chi}},\frac{\kappa \theta}{17(r+1)2^r \cdot c_{\chi}},\frac{3}{4}\right),
\]
we obtain the latter assertion of the theorem. 
Here, since we have fixed a function $\chi$, we omit to write dependences of 
$W_{\PP_{\ideala},S\Psi_{\log}^{\mathrm{SI}}}(\rho,\chi,S)$
and
$M_{\PP_{\ideala},\mathrm{S}\Psi_{\log}^{\mathrm{SI}}}(\rho,\uvarsigma,\bv,\chi,S,\theta)$
on $\chi$.
Thus we write
$W_{\PP_{\ideala},S\Psi_{\log}^{\mathrm{SI}}}(\rho,S)$
and
$M_{\PP_{\ideala},\mathrm{S}\Psi_{\log}^{\mathrm{SI}}}(\rho,\uvarsigma,\bv,S,\theta)$
for short.
In particular, we have the former assertion, $\PP_{\ideala}\in \logpseua$. 
\end{proof}

\subsection{Counting elements by the Chebotarev density theorem}\label{subsection=counting_Cheb}
In order to establish Theorem~\ref{th:constellations-in-ideals}, we aim to apply Theorem~\ref{theorem=package_infinite_a_close} to the given set $A\subseteq \PP _{\idealc }$,
which translates to a subset of $\PP _{\ideala }$ for an ideal $\ideala $ of $\OK $ via Proposition~\ref{theorem=reduction_order}. 
Since we have proved $\PP _{\ideala} \in \logpseua $ in Theorem~\ref{theorem=subpseudorandom_PPa}, 
what remains to be verified is conditions~\eqref{en:counting_infinite_DD} and \eqref{en:counting_ideal_infinite_DD} in Theorem~\ref{theorem=package_infinite_DD}.
As we will explain in Subsection~\ref{subsection=proof_orders}, condition~\eqref{en:counting_ideal_infinite_DD} can be confirmed by Landau's prime ideal theorem.
Thus, the main task %
is the verification of the counting condition~\eqref{en:counting_infinite_DD}.

In this subsection, we will perform this counting argument, with the aid of a finer version of \chebden \ (Theorem~\ref{theorem=Chebotarev_narrow}) than we have been using.
See Theorems~\ref{theorem=counting_PPa} and \ref{theorem=counting_PPc} for the final statements of our \counting s.

Two main differences between the counting argument in this section and that in previous sections are the following: first, we have additional data, such as $\ideala$, $\OO$ and $\idealc$. 
Secondly, since the sign of the norm $\Nelm (\alpha )$ matters in Theorem~\ref{theorem=quadraticformcloseprimes}
and Theorem~\ref{th:quadratic-forms-ideals},
we need to count elements \emph{having a prescribed sign}; %
to count elements with signs, the mere \counting\ of ideals does not suffice
because as the identity $\Nrm(\alpha\OK)=|\Nelm(\alpha)|$ suggests, the ideal generated by a given element does not remember its sign.

\begin{setting}	\label{setting:conductor-and-ideala}
Let $K$ be a number field of degree $n$. 
Let $f\in \NN $.
Let $r_1$ be the number of real embeddings of $K$,
and fix real embeddings $\sigma_1,\ldots,\sigma_{r_1}$ to define the function $\sgn$ (Definition~\ref{def:sign-of-algebraic-numbers}).
Let $\vph_K$ be the totient function of $K$ (Definition~\ref{def=totient}).
\end{setting}

First, we consider the set $\PP_{\ideala}$ of `prime elements' of an ideal $\ideala $ of $\OK $.
Let us collect some pieces of terminology to state %
the \Cheb\ density theorem~\ref{theorem=Chebotarev_narrow}. 

\begin{definition}
\begin{enumerate}[$(1)$]
    \item 
    For $\xi \in K^{\times}$ and $\alpha \in \OK$, we write $\xi \equiv \alpha \pmod \cond $, or say $\xi $ is congruent to $\alpha $ modulo $\cond $, if there exist $x\in \cond\OK$ and $y\in \OK\setminus \{0\}$ with $y$ prime to $\cond$ such that $\xi = \alpha+\frac{x}{y}$.
    \item\label{en:If} 
    We denote by $\If $ the commutative group of fractional ideals of $\OK $ coprime with $\cond $.
    (A fractional ideal is said to be \emph{coprime} with $\cond $ means if it does not share prime factors with $\cond $, cf.\ Theorem~\ref{theorem=primeideals_frac}.
    We have a canonical isomorphism $\If \cong \bigoplus _{\idealp \nmid f} \ZZ $.)
    \item\label{en:Kfp} Let $\Kfp $ be the subgroup of elements of $K^{\times}$ which are congruent to $1$ modulo $\cond $ and totally positive. Namely, set
    \[
    \Kfp \coloneqq \left\{ \xi=1+\frac{x}{y} \in K : \begin{array}{l} x\in \cond\OK ,\ y\in \OK \setminus \{0\}\text{ prime to $\cond $},\\ \sgn (\xi)=(+1,\dots ,+1) \end{array}      \right\}.
    \]
    Then the map $\xi\mapsto \xi\OK $ gives rise to a homomorphism $\Kfp \to \If$.
    \item\label{en:hf} We define the {\em ideal class group $\Clf $ with modulus $\cond $} as the cokernel of the homomorphism $\Kfp \to \If$ above. We define $h_\cond$ to be the order of $\Clf$.
    \item
    Let $\OKf $ be the subgroup of $\OKt $ consisting of the elements 
    which are congruent to $1$ modulo $\cond$ and which are totally positive.
	In other words, $\OKf$ is the kernel of the map $\OKt  \to (\OK /\cond \OK )^{\times} \times \{ \pm 1 \} ^{r_1}; \xi \mapsto (\xi + \cond \OK, \sgn (\xi))$.
    (Since the target is a finite group, we know $\rank (\OKf) =\rank (\OKt)$.)
\end{enumerate}
\label{def:Cl-with-modulus}
\end{definition}
In \eqref{en:hf}, 
we followed the sign convention of \cite{Neukirch} and \cite{Cassels-Froehlich}.
It is well known that $h_\cond $ is finite; for instance, see \cite[Chapter~VI, Proposition~1.8 and Proposition~1.9]{Neukirch}. 
Recall that as usual we regard a residue class $\tau \in \OK/\cond\OK$ %
as a subset of $\OK$. %
From Definition~\ref{def:subset-of-Spec} to Theorem~\ref{theorem=counting_PPa}, we also use the setting below.
\begin{setting} \label{setting:conductor-and-ideala:2}
	Let $\ideala$ be a non-zero ideal of $\OK$.
	Assume that $\ideala$ is prime to $\cond$. 
\end{setting}

\begin{definition}\label{def:subset-of-Spec}
    Given $\tau\in (\OK / \cond \mathcal O_K )^{\times}$ and a sign $s\in \{ \pm 1 \} ^{r_1}$, we define a set $|\Spec(\OK)|_{\ideala, \tau,s}$ of prime ideals by
\[
|\Spec(\OK)|_{\ideala,\tau,s}\coloneqq 
\left\{ \idealp \in |\Spec(\OK)|  :  \idealp\ideala =\alpha \OK \text{ for some }\alpha \in \tau 
\text{ with } \sgn (\alpha )=s
\right\} .
\] 
\end{definition}

Note that $|\Spec(\OK)|_{\ideala,\tau,s}$'s are not necessarily disjoint for different $(\tau ,s)$'s but this does not cause a problem.
In the following lemma, we remark that we can to consider the class $[\ideala ]$ of $\ideala $ to be in the ideal class group $\Clf$ under Setting~\ref{setting:conductor-and-ideala:2}.
\begin{lemma}\label{lem:hoka-naranai}
    Let $\tau\in (\OK/\cond \OK)^{\times}$ and $s\in \{ \pm 1 \} ^{r_1}$. 
    Fix an element $\xi _0\in \OK $ with sign $s$ such that $\xi _0\in \tau$, which exists by Lemma~$\ref{lem:existence-sign-modulo}$. 
Then we have
\[
|\Spec(\OK)|_{\ideala,\tau,s}= \{\idealp\in|\Spec(\OK)| : [\idealp]= -[\ideala]+[\xi_0\OK] \text{ in }\Clf\} .
\] 

\end{lemma}

We write down the proof of this standard fact for the convenience of the reader.
\begin{proof}
    Let $\idealp \in |\Spec(\OK)|_{\ideala ,\tau,s}$. This means that there exists $\alpha \in \tau$ with $\sgn (\alpha )=s $ such that $\idealp \ideala  = \alpha \OK $.
    Then the element $\xi\coloneqq\alpha /\xi_0 \in K^{\times}$ satisfies that $\xi\equiv 1 \pmod \cond $ and     $\sgn(\xi) = (+1,\dots ,+1)$. 
    Therefore we conclude that $[\alpha \OK ]= [\xi _0\OK ]$ in $\Clf $, thus proving that $\idealp $ satisfies $[\idealp]= -[\ideala ]+[\xi_0\OK] $. 

    Conversely, suppose that $\idealp $ satisfies $[\idealp ]+[\ideala ]=[\xi_0\OK ]$.
    Then by the definition of $\Clf $, there exists a totally positive $\eta \in K^{\times}$ with  $\eta\equiv 1\pmod \cond $ 
    such that $\idealp \ideala  = \xi _0\eta \OK$. 
    Now $\alpha\coloneqq\xi _0\eta $ satisfies that $\alpha \in \tau$ and that $\sgn(\alpha) =s$.
    Moreover, since $\idealp$ and $\ideala$ are subsets of $\OK$, we have $\alpha\OK\subseteq\OK$. Therefore, we conclude that $\alpha\in\OKnz$, as desired. 
\end{proof}

The Chebotarev density theorem
 \cite[Theorem 4]{Cassels-Froehlich} 
 asserts that for every class $\mu\in \Clf$, the set $\{ \idealp \in |\Spec (\OK) | : [\idealp ]=\mu \text{ in }\Clf  \}$ has natural asymptotic density $\frac 1 {\hcond }$ in $|\Spec (\OK)|$. 
This is translated via Lemma~\ref{lem:hoka-naranai} as follows.

\begin{theorem}[{Chebotarev density theorem, see \cite[Theorem~4]{Cassels-Froehlich}}]\label{theorem=Chebotarev_narrow}
Let $\tau\in (\OK/\cond\OK)^{\times}$ and $s\in \{ \pm 1 \} ^{r_1}$. Then, we \havethat\ 
\[
\# \{ \idealp \in |\Spec(\OK)|_{\ideala ,\tau,s} : \Nrm(\idealp) \le L \} 
= (1+o_{L\to\infty;f,\ideala}(1))\cdot\frac{1}{\hcond}\cdot\frac{L}{\log L}.
\]
\end{theorem}

Let us explain how Theorem~\ref{theorem=Chebotarev}~\eqref{Chebotarev} is implied by Theorem~\ref{theorem=Chebotarev_narrow} (or \cite[Theorem~4]{Cassels-Froehlich}).
Set $\cond = 1$ for instance. 
By \cite[Theorem~4]{Cassels-Froehlich}, for each class $\mu \in \Cl ^1_K $ we have the estimate 
\begin{equation}\label{eq:Cheb-1}
    \#\{ 
        \idealp \in \SpecOK  : 
        [\idealp]=\mu \text{ in }\Cl ^1_K \text{ and }\Nrm(\idealp )\le L     
        \} 
        =
        (1+o_{L\to\infty}(1))\cdot\frac{1}{h_1}\cdot\frac{L}{\log L}.
\end{equation}
By \cite[Chapter~VI, Proposition 1.11]{Neukirch}
there exists a surjective homomorphism
$\Cl ^1_K \twoheadrightarrow \Cl _K $
(though the surjectivity is not needed for the purpose achieved by Theorem~\ref{theorem=Chebotarev}~\eqref{Chebotarev}).
Summing up \eqref{eq:Cheb-1} over all $\mu \in \ker (\Cl ^1_K \to \Cl _K)$ proves Theorem~\ref{theorem=Chebotarev}~\eqref{Chebotarev}.

Now we turn this \counting\ of prime ideals into counts of elements. 

\begin{definition}\label{definition=Paus}
    For every $\tau\in \OK /\cond \OK $ and every $s\in \{ \pm 1\} ^{r_1}$, 
    define a subset   $\Pavs$ of $\PP_{\ideala }$ as the set of elements $\alpha \in \PP_{\ideala }\cap \tau $ with $\sgn (\alpha ) = s$. %
\end{definition}

\begin{lemma}\label{lem:up-to-OKm}
    \begin{enumerate}[$(1)$]
        \item\label{item:P-is-disjoint-union} The set $\PP_{\ideala }$ can be partitioned as $\PP_{\ideala } = \bigsqcup_{(\tau,s)}\Pavs$. Here $(\tau,s)$ runs over $(\OK/\cond\OK)\times \{\pm1\}^{r_1}$. 
        \item\label{item:when-a-is-not-invertible-element} For each $(\tau,s)\in(\OK/\cond\OK)\times \{\pm1\}^{r_1}$, the $\OKf$-action by multiplication leaves $\Pavs$ invariant.
        If $\tau\in \OK /\cond \OK $ is not invertible, then
         $\#\left(\Pavs/\OKf\right)<\infty$. Here $\Pavs/\OKf$ means the quotient set by the $\OKf$-action.
        \item\label{item:bij-Spec-and-P} 
        Let $\DD' \subseteq \ideala\setminus \{0\} $ be a fundamental domain
        for the action $\OKf \curvearrowright \ideala\setminus \{0\}$. %
        If $\tau\in (\OK /\cond \OK )^{\times}$, then we have a bijection %
    \begin{equation}\label{eq:natural-bij}
        |\Spec(\OK)|_{\ideala,\tau,s} \simeq \Pavs\cap \DD' .
    \end{equation}
    Here, a prime ideal $\idealp\in|\Spec(\OK)|_{\ideala,\tau,s}$ is sent to a unique element 
    $\alpha \in \DD'$ which satisfies the following three conditions: 
    $\idealp \ideala  =\alpha \OK $,
        $\alpha \in \tau$, and $\sgn (\alpha )=s$.
    \end{enumerate}
\end{lemma}
\begin{proof}
Item \eqref{item:P-is-disjoint-union} is trivial.
Next we prove \eqref{item:when-a-is-not-invertible-element}. 
For the former assertion, observe that  the $\OKf $-action changes neither the modulo $\cond $ residue class nor the sign of elements of $\OK$. 
For the latter assertion, suppose that $\tau \not\in (\OK /\cond \OK)^{\times}$.

Take an element $\alpha$ of $\Pavs $.
This in particular means $\alpha \in \tau $ and hence $\alpha $ is not prime to $\cond $. 
    Since  $\ideala$ is prime to $f$, the prime ideal $\idealp =\alpha \ideala ^{-1}$ is not prime to $\cond $;
    in other words $\idealp $ divides $\cond $. 
    We have established an injection 
    \[
        \Pavs 
        \hookrightarrow
        \bigsqcup _{\idealp \vert \cond }  
        \left\{\alpha \in \OK : \alpha \OK = \ideala \idealp \right\} ,
    \]
    where there are only finitely many summands on the right-hand side.
    Observe that each summand is an $\OKt $-orbit.
    Since $\OKf$ is a finite index subgroup of $\OKt $,
    each summand splits into finitely many $\OKf $-orbits.
    This proves that $\Pavs / \OKf $ is finite.

Finally, we prove \eqref{item:bij-Spec-and-P}.
    For a given $\idealp \in  |\Spec(\OK)|_{\ideala,\tau,s}$, by the definition of $|\Spec(\OK)|_{\ideala ,\tau,s}$, there exists $\alpha \in \OKnz$ which fulfills the three conditions in the statement. Such an element $\alpha $ is automatically in $\PP_{\ideala;\tau,s}$ by its definition.
    We will show that this  $\alpha $ is unique modulo the $\OKf $-action. Take two elements $\alpha _1$ and $\alpha _2$ satisfying the three conditions.
    Since they define the same ideal, we \havethat\  $\xi\coloneqq \alpha _1/\alpha _2 \in \OKt $. 
    Since
    $\alpha _1, \alpha _2 \in \tau$ (which is now assumed to be prime to $\cond $) and they both have sign $s$, 
    we in addition know $\xi\equiv 1 \pmod \cond $ and that $\xi$ is totally positive. %
    Therefore $\xi \in \OKf$ as desired. 
    This shows that the map $\idealp \mapsto \alpha$ is well-defined.

    Conversely, suppose that we are given $\alpha \in \Pavs$.
    By Definition~\ref{def:of-set-P}, 
    $\idealp\coloneqq \alpha \ideala ^{-1}$ is a prime ideal.
    By Definition~\ref{def:subset-of-Spec}, $\idealp $ belongs to $|\Spec (\OK ) |_{\ideala , \tau, s}$. Thus, we obtain a map $\alpha \mapsto\idealp $.
    
    Since both maps are constructed from the relation $\ideala \idealp = \alpha \OK $, it is clear that they are inverse to each other.
\end{proof}

\begin{corollary}\label{cor:density-of-P-a-s}
Assume Settings~$\ref{setting:conductor-and-ideala}$ and~$\ref{setting:conductor-and-ideala:2}$. 
Let $\DD'$ be a fundamental domain for the action $\OKf\curvearrowright \ideala\setminus \{0\}$. 
    Let $\tau \in \OK/f\OK$ and $s\in \{\pm1\}^{r_1}$.
\begin{enumerate}[$(1)$]
\item\label{en:usui}
If $\tau\not \in (\OK/f\OK)^{\times}$, then the relative asymptotic density of $\Pavs\cap \DD' $ 
in 
$\PP_{\ideala }\cap \DD' $ measured by the ideal norm equals $0$. That means,
\[
\lim_{L\to \infty}\frac{\#(\Pavs\cap \DD'\cap \OK(L))}{\#(\PP_{\ideala }\cap \DD' \cap \OK(L))}=0.
\]
  \item\label{en:koi}  If $\tau \in (\OK/f\OK)^{\times}$, then 
\[
\lim_{L\to \infty}\frac{\#(\Pavs\cap \DD'\cap \OK(L))}{\#(\PP_{\ideala }\cap \DD' \cap \OK(L))}=\frac{1}{2^{r_1}\varphi _K(\cond )}
\]
holds.
\end{enumerate}
\end{corollary}
\begin{proof}
    Both items follow from Theorem~\ref{theorem=Chebotarev_narrow} and Lemma~\ref{lem:up-to-OKm}: the factors $2^{r_1}$ and $\varphi _K(\cond )$ come from the numbers of possible choices of $s$ and $\tau$, respectively.
\end{proof}

Finally, we have the following estimate from below of $\PP_{\ideala}$.

\begin{theorem}\label{theorem=counting_PPa}
Assume Settings~$\ref{setting:conductor-and-ideala}$ and~$\ref{setting:conductor-and-ideala:2}$. 
Let $\tau \in (\OK/\cond \OK)^{\times}$ and $s\in \{\pm1\}^{r_1}$. Then, for every  $\ZZ$-basis $\bv$ of $\ideala$, we \havethat\ 
\[
\liminf_{M\to\infty}\frac{\#(\PP_{\ideala;\tau,s}\cap \ideala(\bv,M))}{M^n(\log M)^{-1}}>0.
\]
\end{theorem}

\begin{proof}
    Recall that the NL-compatibility is defined for a subset of $\OKnz$, in particular, for a subset of $\ideala\setminus \{0\}$. 
    Recall also for an integral basis $\omom$ of $K$, the restriction of $\|\cdot\|_{\infty,\omom}$ on $\ideala$ is bi-Lipschitz equivalent to $\|\cdot\|_{\infty,\bv}$.

    Fix a $\ZZ$-basis $\bv$ of $\ideala$.
First, we claim that there exists an NL-compatible fundamental domain for the action $\OKf\curvearrowright \ideala \setminus \{0\}$. Indeed, by Proposition~\ref{proposition=normrespectingfundamentaldomain}, we can take an NL-compatible $\OKt$-fundamental domain $\DD\subseteq \OKnz$. Since $\OKf$ has a finite index in $\OKt$, by considering the union of finitely many translates of $\DD$, we have an NL-compatible fundamental domain $\tilde{\DD}\subseteq \OKnz$ for the action $\OKf\curvearrowright \OKnz$. Finally, set $\DD'\coloneqq \ideala\cap \tilde{\DD}$; then this $\DD'$ is an NL-compatible fundamental domain for $\OKf\curvearrowright \ideala \setminus \{0\}$. 

Now the assertion follows from Corollary~\ref{cor:density-of-P-a-s}~\eqref{en:koi} and the NL-compatibility of $\DD'$; the deduction goes along the same lines as the proof of Proposition~\ref{proposition=PK_kouri}.
\end{proof}

In the rest of the current subsection, we treat the case of $\PP_{\idealc}$ for an invertible fractional ideal $\idealc $ of an order $\Or $.
We no longer assume Setting~\ref{setting:conductor-and-ideala:2}, but we continue to use Setting~\ref{setting:conductor-and-ideala}.

\begin{lemma}\label{lem:105a}

Let $\cond$ be a positive integer.
Let $\Or \subseteq \OK $ be an order
and 
$\ideald\subseteq\Or$ be an ideal coprime with $f$. Assume that $f\OK\subseteq\Or$. Then, we \havethat\  $\ideald\OK\cap\Or=\ideald$.
\end{lemma}
\begin{proof}
Since $\ideald$ is coprime with $f$, we have $\ideald+f\Or=\Or$. This implies $\ideald\OK+f\OK=\OK$. %
It in turn implies that
\[
\ideald\OK\cap f\OK=f\ideald\OK\subseteq\ideald\Or=\ideald.
\]
Hence we conclude that 
\[
    \ideald\OK\cap\Or=\ideald\OK\cap(\ideald+f\Or)
    \subseteq \ideald .
\]
Since the inclusion $\ideald\OK\cap\Or \supseteq \ideald$ trivially holds, the proof is completed.
\end{proof}

Now we are ready to present our counting result for $\PP_{\idealc}$. %

\begin{theorem}\label{theorem=counting_PPc}
Assume Setting~$\ref{setting:conductor-and-ideala}$.
Let $\Or$ be an order in $K$, and $\idealc$ be an invertible fractional ideal of $\Or$. 
Let $f\in \NN$ 
be 
such that $f\OK\subseteq \Or$. Take  $\xi\in K^{\times}$ and $\ideala=\xi \idealc \OK \in \Ideals_K$ as in Proposition~$\ref{theorem=reduction_order}$. Let $\tau \in (\OO/\cond \OK)^{\times}$ and $s\in \{\pm1\}^{r_1}$. 
Then we have 
\begin{equation}\label{eq:intersection-is-c'}
    \Pavs \subseteq \xi\PP _{\idealc }
\end{equation}
and for every $\ZZ$-basis $\bv$ of $\ideala$, we \havethat\ 
\begin{equation}\label{eq:theorem=counting_PPc}
    \liminf_{M\to\infty}\frac{\#(\xi\PP_{\idealc}\cap \PP_{\ideala;\tau,s}\cap \ideala(\bv,M))}{M^n(\log M)^{-1}}>0.
\end{equation}
\end{theorem}
\begin{proof}
Set $\ideald \coloneqq \xi\idealc$ so that $\ideala = \ideald \OK $.
By the choice of $\xi $, %
Lemma~\ref{lem:105a} %
applies to give 
$\ideala\cap \OO=\ideald $.
It follows that when $\tau \in (\Or /\cond \OK )$
we have %
$\Pavs \subseteq \ideala \cap \tau \subseteq \ideala \cap \Or = \ideald $.
This together with the definition 
$\PP _{\ideald }= \PP _{\ideala }\cap \ideald  $
implies $\Pavs \subseteq \PP _{\ideald }$, 
which is equivalent to \eqref{eq:intersection-is-c'}
since $\PP_{\ideald }=\xi \PP_{\idealc }$.
Now \eqref{eq:theorem=counting_PPc} follows from \eqref{eq:intersection-is-c'} and Theorem~\ref{theorem=counting_PPa}.
\end{proof}
Note that we have $(\OO/\cond\OK)^{\times} \neq \varnothing$ because $1 \in \OO$.

\subsection{Proof of Theorem~\ref{th:constellations-in-ideals}}\label{subsection=proof_orders}

Now we are ready to establish Theorem~\ref{th:constellations-in-ideals}.

\begin{proof}[Proof of Theorem~$\ref{th:constellations-in-ideals}$]
Fix $f\in \NN$ with $f\OK\subseteq \Or$.
First, apply Proposition~\ref{theorem=reduction_order} to $(\Or,f,\idealc)$
to obtain $\xi\in K^{\times}$ and $\ideala\in\Ideals_K$
satisfying the conditions there. In particular $\ideala $ is coprime with $f$.
Fix a $\ZZ$-basis $\bv$ of $\ideala$. 
Set
$ 
    \ideald \coloneqq \xi \idealc
$,$
    \bz \coloneqq \xi \bw
$ and $
    A' \coloneqq \xi A
$.
    Then, $\ideald \subseteq \ideala$, and the restriction of $\|\cdot\|_{\infty,\bv}$ on $\ideald $ is bi-Lipschitz equivalent to $\|\cdot\|_{\infty,\bz }$. Apply Theorem~\ref{theorem=counting_PPc} with $\tau=1 \bmod  \cond \in (\OO/\cond \OK)^{\times}$. The assumption $\overline{d}_{\PP_{\idealc },\bw}(A)>0$ then implies that
\[
\limsup_{M\to\infty}\frac{\#(A' \cap \ideala(\bv,M))}{M^n(\log M)^{-1}}>0.
\]
By Theorem~\ref{theorem=subpseudorandom_PPa} and Lemma~\ref{lemma=stronglog}~\eqref{en:strlogsubset}, this set $A'$ is a member of $\logpseua$. 
We will verify that $A'$ fulfills condition~\eqref{en:counting_ideal_infinite_DD} of Theorem~\ref{theorem=package_infinite_DD}.
It suffices to prove that the larger set $\PP_{\ideald }$ satisfies this. 

Let $L\in \RR_{\geq 2}$
and 
$\beta\in \PP_{\ideald }\cap \OK(L)$.
Then, by definition, 
$\idealp_{\beta}\coloneqq\beta\ideala^{-1}$ is a prime ideal of $\OK $ 
having norm 
$\Nrm(\idealp_{\beta})\leq L/\Nrm(\ideala)\leq L$. 
Hence the map $\beta \mapsto \idealp _{\beta }$ gives an injection 
$\{\beta \OK : \beta\in \PP_{\ideald }\cap \OK(L)\}
\hookrightarrow 
\{ \idealp \in \SpecOK : \Nrm (\idealp )\le L \}$.
Theorem~\ref{theorem=Chebotarev}~\eqref{Landau} implies that there exists $\Delta>0$ such that for every $L\in \RR_{\geq 2}$, 
\[
\#\{\beta \OK : \beta\in \PP_{\ideald }\cap \OK(L)\}\leq \Delta\cdot \frac{L}{\log L}.
\]
This verifies condition~\eqref{en:counting_ideal_infinite_DD} of Theorem~\ref{theorem=package_infinite_DD}. 

Now Theorem~\ref{theorem=package_infinite_a_close} applies to $A'$, and we have a sequence $(\eT_l')_{l\in \NN}$ of finite subsets in $A'$ satisfying the conclusion of Theorem~\ref{theorem=package_infinite_a_close}.
Finally, via the bijection 
$A' \ni \beta\mapsto \xi^{-1}\beta \in A$, we transfer
$(\eT_l')_{l\in \NN}$ to a sequence 
$(\eT_l)_{l\in \NN}$
of subsets in $A$.
This sequence $(\eT_l)_{l\in \NN}$ fulfills all the conditions of Theorem~\ref{th:constellations-in-ideals}. 
Indeed, to verify~\eqref{en:saga_sukunaiyo}, note that by Proposition~\ref{theorem=reduction_order} for all $\alpha \in A$ we have
\[
	\Nelm(\alpha) = \frac{\Nrm(\idealc)}{\Nrm(\ideala)} \cdot \Nelm(\xi \alpha);
\]
the factor $\frac{\Nrm(\idealc)}{\Nrm(\ideala)}$ does not depend on $\alpha$.
Hence, given $\theta \in (0,1)_{\RR}$, $S\subseteq \idealc$ and $\eta>0$, we can replace $L\subseteq \NN$ with another finite subset of $\NN$ in an appropriate manner.
\end{proof}

\subsection{Proof of Theorem~\ref{theorem=quadraticformcloseprimes}}\label{subsec:proof-constellations-quadratic-forms}

In this subsection, we will establish Theorem~\ref{theorem=quadraticformcloseprimes}. %
Our arguments will actually show %
the following theorem on \emph{norm forms} in general. %

\begin{setting}\label{setting=normform}
Let $K$ be a number field of degree $n$, and  $\Or$ an order in $K$. 
Let $\idealc$ be an invertible fractional ideal. 
Let $\boldsymbol{u}$ be the standard basis of $\ZZ^n$,
$\bw=(\gamma_1,\gamma_2,\ldots,\gamma_n)$ be a $\ZZ$-basis of $\idealc$, 
and $\iota \colon \ZZ ^n \to  \idealc $ be the isomorphism of $\ZZ $-modules which sends $\boldsymbol{u}$ to $\bw$. 
Let $F=F_{(\Or,\idealc,\bw)}\colon \ZZ^n\to \ZZ$ be the \emph{norm form} associated with $(\Or,\idealc,\bw)$, meaning that
for all $(x_1,x_2,\ldots ,x_n)\in \ZZ^n$,
\[
F_{(\Or,\idealc,\bw)}(x_1,x_2,\ldots ,x_n)\coloneqq \frac{\Nelm(\gamma_1x_1+\gamma_2x_2+\cdots +\gamma_nx_n)}{\Nrm(\idealc)}.
\]
In other words, $F_{(\Or,\idealc,\bw)}=(\Nrm(\idealc))^{-1}(\Nelm\circ\iota)$.
Let $r_1$ be the number of real embeddings of $K$.%
\end{setting}

\begin{theorem}[Szemer\'edi-type  theorem on prime representations of norm forms]\label{theorem=normform}
Assume Setting~$\ref{setting=normform}$. If $r_1>0$, then take an arbitrary $\epsilon \in \{\pm 1\}$; if $r_1=0$, then set $\epsilon=+1$. Let $A\subseteq F^{-1}(\epsilon\PP)$
be a set with $\overline{d}_{F^{-1}(\epsilon\PP),\boldsymbol{u}}(A)>0$. Then, there exists a sequence of pairwise disjoint finite subsets $(\eT_l)_{l\in \NN}$ of $A$ which fulfills the following two conditions.
\begin{enumerate}[$(a)$]
  \item\label{en:kotonaru_prime_again_nf} For every $l\in \NN$, $F\mid_{\eT_l}\colon \eT_l\to \PP_{\QQ}$ is injective.
  \item\label{en:chikai_prime_again_nf} For every $\theta\in (0,1)_{\RR}$, for every finite subset $S \subseteq \ZZ^n$ and for every $\eta>0$, there exists a finite subset $L\subseteq \NN$ such that for every $l\in \NN\setminus L$, the following hold:
\begin{enumerate}
  \item[$(b1)$]\label{en:chikai1_nf} $\eT_l$ contains an $S$-constellation,
  \item[$(b2)$]\label{en:chikai2_nf} for every $p_1,p_2\in F(\eT_l)$,
\[
\frac{|p_2|}{|p_1|}\leq 1+\eta\cdot (\min\{|p| : p\in F(\eT_l)\})^{\frac{\theta-1}{n}}
\]
holds.
\end{enumerate}
\end{enumerate}
\end{theorem}

\begin{remark}\label{remark=normform}
If $r_1=0$, or in other words $K$ is totally imaginary,
the norm form $F$ is positive definite. 
Theorem~\ref{theorem=normform} asserts that this is the \emph{only} obstruction to obtaining our constellation theorem with sign $\epsilon =-1$. 
\end{remark}

To prove Theorem~\ref{theorem=normform}, take $\cond \in \NN $ with $\cond \OK \subseteq \Or $. Apply Proposition~\ref{theorem=reduction_order} to $(\Or,f,\idealc)$, and obtain $\xi\in K^{\times}$ and an ideal $\ideala\in \Ideals_K$ coprime with $f$. Recall the definition of the degree of a prime ideal from Subsection~\ref{subsection=pideal}.

\begin{definition}\label{definition=deg1}
Assume Setting~\ref{setting=normform}. 
Let $\epsilon\in \{\pm1\}$. 
Let $\xi \in K^\times $ and $\ideala = \xi \idealc \OK $  be as in Proposition~\ref{theorem=reduction_order}.
\begin{enumerate}[(1)]
\item\label{en:deg1_prime}
Define $|\Spec(\OK)|^1$ as the set of prime ideals of $\OK$ of degree $1$.
\item\label{en:sign}
Define $\PP_{\idealc}^{\epsilon}$ as the set of $\alpha\in \PP_{\idealc}$ such that $\Nelm (\alpha )$ has sign $\epsilon $.
\item\label{en:deg1_sign}
Define $\PP_{\idealc}^{1,\epsilon }$ by
\[
        \PP _{\idealc} ^{1,\epsilon } \coloneqq 
        \{ \alpha  \in \idealc : (\xi\alpha) \ideala^{-1} \in |\Spec(\OK)|^1\textrm{ and $\Nelm (\alpha )$ has sign $\epsilon $}          
        \} .
\]
\end{enumerate}
\end{definition}

\begin{lemma}\label{lemma=degree1sig}
 The isomorphism $\iota \colon \ZZ^n\to \idealc $ from Setting~$\ref{setting=normform}$
 sends the set $F^{-1}(\epsilon\PP)$ bijectively to $\PP _{\idealc} ^{1,\epsilon }$.
\end{lemma}
\begin{proof}
The composite map $F\circ \iota^{-1}$ is computed
    as
    \begin{equation}\label{eq:composite-map-F}
        \idealc\ni \alpha\ \mapsto  {\Nelm (\alpha )\Nrm (\idealc )^{-1}}  .
    \end{equation}
    This last value has sign $\mathrm{sgn} (\Nelm (\alpha ))$.
    Also, the right-hand side of \eqref{eq:composite-map-F}
    has absolute value $\Nrm (\alpha \idealc ^{-1}) $,
    which equals $\Nrm((\xi\alpha)\ideala^{-1})$ by Proposition~\ref{theorem=reduction_order}~\eqref{en:norm_huhen}.
    Note that for $\idealb\in \Ideals_K$, $\idealb\in|\Spec(\OK)|^1$ if and only if $\Nrm(\idealb)\in \PP$. This ends our proof.
\end{proof}

We consider 
the inclusions $\PP _{\idealc }^{1,\epsilon}\subseteq \PP _{\idealc }^{\epsilon}\subseteq \PP _{\idealc }$
multiplied by $\xi $.
The conclusions of the next lemma are equivalent to saying that $\underline{d}_{\PP _{\idealc },\bw}(\PP _{\idealc }^{\epsilon})>0$ and $\underline{d}_{\PP _{\idealc }^{\epsilon},\bw}(\PP _{\idealc }^{1,\epsilon})=1$.

\begin{lemma}\label{lemma=density_degree1}
Assume Setting~$\ref{setting=normform}$. Let $\xi$ and $\ideala$ be as in Proposition~$\ref{theorem=reduction_order}$. Let $\bv$ be a $\ZZ$-basis of $\ideala$. Let $\epsilon$ be an arbitrary sign $\epsilon \in \{\pm 1\}$ if $r_1>0$; otherwise, let $\epsilon=+1$. Set $\PP_1'\coloneqq\xi \PP _{\idealc }^{1,\epsilon }$, $\PP_2'\coloneqq\xi \PP _{\idealc }^{\epsilon }$ and $\PP_3'\coloneqq \xi \PP_{\idealc}$. 
Then, 
we have $\underline{d}_{\PP_3',\bv}(\PP_2')>0$ and $\underline{d}_{\PP_2',\bv}(\PP_1')=1$.
\end{lemma}

\begin{proof}
First, we prove $\underline{d}_{\PP_3',\bv}(\PP_2')>0$. If $r_1=0$, then the norm form $F$ is positive definite. Hence we have $\PP_2'=\PP_3'$; recall we have taken $\epsilon=+1$ in this case. Now we treat the remaining case of $r_1>0$. 

By \eqref{eq:intersection-is-c'} we know
$\PP_2'\supseteq \bigsqcup_{(\tau,s)}\PP_{\ideala;\tau,s}$, 
where $(\tau,s)$ runs over the set of pairs $(\tau,s)$ with $\tau\in (\Or /\cond \OK )^{\times}$ 
and 
$s=(s_1,s_2,\ldots, s_{r_1}) \in 
\{\pm 1\}^{r_1}$ 
satisfying $s_1s_2\cdots s_{r_1} = \epsilon$. Note that this set of pairs is non-empty because $1\in \OO$ and $r_1>0$. Corollary~\ref{cor:density-of-P-a-s} and Theorem~\ref{theorem=counting_PPa} apply. Hence, in this case, we also conclude that $\underline{d}_{\PP_3',\bv}(\PP_2')>0$. See also Remark~\ref{remark=Delta_detekuru}.

Secondly, we prove $\underline{d}_{\PP_2',\bv}(\PP_1')=1$. By Theorem~\ref{theorem=counting_PPa} and the argument above, we \havethat\ 
\[
\liminf_{M\to\infty}\frac{\#(\PP_2'\cap \ideala(\bv,M))}{M^n(\log M)^{-1}}>0.
\]
We claim that
\begin{equation}\label{eq:zerodens}
\limsup_{M\to\infty}\frac{\#((\PP_2'\setminus \PP_1')\cap \ideala(\bv,M))}{M^n(\log M)^{-1}}=0.
\end{equation}
To prove this, take $\beta \in (\PP_2'\setminus \PP_1')\cap \ideala(\bv,M)$. Then, there exists $\idealp_{\beta}\in |\Spec(\OK)|\setminus |\Spec(\OK)|^1$ such that $\beta\ideala^{-1}=\idealp_{\beta}$. Since $\idealp_{\beta}\in |\Spec(\OK)|\setminus |\Spec(\OK)|^1$, there exist $p_{\beta}\in \PP$ and $f_{\beta}\in \NN_{\geq 2}$ such that $\Nrm(\idealp_{\beta})=p_{\beta}^{f_{\beta}}$. Then we have
\[
p_{\beta}\leq \left(\frac{|\Nelm(\beta)|}{\Nrm(\ideala)}\right)^{\frac{1}{2}}.
\]
Here, if $M$ is sufficiently large depending on $\bv$, then the right-hand side of the inequality above does not exceed $M^{(2n+1)/4}$. By Lemma~\ref{lemma=squarefree}, there exist at most $nM^{(2n+1)/4}$ possibilities of $\idealp_{\beta}$. By Corollary~\ref{corollary=OKt_orbit_ideal}~\eqref{en:roughcounting_ideal}, we conclude that 
\[
\limsup_{M\to\infty}\frac{\#((\PP_2'\setminus \PP_1')\cap \ideala(\bv,M))}{M^\frac{2n+1}{4}(\log M)^{n}}<\infty;
\]
note that $k \leq n$, where $k$ is as in Corollary~\ref{corollary=OKt_orbit_ideal}.
Therefore, \eqref{eq:zerodens} holds. It follows that $\underline{d}_{\PP_2',\bv}(\PP_1')=1$. 
\end{proof}

\begin{proof}[Proof of Theorem~$\ref{theorem=normform}$]
We continue to use the same notation as in
Lemma~\ref{lemma=density_degree1}.
By Lemma~\ref{lemma=density_degree1}, we have $\underline{d}_{\PP_3',\bv}(\PP_1')>0$. By Lemma~\ref{lemma=degree1sig}, this implies that $\overline{d}_{\PP_3',\bv}(\xi\cdot \iota(A))>0$.  By Theorem~\ref{theorem=counting_PPc}, Theorem~\ref{theorem=fundamental_Omega_infinite}~\eqref{en:limsup} applies, and there exists an NL-compatible fundamental domain $\DD$ for the action $\OKt\curvearrowright \ideala\setminus \{0\}$ such that $\overline{d}_{\PP_3',\bv}(\DD\cap (\xi\cdot \iota(A)))>0$.  Define 
\[
\PP(A,\DD)\coloneqq \left\{\frac{|\Nelm(\beta)|}{\Nrm(\ideala)}: \beta\in \DD\cap(\xi\cdot\iota(A))\right\}.
\]
By Lemma~\ref{lemma=degree1sig}, this is  a subset of $\PP$. %
Next we will construct a new set $\tilde{A}'$ from $\DD\cap(\xi\cdot\iota(A))$. For each $p\in \PP(A,\DD)$, choose an arbitrary $\tilde{\beta}_p\in \DD\cap(\xi\cdot \iota(A))$ satisfying
\begin{equation}\label{eq:mijika}
\|\tilde{\beta}_{p}\|_{\infty,\bv}=\min \left\{\|\beta\|_{\infty,\bv}: \beta \in \DD\cap(\xi\cdot \iota(A)),\ \frac{|\Nelm(\beta)|}{\Nrm(\ideala)}=p\right\}.
\end{equation}
Then, define $\tilde{A}'$ by $\tilde{A}'\coloneqq \{\tilde{\beta}_p:p\in \PP(A,\DD)\}$. We claim that the map $\tilde{A}'\ni \tilde{\beta} \mapsto |\Nelm(\tilde{\beta})|\in \NN$ is injective, and that 
\begin{equation}\label{eq:A'_ooi}
\underline{d}_{\DD\cap (\xi\cdot \iota(A)),\bv}(\tilde{A}')\geq \frac{1}{n}
\end{equation}
holds true. Indeed, the former assertion holds by construction. To see the latter assertion, recall from Lemma~\ref{lemma=squarefree} that for every $p\in \PP$, the number of prime $p$-ideals in $\OK$ does not exceed $n$. Hence,  for every $p\in \PP(A,\DD)$, we have
\[
\#\left\{\beta \in \DD\cap(\xi\cdot\iota(A)) : \frac{|\Nelm(\beta)|}{\Nrm(\ideala)}=p \right\}\leq n.
\]
Therefore, by \eqref{eq:mijika}, we conclude \eqref{eq:A'_ooi}.

Set $A'\coloneqq \xi^{-1}\tilde{A}'$; this is a subset of $\iota(A)$.
We have $\overline{d}_{\PP_{\idealc},\bw}(A')>0$ by 
$\overline{d}_{\PP_3',\bv}(\DD\cap (\xi\cdot \iota(A)))>0$ and \eqref{eq:A'_ooi}.
Therefore, we can apply Theorem~\ref{th:constellations-in-ideals} to this $A'$
and obtain a sequence $(\eT_l')_{l\in \NN}$ of subsets in $A'$ satisfying the conditions there. Again by Lemma~\ref{lemma=degree1sig}, the sequence $(\eT_l)_{l\in \NN}\coloneqq (\iota^{-1}(\eT_l'))_{l\in \NN}$ of finite subsets in $\ZZ^n$ fulfills all the conditions of Theorem~\ref{theorem=normform}. 
\end{proof}

Finally, we establish Theorem~\ref{theorem=quadraticformcloseprimes}. In the proof below, we do not assume Setting~\ref{setting=normform}.

\begin{proof}[Proof of Theorem~$\ref{theorem=quadraticformcloseprimes}$]
Let $F(x,y)=ax^2+bxy+cy^2\in \ZZ[x,y]$ be a non-degenerate and primitive binary quadratic form with $a>0$. 
By Theorem~\ref{th:quadratic-forms-ideals}, there exist an order $\Or $ in a quadratic field 
$K$, an invertible fractional ideal $\idealc $ of  $\Or $,
a $\ZZ $-basis $\bw=(\gamma_1 ,\gamma_2)$ of $\idealc $ 
and a sign $\signF $
such that \eqref{eq:quadratic-forms-ideals} holds.

If $D_F>0$, then for each $\epsilon_0\in \{\pm 1\}$, we can apply Theorem~\ref{theorem=normform} with sign $\epsilon=\epsilon_0\signF$. This immediately proves the assertion. If $D_F<0$, then $a>0$ implies that $\signF=+1$. 
Apply Theorem~\ref{theorem=normform} with sign $\epsilon=+1$, and obtain the conclusion.
\end{proof}

%% file: appendix_quadratic.tex
\section{Binary quadratic forms and quadratic fields}

The goal of this appendix is to provide a standard proof of the classical fact, Theorem~\ref{th:quad-classical}, 
on the correspondence of binary quadratic forms with integral coefficients 
and
ideals in orders of quadratic fields. 
A part of Theorem \ref{th:quad-classical}, in the form of Theorem~\ref{th:quadratic-forms-ideals},  plays a  key role in the proof of Theorem~\ref{mtheorem=quadraticform}. 
We will need no more algebraic backgrounds from the main body of the paper than 
Proposition~\ref{lem:multiplicative-norm}, or more precisely \eqref{eq:norm-is-multiplicative}.

\subsection{Definitions}
Recall that by a \emph{binary quadratic form with integral coefficients}, we mean 
    a map $F\colon \mathbb Z^2 \to \mathbb Z $ which can be written as 
    $F(x,y)= ax^2+bxy+cy^2$ with $a,b,c\in \ZZ$ 
    with respect to the standard basis of $\ZZ^2$. In this appendix henceforth, we omit the modifier `with integral coefficients.'
Note that the property that $F$ is written in the form $ax^2+bxy+cy^2$ is preserved under a $\ZZ$-linear isomorphism $\ZZ^2\stackrel{\simeq}{\to} \ZZ^2$, in other words, a change of basis. Hence, the following concept of equivalence is natural.
Here we consider the so-called {\em proper} equivalence which respects the orientation, but we drop the adjective `proper' because we will never consider the improper one in this paper.

\begin{definition}
    Two binary quadratic forms $F,G\colon \mathbb Z^2\to\mathbb Z$ are said to be
     {\em equivalent}
    if there exists a $\ZZ $-linear isomorphism 
    $\iota \colon \mathbb Z^2 \xrightarrow{\simeq} \mathbb Z^2$ 
    preserving the orientation such that 
    $F=G\circ \iota $, namely such that the following diagram 
\[  
        \xymatrix{
            \mathbb Z^2 \ar[r]^F\ar[d]_{\iota }^{\simeq} & \ZZ  \\
            \mathbb Z^2 \ar[ru]_G
        }
        \]
commutes.
\end{definition}

Recall that the \emph{discriminant} of a binary quadratic form $F(x,y)=ax^2+bxy+cy^2$ is the integer $D_F\coloneqq b^2-4ac$. Note that $D_F\equiv 0\text{ or }1 \pmod 4$ always holds. The discriminant is preserved by equivalence of binary quadratic forms. Recall also that $F$ is said to be \emph{non-degenerate} if $D_F$ is not a perfect square, and that $F$ is said to be \emph{primitive} if $\mathrm{gcd}(a,b,c)=1$ holds.

\begin{definition}
    For an integer $D\in \ZZ $,
    we denote by $Q(D)$ the set of equivalence classes of non-degenerate primitive quadratic forms with discriminant $D$.
\end{definition}
The set $Q(D)$ is empty unless $D\equiv 0\text{ or }1 \pmod 4.$

For a square-free integer $d\in \mathbb Z\setminus\{1\}$, let us consider the quadratic field $\mathbb Q(\sqrt d)$ as a subfield of $\mathbb C$.
We choose a square root $\sqrt{d}\in\CC$ of $d$ in the following manner: if $d>0 $, we take $\sqrt{d}$ to be the positive real one, and if $d<0$, the one with positive imaginary part.  
\begin{definition}\label{definition=disc}
Let $d$ be a square-free integer not equal to $1$ and $K$ the quadratic field $K=\QQ(\sqrt{d})$.
\begin{enumerate}[(1)]
\item\label{en:dK} Define the discriminant $d_K$ of $K$ by
\[
d_K\coloneqq \begin{cases}d & \text{if} \ d\equiv 1\pmod{4}, \\ 4d& \text{if} \ d\equiv 2,3\pmod{4}.\end{cases}
\]
\item\label{en:Dorder} Define the discriminant $D$ of an order $\Or$ in $K$ by $D\coloneqq \# (\OK /\Or )^2\cdot d_K$. 
\end{enumerate}
\end{definition}
There exists a bijective correspondence between the pairs of a quadratic field $K=\mathbb Q(\sqrt d)$ and an order $\Or$ in $K$, and the discriminants $D$:
\begin{equation}\label{eq:orders-and-D}
\left\{(K,\Or) : \begin{array}{l} K \text{ a quadratic field}, \\ \Or \text{ an order in } K\end{array}  \right\}
\simeq
\left\{ D\in\ZZ : \begin{array}{l} D\equiv 0,1 \pmod{4},\\ \text{not a square}\end{array}\right\} .   
\end{equation} 
As it is fundamental in the proof of Theorem~\ref{th:quad-classical},
we present the explicit form of correspondence~\eqref{eq:orders-and-D}.
The one from the left-hand side to the right-hand side is given in Definition~\ref{definition=disc}~\eqref{en:Dorder}.
To explain the reverse correspondence, 
set $d$ to be the square-free part of $D$ (i.e., the square-free integer $d$ such that $D/d$ is a square) and define 
\begin{equation}\label{eq:KandO}
K\coloneqq\QQ (\sqrt{d}),\quad \textrm{and}\quad \Or\coloneqq \mathbb Z \oplus  \sqrt{\frac{D}{d_K}}\omega\mathbb Z,
\end{equation}
where $\omega$ is defined by 
\begin{equation}\label{eq:def-of-omega}
\omega \coloneqq\begin{cases}
\frac{1+\sqrt{d}}{2} & \text{if} \ d\equiv 1\pmod{4}, \\
\sqrt{d} & \text{if} \ d\equiv 2,3\pmod{4}.
\end{cases}
\end{equation}
These two maps are inverse to each other by the fact that every order of $K=\QQ (\sqrt{d})$ 
has the form $\Or =\mathbb Z \oplus \cond \omega\mathbb{Z}$
for some positive integer $\cond$; an easy consequence of the fact $1\in \Or $, see \cite[Lemma~7.2]{Cox89} or the discussion preceding \eqref{eq:Or-tau}. With this $f$, we have $\# (\OK /\Or ) = f$ and hence $D=f^2 d_K$.
\begin{definition}\label{def:signed-ideal}
Let $d$ be a square-free integer not equal to $1$.
Let $\Or $ be an order of the quadratic field $K=\QQ(\sqrt{d})$.
\begin{enumerate}[(1)]
\item\label{en:Cl+}  Define the commutative group ${\Cl }^+(\Or )$, called the \emph{narrow class group}, as follows.
Let $I_{\Or}$ be the group of invertible fractional ideals of $\Or$, where the group law is given by multiplication.
Then, we define ${\Cl }^+(\Or )$ to be the cokernel of the homomorphism from $K^{\times}$ to $I_{\Or} \times \{\pm1\}$ defined by $x\mapsto (x\Or, \sgn(N_{K/\mathbb{Q}}(x)))$.
\item\label{en:orient} Endow the $2$-dimensional $\QQ$-vector space $K$ with the orientation given by the basis $(1,\sqrt{d})$.
Let $\idealc \in I_\Or$ and $\epsilon\in\signs$.
We say that a $\ZZ$-basis $(\gamma_1,\gamma_2)$ of $\idealc$ \emph{has sign $\epsilon$} if the representing matrix of the inclusion map $\idealc \hookrightarrow K$ with respect to the bases $(\gamma_1,\gamma_2)$ and $(1,\sqrt{d})$ has determinant with sign $\epsilon$.
\end{enumerate}
\end{definition}

\subsection{The correspondence}
The following theorem is essentially due to Gauss, Dirichlet and Dedekind, and the goal of this appendix.
This specific statement 
is taken from \cite[Theorem 10]{Bhargava04}. 
For a binary quadratic form $F$ with discriminant $D_F=D$, write $[F]$ for its equivalence class in $Q(D)$. Similarly, for $(\idealc,\epsilon ) \in I_\Or \times \signs $, write $[(\idealc,\epsilon)]$ for its equivalence class in ${\Cl }^+(\OO )$.
\begin{theorem}\label{th:quad-classical}
Let $\Or$ be an order of a quadratic field $K$ and let $D$ be its discriminant.
Then the following maps are well-defined, and one is the inverse map to the other.
In particular, they provide a bijective correspondence
\[
\Cl^{+}(\Or)\simeq Q(D).
\]
\begin{enumerate}[$(1)$]
   \item\label{en:CltoQ}$($From ideals to quadratic forms$)$ For $[(\idealc,\epsilon)]\in {\Cl }^+(\OO )$ with $(\idealc,\epsilon ) \in I_\Or \times \signs $, choose a $\ZZ $-basis $(\gamma_1 ,\gamma_2)$ of $\idealc$ which has sign $\epsilon$.
   Then, define the corresponding class $[F]$ of quadratic forms  by setting
\begin{equation}\label{eq:from-ideal-to-quad}
F(x,y)\coloneqq\frac{N_{K/\QQ}(\gamma_1x+\gamma_2y)}{\epsilon\Nrm(\idealc)}\quad\text{ for}\quad (x,y)\in\ZZ^2.
\end{equation} 
\item\label{en:QtoCl}$($From quadratic forms to ideals$)$ For $[F]\in Q(D)$, consider a representative $F$ of the form $F(x,y)=ax^2+bxy+cy^2$.
Let $d$ be the square-free part of $D$ so that $K=\QQ(\sqrt{d})$.
Set $f\coloneqq \sqrt{D/d_K}$ and $\tau  \coloneqq \frac{-b-f\sqrt{d_K}}{2} \in \OK $. Then define the corresponding element $[(\idealc,\epsilon)]\in {\Cl }^+(\OO )$ as follows: set
 \begin{equation}\label{eq:from-quad-to-ideal}
     \idealc \coloneqq a \mathbb Z \oplus \tau  \mathbb Z \quad \textrm{and} \quad        \epsilon \coloneqq \mathrm{sgn}(a )  .  
 \end{equation}
\end{enumerate}
\end{theorem}

We will prove Theorem~\ref{th:quad-classical} in Subsections~\ref{subsection=correspondence1}, \ref{subsection=correspondence2} and \ref{subsection=endoftheproof}. 
In what follows, we use the following setting. 
\begin{setting}\label{setting=A}
Let $D$ be an integer not a square with $D \equiv0,1\pmod{4}$.
Let $d$ be the square-free part of $D$ and $K\coloneqq\QQ(\sqrt{d})$.
Set $f\coloneqq\sqrt{D/d_K}$ and $\Or\coloneqq \ZZ \oplus f\omega \ZZ $ as in \eqref{eq:KandO}. 
\end{setting}

\subsection{Well-definedness of correspondence~\eqref{en:CltoQ}}\label{subsection=correspondence1}
In this subsection, we prove that correspondence~\eqref{en:CltoQ} in Theorem~\ref{th:quad-classical} is well-defined. 
By \eqref{eq:norm-is-multiplicative}, the absolute value of the right-hand side of \eqref{eq:from-ideal-to-quad} equals $\Nrm((\gamma_1x+\gamma_2y)\idealc^{-1})$.
Since $\gamma_1x+\gamma_2y\in\idealc$, we have $(\gamma_1x+\gamma_2y)\idealc^{-1}\subseteq \Or$.
Hence, for every $(x,y)\in\ZZ^2$, the value $F(x,y)$ is an integer.
This implies that $F$ has integer coefficients.

The equivalence class of $F(x,y)$ clearly does not depend on the auxiliary choice of the basis $(\gamma_1 ,\gamma_2)$. 
We claim that it depends only on the class of 
$(\idealc ,\epsilon )$ in $\ClOr$.
To see this, note that for every element $\xi\in K^{\times}$, the following commutative diagram
\[
\xymatrix@C=2cm{
    &\idealc \ar[dd]_{\xi \times } \ar[dr]^{\frac{N_{K/\mathbb Q}(-) }{\epsilon \Nrm ( \idealc)}} & \\
   \ZZ ^2 \ar[ur]^{(\gamma_1 ,\gamma_2 )} \ar[dr]_{(\xi \gamma_1 ,\ \xi \gamma_2 )}&& \ZZ \\
    &\xi \idealc \ar[ur]_{\frac{N_{K/\mathbb Q}(-) }{\epsilon\cdot \mathrm{sgn}(\Nelm(\xi))\cdot  \Nrm (\xi \idealc)} }&
}
\]
commutes. Indeed, the commutativity of the left-hand triangle is obvious.
The commutativity of the right-hand triangle follows from the fact that the norm map
$N_{K/\mathbb Q}(-)$ is multiplicative and from \eqref{eq:norm-is-multiplicative}.
Note that the change of the basis from $(\gamma_1,\gamma_2 )$ to $(\xi\gamma_1,\xi\gamma_2)$ changes the sign by the factor $\sgn(\Nelm(\xi))$,
so that the lower composition map in the diagram is a quadratic form obtained from the pair $(\xi \idealc,\sgn(\Nelm(\xi))\epsilon)$.

It remains to check that the discriminant $D_F$ of the quadratic form $F(x,y)$ is equal to the discriminant $D$ of $\Or $, and that $F$ is primitive. First, we will prove that $D_F=D$. For this, we may assume $\idealc \subseteq \Or $. Indeed,  replace $\idealc $ with an appropriate $\xi \idealc $; the argument above justifies this process.
Let us denote 
the unique non-trivial automorphism of $K=\QQ(\sqrt{d})$ over $\QQ$
by $\alpha \mapsto \overline{\alpha} $.
Then for every $(x,y)\in \ZZ^2$, we have 
$N_{K/\mathbb Q} (\gamma_1 x+\gamma_2 y)=(\gamma_1 x +\gamma_2 y)(\overline{\gamma_1} x +\overline{\gamma_2} y) $. Hence, 
\[
F(x,y)= \frac{\gamma_1\overline{\gamma_1} x^2+(\gamma_1\overline{\gamma_2} +\overline{\gamma_1}\gamma_2)xy+\gamma_2\overline{\gamma_2} y^2}{\epsilon \Nrm (\idealc)} .
\] 
It then follows that $D_F=\frac{ (\gamma_1\overline{\gamma_2} -\overline{\gamma_1}\gamma_2 )^2 }{\Nrm (\idealc)^2 }$. The numerator is the square of
$\mathrm{det}\begin{pmatrix}\gamma_1 & \gamma_2 \\ \overline{\gamma_1} & \overline{\gamma_2} \end{pmatrix} $. 
By the assumption $\idealc \subseteq \Or $, there exists a unique $2$-by-$2$ integer matrix $T$ satisfying 
$( \gamma_1\ \gamma_2 )= (1\ f\omega ) T$.
It follows $\begin{pmatrix} \gamma_1 & \gamma_2 \\ \overline{\gamma_1} &\overline{\gamma_2}  
\end{pmatrix}= \begin{pmatrix} 1 & f\omega  \\ 1 & f\overline{\omega}  
\end{pmatrix} T $.
By the definition of $T$, we have $|\det (T)| = \Nrm (\idealc )$. Hence, we obtain that
\[ \mathrm{det}\begin{pmatrix}\gamma_1 & \gamma_2 \\ \overline{\gamma_1} & \overline{\gamma_2} \end{pmatrix} ^2 
= 
\mathrm{det}\begin{pmatrix}1 & f\omega \\ 1& f \overline\omega \end{pmatrix} ^2
\Nrm (\idealc ) ^2. \] 
A direct calculation shows $\det \begin{pmatrix}1 & f\omega \\ 1& f \overline\omega \end{pmatrix} ^2 = D$. Thus, we conclude that $D_F=D$.

Secondly, we will prove that $F(x,y)$ is primitive.
By Corollary \ref{cor:coprime}, we may assume that $\idealc $ satisfies
$\idealc +f \Or = \Or $; in this case $\Nrm (\idealc )$ is 
prime to $f$ because $f$ is invertible in the ring $\Or / \idealc$.
The integer $\Nrm (\idealc )$ annihilates the abelian group $\Or /\idealc $.
It follows that $\Nrm (\idealc )=\Nrm (\idealc ) \cdot 1 \in\idealc $.
Take the unique $(x,y)\in \ZZ ^2$ with $\gamma_1 x+\gamma_2 y = \Nrm (\idealc ) $.
For this $(x,y)$, we have $F(x,y) = \frac{\Nelm ( \Nrm (\idealc ))}{\epsilon \Nrm (\idealc )} = \epsilon \Nrm (\idealc )$. Now suppose that $\mathrm{gcd}(a,b,c)\ne 1$. Take a rational prime number $p$ dividing $\mathrm{gcd}(a,b,c)$. 
Then, $p$ divides every absolute value of $F$; in particular, it divides $\Nrm (\idealc )$. Since $\Nrm (\idealc )$ is now 
prime to $f$, it follows that $p$ is 
prime to $f$. 
Also $p^2$ divides $b^2-4ac = D = f^2 d_K$.
It follows that $p^2$ divides $d_K$ ($= d $ or $4d$).
If $p$ is odd, this contradicts the condition that $d$ is square-free. 
We can also deduce a contradiction for the case of $p=2$ by considering the reduction modulo $16$. 
More precisely, since $p^2=4$ divides $d_K$, it must be the case that $d_K=4d$ and $d\equiv 2,3 \pmod 4$.
It follows that $d_K\equiv 8,12 \pmod{16}$.
Since $f$ is now odd we have $f^2\equiv 1,9 \pmod {16}$.
It follows that $f^2d_K\equiv 8,12 \pmod {16}$.
On the one hand, since $a,b,c$ are even, we have 
$b^2-4ac \equiv 0,4 \pmod{16}$.
This contradicts $b^2-4ac = f^2 d_K$.
Therefore, $\mathrm{gcd}(a,b,c)=1$, as desired.

This ends the proof of well-definedness of correspondence~\eqref{en:CltoQ}.
\subsection{On well-definedness of correspondence~\eqref{en:QtoCl}}\label{subsection=correspondence2}

This subsection  is devoted to the proof of well-definedness of correspondence~\eqref{en:QtoCl} in Theorem~\ref{th:quad-classical}. 
Strictly speaking, we will prove that the map $F\mapsto [(\idealc,\epsilon)]$ from the set of all primitive binary quadratic forms $\ZZ^2\to \ZZ$ with discriminant $D$ to $\ClOr$ is well-defined, where $(\idealc,\epsilon)$ is defined in \eqref{eq:from-quad-to-ideal}. First, note that $a\neq 0$, since $D=b^2-4ac $ is not a perfect square. Hence $\mathrm{sgn}(a)$ in \eqref{eq:from-quad-to-ideal} does not cause a problem. 
We  claim that $\Or = \ZZ \oplus \tau \ZZ$
as an abelian group. Indeed, observe that $f\omega + \tau \equiv \frac{fd_K -b}{2}\pmod{\ZZ}$.
From $f^2d_K =D= b^2-4ac $, observe also that $fd_K $ and $b$ must have the same parity. Hence, $f\omega =  -\tau  $ in the quotient group $\OK /\ZZ $, and we conclude that
\begin{equation}\label{eq:Or-tau}
    \Or = \ZZ \oplus f\omega \ZZ = \ZZ \oplus \tau \ZZ .
\end{equation}
Since $\tau  ^2 +b \tau  + ac =0 $, the subgroup $\idealc = a\ZZ \oplus \tau \ZZ \subseteq \Or$ is in fact an ideal of $\Or$.
We will moreover check that $\idealc$ is an invertible ideal.
Consider the conjugate $\overline\idealc = a\ZZ \oplus \overline\tau \ZZ \subseteq \Or $
and take the product,
\[
\idealc \overline\idealc = (a\Or+\tau\Or)\cdot(a\Or+\overline{\tau}\Or)=a^2\Or+a\tau\Or+a\overline{\tau}\Or+\tau\overline{\tau}\Or.
\]
Since $\tau ^2+b\tau +ac =0$, we have $\tau\overline\tau =ac $. 
Since $\tau +\overline\tau = -b$ and $F$ is primitive, we moreover \obtainthat\ 
$\idealc \overline\idealc= a\Or$.
This shows that $\idealc $ is invertible with the inverse $a^{-1}\overline\idealc$.

The arguments above in this subsection show that the map $F\mapsto [(\idealc,\epsilon)]$ is well-defined. To verify that correspondence~\eqref{en:QtoCl} is well-defined, it remains to check that the class of $(\idealc ,\epsilon )$ in $\ClOr $ is invariant under changes of  the representative $F$ of the equivalence class $[F]\in Q(D)$. 
This will be done in 
Subsection~\ref{subsection=endoftheproof}.

\subsection{End of the proof of Theorem~\ref{th:quad-classical}}\label{subsection=endoftheproof}

In Subsections~\ref{subsection=correspondence1} and \ref{subsection=correspondence2}, we have checked that the maps in the following diagram 
\begin{equation}\label{eq:commutative-triangle}
        \raisebox{1cm}{$\xymatrix{  & \{\text{primitive quadratic forms $F\colon \ZZ ^2\to \ZZ $ with }D_F=D \} 
        \ar[dl]_{\eqref{eq:from-quad-to-ideal}}
        \ar@{->>}[d]^{\text{quotient map} } \\
        \ClOr \ar[r]_{\eqref{eq:from-ideal-to-quad}} & Q(D)
        }$}
\end{equation}
are all well-defined.
To establish Theorem \ref{th:quad-classical}, it suffices to show that the diagram is commutative,
and that the horizontal map is injective.
That will also establish that the map $Q(D)\overset{\eqref{eq:from-quad-to-ideal}}{\longrightarrow}\ClOr $ is well defined, thus completing the arguments in Subsection~\ref{subsection=correspondence2}. In this subsection, we will prove the two assertions above.

First, we will show that the diagram \eqref{eq:commutative-triangle} is commutative. Take a primitive binary quadratic form $F(x,y)=ax^2+bxy+cy^2$ with discriminant $D$.
Let $ \tau$ and $\idealc$ as in Theorem \ref{th:quad-classical}~\eqref{en:QtoCl}. The basis $(a,-\tau  )$ of $\idealc$ has sign $\epsilon \coloneqq\mathrm{sgn}(a) $.
Therefore the quadratic form associated with this pair $(\idealc ,\epsilon )$ and the basis is the map
\[
(x,y)\mapsto\frac{N_{K/\QQ}(ax-\tau y)}{\epsilon\Nrm(\idealc)}.
\] 
By \eqref{eq:Or-tau}, the denominator is $\epsilon \cdot |a| = a$. 
The numerator is 
\[
N_{K/\mathbb Q} (a x -\tau  y) = (ax -\tau  y) (ax -\overline \tau  y) = a^2x^2 -a(\tau  +\overline \tau  ) xy +\tau \overline{\tau  } y^2 . 
\]
Since $\tau  +\overline \tau  = -b $ and $\tau \overline \tau  = ac $, we conclude that 
\[ 
N_{K/\mathbb Q} (a x -\tau  y)= a^2x^2 +abxy +ac y^2 = a\cdot F(x,y).
\] 
This proves the commutativity of \eqref{eq:commutative-triangle}.

In the final part of the proof of Theorem~\ref{th:quad-classical}, we will show that the map $\ClOr \overset{\eqref{eq:from-ideal-to-quad}}{\longrightarrow}Q(D)$ is injective. Suppose that two pairs $(\mathfrak{c_1},\epsilon _1)$ and $(\mathfrak{c_2},\epsilon _2)$ of invertible fractional ideals and signs give equivalent quadratic forms.
Choose appropriate $\mathbb Z$-bases $(\gamma_1,\gamma _2)$ and $(\gamma_1',\gamma_2')$ respectively so that we obtain the following identity for two quadratic forms: 
\begin{equation}\label{eq:we-get-the-same-quadratic-form}
    \frac{N_{K/\mathbb Q} (\gamma_1x + \gamma_2 y ) }{\epsilon _1 \Nrm (\mathfrak{c_1})}
    =
    \frac{N_{K/\mathbb Q} (\gamma_1'x + \gamma _2' y ) }{\epsilon _2 \Nrm (\mathfrak{c_2})} .
\end{equation} 
Recall that for all $(x,y)\in \ZZ^2$, we have $N_{K/\mathbb Q}(\gamma_1 x+\gamma_2 y)=(\gamma_1 x+\gamma_2 y )(\overline{\gamma_1} x+\overline{\gamma_2} y)$. Similarly, $N_{K/\mathbb Q}(\gamma_1' x+\gamma_2' y)=(\gamma_1' x+\gamma_2' y )(\overline{\gamma_1'} x+\overline{\gamma_2'} y)$ holds.
In each of the two equalities above, the right-hand side makes sense even for $x,y \in K=\QQ (\sqrt d)$.
Set $y=1$. The values of $x\in K$ satisfying $(\gamma_1 x+\gamma_2)(\overline{\gamma_1} x+\overline{\gamma_2} )= 0$ are $-\gamma_2 / \gamma_1 $ and $-\overline{\gamma_1}/ \overline{\gamma_2}$. A similar fact holds for $x\mapsto (\gamma_1' x+\gamma_2')(\overline{\gamma_1'} x+\overline{\gamma_2'} )$. 
From \eqref{eq:we-get-the-same-quadratic-form}, either $\gamma_2/\gamma_1 = \gamma_2'/\gamma_1'$ or $\gamma_2/\gamma_1 = \overline{\gamma_2'}/\overline{\gamma_1'}$  holds true.
In other words, there exists  $\xi \in K^{\times}$ such that the following equality holds in $K^2$: 
\begin{equation}\label{eq:equality-in-K^2}
 (\gamma _1,\gamma _2)=\xi (\gamma _1',\gamma _2')\quad  \text{ or } \quad (\gamma _1,\gamma _2)=\xi (\overline{\gamma _1'},\overline{\gamma _2'}). 
\end{equation}
In either case,
if we substitute it into \eqref{eq:we-get-the-same-quadratic-form}, we \obtainthat\ 
\[ 
\frac{\xi \overline\xi (\gamma _1'x+\gamma _2' y)(\overline{\gamma _1'}x+\overline{\gamma _2'}y)}{\epsilon _1 \Nrm (\mathfrak{c_1})} 
=
\frac{ (\gamma _1'x+\gamma _2' y)(\overline{\gamma _1'}x+\overline{\gamma _2'}y)}{\epsilon _2 \Nrm (\mathfrak{c_2})}
\]
as binary quadratic forms.
Therefore, we have  $\xi\overline\xi = \epsilon _1 \Nrm (\mathfrak{c_1})/(\epsilon _2 \Nrm (\mathfrak{c_2}))$.
In particular, $\Nelm (\xi )= \xi \overline\xi $ has sign $\epsilon _1 / \epsilon _2$. 
It then follows that the basis $\xi(\gamma _1',\gamma _2')$ of $\xi \idealc_2$ has sign
$\epsilon_1$ and $\xi(\overline{\gamma _1'},\overline{\gamma _2'})$ of $\xi \overline{\idealc_2}$ has sign $-\epsilon _1$.
Since the basis $ (\gamma _1,\gamma _2)$ has sign $\epsilon _1$,
we conclude that in \eqref{eq:equality-in-K^2}, only the first case can hold. This also implies that $\idealc_1=\xi \idealc_2$ as ideals. Since $\mathrm{sgn}(\Nelm(\xi))=\epsilon_1/\epsilon_2 $, we obtain the equality $(\idealc_1,\epsilon_1)=(\xi \Or, \sgn(\Nelm(\xi)))\cdot(\idealc_2,\epsilon _2)$ in $I_\Or \times \signs $. This proves the desired injectivity.

This completes the proof of Theorem \ref{th:quad-classical}.

%% file: address.tex
\bigskip

{\sc Mathematical Institute, Tohoku University, Sendai, 980-8578, Japan}

{\it E-mail address}: {\tt kaiw@tohoku.ac.jp}

\bigskip

{\sc Mathematical Institute, Tohoku University, Sendai, 980-8578, Japan}

{\it E-mail address}: {\tt m.masato.mimura.m@tohoku.ac.jp}

\bigskip

{\sc Graduate School of Information Sciences, Tohoku University, Sendai, 980-8579, Japan}

{\it E-mail address}: {\tt munemasa@math.is.tohoku.ac.jp}

\bigskip

{\sc Department of Mathematical Sciences, Aoyama Gakuin University, Sagamihara, 252-5258, Japan}

{\it E-mail address}: {\tt seki@math.aoyama.ac.jp}

\bigskip

{\sc Graduate School of Information Sciences, Tohoku University, Sendai, 980-8579, Japan}

{\it E-mail address}: {\tt kiyoto.yosino.r2@dc.tohoku.ac.jp}

%% file: PECNF_E.bbl
\providecommand{\bysame}{\leavevmode\hbox to3em{\hrulefill}\thinspace}
\providecommand{\MR}{\relax\ifhmode\unskip\space\fi MR }
\providecommand{\MRhref}[2]{%
  \href{http://www.ams.org/mathscinet-getitem?mr=#1}{#2}
}
\providecommand{\href}[2]{#2}
\begin{thebibliography}{RSTT06}

\bibitem[AM16]{Atiyah-Macdonald}
M.~F. Atiyah and I.~G. Macdonald, \emph{Introduction to commutative algebra},
  economy ed., Addison-Wesley Series in Mathematics, Westview Press, Boulder,
  CO, 2016, For the 1969 original see [MR0242802].

\bibitem[Bha04]{Bhargava04}
M.~Bhargava, \emph{Higher composition laws. {I}. {A} new view on {G}auss
  composition, and quadratic generalizations}, Ann. of Math. (2) \textbf{159}
  (2004), no.~1, 217--250.

\bibitem[BHP01]{Baker-Harman-Pintz2001}
R.~C. Baker, G.~Harman, and J.~Pintz, \emph{The difference between consecutive
  primes. {II}}, Proc. London Math. Soc. (3) \textbf{83} (2001), no.~3,
  532--562.

\bibitem[BS20]{Bloom-Sisask}
T.~F. Bloom and O.~Sisask, \emph{Breaking the logarithmic barrier in {R}oth's
  theorem on arithmetic progressions}, preprint, arXiv:2007.03528 (2020).

\bibitem[CFZ14]{Conlon-Fox-Zhao14}
D.~Conlon, J.~Fox, and Y.~Zhao, \emph{The {G}reen-{T}ao theorem: an
  exposition}, EMS Surv. Math. Sci. \textbf{1} (2014), no.~2, 249--282.

\bibitem[CFZ15]{Conlon-Fox-Zhao15}
\bysame, \emph{A relative {S}zemer\'{e}di theorem}, Geom. Funct. Anal.
  \textbf{25} (2015), no.~3, 733--762.

\bibitem[CMT18]{Cook-Magyar-Titichetrakun18}
B.~Cook, \'{A}. Magyar, and T.~Titichetrakun, \emph{A multidimensional
  {S}zemer\'{e}di theorem in the primes via combinatorics}, Ann. Comb.
  \textbf{22} (2018), no.~4, 711--768.

\bibitem[Cox13]{Cox89}
D.~A. Cox, \emph{Primes of the form {$x^2 + ny^2$}}, second ed., Pure and
  Applied Mathematics (Hoboken), John Wiley \& Sons, Inc., Hoboken, NJ, 2013,
  Fermat, class field theory, and complex multiplication.

\bibitem[EF19]{Elsholtz-Frei19}
C.~Elsholtz and C.~Frei, \emph{Arithmetic progressions in binary quadratic
  forms and norm forms}, Bull. Lond. Math. Soc. \textbf{51} (2019), no.~4,
  595--602.

\bibitem[FK78]{Furstenberg-Katznelson78}
H.~Furstenberg and Y.~Katznelson, \emph{An ergodic {S}zemer\'{e}di theorem for
  commuting transformations}, J. Analyse Math. \textbf{34} (1978), 275--291
  (1979).

\bibitem[FZ15]{Fox-Zhao15}
J.~Fox and Y.~Zhao, \emph{A short proof of the multidimensional {S}zemer\'{e}di
  theorem in the primes}, Amer. J. Math. \textbf{137} (2015), no.~4,
  1139--1145.

\bibitem[Gow07]{Gowers07}
W.~T. Gowers, \emph{Hypergraph regularity and the multidimensional
  {S}zemer\'{e}di theorem}, Ann. of Math. (2) \textbf{166} (2007), no.~3,
  897--946.

\bibitem[GT08]{Green-Tao08}
B.~Green and T.~Tao, \emph{The primes contain arbitrarily long arithmetic
  progressions}, Ann. of Math. (2) \textbf{167} (2008), no.~2, 481--547.

\bibitem[Hec81]{Hecke}
E.~Hecke, \emph{Lectures on the theory of algebraic numbers}, Graduate Texts in
  Mathematics, vol.~77, Springer-Verlag, New York-Berlin, 1981, Translated from
  the German by George U. Brauer, Jay R. Goldman and R. Kotzen.

\bibitem[Hei67]{Cassels-Froehlich}
H.~Heilbronn, \emph{Zeta-functions and {$L$}-functions}, Algebraic {N}umber
  {T}heory ({P}roc. {I}nstructional {C}onf., {B}righton, 1965), Thompson,
  Washington, D.C., 1967, pp.~204--230.

\bibitem[HW08]{Hardy-Wright}
G.~H. Hardy and E.~M. Wright, \emph{An introduction to the theory of numbers},
  sixth ed., Oxford University Press, Oxford, 2008, Revised by D. R.
  Heath-Brown and J. H. Silverman, With a foreword by Andrew Wiles.

\bibitem[KRE20]{Kuperberg-Rodgers-RodittyGershon20}
V.~Kuperberg, B.~Rodgers, and Roditty-Gershon E., \emph{Sums of singular series
  and primes in short intervals in algebraic number fields}, preprint,
  arXiv:2001.09513 (2020).

\bibitem[Lan53]{Landau}
E.~Landau, \emph{Handbuch der {L}ehre von der {V}erteilung der {P}rimzahlen. 2
  {B}\"{a}nde}, Chelsea Publishing Co., New York, 1953, 2d ed, With an appendix
  by Paul T. Bateman.

\bibitem[May20]{maynard_2020}
J.~Maynard, \emph{Primes represented by incomplete norm forms}, Forum of
  Mathematics, Pi \textbf{8} (2020), e3.

\bibitem[Mit56]{Mitsui56}
T.~Mitsui, \emph{Generalized prime number theorem}, Jpn. J. Math. \textbf{26}
  (1956), 1--42.

\bibitem[Neu99]{Neukirch}
J.~Neukirch, \emph{Algebraic number theory}, Grundlehren der Mathematischen
  Wissenschaften [Fundamental Principles of Mathematical Sciences], vol. 322,
  Springer-Verlag, Berlin, 1999, Translated from the 1992 German original and
  with a note by Norbert Schappacher, With a foreword by G. Harder.

\bibitem[NRS06]{Nagle-Rodl-Schacht06}
B.~Nagle, V.~R\"{o}dl, and M.~Schacht, \emph{The counting lemma for regular
  {$k$}-uniform hypergraphs}, Random Structures Algorithms \textbf{28} (2006),
  no.~2, 113--179.

\bibitem[RS06]{Rodl-Skokan04}
V.~R\"{o}dl and J.~Skokan, \emph{Applications of the regularity lemma for
  uniform hypergraphs}, Random Structures Algorithms \textbf{28} (2006), no.~2,
  180--194.

\bibitem[RS07a]{Rodl-Schacht072}
V.~R\"{o}dl and M.~Schacht, \emph{Regular partitions of hypergraphs: counting
  lemmas}, Combin. Probab. Comput. \textbf{16} (2007), no.~6, 887--901.

\bibitem[RS07b]{Rodl-Schacht071}
\bysame, \emph{Regular partitions of hypergraphs: regularity lemmas}, Combin.
  Probab. Comput. \textbf{16} (2007), no.~6, 833--885.

\bibitem[RSTT06]{Rodl-Schacht-Tengan-Tokushige06}
V.~R\"{o}dl, M.~Schacht, E.~Tengan, and N.~Tokushige, \emph{Density theorems
  and extremal hypergraph problems}, Israel J. Math. \textbf{152} (2006),
  371--380.

\bibitem[RW19]{RW}
L.~Rimani\'{c} and J.~Wolf, \emph{Szemer\'{e}di's theorem in the primes}, Proc.
  Edinb. Math. Soc. (2) \textbf{62} (2019), no.~2, 443--457.

\bibitem[Sol03]{Solymosi03}
J.~Solymosi, \emph{Note on a generalization of {R}oth's theorem}, Discrete and
  computational geometry, Algorithms Combin., vol.~25, Springer, Berlin, 2003,
  pp.~825--827.

\bibitem[Sta67]{Stark67}
H.~M. Stark, \emph{A complete determination of the complex quadratic fields of
  class-number one}, Michigan Math. J. \textbf{14} (1967), 1--27.

\bibitem[Tao06]{Tao06Gaussian}
T.~Tao, \emph{The {G}aussian primes contain arbitrarily shaped constellations},
  J. Anal. Math. \textbf{99} (2006), 109--176.

\bibitem[TZ08]{Tao-Ziegler08}
T.~Tao and T.~Ziegler, \emph{The primes contain arbitrarily long polynomial
  progressions}, Acta Math. \textbf{201} (2008), no.~2, 213--305.

\bibitem[TZ15]{Tao-Ziegler15}
\bysame, \emph{A multi-dimensional {S}zemer\'{e}di theorem for the primes via a
  correspondence principle}, Israel J. Math. \textbf{207} (2015), no.~1,
  203--228.

\bibitem[Var59]{Varnavides59}
P.~Varnavides, \emph{On certain sets of positive density}, J. London Math. Soc.
  \textbf{34} (1959), 358--360.

\end{thebibliography}
